\documentclass{amsart}
\usepackage{pdfsync}
\usepackage{amssymb,amsmath}
\usepackage{graphicx}
\usepackage{xcolor}
\usepackage{hyperref}
\usepackage{subfigure}
\usepackage{ upgreek }
\usepackage{tikz}
\usepackage{comment}
\usepackage[all]{xy}
\usetikzlibrary{matrix}
\usepackage{appendix}
\graphicspath{{../}}
\newcommand{\bC}{{\mathbb C}}

\newcommand{\bR}{{\mathbb R}}
\newcommand{\bZ}{{\mathbb Z}}

\newcommand{\bT}{{\mathbb T}}

\newcommand{\cA}{\mathcal A}

\newcommand{\cC}{\mathcal C}

\newcommand{\cN}{\mathcal N}
\newcommand{\cM}{\mathcal M}
\newcommand{\cO}{\mathcal O}

\newcommand{\cL}{\mathcal L}
\newcommand{\cR}{\mathcal R}

\newcommand{\cH}{\mathcal H}

\newcommand{\cK}{\mathcal K}

\newcommand{\cJ}{\mathcal J}

\newcommand{\fh}{\mathfrak h}
\newcommand{\cSH}{\mathcal{SH}}
\newcommand{\fH}{\mathfrak{H}}

\DeclareMathOperator{\Sec}{Sec}
\DeclareMathOperator{\inj}{inj}
\DeclareMathOperator{\Area}{Area}

\DeclareMathOperator{\Hess}{Hess}
\DeclareMathOperator{\Per}{Per}

\newtheorem{lemma}{Lemma}[section]
\newtheorem{theorem}[lemma]{Theorem}
\newtheorem{proposition}[lemma]{Proposition}
\newtheorem{remark}[lemma]{Remark}
\newtheorem{example}[lemma]{Example}
\newtheorem{definition}[lemma]{Definition}
\newtheorem{ques}{Question}
\newtheorem{corollary}[lemma]{Corollary}

\newtheorem{obs}{Observation}
\newtheorem{conjecture}{Conjecture}
\begin{document}

\title{Locality of relative symplectic cohomology for complete embeddings}

\author{Yoel Groman 
 and Umut Varolgunes
}

\address{ Yoel Groman,
Hebrew University of Jerusalem,
Mathematics Department}
\address{Umut Varolgunes, Department of Mathematics, Bo\u{g}azi\c{c}i University}

\begin{abstract}
A complete embedding is a symplectic embedding $\iota:Y\to M$ of a geometrically bounded symplectic manifold $Y$  into another geometrically bounded symplectic manifold $M$ of the same dimension. When $Y$ satisfies an additional finiteness hypothesis, we prove that the truncated relative symplectic cohomology of a compact subset $K$ inside $Y$ is naturally isomorphic to that of its image $\iota(K)$ inside $M$. Under the assumption that the torsion exponents of $K$ are bounded we deduce the same result for relative symplectic cohomology. We introduce a technique for constructing complete embeddings using what we refer to as integrable anti-surgery. We apply these to study symplectic topology and mirror symmetry of symplectic cluster manifolds and other examples of symplectic manifolds with singular Lagrangian torus fibrations satisfying certain completeness conditions. 

%We introduce the notion of a symplectic manifold being of geometrically finite type, which is a strengthening of geometrically boundedness. We prove that truncated symplectic cohomology of a compact subset $K$ that is contained in a geometrically finite type $Y^{2n}$ does not change if it is considered as a subset of geometrically bounded $M^{2n}$ via an arbitrary symplectic embedding $Y\to M$. Under the assumption that the torsion exponents of $K$ (as a subset of $Y$) are bounded we deduce the same result for relative symplectic cohomology. We provide examples of geometric setups in the context of complete Lagrangian torus fibrations with singularities where our locality result applies.
\end{abstract}

\maketitle
\tableofcontents
\section{Introduction}
 A symplectic manifold  $(M,\omega)$ is said to be \emph{geometrically bounded} if it carries a compatible almost complex structure whose associated Riemannian metric is equivalent to a complete Riemannian metric with bounds on its sectional curvature and radius of injectivity. This includes closed symplectic manifolds but also many open ones. See the beginning of \S \ref{Sec:Overview} for a discussion. 

The geometric boundedness assumption allows one to do Hamiltonian Floer theory on $M$. In particular we can associate a $\Lambda_{\geq 0}$-module $SH_M(K),$ called relative symplectic cohomology, to each compact $K\subset M$. A key feature  is  that for nested inclusions of compact subsets $K\subset K'$, we have canonical and functorial restriction maps $SH_M(K')\to SH_M(K).$ Recent years saw a flourishing of the study of similar invariants which were initially introduced, for symplectically a-spherical manifolds, in the the series of papers \cite{FH94,CFH95,CFHW94,CFHW96}.

A natural (but rather vague) question is the following:
\begin{ques}
How much does $SH_M(K)$ depend on $M$?
 \end{ques}
This can be thought of as \emph{the question of locality}.  This question is of broad importance in reducing the study of global Floer theoretic invariants to local ones (especially when it is combined with the results of \cite{varolgunes}) and it is closely related to the notion of obstructedness in Lagrangian Floer theory.  It is expected that systematic answers exist in numerous settings, e.g. \cite{borman}, \cite{yuhan}, \cite{tonkonog}, \cite{GromanToAppear}, also compare with \cite{tonkonog2019}.

Here is a more precise sub-question that we will concern ourselves with in the present paper:

\begin{ques}
Let $Y$ be a geometrically bounded symplectic manifold of the same dimension together with a symplectic embedding $\iota:Y\to M$. % manifold with image $Y=\iota(X)\subset M$. % be an open subset, which is geometrically bounded with the symplectic form $\omega|_Y$. 
When is there a natural isomorphism between $SH_Y(K)$ and $SH_M(\iota(K))$ for a class of compact subsets $K\subset Y$? We use the word natural in the sense that the isomorphisms commute with restriction maps.
 \end{ques}

When such isomorphisms exist we call them \emph{locality isomorphisms}. In this paper we consider the case when $Y$ %$ (Y,\omega|_Y)$
is of \emph{geometrically finite type} and construct locality isomorphisms for compact subsets of $Y$ that have \emph{homologically finite torsion}. These notions are introduced in the next paragraph. Let us immediately comment that the homologically finite torsion assumption is there for technical reasons and we fully believe it is unnecessary. On the other hand, the geometrically finite type assumption (as opposed to mere geometric boundedness) is essential for our argument.% and we believe that it is essential. That being said, in all natural examples this assumption is satisfied and proving the result or constructing a counter-example in the general case seems to be,  while possibly interesting in itself, irrelevant for applications. We do not address the issue in this paper.

A symplectic manifold $Y$ is said to be of \emph{geometrically finite type} if it is geometrically bounded and carries a smooth exhaustion function $f$ which has a finite number of critical points and whose Hamiltonian vector field is $C^1$-bounded with respect to a geometrically bounded $J$. See the beginning of \S \ref{sec-loc-iso} for a discussion including examples and non-examples. Note that closed manifolds again automatically satisfy this condition. We emphasize that $J$ and $f$ are not part of the structure of $Y$. That is, any invariant we construct  will eventually be independent of $J$ and $f$. A compact set $K\subset Y$ is said to have \emph{homologically finite torsion} if the torsion exponents of $SH^*_Y(K)$ over the Novikov ring $\Lambda_{\geq 0}$ are bounded above. See \S \ref{sss-on-tor} for a discussion.

\begin{theorem}[The Main Theorem]\label{thmMainIntro}
Let $M$ and $Y$ be symplectic manifolds of the same dimension and $\iota:Y\to M$ be a symplectic embedding. Assume that $M$ is geometrically bounded and $Y$ is geometrically of finite type. % and the open subset $Y\subset M$ be geometrically of finite type. 
Then, we can construct natural isomorphisms $SH^*_{Y}(K)\simeq SH^*_{M}(\iota(K))$ for each homologically finite torsion compact subset $K\subset Y$.
\end{theorem}

Before proceeding we note that if $Y$ is compact then it has to be a connected component of $M$ and the result is trivial. Therefore, let's assume that $Y$ is non-compact. This implies (see Remark \ref{remVolInf}) that $Y$, and therefore $M$, has infinite volume, and in particular that $M$ is also non-compact. An important special case of the setup in Theorem \ref{thmMainIntro} is when $K$ is a symplectic manifold with convex contact boundary, $Y$ is its completion and $M$ is an \emph{arbitrary} geometrically bounded symplectic manifold \footnote{For the present version of the theorem to apply we need $K\subset Y$ to satisfy the torsion hypothesis but the geometric setup should be considered without that assumption.}. This case is what gave rise to the phrase \emph{complete embedding}: a codimension zero symplectic embedding of a geometrically bounded symplectic manifold into another.

We briefly enumerate some settings in which Theorem \ref{thmMainIntro} can be applied.  The last two items will be expanded upon later in the introduction.

\begin{itemize}
\item
One class of examples comes from not necessarily exact strong symplectic fillings of Liouville cobordisms, see Example \ref{Ex:Liouville-Cobordisms}.  We do not focus on this setup in this paper.

\item In Sections \S \ref{sec:EignenRayCompleteEmb} and \S \ref{Sec:GrossFibrations} we introduce a technique for constructing interesting complete embeddings in the context of singular Lagrangian torus fibrations. This technique applies among others to symplectic cluster varieties and the Gross fibrations on the complement of an anti-canonical divisor in a toric Calabi-Yau manifold. We expect this technique to apply also to torus fibrations with singularities of Gross-Siebert type \cite{Gammage2021,AbouzaidSylvan2021} equipped with global torus symmetries, or, more generally, satisfying a certain slidability hypothesis. We discuss some applications of Theorem \ref{thmMainIntro} to mirror symmetry below. This will be developed further in forthcoming work.

%Let $(X,\omega_1, V)$ be a Liouville cobordism from the contact manifold $(Y^-,\lambda^-)$ to the contact manifold $(Y^+,\lambda^+)$. Let us also take a strong (i.e. convex) symplectic filling $(F,\omega_2)$ of $(Y^-,\lambda^-)$. Finally, consider the symplectic manifold $W$ obtained as the completion of the union $F\cup_{Y^-} X$. Completion here has the standard meaning of attaching a semi-infinite symplectization of $(Y^+,\lambda^+)$ to the contact boundary of $F\cup_{Y^-}  X$. Then there exists a unique symplectic embedding of the completion $\widehat{F}$ of $F$ into $W$, which restricts to identity on $F$ and sends the Liouville vector field on the semi-infinite symplectization of $(Y^-,\lambda^-)$ to $V$.

\item We can also use Theorem \ref{thmMainIntro} to show non-existence of complete embeddings.  This is analogous to the way Viterbo functoriality is used  to rule out exact Lagrangian embeddings in $\bC^n$ and other flexible Weinstein manifolds.  For a general geometrically bounded symplectic manifold the notion of complete embedding replaces that of Liouville embedding. A simple example of this phenomenon is that \emph{there is no open subset of a toric Calabi-Yau $n$-fold which is symplectomorphic to $T^*\bT^n$}. We will discuss this in more detail later in the introduction.
\end{itemize}

We assume throughout the paper that $c_1(M)=0$, and moreover, when we do Hamiltonian Floer theory on $M$, we fix a homotopy class of a trivialization of $\Lambda^n_\bC(TM^{2n})$. This allows us to work in the $\mathbb{Z}$-graded set-up (which sometimes leads to sharper statements, e.g. Proposition \ref{prop-intro-rel-red}) and get by without using virtual techniques. We hope that it will be clear to the reader that none of these are crucial for our techniques.

\subsection{More details on Theorem \ref{thmMainIntro}}
\subsubsection{On torsion}\label{sss-on-tor}
We now discuss the assumption concerning torsion in Theorem \ref{thmMainIntro}. We repeat that we fully expect this assumption can be lifted from the statement. Proving this would have required working a lot more at the chain level,  making the paper more technical than it already is, but we do not expect any serious difficulties.

In the text we prove a statement that is more refined than Theorem \ref{thmMainIntro}, which applies without  assumptions on the torsion. Namely, we prove that the truncated relative\footnote{we will omit relative from this phrase from now on and say truncated symplectic cohomology} symplectic cohomologies $SH^*_{Y,\lambda}(K)$ and $SH^*_{M,\lambda}(\iota(K))$ (see Section \ref{ss-trunc}) satisfy locality in the sense of Theorem \ref{thmMainIntro} (with the same assumptions on $M,Y$ and $\iota$) for all $\lambda\geq 0$ and \emph{all} compact $K\subset Y$. See Theorem \ref{thmLocalityTrunc} for the precise statement. 

The torsion assumption comes in when we try to recover relative symplectic cohomology from truncated symplectic cohomologies. We need some definitions to explain this. The maximal torsion of a module $V$ over the Novikov ring $\Lambda_{\geq 0}$ is the supremum of the torsions of all of its torsion elements (Definition \ref{def-tor}), which is possibly equal to $\pm\infty$. If it is not $+\infty$, then we say that $V$ has \emph{finite torsion}.  
%
%For any cochain complex\footnote{all of our chain complexes are $\mathbb{Z}$-graded in this paper} over $\Lambda_{\geq 0}$, we define the \emph{homological torsion in degree $i$ of $C$} as the maximal torsion of $H^i(C)$. We then prove a result in homological algebra, Lemma \ref{lemTorsMitt}, that the full symplectic cohomology can be recovered from the truncated symplectic cohomologies provided the torsion is homologically finite.

Let us state the specific result that we use to relate the relative symplectic cohomology to truncated symplectic cohomologies as we think it is instructive in its own right. We define

$$SH_{M,red}^*(K):=\varprojlim_{\lambda}SH_{M,\lambda}^*(K).$$

\begin{proposition}\label{prop-intro-rel-red}Assume that $SH^{i}_{M}(K)$ has finite torsion for some $i\in\mathbb{Z}$, then the canonical map $$SH_M^i(K)\to SH_{M,red}^i(K)$$ is an isomorphism.\end{proposition}

The proof is an exercise in homological algebra and it is provided in Section \ref{ss-tor-fin}. Let us stress that in Theorem \ref{thmMainIntro} (or even better in its slightly sharper version Theorem \ref{thmLocalityRelFinal}) the assumption on finiteness of torsion is made only with respect to $Y$ (and not $M$), making the result much more useful. 

\begin{remark}
The maximal torsion of the homology of a chain complex over the Novikov ring is closely related to the \emph{boundary depth} of the corresponding filtered chain complex over the Novikov field. We do not discuss this in detail as it would require us to introduce some notation that is not used elsewhere in the paper.
\end{remark}

We now comment on when the assumption on homological finiteness of torsion is known or expected to hold and when it does not hold. Let us define the \emph{homological torsion in degree $i$} of $K\subset M$ as the maximal torsion of $SH_M^i(K)$.

\begin{itemize}
\item By the computations in  \cite{CFHW94} ellipsoids and polydisks in $\bC^n$ have homologically finite torsion in each degree.  
\item Let $Y=T^*\bT^n$ with its projection $\pi:T^*\bT^n\to\bR^n$. Let $K=\pi^{-1}(P)\subset Y$ for $P\subset \bR^n$ a compact convex domain. Then $SH_Y^*(K)$ has no torsion, so finiteness of torsion in any degree holds trivially. This result can be derived by a direct computation using Viterbo type acceleration data along with careful perturbations to deal with the Morse-Bott $T^n$-families of orbits that appear (similar to Section \ref{ss-non-fin-tor}). 
\item For $Y$ as in the previous item and $K=\pi^{-1}(P)\subset  Y$ such that $P\subset\bR^n $ is \emph{star-shaped} but not convex, it can be shown that homological finiteness of torsion does \emph{not} hold in all degrees. We give the argument for a particular such domain in Section \ref{ss-non-fin-tor} below. 
\item Let $Y$ be a symplectic cluster manifold with nodal almost toric fibration $\pi:Y\to B$ of cluster type (see Section \ref{sss-intro-cluster} for quick definitions).  Let $P\subset B$ be a compact domain (possibly with corners) whose boundary does not contain any nodes. Let $K=\pi^{-1}(P)$. Then $K$ has homologically finite torsion in degree $0$. If we assume further that $P$ is a convex\footnote{this means that it is locally convex near its boundary} polygon with rational slope sides and the restriction of the symplectic form to $K$ is exact then $K$ has homologically finite torsion in degree $1$ as well. This can be shown by combining Theorems \ref{tmCompEmbLagTails} and \ref{thmMainIntro} and the argument given in \cite[Theorem 6.19, Remark 6.20]{Pascaleff}. 
\item For $Y$ as in the previous item, it can again be shown that if $P$ is star shaped but not convex, homological finiteness of torsion does not hold in degree 1. Cf \cite[Remark 6.17]{Pascaleff}. 
\item If $K$ is an index bounded Liouville domain (as in e.g. Definition 2.2 of \cite{yuhan}) inside a symplectic manifold $Y$ with some (homotopy class) of a trivialization of its canonical bundle, then $SH^{i}_{Y}(K)$ has finite torsion for all $i$. Compare with Proposition 4.6 of \cite{yuhan}.
\end{itemize}
\begin{remark}
We have not rigorously studied the question of homologically torsion finiteness in degrees $2$ and $3$ inside symplectic cluster manifolds, but we expect that it would not be too hard to analyze using similar methods. %We also expect the exactness assumption to be non-essential.
\end{remark}
%{\color{red} We also remark that using the Mayer-Vietoris theorem of \cite{varolgunes}, one can obtain locality isomorphisms for a larger class of compact subsets. Assume that $K_1$ and $K_2$ are compact subsets of $Y\subset M$ that satisfy descent  as in Definition 1 \cite{varolgunes} (for examples see Theorem 1.3.4 of the same paper). If the locality isomorphisms exist for $K_1$, $K_2$ and $K_1\cap K_2$ (e.g if they are all homologically finite torsion) then locality isomorphism also exists for $K_1\cap K_2$. For example this can be used to deduce locality for the subset of $T^*T^2$ (used as $Y$, where $M$ for example could be a symplectic cluster manifold using Theorem \ref{tmCompEmbLagTails}) considered in Section \ref{ss-non-fin-tor}, which we know is not homologically finite type in degree $1$.}

\subsubsection{Idea of the proof}
We conclude this part of the introduction with a brief discussion of the method of proof. Consider the setup of Theorem \ref{thmMainIntro}. We start with a comment on why one should generally expect any type of locality result. %In general the underlying chain complex of $SH^*_M(K)$  has two types of generators:
%\begin{itemize}
%\item
%\emph{Inner generators} which are periodic orbits occuring in or near $K$. These are common to invariants relative to $Y$ and $M$.
%\item \emph{Outer generators} which occur outside of a neighborhood of $K$ and are present in $M$ but not in $Y$.
%\end{itemize}
The first condition for such result is that the contribution of the $1$-periodic orbits which are, roughly speaking, present in $M$ but not in $Y$ be negligible. This is expected to always hold as a result of the truncation and completion operations involved in the definition of our invariants, but it is not always easy to prove. If this holds, we can consider the underlying $\Lambda_{\geq 0}$-modules for $SH^*_Y(K)$ and $SH^*_M(\iota(K))$ at the chain level as being identical. The locality question then involves understanding the contribution, or lack thereof, of Floer trajectories that connect orbits in $\iota(Y)$ but are not contained in $\iota(Y)$.

Let us assume that $\iota$ is the inclusion of an open subset $Y\subset M$ equipped with the symplectic form $\omega|_Y$.\footnote{For the construction of locality isomorphisms, the setup we presented thus far with a complete embedding $\iota: Y\to M$ is more general only in a psychological sense. One can always identify $Y$ with its image under $\iota$ and reduce to the case of an open subset $Y\subset M$ equipped with the symplectic form $\omega|_Y$, which is by assumption geometrically of finite type.} The geometrically finite type hypothesis made on $Y$ implies that we can find a geometrically bounded almost complex structure $J_Y$, an admissible function $f$ as defined in the beginning of Section \ref{sec-loc-iso}, and a constant $r$ such that all the $1$-periodic orbits of $f$ are contained in the sub-level $V_r:=f^{-1}(-\infty,r)$. For any $R>0$  the annulus $A_R:=f^{-1}([r,R+r])$ forms a \emph{separating region} in the following sense.
\begin{proposition}
Fix $J_Y$ as above.  There are constants $\epsilon>0, C>0,D,$ depending on the bounds on the geometry of $J_Y$ such that the following holds. Let $f$ be a function on $Y$ as above and assume the Lipschitz constant of $f$ is $\leq\epsilon$. Let $R>0$ and let $H:M\to\bR$ be a Hamiltonian such that $H|_{A_R}= f|_{A_R}$. Let $J$ be a compatible almost complex structure on $M$ such that $J|_{A_R}=J_Y|_{A_R}$. Then any Floer trajectory $u$ for the datum $(H,J)$ which meets both $V_{r}$ and $M\setminus V_{R+r}$ satisfies
\begin{equation}
E(u)\geq CR-D.
\end{equation}
\end{proposition}
The proofs of more detailed forms of this statement are given in Sections \S\ref{ss-diss-c0} and \S\ref{SecLipBoundDiff}.

\begin{remark}\label{remHeinUsher}
A closely related proposition is proven in \cite[Lemma 3.2]{Hein} following \cite[Lemma 2.3]{Usher} and is referred to as Usher's Lemma. What the present Proposition adds to Usher's Lemma is that the right hand side goes to infinity as $R$ goes to infinity provided the Hamiltonian has small Lipschitz constant. The proof is non-trivial.
\end{remark}

%Let us now consider the setting of complete embedding $Y\subset M$ and a compact set $K\subset Y$. Fix a geometrically bounded compatible almost complex structure in $Y$. The finite type assumption implies that for every $R>0$ we can choose acceleration data for $SH^*_{M}(K)$ with bounds on the geometry inside the $R$ neighborhood of $K$ inside $M$ (which by construction is contained in $Y$) using the almost complex structure we fixed on $Y$ with more properties only one of which we mention: all generators are either near $K$ or outside $B_R(K)$. Theorem ** below, which is a quantitative version of the dissipativity method (reference to Yoel's paper), states that for each real number $E>0$ we can find an $R$ such that any Floer trajectory with one end on an inner orbit and having energy at most $E$ is in fact contained inside $B_R(K)$.
Given a compact set $K$ we can choose $r$ so that, in addition, $K\subset V_r$. The proposition implies that at each energy truncation level we can Floer theoretically separate $K$ from $M\setminus Y$ by making $R$ large enough. The upshot is that at each truncation level we can make the underlying complexes for $K\subset Y$ and ${\iota(K)}\subset M$ coincide after justifying that the orbits that are not common to both do not contribute.

\subsection{Complete embeddings and singular Lagrangian torus fibrations}
\subsubsection{Symplectic cluster manifolds}\label{sss-intro-cluster}

Let $X$ be a symplectic manifold and $\pi: X\to B$ be a \emph{nodal Lagrangian torus fibration}. That is, an almost toric fibration as in \cite{symington} with a finite number of singularities all of which are of focus-focus type. The torus fibration $\pi$  induces on $B$ the structure of an integral affine manifold with singularities. That is, denoting by $N\subset B$ the critical values  of $\pi$, the open and dense subset $B^{reg}:=B\setminus N$ carries an integral affine structure. Around each $n\in N$ the integral affine structure has monodromy of shear type. A detailed description is in Section \S \ref{sec:EignenRayCompleteEmb} below.  

%We say that $\pi$ is \emph{complete} if any sequence of points in $B_{reg}$ that is contained in an integral affine chart with bounded integral affine coordinates has a subsequence convergent in $B$.

A \emph{Lagrangian tail $L$ emanating from a focus-focus point $p$} is a properly embedded Lagrangian plane in $X$ which surjects under $\pi$ onto a {smooth ray} $l$ emanating from $m=\pi(p)$ with fiber over $m$ being $p$ and with all the other fibers diffeomorphic to circles disjoint from the critical points of $\pi$. Here, by a smooth ray we mean the image of a smooth embedding $[0,\infty)\to B$.   We show in Lemma \ref{lem-tail-eigenray} that $l$ is in fact an {eigenray} of $p$ in the sense of Definition \ref{def-eigenray}. See Figure \ref{fig-tail}. Note that $X$ might admit different nodal Lagrangian torus fibrations and each of these give rise to Lagrangian tails in their own right.

\begin{figure}
\includegraphics[width=0.6\textwidth]{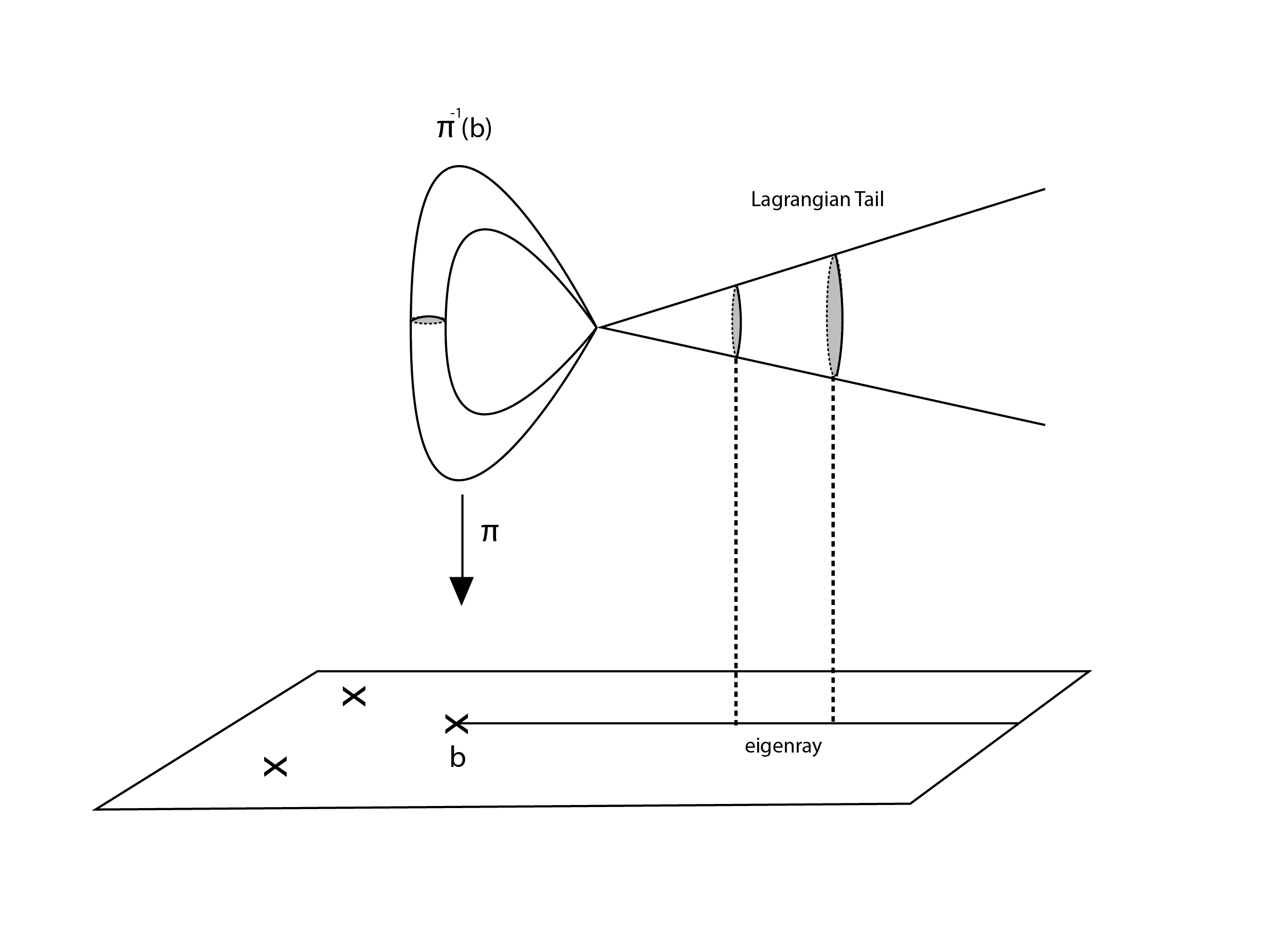}
\caption{A depiction of a Lagrangian tail}
\label{fig-tail}
\end{figure}

\begin{definition}\label{def-symp-clus}A \emph{4-dimensional symplectic cluster manifold} is a symplectic 4-manifold $X$ for which there exists a nodal Lagrangian torus fibration $\pi:X\to B$ with $B$ simply connected, and a choice of pairwise disjoint Lagrangian tails $L_p$ for each critical point $p$ of $\pi$ such that the following condition holds: \begin{itemize}
\item The eigenrays $l_p:=\pi(L_p)$ are proper in the sense of Definition \ref{def-eigenray}. Moreover, an affine geodesic starting at any point $b\in B^{reg}$ in any direction can be extended indefinitely unless it converges in finite time to a point on  $\bigcup_{p\in Crit (\pi)}l_p$.
\end{itemize}Let us refer to this condition as \emph{weak geodesic completeness}.\end{definition}

{For the remainder of the paper we will refer to 4-dimensional symplectic cluster manifold simply as symplectic cluster manifolds.} In Lemma \ref{prpClstGeoBd} below we show that \emph{symplectic cluster manifolds are of geometrically finite type}.

We refer to the data $P=\{\pi,\{L_p\}_{p\in Crit \pi}\}$ as a \emph{cluster presentation} of $X$.   We say {that} a nodal Lagrangian torus fibration $\pi$ on $X$ is \emph{of cluster type} if it is part of a cluster presentation on $X$. The multiset $\{\pi(L_p)\}_{p\in Crit\pi}$  of rays in $B$ associated to a given cluster presentation, together with the integral affine structure induced by $\pi$ give rise to a combinatorial structure called an \emph{eigenray diagram} in the beginning of Section \S\ref{SecEigenrayDiagrams} {\footnote{{See Remark \ref{rem-eigenray-def}.}} }. We refer to it as an \emph{eigenray presentation} of $X$ and denote it {by} $\cR_P$. It is shown in Proposition \ref{prpClustEquiv} that an eigenray diagram $\cR$ gives rise to a symplectic cluster manifold $M_{\cR}$ which is unique up to symplectomorphism (see Section \ref{ss-nodal-int} for the uniqueness part). In particular, for the case where $\pi$ in a cluster representation has no singularities we have $M_0=T^*T^2$. 

\begin{remark} We could define a symplectic cluster manifold simply as a symplectic manifold that is symplectomorphic to $M_\cR$ for some eigenray diagram $\cR$, but we tried to give a more intrinsic definition in the introduction. The equivalence is shown in Proposition \ref{prpClustEquiv}. 
\end{remark}

A given cluster manifold may possess many different eigenray presentations. These may be related by \emph{nodal slides} which change the fibration $\pi$ and by $\emph{branch moves}$ which roughly amount to replacing an eigenray by its opposite (see Section \ref{SecEigenrayDiagrams} for precise definition).  Eigenray diagrams are the symplectic counterpart of the \emph{seeds} familiar from the cluster variety literature, and one could think of a cluster representation as being analogous to that of \emph{toric models} from \cite{ghk}.

The defining property of cluster varieties in algebraic geometry is  the existence of a family $T_{\alpha}$ of embeddings of the algebraic torus parametrized by the set of seeds $\alpha$. The following theorem gives the symplectic counterpart of this story.

\begin{theorem}\label{tmCompEmbLagTails}
Let $X$ be a symplectic cluster manifold, and let $$P=\{\pi,\{L_p\}_{p\in Crit \pi}\}$$ be a cluster presentation of $X$.  For any subset $S\subset Crit \pi$ let $\cR_{P\setminus S}$ be the eigenray diagram obtained from $\cR_P$ by deleting (in the multiset description used above) the eigenrays $\pi(L_p)$ for $p\in S$.
Then $X\setminus\cup_{p\in S} L_p$ is symplectomorphic to $M_{\cR_{P\setminus S}}$. In particular,
\begin{equation}\label{eqCompEmbLagTails}
X\setminus \cup_{p\in Crit\pi}L_p\simeq T^*T^2.
\end{equation}
\end{theorem}

{The proof of Theorem \ref{tmCompEmbLagTails} is given in Section \S\ref{SecProofOfThm2}. There, we also show that we have good control over how the nodal Lagrangian torus fibrations on the two symplectic cluster manifolds relate to each other.} The basic idea is explained in Section \ref{sss-idea}. 
\begin{remark}\label{remExacEigenray}
Let us call an eigenray diagram \emph{exact} if the lines containing all of the eigenrays pass through the same point $b_0$. It is well known that if $\mathcal{R}$ is exact then $M_{\mathcal{R}}$ admits a Liouville structure which makes it a complete finite type Weinstein manifold\footnote{This is a folk result. Let us give a quick proof. Possibly using nodal slides, we can assume that eigenrays do not contain $b_0$ in their interior. Let  $K\subset B_\cR$ be  a star shaped compact domain which contains the union of the straight line segments from $b_0$ to all the other nodes. The assumption implies that outside of $K$ the Euler vector field with respect to $b_0$ in connected integral affine charts is preserved under the transition maps of the integral affine structure. Using action angle coordinates we lift it to a Liouville field whose $\omega$-dual is a primitive $\theta$ of the symplectic form $\omega$ in the complement of $\pi_\cR^{-1}(K)$. In particular, we obtain a class $a=[(\omega,\theta)]$ in relative deRham cohomology $H^2_{dR}(M_\cR, M\setminus \pi_\cR^{-1}(K))$.  
To extend $\theta$ to all of $M$ it suffices to show this class vanishes. %If we can show that $a=0$, then it is easy to construct a primitive of $\omega$ on $M_\cR$, which agrees with $\theta$ outside a compact set slightly larger than $K$. That $a=0$ follows 
This follows from the relative deRham isomorphism because the class of a Lagrangian section and the classes of Lagrangian tails % {\color{blue}maybe you mean discs in the complement with boundary on the tails? NO I DON'T. TO SEE THIS USE EXCISION AND POINCARE DUALITY. WE KNOW THE GENERATORS OF $H_2(\pi_\cR^{-1}(N(K)))$, TEST CLASSES. } 
over the given eigenrays generate the relative homology group $H_2(M_\cR, M\setminus \pi_\cR^{-1}(K))$ and $a$ vanishes on all of them.
%{\color{blue}why does $a$ vanish on the Lagrangian tails and what does it have to do with the condition? YOU HAVE TO INTEGRATE THE RELATIVE DE RHAM CYCLE ON THESE RELATIVE CYCLES. I AM CONFUSED ABOUT THE SECOND QUESTION. WE ARE TRYING TO SHOW THAT a PAIRS TRIVIALLY WITH ALL THESE GENERATORS?}
}.  Such exact symplectic cluster manifolds and their mirror symmetry have been studied by various authors in the literature \cite{shende}, \cite{keating}, \cite{hacking2020homological}. Conversely, it is easy to show that if $\cR$ is non-exact then $M_{\cR}$ is non-exact by constructing a second homology class that pairs non-trivially with the symplectic class.

We are not aware of any systematic study of the cluster symplectic manifolds that are associated to non-exact eigenray diagrams. It appears that for a non-exact eigenray diagram $\mathcal{R}$, the symplectic manifold $M_{\mathcal{R}}$ is not symplectomorphic to the positive half of the symplectization of a contact manifold outside of a compact subset, but there might be a weaker structure at infinity involving stable Hamiltonian structures.

\end{remark}

Let us illustrate how Theorem \ref{tmCompEmbLagTails} is used in an example. 
\begin{example}\label{ExFiveCharts}
Consider the eigenray diagram $\cR$ with two nodes  of multiplicity one at the points $(1,0)$ and $(0,1)$ in $\bR^2$ with rays going along the positive $x$ and $y$-axes, respectively. See the left side of Figure \ref{fig-five}. Let $\pi : M\to B$ be a compatible nodal Lagrangian torus fibration, i.e. in the notation of Section \ref{SecEigenrayDiagrams}, one of the form $\pi_\cR: M_\cR\to B_\cR$. We  describe $5$ symplectic embeddings of $M_0=T^*T^2$ into $M$ using the statement of Theorem \ref{tmCompEmbLagTails}.
\begin{figure}
\includegraphics[width=0.6\textwidth]{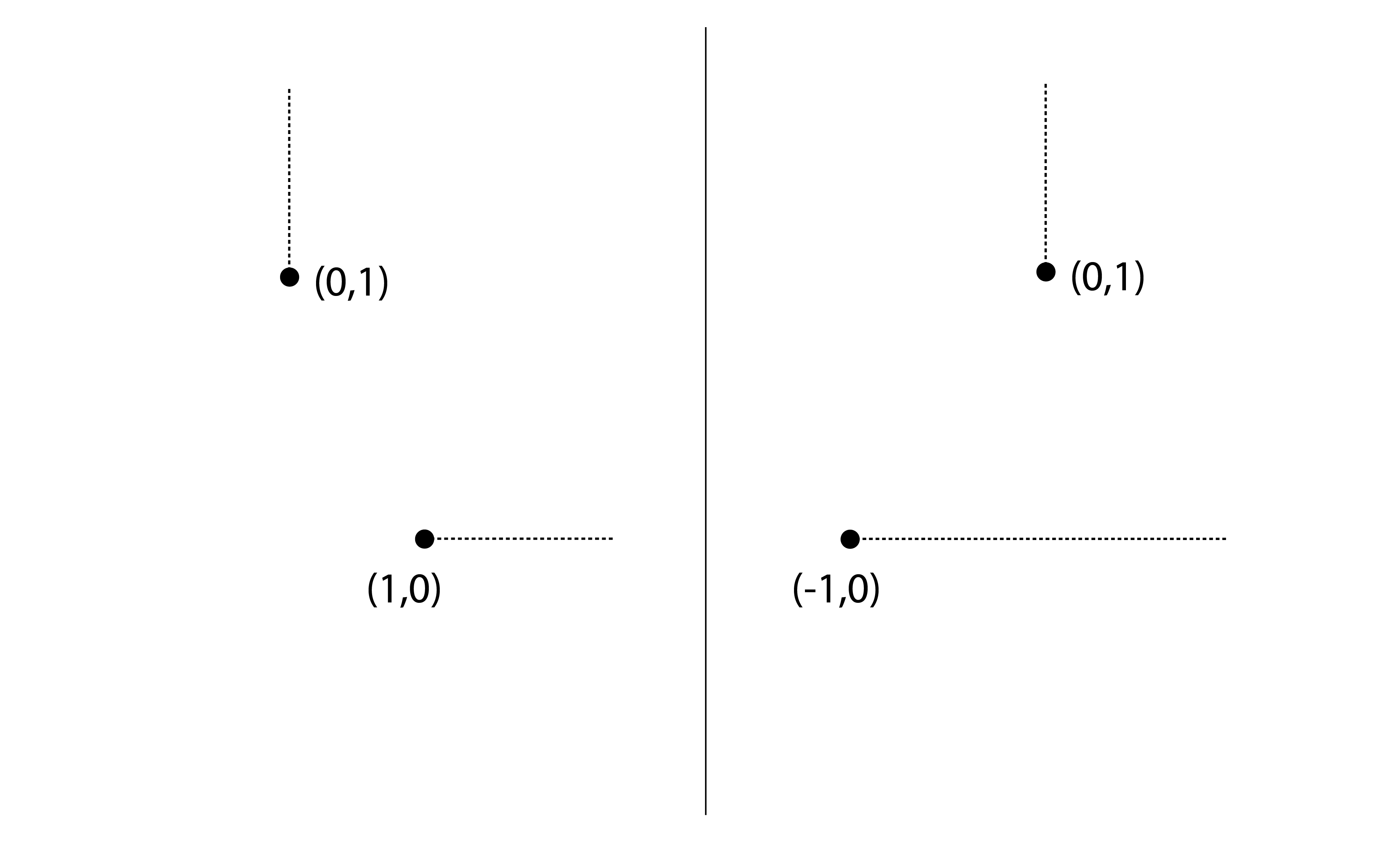}
\caption{Two eigenray diagrams related by a nodal slide.}
\label{fig-five}
\end{figure}
Let us denote the critical values in $B$ corresponding to $(1,0)$ and $(0,1)$ by $n_x$ and $n_y$. We then denote the eigenrays of $n_x$ and $n_y$ corresponding to the defining rays by $l_x$ and $l_y$, and the other eigenrays by $l_x^-$ and $l_y^-$. There are $4$ different ways of choosing an eigenray for each node. Three out of these $4$ choices, namely $(l_x,l_y),(l_x^-,l_y),$ and $(l_x,l_y^-)$ lead to non-intersecting eigenrays. The corresponding Lagrangian tails are clearly disjoint, and hence for each of these three cases we obtain a cluster representation of $M$ (along with $\pi$). By Theorem \ref{tmCompEmbLagTails}, we therefore obtain $3$ embeddings $M_0\to M$.  For the choice $(l_x^-,l_y^-)$, however, one immediately sees that, no matter how the Lagrangian tails are chosen, they cannot be disjoint as they  intersect the torus over the intersection point along cycles representing mutually independent homology classes (well defined up to sign). Indeed, under the identification of the tangent space to the base with the cohomology of the fiber, these non-zero homology classes annihilate the tangent line to the respective tail. % \footnote{\color{purple}here we are using the basic observation that a Lagrangian cylinder that projects onto a straight affine line segment $s$ with rational slope must intersect the fiber tori above $s$ along straight circles (in angle coordinates) in the homology class corresponding to the primitive vector that annihilates $s$.}

Using nodal slides we can produce two additional embeddings. As mentioned above, here we are thinking of a nodal slide operation as changing the initial nodal Lagrangian torus fibration on the underlying symplectic manifold $M$ to a new fibration $\pi^{new}:M\to B^{new}$. Here $B^{new}=B$ as smooth manifolds, but with a different induced nodal integral affine structure. Moreover, $B^{new}$ as a nodal integral affine manifold can be described by an eigenray diagram $\cR^{new}$ obtained by applying a nodal slide to $\cR$, which we will use in what follows. The reader might want to consult the end of Section \ref{ss-3d} for how this works in an example.

The first option is to slide $n_x$ in the direction of $l_x^{-}$ until in the eigenray diagram the node $(1,0)$ arrives at $(-1,0)$. See the right side of Figure \ref{fig-five}. This results in a new nodal Lagrangian torus fibration $\pi':M\to B'$ as in the previous paragraph. Let us denote the nodes in $B'$ by $n_x'$ and $n_y'$ with eigenrays $l'_x$ and $l'_y$ corresponding to the rays of the eigenray diagram, and $(l'_x)^{-}$ and $(l'_y)^{-}$ for the others. Lagrangian tails corresponding to $l'_x, l'_y, (l'_x)^{-}$ are Hamiltonian isotopic to the ones of $l_x, l_y, l_x^{-}$ respectively, but crucially,  the one for $(l'_y)^-$ has changed drastically and it is not Hamiltonian isotopic to the one for $l_y^-$\footnote{The justification of this point is given below, right after Corollary \ref{corNonHamIsotopicTails}. This relies on Conjecture \ref{ConjSupport}, for which we give a proof sketch in the required generality.}. More important for our point here, $(l'_y)^-$ does not intersect $(l'_x)^-$ and hence their Lagrangian tails give rise to a cluster presentation (along with $\pi'$) of $M$. This gives our fourth embedding via Theorem \ref{tmCompEmbLagTails}. We later refer to the Lagrangian tail of $(l'_y)^-$  as a \emph{scattered Lagrangian for the diagram $\cR$}.

Finally, we can apply the same procedure but this time sliding $n_y$. This produces the fifth embedding. The fact that these five embeddings are not Hamiltonian isotopic is not trivial but we give a convincing sketch of an argument below.

Interestingly, the two scattered Lagrangians are Hamiltonian isotopic to each other. In fact, it can be shown that they are both Hamiltonian isotopic to the tropical Lagrangian from Figure \ref{fig-scattered}. Working slightly harder one can also show that iterations of the nodal slide and mutate procedure produces Lagrangian tails, which are either Hamiltonian isotopic to the initial Lagrangian tails (the ones of $l_x,l_y,l_x^-$ and $l_y^-$) or to the scattered Lagrangian.
\end{example}

\begin{figure}
\includegraphics[width=0.6\textwidth]{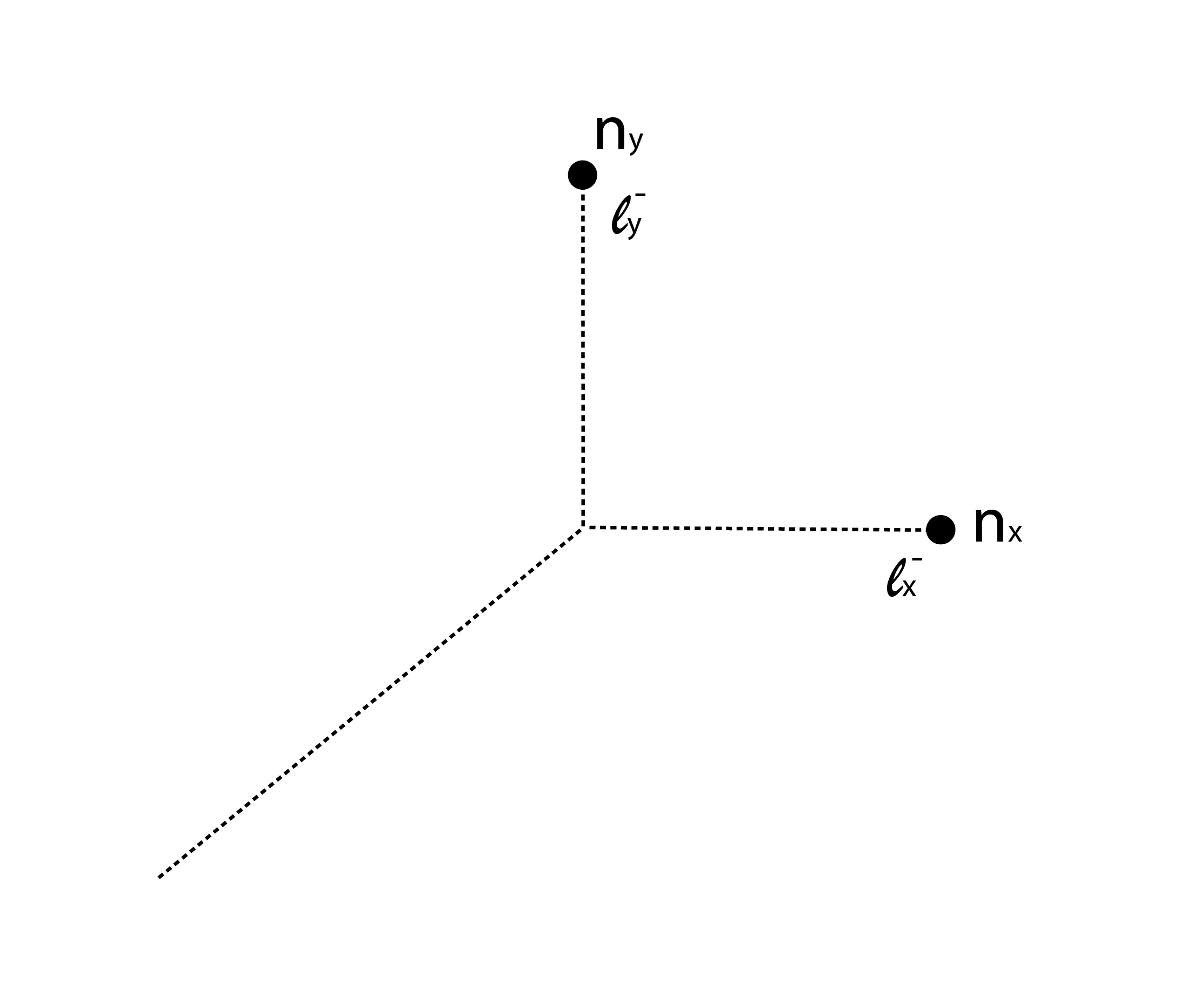}
\caption{The tropical curve drawn above in the base of the original fibration $\pi:M\to B$ gives rise to a tropical Lagrangian inside $M$. The scattered Lagrangians in Example \ref{ExFiveCharts} are both Hamiltonian isotopic to this tropical Lagrangian.}
\label{fig-scattered}
\end{figure}
We now formulate the following conjecture inspired by the combination of mirror symmetry, Theorems \ref{tmCompEmbLagTails}  and \ref{thmMainIntro}, and the theory of the local Fukaya category being developed by M. Abouzaid jointly with the authors. It is also a special case of the more general expectation (see \cite{hicks} for first steps) that the support a tropical Lagrangian should be equal to the defining tropical variety.

\begin{conjecture}\label{ConjSupport}
Let $X$ be a symplectic cluster $4$-manifold equipped with a  cluster presentation $\{\pi,L_1,\dots, L_n\}$.  Then for any $1\leq i\leq n$ no fiber of $\pi$ over a point in $\pi(L_i)$ can be displaced from $L_i$ by a Hamiltonian isotopy.  
 \end{conjecture}
 
%The  argument for this is straightforward, but relies on tools that are still being developed. We give an outline. First it is easy to show $L$ can be completed to a cluster presentation $\{\pi,L,L_2,\dots, L_n\}$. 
Without loss of generality fix $i=1$.  We will assume for simplicity that $\pi(L_1)$ does not contain any $\pi(L_i)$ in its interior and that the fiber in question is a smooth one. By nodal slides supported away from $\pi(L_1)$ the ray $\pi(L_1)$ can be taken to be in mutable position. Taking $L'_1$ to be an opposite tail (i.e. one that lies over the opposite ray and emanates from the same critical point), we get a new cluster presentation $\{\pi',L'_1,L_2,\dots, L_n\}$. Taking $Y=X\setminus L_1'\cup_{i\geq 2}L_i$ we have that $L_1\cap Y$ is a properly embedded Lagrangian cylinder. Under the symplectomorphism \eqref{eqCompEmbLagTails} $Y\cap L_1$ maps to the conormal $N$ of a rational affine line $l$ in the base of the standard fibration $\pi_0:T^*\bT^2\to\bR^2$. It is straightforward to show that 
for any rational convex polygon $P\subset\bR^2$ which meets $l=\pi_0(N)$ we have 
\begin{equation}\label{eqHFneq0}
HF^*_{T^*\bT^2}(N,N; \pi^{-1}(P))\neq 0.
\end{equation}
The left hand side is the relative Lagrangian Floer cohomology, which was introduced in Section 2.3 of  \cite{tonkonog}. In this case $HF^*_{T^*\bT^2}(N,N; \pi^{-1}(P))$ can be shown to be a certain completion of the wrapped Floer cohomology of $N$ which is straightforward to compute and so prove \eqref{eqHFneq0}. It should not be hard to generalize the locality Theorem \ref{thmMainIntro} to give an isomorphism
\begin{equation}
HF^*_Y(N_1\cap Y,N_2\cap Y;K)\simeq HF^*_M(N_1,N_2;K)
\end{equation}
whenever $Y\subset M$ is a complete embedding, $N_1,N_2\subset M$ are properly embedded, and for a compact subset $K\subset Y$ the Lagrangians $N_1\cap Y\setminus K, N_2\cap Y\setminus K$ can be displaced a distance $\epsilon$ from each other by a Lipschitz Hamiltonian with respect to a metric determined by an almost complex structure for which $Y$ is geometrically bounded. We deduce from this that 
\begin{equation}\label{eqHFneq02}
HF^*_{X}(L_1,L_1; \pi^{-1}(P))\neq 0 
\end{equation}
for any polygon $P\subset B$ meeting $\pi(L_1)$. Strictly speaking, this requires discussing potential obstructions, but we ignore this at the level of rigor of the current discussion. Since \eqref{eqHFneq02} holds for arbitrary  $P\subset B$ meeting $\pi(L_1),$ Conjecture \eqref{ConjSupport} follows. 

\begin{remark}
Interpreting this in terms of mirror symmetry, one expects that for any torus fiber $\pi_x$, with $x\in\pi(L),$ there exists a bounding cochain $b$ such that   $HF^*(L,(\pi_x,b))\neq 0$. Conceptually, $L$ is mirror to a certain holomorphic plane $C$, and the set of objects of the Fukaya category supported on the torus fiber $\pi_x$ contains objects that are mirror to points of $C$. This is well known to experts in the case where $X$ contains a single node. The more general case can be reduced to the case of a single node by Theorem \ref{tmCompEmbLagTails} together with an appropriate variant of the locality theorem \ref{thmMainIntro}. 

%A somewhat more convoluted but perhaps also more illuminating argument is the following. First one considers the case where $X$ contains a single nodal point. Mirror symmetry for this case can be made sufficiently explicit to show that for any torus fiber $\pi_x$ for $x\in\pi(L)$ there exists a bounding cochain $b$ such that the pair $(\pi_x,b)$ is split generated by $L$. Conceptually, $L$ is mirror to a certain holomorphic plane $C$, and the set of objects of the Fukaya category supported on the torus fiber $\pi_x$ contains objects that are mirror to points of $C$.  The split generation of  $(\pi_x,b)$ by $L$ continues to hold in the local Fukaya category associated with appropriate subsets $K\subset X$. Theorem \ref{tmCompEmbLagTails} together with the locality theorem \ref{thmMainIntro} applied to the open string case would the imply that for general $X$ there is a compact subset $K$ and a bounding cochain so that $(\pi_x,b)$ is locally split generated by $L$. In particular, $\pi_x$ cannot be  displaced from $L$. 
\end{remark}
We can use Conjecture \ref{ConjSupport} to produce the complete embeddings version of \cite{shende} on distinguishing exact Lagrangian tori in exact symplectic cluster manifolds.  Given two complete embeddings $i_1,i_2$ of, say, $T^*T^2$ inside $X$ each  obtained as in Theorem \ref{tmCompEmbLagTails} by  choosing a cluster presentation and removing Lagrangian tails we may ask whether the two complete embeddings are related by Hamiltonian isotopy. Equivalently we may ask the same question for their complements. This leads us to the question of distinguishing Lagrangian tails associated with possibly different fibrations $\pi$ up to Hamiltonian isotopy.

%We have defined the notion of Lagrangian tail given an almost toric fibration. For the remainder of this subsection given a symplectic cluster variety  we call a Lagrangian disc $L$ in $X$  a Lagrangian tail if there exists an almost toric fibration $\pi$ for which it is a Lagragian tail. 

Conjecture \ref{ConjSupport} has the following immediate corollary.

\begin{corollary}\label{corNonHamIsotopicTails}
Let $X$ be a symplectic cluster manifold, let $\pi_1,\pi_2:X\to B$ be {nodal Lagrangian torus fibrations} of cluster type, and let $L_1,L_2$ be Lagrangian tails associated with  $\pi_1,\pi_2$ respectively. Suppose  $\pi_1,\pi_2$ coincide on some open set $K=\pi_1^{-1}(V)=\pi_2^{-1}(V)$.  Then if $\pi_1(L_1)\cap V\neq \pi_2(L_2)\cap V$, $L_1$ is not Hamiltonian isotopic to $L_2$. 
%Let $X,\pi$ be a symplectic cluster manifold with complete nodal fibration $\pi:X\to B$. Let $\cR_1,\cR_2$ be eigenray diagrams that arise by removal of Lagrangian tails from $X$, then if the symplectic embeddings provided by Theorem \ref{tmCompEmbLagTails}  of $M_{\cR_1}$ and $M_{\cR_2}$ into $X$ are related by Hamiltonian isotopy then the diagrams $\cR_1$ and $\cR_2$ are equivalent.

\end{corollary}

This corollary can be applied to distinguish complete embeddings. We expand on how this works in Example \ref{ExFiveCharts}. In this case the corollary shows that the 4 initial Lagrangian tails corresponding to $l_x^\pm,l_y^\pm$  together with the scattered Lagrangian are pairwise Hamiltonian non-isotopic. Indeed, for the initial rays this is an immediate consequence of Conjecture \ref{ConjSupport}. We now consider a pair consisting of an initial tail and the scattered one. According to Proposition \ref{prop-nodal-slide} we can take the nodal slide involved in producing the scattered ray to be supported on arbitrarily small neighborhood of the segment containing the intersection point and one of the nodes. Thus $V$ in the corollary can be taken to be the complement of the closure of this neighborhood and the result follows. %{\color{gray}Interestingly, the two scattered rays cannot be distinguished with the help of the Corollary. However they are distinguished by  intersection Floer homologies with the initial tails.  Having established that the tails are pairwise non-isotopic we immediately deduce that the five embeddings of Example \ref{ExFiveCharts} are not Hamiltonian isotopic. We also deduce that the six embeddings of the nodal surface $M_1$ obtained by Theorem \ref{tmCompEmbLagTails} are all Hamiltonian non-isotopic.  {\color{red}see emails}}

\subsubsection{Some three dimensional examples}

{The ideas presented so far are not limited to dimension $4$. There are open symplectic manifolds in any even dimension for which analogues to Theorem \ref{tmCompEmbLagTails} hold. Here we content ourselves by analyzing one of the two classes of $6$ dimensional analogues to a symplectic cluster manifold with global $S^1$ symmetry (e.g., one with a single node). Specifically, we consider symplectic $3$-manifolds carrying Lagrangian torus fibrations $\cL:M\to \bR^3$ which
\begin{itemize}
 \item have global $\bT^2$ symmetry generated by the last two coordinate functions $x_1$ and $x_2$,
 \item satisfy an appropriate completeness condition, and,
 \item has singular values along a $1$-dimensional graph $\Delta\subset \bR^3$ which lies in the plane $x_0=0$.
 \end{itemize} For a complete definition of the class we consider see Definitions \ref{def-3d-semitoric} and \ref{dfHorizontalCompleteness}. Important examples of this geometric setup is given by Gross fibrations from \cite[Theorem 2.4]{Gross00}. We describe in detail two special cases of Gross' construction in Examples \ref{ex-gross1} and \ref{ex-gross2}.
 
 Let $P^{\pm}= \bR_{\pm}\times \Delta\subset\bR^3$. Let $C^{\pm}\subset M $ be a $\bT^2$ invariant lift of  $P^{\pm}$. Note that $C^{\pm}$ is a stratified coisotropic subspace. The pair of spaces $C^{\pm}$ are the analogues of the pair of Lagrangian tails emanating from the singular point in the symplectic cluster manifold with a single node. We prove the following theorem in Section \ref{Sec:GrossFibrations} }

\begin{theorem}\label{tmSemiToricCo}
There are symplectomorphisms $\iota^{\pm}:M\setminus C^{\pm}=T^*\bT^3$. Moreover, $\iota^{\pm}$ can be chosen to intertwine $\cL$ with the standard projection $T^*\bT^3\to\bR^3$ on the complement of any open neighborhood of $P^{\pm}$. 
\end{theorem}
{
We briefly comment on the significance of this result {in relation to the locality theorem}. The Gross fibrations have been studied from the point of view of SYZ mirror symmetry in \cite{Gross00, aak,CSN}. It has been shown that, up to real codimension 4, the SYZ mirror is a conic bundle which is glued together from a pair of algebraic tori $(\Lambda^*)^3$ over the Novikov field according to some wall crossing formula. Theorem \ref{tmSemiToricCo} together with our locality theorem gives a reinterpretation of this SYZ construction using relative symplectic cohomology. Explaining this further is outside the scope of the present paper.}

\subsubsection{Basic idea}\label{sss-idea}
The method  we use to prove the embedding theorems above is an anti-surgery operation on symplectic manifolds equipped with {nodal} Lagrangian submersions. We can also define a surgery operation, which we develop in Section \ref{ss-surgery} and give applications, but our focus is on the anti-surgery operation in the introduction. 

As a toy version of the anti-surgery operation we describe below, the reader may  consider the case of removing the ray $\gamma=\bR_{\geq 0}\times \{1\}$ from the cylinder $\bR\times S^1.$ The operation would start with the standard Lagrangian {circle} fibration  $\bR\times S^1\to \bR$, produce a Lagrangian {circle} fibration $\bR\times S^1\setminus \gamma\to \bR$ {so that the integral affine structure induced on $\bR$ is the standard one}, and {hence, in particular,} prove that $\bR\times S^1\setminus \gamma$ is symplectomorphic to $\bR\times S^1$.\\

To explain the basic idea in the context of symplectic cluster manifolds, we focus on the embedding described in Equation \eqref{eqCompEmbLagTails} of Theorem \ref{tmCompEmbLagTails} in the $S=Crit\pi$ case. The procedure starts with a nodal Lagrangian torus fibration $\pi:X\to B$ that is part of a cluster presentation and modifies it in a neighborhood $U_p$ of each Lagrangian tail $L_p$. The end result is a new Lagrangian torus fibration $$\pi^{new}:X\setminus \cup_{p\in Crit\pi}L_p\to B$$ which is 
\begin{itemize}
\item proper, 
\item has no singularities, and, crucially, 
\item induces on $B$ an integral affine structure which is isomorphic to the standard one on $\bR^2$.
\end{itemize}

We describe this now in more detail. {We show in Proposition \ref{prpClustEquiv}  that we may identify the base $B$ with $\bR^2$ so that with respect to the induced integral affine structure, the identity map is a PL homeomorphism which is an affine isomorphism on the complement of the projections $l_p=\pi(L_p)$ of the tails. For simplicity, we assume that the rays of the form $l_p$ for $p\in Crit \pi$ are pairwise disjoint.} We focus on one Lagrangian tail at a time. A local model for the fibration in a neighborhood of such a tail is as follows. Consider $\mathbb{C}^2$ with its standard Kahler structure and denote the complex coordinates by $z_1$ and $z_2$. The Hamiltonian function  $\mu(z_1,z_2):=\pi(|z_1|^2-|z_2|^2)$ generates a $S:=\mathbb{R}/\mathbb{Z}$ action on $\mathbb{C}^2$ by $\theta\cdot (z_1,z_2)=(e^{2\pi i\theta}z_1, e^{-2\pi i\theta}z_2)$. We have a smooth map $Hopf: \mathbb{C}^2\to \mathbb{C} \times\mathbb{R}$ defined by 
\begin{equation}
(z_1,z_2)\mapsto \left(2\pi z_1z_2, \mu(z_1,z_2)\right).
\end{equation}
 The fibers of $Hopf$ are precisely the orbits of the $S$-action. Let  
 \begin{equation}
 \pi^{st}=(Re(z_1z_2),\mu(z_1,z_2)), 
 \end{equation}
  and let 
  \begin{equation}
  L^{st}:=\{Re(z_1z_2)\geq 0\}\cap \{Im(z_1z_2)= 0\}\cap \{|z_1|^2-|z_2|^2=0\}.
  \end{equation}
Then there is an open neighborhood $W_p$ of $L_p$, an $S$-equivariant neighborhood $W^{st}$ of $L^{st}$ and a symplectomorphism $\psi:W_p\to W^{st}$ mapping $L_p$ to $L^{st}$ and intertwining $\pi$ with $\pi^{st}$. 

%Denote by $Z^{st}=\bC^2/S$ and by $q^{st}:\bC^2\to Z^{st}$ the quotient map. 
%Then there is a homeomorphism $Z^{st}\simeq\bC\times\bR$ intertwining $q^{st}$ with $Hopf$. {\color{red} I debated whether to remove the discussion of  $q^{st}$ and $Z^{st}$ from here entirely as they are not used again, but I decided to not take a risk. I leave this decision to you}
We now make the following crucial observation: \emph{if $f:\bC\times\bR\setminus Q\to\bR$ is a smooth submersion with $Q$ a closed subset containing $(0,0)$, then the fibers of the map $(f\circ Hopf,\mu):\bC^2\setminus Hopf^{-1}(Q)\to \bR^2$ are Lagrangian submanifolds}. 

We apply this as follows. Let $l:=Hopf(L^{st})$ and let $U=Hopf(W^{st})$ .  By being a little careful in the choice of $W^{st}$ we can ensure that there is a diffeomorphism  $
\Phi: U-l\to U$ which preserves the projection to $\mathbb{R}$ and is the identity near the boundary of $U$. Let $h:\bC\times\bR\to \bR$ be the map $(z,x)\mapsto Re (z)$. Then the map $$\pi':=(h\circ \Phi\circ Hopf,\mu):W^{st}\setminus L^{st}\to U$$ is a Lagrangian submersion with no singularities. Moreover $\pi'=\pi^{st}$ near the boundary of $W_p$.  

We may thus define $\pi^{new}|_{W_p}=\pi'\circ\psi$ for each of the critical points of $\pi$, and $\pi^{new}=\pi$ on  $X\setminus\cup_{p\in Crit\pi}W_p$. The resulting Lagrangian submersion is readily seen to be proper. A graphic illustration of the modification of the foliation near one of the Lagrangian tails is given in Figure \ref{fig-node-removal}. The result is a proper Lagrangian submersion $\pi^{new}:X^{new}:=X\setminus \cup_{p\in Crit\pi} L_p\to B$. 

\begin{figure}
\includegraphics[width=0.6\textwidth]{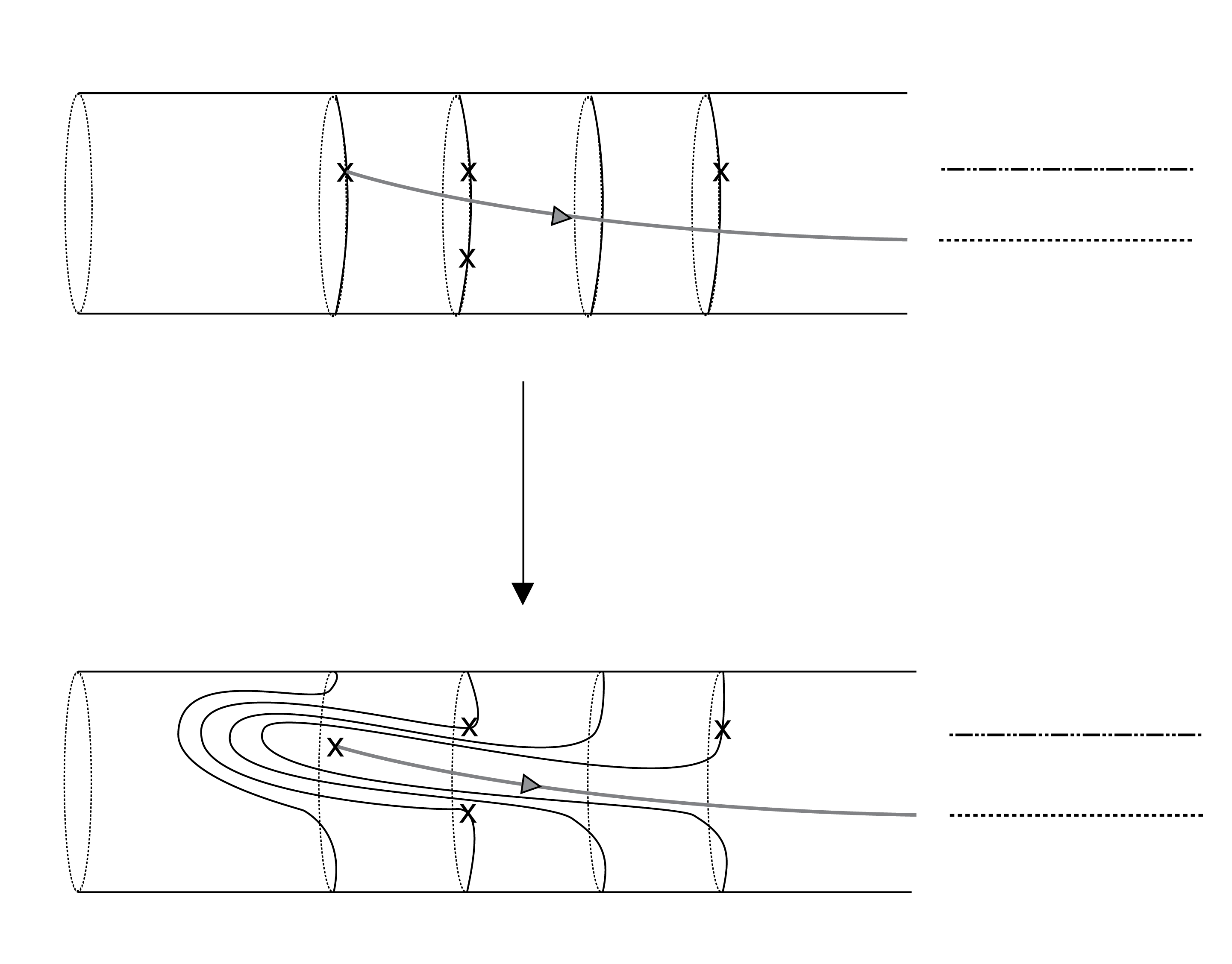}
\caption{The foliation before (top) and after (bottom) an anti-surgery (along a Lagrangian tail) on the reduced space containing the $S$ quotient of the tail (depicted as the smooth ray going to infinity with an arrow on it). Note that at the bottom image, this ray has been removed. }
\label{fig-node-removal}
\end{figure}

To conclude, we need to show that the integral affine structure induced on $B$ by $\pi^{new}$ via the Arnold-Liouville Theorem makes it isomorphic to $\bR^2$ with its standard integral affine structure.  %\red{Using simply connectedness and a classical result of Auslander and Markus \cite{Auslander}}, a sufficient condition is  that the integral affine structure is \red{geodesically} complete\footnote{Without this, $B$ might be integral affine isomorphic to some proper subset of $\bR^2$, or, worse, of the universal cover of $\bR^2\setminus\{0,0\}$.}. {\color{red}Moreover, to establish this it suffices that if we take the point $x\in B$ which is the meeting point of all the lines through the eigenrays associated with $\pi$, then all the affine rays emanating from $x$ in any direction exist for all non-negative times.} 
To proceed,  recall that $L_p$ projects under $\pi$ {to an eigenray $l_p$}, i.e. an affine ray whose direction is fixed under the integral affine monodromy around $\pi(p)$. In fact, there is an integral affine coordinate $i:\pi(W_p)\to\bR$,  such that $i$ generates the local $S$-action and such that $\pi(L_p)$ is contained in a level set of $i$. The function $i$ is determined by the equation $i\circ\pi=\mu\circ\psi_p$.  Without loss of generality we choose $W_p\subset B$ { so that} the non-empty level sets of $i$ are affine isomorphic to half infinite intervals. 

{The modifications do not affect the integral affine structure outside of the the neighborhoods $\pi(W_p)$, so we focus on analyzing what happens in the  neighborhoods $\pi(W_p)$}. Note that $i\circ\pi^{new}|_{W_p\setminus L_p}=i\circ\pi_{W_p\setminus L_p}$ generates a $1$-periodic flow and so is still an integral affine coordinate with respect to $\pi^{new}$.  Let us now show that, with respect to $\pi^{new}$, the non-empty level sets of $i$ are integral affine isomorphic to half infinite intervals. Equivalently, we need to show that if $\gamma$ is any simple loop in a fiber of $\pi^{new}$ which is transverse to the $S$-orbits then the cylinder $C$ traced out by transporting $\gamma$ along any non-empty level set $i^{-1}(c)$ has infinite area. This is readily seen by comparing the area of such a cylinder $C$ with that of a cylinder  $C'$ traced by transporting $\gamma$ along fibers of the original fibration $\pi$ over the same level set $i^{-1}(c)$. Figure \ref{fig-node-removal} extended to a larger portion of the reduced space so as to include some of the unchanged fibers could help visualizing this. {The effect of the modification in the neighborhood $W_p$ is thus to make it affine isomorphic to the corresponding neighborhood in $\bR^2$. Doing this for all the tails, the integral affine structure induced by $\pi^{new}$ on $B$ becomes isomorphic to $\bR^2$.} 

%{\color{red}A particular consequence is that there is a PL homeomorphism between the integral affine structures induced by $\pi$ and $\pi^{new}$ which is identity away from some open neighborhoods of the  $l_p$ and is an affine isomorphism on the complement of the rays $l_p$. Indeed we may extend the isomorphism defined by identity on the complement of the neighborhoods $\pi(W_p)$ to such a PL homeomorphism by identifying the level sets of $i$ on each side inside the $\pi(W_p)$ in the obvious way.  Thus the ray from the central point $x$ to any point not on an eigenray still extends indefinitely. In addition, with respect to the affine structure induced from $\pi^{new}$ rays from $x$ towards the nodal values also extend indefinitely. The claim now follows.}

%\red{Now one can argue that the desired \red{geodesic} completeness holds in the following way. Assume for the sake of contradiction that there is a forward incomplete geodesic $\gamma$. Using that there is no monodromy of the affine connection, we can easily prove that $\gamma$ can hit $\cup_{p\in Crit\pi} l_p\subset B$ at most finitely many times. This is because if it intersects $l_p$ it needs to spend at least some $\epsilon_p>0$ time in $\pi(W_p)$. Therefore, without loss of generality we can assume that $\gamma$ never intersects $\cup_{p\in Crit\pi} l_p$. We easily obtain a contradiction to the weak geodesic completeness of the original nodal integral affine structure on $B$. COMPARE WITH PROP 7.14.} 

Let us also discuss the application of the surgery operation that we have in mind. Recall that a Looijenga interior $U$ is a log Calabi-Yau surface with maximal boundary. By \cite{ghk}, $U$ admits a toric model. Fixing a toric model, we obtain a decomposition of $U$ into an open dense $(\mathbb{C}^*)^2$ and some disjoint union of exceptional curves coming from the Kaliman modifications.  Assume that the holomorphic volume form $\Omega$ of $U$ is normalized so that it is equal to $dlog(x)\wedge dlog(y)$ on $(\mathbb{C}^*)^2.$ Using integrable surgery one can prove the following result, which could be of independent interest. 

\begin{theorem}\label{thm-Loo}
$(U,Im(\Omega))$ is a symplectic cluster manifold. {In particular, it is geometrically of finite type}.
\end{theorem}

The sketch of the proof is given in Remark \ref{rem-Loo} in the first non-trivial example. Note that with this symplectic form the exceptional curves become Lagrangians. In the case exceptional curves are all complex planes the theorem follows immediately from doing integrable surgery to the standard Lagrangian torus fibration on $$((\mathbb{C}^*)^2, Im(\Omega))=(T^*T^2, d\lambda_{taut})$$ along the exceptional curves. When some of the exceptional curves are a chain of $\bC P^1$'s followed by a complex plane, we {can} modify them to pairwise disjoint Lagrangian {tails} emanating from each of the double points in the exceptional curves {or use a straightforward generalization of the integrable surgery}.

\begin{remark}
In fact, all symplectic cluster manifolds can be presented as $(U,Im(\Omega))$ from Theorem \ref{thm-Loo}. Still, we find the presentation that relies on eigenray diagrams and nodal Lagrangian torus fibrations more useful for symplectic geometry.

Incidentally, a similar definition of a symplectic cluster manifold in terms of combinatorial data on $\mathbb{R}^{2n}$ can be given in all $n\geq 1$. This already appears in an unpublished manuscript of Kontsevich-Soibelman. Since we restrict to 4 dimensional symplectic cluster manifolds in this paper, we omit the definition.
\end{remark}

In \cite[Lemma 4.3]{bardwell}, which came out only several days after the first version of our preprint appeared, the authors construct special Lagrangian fibrations $\pi'_\epsilon:U\to B$ on $(U,\Omega)$ for a family of Kahler forms $\omega_\epsilon$. In particular, the $\pi'_\epsilon$'s are all nodal Lagrangian torus fibrations with respect to $Im(\Omega)$. This appears to be an alternative construction to the one from Theorem \ref{thm-Loo}. It will be interesting to compare the two approaches more thoroughly in the future. %It is unclear to us whether the induced nodal integral affine structures (the complex one in their context) on $B$ are complete, but it appears to be close, especially as $\epsilon\to 0$. We added this paragraph to indicate briefly one of the relationships between the two works, a more thorough comparison is left for future.

\subsection{Relation to wall crossing and  mirror symmetry}\label{ss-WC}

Let us first illustrate in the simplest example how the wall-crossing phenomenon makes an appearance in our framework of locality isomorphisms for complete embeddings.

Consider the eigenray diagram with a single multiplicity one node at $(0,0)$ with its ray being the non-negative real axis, and take compatible a nodal Lagrangian torus fibration $\pi_1: M_1\to B_1$. Note that $\pi_1$ has a single focus-focus singularity, which we assume is in the fiber $\pi_1^{-1}(b)$. There are two monodromy invariant rays of $b\in B_1$ and choosing arbitrary Lagrangian tails above each of them we obtain two non-Hamiltonian isotopic symplectic embeddings $$\Phi_i: T^*T^2\to M_1, \text{ for } i=1,2$$ using Theorem \ref{tmCompEmbLagTails}. In fact using the slightly stronger statement in Theorem \ref{prop-node-removal}, for any connected compact convex polygon $P\subset B_1$ that is disjoint from both eigenrays of $b$, we can choose the embeddings $\Phi_i$ such that there is a convex polygon $Q\subset B_0$ where $\Phi_i$ are both fiber preserving over $Q$ and induce bijections $Q\simeq P$.

Assume that we choose the homotopy class of trivializations of the canonical bundles such that the fibers of $\pi_0$ and $\pi_1$ are Maslov zero Lagrangians, see the beginning of Section \ref{ss-rel-koso} for a discussion.

Our results show that there are two locality isomorphisms $$(\Phi_i)_*: SH^*_{T^*T^2}(\pi_0^{-1}(Q))\to SH^*_{M_1}(\pi_1^{-1}(P)),  \text{ for } i=1,2.$$

Let us also note that an extension of the computation of Seidel from \cite{seidel}  Equation (3.4) shows that $SH^0_{T^*T^2}(\pi_0^{-1}(Q))$ is isomorphic to the algebra $KS(Q)$ of non-archimedean analytic functions convergent on $Q$ (a la \cite{koso}, see Section \ref{ss-rel-koso}). The main point which we return to in a future paper is that the automorphism $((\Phi_2)_*)^{-1}(\Phi_1)_*$ of $KS(Q)$ is not the identity, not even monomial. It is given by the well-known wall crossing transformation as in e.g. pg. 6 of \cite{hacking}.

%We note that this story extends to the Gross fibration discussed above (and in fact many other examples), opening the way to a Floer theoretic understanding of the constructions of \cite{aak}.

We end this section by noting that in the brief and sketchy Section \ref{ss-ms}, we discuss mirror symmetry for symplectic cluster varieties. Our goal is to illustrate the kind of results we will be aiming for in future work and is limited in scope and detail.

\subsection{Further applications to symplectic topology}

We say that a geometrically bounded symplectic manifold $M$ is \emph{$SH$-invisible} if for each compact $K\subset M$ we have that $SH^*_M(K)\otimes_{\Lambda_{\geq 0}}\Lambda=0$. By unitality of restriction maps in relative symplectic cohomology \cite{tonkonog}, this condition is equivalent to $SH^*_M(K_i)\otimes_{\Lambda_{\geq 0}}\Lambda=0$, $i=1,2\dots$ for an exhaustion $K_1\subset K_2\subset \ldots $ of $M$ by compact subsets. {We now list some examples and non-examples without detailed proofs.}
\begin{itemize}
\item  A closed $M$ cannot be $SH$-invisible, since $SH^*_M(M)\otimes_{\Lambda_{\geq 0}}\Lambda\simeq H^*(M;\Lambda)$.
\item { If $M$ contains a Floer theoretically essential Lagrangian submanifold it cannot be $SH$-invisible. We expect this to be an immediate consequence of the unitality of closed-open maps as in \cite{tonkonog}, but omit an actual proof.}
\item If $M$ is a finite type Weinstein domain whose skeleton is stably displaceable,  then a neighbourhood of the skeleton is stably displaceable. Using the Liouville flow, it follows that in fact any precompact neighborhood of the skeleton is stably displaceable. Finally,  by \cite{varolgunesthesis} (compare with \cite{kang}) $M$ is $SH$-invisible. In particular,  this holds for flexible Weinstein domains. Using the technique used to prove Corollary 3.9 in \cite{Kyler}, one can also show that every subflexible Weinstein domain is $SH$-invisible.
\item By the Kunneth formula of \cite{groman} {and a straightforward analysis of the units}, the product of a geometrically bounded symplectic manifold with an $SH$-invisible one is $SH$-invisible.
\item If $c_1(M)=0$ and $M$ carries a Hamiltonian $S^1$ action whose moment map is bounded from below, then $M$ is $SH$-invisible (\cite[Lemma 10.10]{groman}).  In particular, by \cite[Example 10.12]{groman}  this holds for any toric Calabi-Yau manifold.
\item If $M$ is a non-aspherical smooth manifold and $\sigma$ is a non-aspherical two form on $M$ then the twisted cotangent bundle $(T^*M,d\lambda+\pi^*\sigma)$ is $SH$-invisible \cite{GromanMerry}.
\end{itemize}

We say that $M$ has \emph{homologically finite torsion} if it has an exhaustion by compact subsets with homologically finite torsion. By the discussion in subsection \ref{sss-on-tor} this holds for $T^*T^n$ and $ \bC^n$.

An immediate corollary of the main theorem is
\begin{corollary}
Let $X\hookrightarrow Y$ be a complete embedding and  suppose $X$  as homologically finite torsion. Then if $Y$ is $SH$-invisible so is $X$.
\end{corollary}
\begin{corollary}
An $SH$-invisible manifold $X$ does not admit a complete embedding of $T^*T^n$.
\end{corollary}
\begin{remark}
We expect the hypothesis on $X$ to be fully liftable. Thus an $SH$-invisible manifold $M$ should not admit any equidimensional symplectic embedding $T^*Q\to M$, where $Q$ is a closed manifold. In fact, here we believe one can use the complete embedding to show that $Q\subset X$ is tautologically unobstructed, and hence Floer theoretically essential, leading to an alternative proof.
\end{remark}
%A particular corollary is
%\begin{corollary}
%If $M$ is algebraically sub-critical and $L$ is any Lagrangian embedding then for any Riemannian metric $g$ on $L$ there is an $r=r_{L,g}$ such that that the if the embedding of $L$ extends to the disc bundle of radius $r'$ over $L$  then $r'\leq r$.
%\end{corollary}

%A similar statement can be made were $L$ is replaced by the skeleton of a Liouville domain and the choice of $g$ is replaced by the choice of the ambient Liouville domain.
\subsection{Structure of the paper}
In Section \S\ref{Sec:Overview} we give an overview of the Hamiltonian Floer theory package for truncated Floer cohomology on geometrically bounded manifolds. In Section \S\ref{sec-loc-iso}
we introduce the notion of geometrically finite symplectic manifold and formulate a more refined version of Theorem \ref{thmMainIntro} involving truncated relative symplectic cohomology. In Section \S\ref{SecSeparatingFlDt} we introduce the notion of separating Floer data and develop the $C^0$ estimates needed for proving our locality results. The proof of locality for truncated relative $SH$ is carried out in Section \S\ref{secMainProof}. In Section \S\ref{Sec:Torsion} we discuss lifting the locality result from truncated relative $SH$ to relative $SH$ under the homological finiteness assumption and prove Theorem \ref{thmMainIntro}. In Section \S\ref{sec:EignenRayCompleteEmb} we develop the theory of symplectic cluster manifolds and prove Theorem \ref{tmCompEmbLagTails}. In section \S\ref{Sec:GrossFibrations} we prove Theorem \ref{tmSemiToricCo}. Appendix \S\ref{AppDissRev} summarizes the results of \cite{groman} that are used.
\subsection{Acknowledgements}
Y.G. was supported by the ISF (grant no. 2445/20). U.V. was supported by the ERC
Starting Grant 850713, and by the T\"{U}B\.{I}TAK 2236 (CoCirc2) programme with a grant numbered 121C034. We thank Mohammed Abouzaid for helpful discussions.

\section{Overview of truncated symplectic cohomology}\label{Sec:Overview}
A symplectic manifold $(M,\omega)$ is said to be \emph{geometrically bounded} if there exists an almost complex structure $J$, a complete Riemannian metric $g$ and a constant $c>1$ so that % the induced metric $g_J$ satisfies
\begin{itemize}
\item $g_J$ is $c$-equivalent to $g$. That is
\begin{equation}
\frac1{c}|v|_{g_J}<|v|_g<c|v|_{g_J}
\end{equation}
for any tangent vector $v$.
\item
\begin{equation}
\max\left\{\left|\Sec_{g}\right|,\frac1{\inj_{g}}\right\}<c.
\end{equation}
\end{itemize}
For a such a $J$ and constant $c$ we say that \emph{$J$ is geometrically  bounded by $c$} or {$c$-bounded}.

\begin{remark}\label{remVolInf}
The bounds on sectional curvature and radius of injectivity imply a uniform bound from below on the volume of balls of radius $\epsilon$ with respect to $g$ for $\epsilon>0$ small enough by standard comparison estimates. This implies a related bound for the metric $g_J$. Thus by the completeness  assumption we obtain that \emph{a geometrically bounded symplectic manifold is either closed or has infinite volume}.
\end{remark}

We now give a quick review of Floer theory on geometrically bounded symplectic manifolds. This is to be expanded upon in the main body of this section and Appendix \S\ref{AppDissRev}. On a geometrically bounded symplectic manifold there exists a set of pairs $(H,J)\in C^{\infty}(S^1\times M,\bR)\times\cJ(M,\omega)$ of \emph{dissipative Floer data}. These satisfy two assumptions that are independent of each other : \begin{enumerate} \item geometric boundedness of the almost complex structure $J_H$ on $\bR\times S^1\times M$ obtained by the Gromov trick, \item loopwise dissipativity.\end{enumerate}
We recall the definition in detail in Appendix \S\ref{AppDissRev}. Such pairs $(H,J)$ satisfy the required $C^0$-estimates for the definition of the Floer differential. We stress that dissipativity is a property of the pair  $(H,J)$.

In  Appendix \S\ref{AppDissRev} we also recall notions of dissipative Floer continuation data, dissipative families of dissipative Floer continuation data, and also families of dissipative Floer data on punctured Riemann surfaces  used in the constructions of the algebraic structures. These notions  allow one to generalize the \emph{Hamiltonian Floer theory package} that is well-established for closed symplectic manifolds to symplectic manifolds which are merely geometrically bounded.

Our end goal in this section is to use the methods of \cite{groman} to define what we will call the truncated symplectic cohomology of a compact set in a geometrically bounded symplectic manifold.

We make the assumption $c_1(M)=0$ and fix the homotopy class of a trivialization of $\Lambda^{n}T_{\mathbb{C}}M$, where $n=\frac{\text{dim}(M)}{2}$. This assumption holds in our intended applications and is only made for convenience  to not distract the reader with discussions of various possibilities of different types of gradings.  In particular, we expect no difficulty in generalizing our methods to general symplectic manifolds, where Hamiltonian Floer theory requires virtual techniques.

\subsection{Floer cohomology for dissipative data}\label{FloerPackage}
Denote by $\Lambda_{\geq 0}$ the Novikov ring
\[
\Lambda_{\geq 0}:=\left\{\sum a_iT^{\lambda_i}|\lambda_i\in\mathbb{R}_{\geq 0},a_i\in\bZ,\lim_{i\to\infty}\lambda_i=\infty\right\}.
\]

Let $H:\bR/\bZ\times M\to\bR$ be a smooth function and let $J$ be an $\omega$ compatible periodically time dependent almost complex structure. We assume that $(H,J)$ is a dissipative Floer datum which is regular for the definition of Floer cohomology. That is, all the $1$-periodic orbits are non-degenerate, all the Floer moduli spaces are cut out transversally and all sphere bubbling is of codimension $\geq 2$. The existence and abundance of such data without the dissipativity condition is well established \cite{HoferSalamon}. As discussed in Appendix \S\ref{AppDissRev} the dissipativity condition is open in
some natural topology and non-empty so does not impose a severe  restriction in this regard. %The same remarks go for the more general Floer problems discussed below. In fact this is easier for problems with fixed endpoints since loopwise dissipativity is inherited from the endpoints. We shall not comment on this further.
%To avoid some minor technicalities in the definition of untruncated Floer cohomology we  assume for the remainder of this subsection that the set $\Per(H)$ of $1$-periodic orbits of $H$ is finite. This assumption is not made in subsequent subsections when we consider truncated Floer cohomology. %We call such $(H,J)$ \textit{admissible}.

Given a pair of elements $\gamma_1,\gamma_2\in\Per(H)$ a \emph{Floer trajectory} from $\gamma_1$ to $\gamma_2$ is a solution $u$ to Floer's equation
\begin{equation}
\partial_su+J(\partial_t-X_H)=0,
\end{equation}
such that $\lim_{s\to-\infty}u(s,t)=\gamma_1(t)$ and $\lim_{s\to\infty}u(s,t)=\gamma_2(t)$. We define the \emph{topological energy} of $u$ by the formula
\begin{equation}
E_{top}(u):=\int u^*\omega+\int_{\bR/\bZ}\left(H_t(\gamma_2(t))-H_t(\gamma_1(t))\right)dt.
\end{equation}
It is a fact that
\begin{equation}\label{eqGeoTopEn1}
E_{top}(u)=\int\|\partial_su\|^2ds\geq 0.
\end{equation}
The quantity on the right hand side is referred to as the \emph{geometric energy} $E_{geo}(u)$. According to proposition \ref{tmSummary00} the dissipativity assumption implies the set $\overline{\cM}(\gamma_1,\gamma_2,E)$ of Floer trajectories from $\gamma_1$ to $\gamma_2$ of energy at most $E$ is contained in an a priori compact set $K=K(E,\gamma_1\gamma_2)$.

Our choice of trivialization  for $\Lambda^{n}T_{\mathbb{C}}M$ gives rise to a grading $i_{CZ}:\Per(H)\to\bZ$ using the induced trivialization of $\gamma^*TM$.  Assuming the index difference is  $i_{CZ}(\gamma_2)-i_{CZ}(\gamma_1)=1$, the regularity assumption implies the quotient
\[
\cM(\gamma_1,\gamma_2,E)=\overline{\cM}(\gamma_1,\gamma_2,E)/\bR
\]
 is a compact oriented $0$-dimensional manifold.

Because $H$ might have infinitely many $1$-periodic orbits (even in a fixed degree), we will be need the following definition. Let $\{\Lambda_{\geq 0}\langle\gamma_i\rangle\}_{i\in I}$ be a collection of free rank one $\Lambda_{\geq 0}$-modules with given generators $\gamma_i$. We define their completed direct sum as follows:
\begin{equation*}
\widehat{\oplus}_{i\in I}\Lambda_{\geq 0}\langle\gamma_i\rangle:=\left\{\sum_{i\in I}\alpha_i\gamma_i\mid\alpha_i\in \Lambda_{\geq 0},\quad \# \{i\in I \mid val(\alpha_i)<R\}<\infty\quad \forall R>0\right\}
\end{equation*}

We may thus proceed to define the Floer complex  $CF^*(H,J)$. As  a $\bZ$-graded $\Lambda_{\geq 0}$-module $CF^*(H,J)=\bigoplus_{k\in \bZ} CF^k(H,J)$, 
\begin{equation}\label{eq-ham-completed}
CF^k(H,J):=\widehat{\oplus}_{\gamma\in Per^k(H)}\Lambda_{\geq 0}\langle\gamma\rangle
\end{equation}
where $Per^k(H)$ are the $1$-periodic orbits of $H$ with degree\footnote{This means Conley-Zehnder index computed using the fixed grading datum on $M$ plus $n=\frac{\text{dim}(M)}{2}$ for us.}  $k$. For each periodic orbit $\gamma$ we fix an isomorphism
\begin{equation}
\Lambda_{\geq 0}\langle\gamma\rangle\simeq o_{\gamma}\otimes \Lambda_{\geq 0},
\end{equation}
where $o_{\gamma}$ is the orientation line associated with $\gamma$ as in Definition 4.19 of \cite{Abouzaid14}. Each $u\in \cM(\gamma_1,\gamma_2)$ for $\gamma_1,\gamma_2$ with index difference $1$ induces an isomorphism
\[
d_u:o_{\gamma_1}\to o_{\gamma_2}.
\]
The differential is defined by
\begin{equation}
d|_{\Lambda_{\geq 0}\langle\gamma_1\rangle}=\sum_{\gamma_2:i_{CZ}(\gamma_2)-i_{CZ}(\gamma_1)=1}\sum_{u\in \cM(\gamma_1,\gamma_2)}T^{E_{top}(u)}d_u.
\end{equation}
It follows from \eqref{eqGeoTopEn1} that the Floer complex is indeed defined over the Novikov ring $\Lambda_{\geq 0}$. Since all the moduli spaces with bounded energy are contained in an a priori compact set, it follows from standard Floer theory that the differential indeed defines a map of completed direct sums and that it squares to $0$.

For a pair $F_1=(H_1,J_1), F_2=(H_2,J_2)$, a \emph{monotone homotopy from $F_1$ to $F_2$} is a family $(H^s,J^s)$ which coincides with $(H_1,J_1)$ for $s\ll0$, with $(H_2,J_2)$ for $s\gg0$, and satisfies
\begin{equation}
\partial_sH^s\geq 0.
\end{equation}
Evidently, a monotone homotopy exists if only if $H_{1,t}(x)\leq H_{2,t}(x)$ for all $t\in S^1$ and $x\in M$. Solutions to the Floer equation corresponding to a monotone datum satisfy the variant of estimate \eqref{eqGeoTopEn1}
\begin{equation}\label{eqGeoTopEn2}
E_{top}(u)\geq\int\|\partial_su\|^2ds\geq 0.
\end{equation}
Moreover, as we recall in Proposition \ref{tmSummary}, when $F_1,F_2$ are dissipative, we can take the regular monotone homotopy to be  dissipative. Note that dissipativity in this case only involves the intermittent boundedness condition which is an open condition. The loopwise dissipativity is inherited from the ends.  Thus, relying on the $C^0$ estimate of Proposition \ref{tmSummary00}  and standard Floer theory, given a generic dissipative monotone homotopy $(H^s,J^s)$ there is an induced chain map
\begin{equation}\label{eqContinuationMap}
f_{H^s,J^s}:CF^*(H_1,J_1)\to CF^*(H_2,J_2).
\end{equation}
These are again defined by counting appropriate Floer solutions weighted by their topological energy.
By the standard Hamiltonian Floer theory package and the estimates of Proposition \ref{tmSummary00} the induced map on homology is independent of the choice of homotopy $(H^s,J^s)$. Moreover, the map is functorial in the sense that if $H_1\leq H_2\leq H_3$, the continuation map associated with the relation $H_1\leq H_3$ is the composition of those associated with $H_1\leq H_2$ and $H_2\leq H_3$.

A particular consequence of the discussion above is that given a Hamiltonian $H$ and two different choices $J_1,J_2$ so that $(H,J_i)$ is dissipative for $i=1,2$ there is a canonical isomorphism  $HF^*(H,J_1)=HF^*(H,J_2)$. For this reason we sometimes allow ourselves to drop $J$ from the notation. Similarly, given a pair of Hamiltonians $H_1\leq H_2$ we refer to the canonical continuation map $HF^*(H_1,J_1)\to HF^*(H_2,J_2)$ without specifying any choice of homotopy. We will sometimes omit from the notation the dependence of Hamiltonian Floer groups on the almost complex structure to not clutter up the already cluttered notation. They are there, and we spell out what they are in the surrounding discussion. We hope this will not cause confusion.

%Henceforth, at the level of Floer homology we \textcolor{blue}{sometimes} drop $J$ from the notation and write $HF^*(H)$ for the homology of the complex $CF^*(H,J)$ where $J$ is any choice so that $(H,J)$ is dissipative. (\textcolor{red}{I am completely against dropping $J$. I think all it does is to make actual explanations impossible. I think sometimes you are using that $J$'s are equivalent  - maybe also that $H$'s are Lipschitz? - and sometimes intermittent boundedness is used. I honestly cannot tell which one is used when. By omitting $J$ you are running away from this explanation. })By the above discussion, \textcolor{blue}{\emph{for fixed $H$}}, any two choices of $J$ so that $(H,J)$ is dissipative(\textcolor{red}{what kind of $J$'s write fully what this means. Is an $H$ fixed?}) give rise to homology groups which are canonically isomorphic. Similarly, given a pair of Hamiltonians $H_1\leq H_2$ we refer to the canonical continuation map $HF^*(H_1)\to HF^*(H_2)$ without specifying any choice of homotopy (\textcolor{red}{Just to repeat: I am against not specifying anything about the almost complex structure here. The role of $J$ needs to be kept explicit in the notation. Especially because I cannot tell what is being used.}).

\subsection{Truncated symplectic cohomology}\label{ss-trunc}
For a dissipative pair $(H,J)\in C^{\infty}(S^1\times M,\bR)\times\cJ(M,\omega)$ and a non-negative real number $\lambda$ we denote by $$HF^*_{\lambda}(H,J):=H^*\left(CF^*(H,J)\otimes_{\Lambda_{\geq 0}}\Lambda_{\geq0}/T^{\lambda}\Lambda_{\geq 0}\right)$$ the $\lambda$-truncated Floer homology. This is a module over $\Lambda_{\geq0}/T^{\lambda}\Lambda_{\geq 0}$.

\begin{remark}
Note that the underlying $\Lambda_{\geq0}/T^{\lambda}\Lambda_{\geq 0}$-module of $\lambda$-truncated Floer homology is free: $$CF^*(H,J)\otimes_{\Lambda_{\geq 0}}\Lambda_{\geq0}/T^{\lambda}\Lambda_{\geq 0}=\oplus_{\gamma\in Per(H)}\Lambda_{\geq 0}/T^{\lambda}\Lambda_{\geq 0}\langle\gamma\rangle.$$
\end{remark}
It follows from \eqref{eqGeoTopEn2} and the definition that the continuation maps \eqref{eqContinuationMap} induce natural maps of $\Lambda_{\geq0}/T^{\lambda_1}\Lambda_{\geq 0}$-modules
\[
HF^*_{\lambda_1}(H_1)\to HF^*_{\lambda_2}(H_2)
\]
whenever $\lambda_1\geq\lambda_2$ and $H_2\geq H_1$.

Denote by $\cH_K$ the  set of dissipative regular Floer data $(H,J)$ such that $H<0$ on $K$. It is shown in \cite[Theorem 6.10]{groman} and recalled in the appendix that the set  $\cH_K$ is a non-empty directed set. Namely, for any pair $(H_1,J_1), (H_2,J_2)\in\cH_K$ there exists a third datum $(H_3,J_3)\in \cH_K$ such that $\max \{H_1(x),H_2(x)\}\leq H_3(x)$. Moreover, we have
\begin{equation}
\sup_{(H,J)\in\cH_K}H=\chi_{M\setminus K},
\end{equation}
where the right hand side is the characteristic function $\chi_{M\setminus K}$ which is $0$ on $K$ and $\infty$ everywhere else. The \emph{truncated relative symplectic cohomology} is defined by
\begin{equation}
SH^*_{M,\lambda}(K) :=\varinjlim_{H\in\cH_K} HF^*_{\lambda}(H).
\end{equation}
Given an inclusion $K_2\subset K_1$ and a pair of real numbers $\lambda_1\leq\lambda_2$ we obtain an induced $\Lambda_{\geq0}/T^{\lambda_1}\Lambda_{\geq 0}$-module map
\begin{equation}\label{eqrestriction}
SH^*_{M,\lambda_1}(K_1)\to SH^*_{M,\lambda_2}(K_2).
\end{equation}
We refer to the map associated to the inclusion $K_2\subset K_1$ with $\lambda$ fixed as \emph{restriction}. To the morphism associated to $\lambda_1\geq\lambda_2$ we refer to as \emph{truncation}. Note that each map \eqref{eqrestriction} can be canonical factored into a restriction map followed by a truncation map.\\

Let $Y\subset M$ be an open set. Denote by $\cK(Y)$ the set of compact sets $K\subset Y$. Consider the category whose objects are the set $\cK(Y)\times\bR_+$ and the morphism sets are
\[
Hom_{\cK(Y)\times\bR_+}\left((K_1,\lambda_1),(K_2,\lambda_2)\right)=\begin{cases} *,&\quad K_2\subset K_1,\lambda_2\leq \lambda_1\\
\emptyset&\quad\mbox{otherwise}.
\end{cases}
\]
We summarize the above discussion with the following proposition.
\begin{proposition}
The assignment
\[
\cSH^*_{Y\subset M}:(K,\lambda)\mapsto SH^*_{M,\lambda}(K),
\]
which acts on morphisms by the maps in \eqref{eqrestriction} is a functor
\[
\cSH^*_{Y\subset M}: \cK(Y)\times\bR_+\to \Lambda_{\geq 0}-Mod.
\]
\end{proposition}
\subsection{Truncated symplectic cohomology from acceleration data}
In the construction above we have considered the set $\cH_K$ of all  dissipative Floer data. This is a set of Floer data which is invariant under symplectomorphisms. However, to get a more concrete handle on relative $SH$ we make use of the following framework. % involving proper (instead of bounded) Hamiltonians with no assumption on the growth rate at infinity.

\begin{definition} Let $K\subset M$ be a compact subset. We call the following datum the \emph{Hamiltonian part of an acceleration datum for $K$}:
\begin{itemize}
\item $H_1\leq H_2\leq\ldots$ a monotone sequence of non-degenerate  one-periodic Hamiltonians $H_i: M\times S^1_t\to \bR$ satisfying $H\mid_{S^1\times K}<0$ and for every $(x,t)\in M\times S^1,$  $$
H_i(x,t)\xrightarrow[i\to+\infty]{}\begin{cases}
0,& x\in K,\\
+\infty,& x\notin K.
\end{cases}
$$
\item A monotone homotopy of Hamiltonians $H_{i,i+1}:[i,i+1]\times M\times S^1\to\mathbb{R}$, for all $i$, which is equal to $H_i$ and $H_{i+1}$ in a fixed neighborhood of the corresponding end points.
\end{itemize}
One can combine the Hamiltonian part  an acceleration datum into a single family of time-dependent Hamiltonians $H_{\tau}: M\times S^1\to \bR$, $\tau\in \bR_{\geq 1}$. Note that we are choosing to omit the $t$-parameter.
\end{definition}

%\begin{remark}
%The properness of the Hamiltonians in this definition is required for definiteness considering what we will do in this paper. There is no reason one should not consider bounded Hamiltonians in this definition in general.
%\end{remark}
\begin{remark}
Note that on a non-compact $M$ the pointwise convergence condition is not equivalent to cofinality.
\end{remark}

\begin{definition}We call a $[1,\infty)_\tau\times S^1_t$ family of geometrically bounded compatible almost complex structures $J_{\tau}$ on $M$ the \emph{almost complex structure part of an acceleration datum}. We similarly omitted the $t$-parameter from the notation. Note that this notion is independent of $K$. \end{definition}

We also fix a non-decreasing surjective map $$(-\infty, \infty)\to [0,1],$$ which is used to turn an $[i,i+1]$-family of Hamiltonians and almost complex structures to a $(-\infty,\infty)$-family, which is then used to write down the Floer equations. This is how the $\tau$ parameter is related to the $s$-parameter as it is commonly used. When we say the Floer data associated to $(H_\tau, J_\tau)_{\tau\in [i,i+1]}$, we mean the data on the infinite cylinder $\mathbb{R}_s\times S^1_t$ after this operation. We assume that the choices are made so that such Floer data is locally $s$-independent outside of $[-1,1]\times S^1$.

\begin{definition}An acceleration datum for $K$ is data of the Hamiltonian part $H_{\tau}$ and the almost complex structure part $J_\tau$ which satisfy the following properties:

\begin{enumerate}
\item For each $i\in\mathbb{N}$, $(H_i,J_i)$ is dissipative and regular.
\item For each $i\in\mathbb{N}$, the Floer data associated to  $(H_\tau, J_\tau)_{\tau\in [i,i+1]}$ is dissipative and regular.
\end{enumerate}

\end{definition}

\begin{proposition}\label{propTruncAccDef}
\begin{itemize}
\item For two different choices of acceleration data for $K$, $(H_\tau,J_\tau)$ and $(H_\tau',J_\tau')$ such that $H_{\tau}\leq H_{\tau'}$  and for any $\lambda\geq 0$
there is a canonical $\Lambda_{\geq 0}$-module isomorphism $$\varinjlim_{i}HF^*_{\lambda}(H_i,J_i)\to \varinjlim_{i}HF^*_{\lambda}(H_i',J_i')$$ defined using the Hamiltonian Floer theory package on geometrically bounded manifolds. Let us call these maps \emph{comparison maps}. They automatically commute with truncation maps.
 \item The comparison map from an acceleration datum to itself is the identity map. Moreover, the comparison maps are functorial for composite inequalities $H_{\tau}\leq H_{\tau'}\leq H_{\tau''}.$
 \end{itemize}
\end{proposition}
\begin{proof}
This is \cite[Lemma 8.12]{groman}.
\end{proof}
In a similar way we have the following proposition. %The lemma below shows that we can compute truncated symplectic cohomology using an acceleration data as well. The proof is given in appendix \ref{AppDissRev}.

\begin{proposition}\label{lmSHCofinal}
Given any acceleration datum $(H_\tau,J_\tau)$ we have an isomorphism
\begin{equation}\label{eqSHcofinal}
\varinjlim_{i} HF^*_{\lambda}(H_i,J_i)\to SH^*_{M,\lambda}(K).
\end{equation}
This isomorphism is natural with respect to restriction and truncation maps. Moreover, if we are given a second acceleration datum $(H_\tau',J_\tau')$, we have a commutative diagram\begin{align*}
\xymatrix{
 \varinjlim_{i} HF^*_{\lambda}(H_i,J_i)\ar[d]\ar[dr]&\\
 \varinjlim_{i} HF^*_{\lambda}(H_i',J_i')\ar[r]&SH^*_{M,\lambda}(K),}
\end{align*} where the vertical map is the comparison map.
\end{proposition}
\begin{proof}
The sequence $\{(H_i,J_i)\}$ embeds as a directed set into $\cH_K$, so a map as in \eqref{eqSHcofinal} is induced by the universal mapping property of direct limits.  It remains to {show that it is an isomorphism}. For this fix a proper dissipative datum $(H,J)\in\cH_K$ such that $H\leq H_i$ for all $i$. Such a datum exists according to \cite[Theorem 6.10]{groman}.  The subset $\cH'_K\subset\cH_K$ consisting of $(H',J')$ such that $H'>H$ is cofinal. Define the set $\cH''_K\subset\cH_K$ to consist of all elements $H''$ for which there is a compact set $K'$ and a constant $c$ so that $H''|_{M\setminus K'}=H|_{M\setminus K'}+c$. 
is an isomorphism. The set $\cH''_K$ admits a cofinal sequence $(H''_i,J''_i)$. Moreover, such a cofinal sequence can be constructed with $H''_i\leq H_i$. { We obtain a commutative diagram
\begin{equation}\label{eqSHcofinal2}
\xymatrix{\varinjlim_{i} HF^*_{\lambda}(H''_i,J''_i)\ar[d]\ar[r]&\varinjlim_{i} HF^*_{\lambda}(H_i,J_i)\ar[d]\\
\varinjlim_{(H'',J'')\in\cH''_K}HF^*_{\lambda}(H'')\ar[r]& SH^*_{M,\lambda}(K)}
\end{equation}
where the upper horizontal map is the comparison map. The upper horizontal map is an isomorphism by the previous lemma. The left vertical map is an isomorphism by cofinality. We show that the lower horizontal map is an isomorphism.} For any $H'\in\cH'_K$ we can pick a monotone sequence $H''_i\in\cH_K''$  converging to $H'$ on compact sets. According to \cite[Theorem 8.9]{groman} we have that the natural map $\varinjlim_iHF^*_{\lambda}(H''_i)\to HF^*_{\lambda}(H')$ is an isomorphism. 
{Consider the maps
\begin{equation}\label{eqSHcofinal1}
\varinjlim_{(H'',J'')\in\cH''_K}HF^*_{\lambda}(H'')\to \varinjlim_{(H',J')\in\cH'_K}HF^*_{\lambda}(H')\to SH^*_{M,\lambda}(K).
\end{equation}
By the above discussion and Fubini for colimits, the arrow on the left is an isomorphism. The arrow on the right is an isomorphism by cofinality of  $\cH'_K\subset\cH_K$. So, the lower horizontal map in \eqref{eqSHcofinal2} is an isomorphism. We conclude that the right vertical map is also an isomorphism as was to be proven. }
 
%\begin{equation}\label{eqSHcofinal2}
%\varinjlim_{(H'',J'')\in\cH''_K}HF^*_{\lambda}(H'')\simeq \varinjlim_{i} HF^*_{\lambda}(H''_i,J''_i)\simeq\varinjlim_{i} HF^*_{\lambda}(H_i,J_i),
%\end{equation}
%where the right hand side is the comparison map. The composition of the maps in this equation with the inverse of \eqref{eqSHcofinal1} produces the desired inverse.
\end{proof}

\subsection{Operations in truncated symplectic cohomology.}\label{ssoperations}

To construct the pair of pants product in local $SH$ we first discuss it as an operation
\begin{equation}
HF^*_{\lambda} (H_1)\otimes HF^*_{\lambda}(H_2)\to HF^*_{\lambda}(H_3),
\end{equation}
for a triple $(H_1,H_2,H_3)$ satisfying the inequality
\begin{equation}\label{eqProdTrip}
\min_tH^3_t(x)>H>2\max_{t,i=1,2}\{H^1_t(x),H^2_t(x)\}.
\end{equation}
We first explain the source of this condition. On a compact manifold and working over the Novikov field instead of over the Novikov ring, the pair of pants product is a standard construction in Floer theory with no need to impose additional conditions. However to guarantee positivity of topological energy we need to require that our Hamiltonian one form satisfy the inequality
\begin{equation}\label{EqMonFlD}
d\fH +\{\fH,\fH\}\geq 0.
\end{equation}
Being non-linear this condition is difficult to work with. For this reason we restrict the discussion to \emph{monotone split} one forms. These are of the form $\fH=H\otimes\alpha$ where $\alpha$ is a closed $1$-form (which for the purpose of the present discussion can be fixed once and for all)  on the pair of pants equaling $dt$ near the input and $2dt$ near the output, and $H:\Sigma\times M\to\bR$ is a function satisfying
\begin{equation}
dH(x)\wedge\alpha\geq 0,\forall x\in M.
\end{equation}
Here we consider $H(x)$ for fixed $x$ as a function on $\Sigma$ and take the exterior derivative of it. The last equation implies \eqref{EqMonFlD} for a split datum. On the other hand the condition \eqref{eqProdTrip} implies the existence of a split Floer datum $H$, equaling $H_i$ near the $i$th end.

It is different question whether there exists a \emph{dissipative} monotone split Floer datum (and whether two such can be connected by  dissipative family etc.). The answer should be an unqualified yes.  However this general claim has not been established in \cite{groman}. Rather a more limited claim is established which suffices for our purposes.

\begin{lemma}\label{lmDissProd}
For $i=1,2,3$ let $(H_i,J_i)$ be dissipative data satisfying \eqref{eqProdTrip}. Suppose
 that away from a compact set we have $(H_1,J_1)=(H_2,J_2)$.  Then
we can find a dissipative datum $(H,\alpha,J)$ on the pair of pants so that on the $i$th end $H\alpha=H^idt$ and such that $dH\wedge\alpha\geq0$.   \end{lemma}
\begin{proof}
Denote by $\Sigma$ the pair of pants and let $\pi:\Sigma\to\bR\times S^1$ be a $2:1$ branched cover which in cylindrical coordinates is the identity near each negative end and $(s,t)\mapsto (s,2t)$ near the positive end. Let $(F_s,J_s)$ be a monotone dissipative homotopy from $(H_1,J_1)=(H_2,J_2)$ to $(\frac1{2}H_3,J_3)$. Such a homotopy exists by \eqref{eqProdTrip} and Lemma \ref{prpMonHo}. Let  $H_z=F_{\pi(z)}$, $\alpha=\pi^*dt$, $J_z=\pi^*J_{\pi(z)}$ and $\alpha=\pi^*dt$. Then $(H,\alpha,J)$ is a monotone dissipative datum as required. Moreover, a $C^0$ bounded perturbation of  $J$ maintains dissipativity. So this Floer datum can be made regular.
\end{proof}
%\begin{remark}
%The condition $H_1=H_2$ essentially reduces the claim concerning the pair of pants product to the claim concerning continuation maps that whose proof is recalled in the appendix.
%\end{remark}
 Any such datum gives rise to a map
\begin{equation}
*:HF^*_{\lambda}(H^1)\otimes HF^*_{\lambda}(H^2)\to HF^*_{\lambda}(H^3)
\end{equation}
by counting solutions to Floer's equation
\begin{equation}\label{eqFlProd}
(du+\fH)^{0,1}=0
\end{equation}
and weighting them by topological energies.

Indeed, since $\fH$ is split, the monotonicity of $\fH$ guarantees that positivity of the topological energy of $u$. On the other hand we have that the space of dissipative monotone data as above is connected. Thus the operation $*$ is independent of the choice. Moreover, it commutes with continuation maps which preserve the inequality \eqref{eqProdTrip}. Therefore, applying the direct limit procedure, we obtain a pair of pants product on truncated local $SH$. Namely we pick a pair of of acceleration $H_{\tau}, G_{\tau}$ such that
\begin{equation}
\min_tG_{\tau,t}(x)>2\max_t H_{\tau,t}(x).
\end{equation}

We get an induced product
\begin{equation}
SH^*_{K,\lambda}\otimes SH^*_{K,\lambda}=\varinjlim_i\left(HF^*_{\lambda}(H_i) \otimes HF^*_{\lambda}(H_i)\right)\to \varinjlim_i HF^*_{\lambda}(G_i)= SH^*_{\lambda,K}.
\end{equation}
The induced product is associative and super-commutative by a standard argument.

%Underlying this construction is a choice of the asymptotic behavior of all our Hamiltonians at infinity. However, it is shown in \cite{groman} that the construction is natural and in particular independent of all the choices.

In a similar way $SH^*_{M,\lambda}(K)$ carries a $BV$ operator. For details on its construction and relations with the product see \cite{Abouzaid14}.

It should be emphasized that all the structure introduced here are natural with respect to both restriction maps and truncation maps.

Let us summarize our final discussion as a theorem.

\begin{theorem}
The assignment
\[
\cSH^*_{Y\subset M}:(K,\lambda)\mapsto SH^*_{M,\lambda}(K),
\]
which acts on morphisms by the maps in \eqref{eqrestriction} is a functor
\[
\cSH^*_{Y\subset M}: \cK(Y)\times\bR_+\to \Lambda_{\geq 0}-BValg.
\]
\end{theorem}

From now on we consider  the functor $\cSH^*_{Y\subset M}$ is it appears in this theorem, i.e. with $BV$-algebras over the Novikov ring as its target.

\section{Symplectic manifolds of geometrically finite type}\label{sec-loc-iso}

%(\red{Below, I will make some comments in red, but also change things without mentioning when I am fixing a mistake or clarifying something in a part that I wrote})

%We start this section by introducing the notion of geometrically finite type symplectic manifold. These are geometrically bounded symplectic manifolds that also satisfy a certain finiteness assumption. We will need to give some preliminary definitions first.

A function $f$ on a geometrically bounded symplectic manifold is called \emph{admissible} if it is proper, bounded below, and there is a constant $C$ such that with respect to a geometrically bounded almost complex structure $J$ we have $\|X_f\|_{g_J}<C$ and $\|\nabla X_f\|_{g_J}<C.$ According to \cite{Tam10}  on any complete geometrically bounded Riemannian manifold  there exists a function $f$ satisfying $\|\nabla f\|<C$, $\|\Hess f\|<C$ and $d(x,x_0)<f(x)<d(x,x_0)+C$. To obtain by this an admissible function, i.e one with estimates on $X_f$ rather than $\nabla f$, we need to assume in addition that $\|\nabla J\|_{g_J}$ is uniformly bounded, or, equivalently, that $\|\nabla\omega\|_{g_J}$ is uniformly bounded.

\begin{lemma}\label{lemNoNonTOrbits}
Let $f$ be an admissible function on a geometrically bounded symplectic manifold.  Then there is an $\epsilon_0>0$ such that any Hamiltonian $H$ which outside of a compact set $K$ satisfies $H|_K=\epsilon_0 f$ has no non-constant $1$-periodic orbits outside of $K$.
\end{lemma}
\begin{proof}

It suffices to show that the function $\epsilon f$ has no non-constant $1$-periodic orbits for $\epsilon$ small enough.  A flow line $\gamma$ of $\epsilon f$ satisfies $|\gamma'|<\epsilon C$. So a $1$-periodic orbit has diameter at most $\epsilon C/2$. Thus, since the geometry is bounded, by making $\epsilon$ small enough, any $1$ periodic orbit is contained in a geodesic ball. Now, in the particular case of $\bR^{2n}$ the claim of the lemma is \cite[Ch. 5, Prp. 17]{HoferZhender}. The proof there adapts immediately to the case of a geodesic ball. For completeness we spell this out.

Let $\gamma:\bR/\bZ \to M$ be a $1$-periodic orbit of $\epsilon f$ mapping into a geodesic ball of radius $\epsilon C/2$.  Then for $\epsilon$ small enough we have
\begin{equation}\label{eqF1st2ndCovDer}
\int\|\gamma'\|^2dt\leq C'\int\|\nabla_{\gamma'}\gamma'\|^2dt
\end{equation}
 for some $C'$ depending on the bounds on the geometry. Indeed, as stated in  \cite{HoferZhender}, for the Euclidean metric this is true with the constant $C'=\frac1{2\pi}$ by Fourier analysis. The covariant derivative with respect to $g_J$ is arbitrarily close to the covariant derivative in the flat metric as $\epsilon$ is made small enough \cite{Eichhorn91}. So the inequality \eqref{eqF1st2ndCovDer} follows at least for sufficiently small $\epsilon$.

We have
\begin{equation}
\int\|\nabla_{\gamma'}\gamma'\|^2=\int \|\nabla_{\epsilon X_f}\epsilon X_f\|^2\leq \epsilon^2 C^2\int\|\gamma'\|^2.
\end{equation}
Combining this with \eqref{eqF1st2ndCovDer} we get that for $\epsilon<\frac1{C\sqrt{C'}}$ any periodic orbit is constant.

%Moreover, considering the Eucidean metric on such a ball, there are a priori bounds on the first and second derivative of $\gamma$ in terms of $\epsilon, C$ and the bound on sectional curvature. These bounds go to $0$ with $\epsilon$ uniformly in $M$. Moreover, the norm of $\gamma'$ withe respect to the Euclidean metric is arbitrarily close to the norm the  gradient of $f$ with respect the same metric.  This reduces the claim to Euclidean space. That is, give an embedded curve $\gamma:S^1\to\bR^n$ contained in $f=Const$ and such that $\gamma'$.  Projecting
\end{proof}

To rule out all orbits outside of some compact subset we need further assumptions which are incorporated into the following definition.

\begin{definition} We say that a symplectic manifold $Y$ is \emph{geometrically of finite type} if it admits a compatible geometrically bounded almost complex structure $J$ and an exhaustion function $f$ which is admissible with respect to $g_J$ and all its critical points are  contained in a compact subset.
\end{definition}

%\begin{remark}
%In all the examples we consider a stronger property holds.  We say $Y$ is \emph{geometrically strongly of finite type} if it admits $J,f$ as above, and in addition $|\nabla f|$ is bounded away from $0$ outside of a compact set. This could be used to simplify some proofs of $C^0$ estimates below. We have not pursued this.
%\end{remark}

Note that the completion of a symplectic manifold with contact boundary is geometrically of finite type. An example of a geometrically bounded symplectic manifold which is not geometrically  of finite type is a Riemann surface of infinite genus with the area form induced by a complete geometrically bounded Riemannian metric.

%\begin{remark}
%\cite{Seidel} introduces a notion of finite type for Liouville manifolds in terms of symplectic cohomology. An example by Mclean is given of a Liouville manifold which is topologically trivial but nontheless is not of finite type as a Liouville manifold. We

%\end{remark}

%All the results in this paper are formulated for $Y$ which is geometrically weakly of finite type. However, the proofs of the underlying $C^0$ estimates are simpler under the stronger assumption (after making stronger assumptions on the involved Floer data as we shall spell out in the appendix). All the examples we consider satisfy the stronger assumption.(\red{needs to be put in the Remark mentioned above})\\

%\begin{remark}
%The assumption that a Riemannian manifold admits a Lipschitz exhaustion function with no critical points outside a compact set implies there exists a compact set $K\subset Y$ such that $Y\setminus K$ is diffeomorphic to the cone over $\partial K$  \textcolor{red}{(is this obvious? Don't you need some result like the forward flow of the gradient vector field is defined for all times? Is the Lipschitz property enough for this?)}.  We do not know whether the converse holds true \textcolor{red}{(it's hard to believe that there would not be literature about this)}. The difficulty is that while every function can be made Lipschitz by postcomposing with an appropriate function, this does not necessarily preserve properness.
 %\end{remark}

%Let us now explain the geometric setup in which we will discuss our locality results.
In the following we will consider a geometrically bounded symplectic manifold $M$ and an open subset $U\subset M$ which is geometrically of finite type with respect to the restriction symplectic form. Such an open subset is the  image of a symplectic embedding $Y\to M$ where $Y$ is a geometrically finite type symplectic manifold with the same dimension as $M$. We do not make any assumptions on the relation between an almost complex structures on $U$ witnessing its being geometrically of finite type and one on $M$ witnessing its geometric boundedness.\\

\begin{example}\label{Ex:Liouville-Cobordisms}
Let $(X,\omega_1)$ be a symplectic manifold with boundary and let $V$ be a Liouville vector field which is negatively transverse to $\partial X$. This makes $\partial X$ into a contact manifold which we denote by $(Y^-,\lambda^-)$. We moreover assume that the Liouville flow is complete in the forward direction and that $X$ admits an exhaustion by (compact) Liouville cobordisms (as in \cite[Section 11.1]{Eliashberg}) with $\partial X$ as their negative boundary. Let $(F,\omega_2)$  be a  strong (but not necessarily exact) symplectic filling of $(Y^-,\lambda^-)$. Consider the symplectic manifold $W$ obtained as the union $F\cup_{Y^-} X$. We note that $W$ is geometrically bounded but not necessarily of geometrically finite type. 

\begin{lemma}\label{lem-com-lio-emb}
There exists a unique symplectic embedding of the completion $\widehat{F}$ of $F$ into $W$, which restricts to identity on $F$ and sends the Liouville vector field on the semi-infinite symplectization of $(Y^-,\lambda^-)$ to $V$.
\end{lemma}
\begin{proof}[Sketch proof] Using the completeness of the Liouville flow in the positive direction we immediately obtain a smooth embedding $\widehat{F}\to W$ with the desired properties. The symplectically expanding property of Liouville vector fields shows that this is a symplectic embedding. Compare with Section 1.8.4 of \cite{EG}.

\end{proof}

The image of the embedding $\widehat{F}\to W$ is an example of the geometric setup we consider in this section, i.e. a \emph{complete embedding}.% We note that there are other examples and we will consider them in this paper.\\
\end{example}
The main technical result of this section is the following theorem.

\begin{theorem}\label{thmLocalityTrunc}
Let $Y$ be a geometrically finite type symplectic manifold and let $M$ be one that is geometrically bounded. Let $\text{dim}(M)=\text{dim}(Y)$ and $\iota:Y\to M$ be a symplectic embedding. Also denote by $\iota$ the induced functor $\cK(Y)\times\bR_+\to\cK(M)\times\bR_+$. Then there exists a distinguished isomorphism of functors
\begin{equation}
T_{\iota}: \cSH^*_{Y\subset Y}\simeq \cSH^*_{\iota(Y)\subset M}\circ\iota.
\end{equation}
\end{theorem}
We refer to this natural transformation as the \emph{locality isomorphism.}

The following theorem states that the locality isomorphisms are functorial with respect to  nested inclusions.

\begin{theorem}\label{thmLocalityTruncFunc}
Let $X$ and $Y$ be symplectic manifolds of geometrically finite type and let $M$ be one that is geometrically bounded. Let $\text{dim}(M)=\text{dim}(Y)=\text{dim}(X)$ and $\iota_X:X\to Y$, $\iota_Y:Y\to M$ be symplectic embeddings. Then there is an equality of natural transformations $$T_{\iota_Y}\circ T_{\iota_X}=T_{\iota_Y\circ\iota_X}.$$
\end{theorem}

After proving these two theorems, we will discuss how to extend these two results about truncated symplectic cohomology to relative symplectic cohomology as in Theorem \ref{thmMainIntro} from the introduction.

\section{Separating Floer data}\label{SecSeparatingFlDt}

We start with a basic but very crucial definition.

\begin{definition}\label{dfLambdaSeparation} Let $Y,M$ and $\iota:Y\to M$ be as in Theorem \ref{thmLocalityTrunc} and $K\subset Y$ be a compact subset. For a real number $\lambda>0$, a Floer datum $(H,J)$  inside $M$ is called \emph{$\lambda$-separating for $K$} if there exist open neighborhoods $V_1$ of $K$ and $V_2$ of $M\setminus\iota(Y)$ with the following property. \textit{All $1$-periodic orbits are contained in $V_1\cup V_2$ and any solution $u$ to Floer's equation which meets both $V_1$ and $V_2$ has topological energy at least $\lambda$.}

A homotopy $(H^s,J^s)$ between $\lambda$-separating Floer data is called \emph{$\lambda$-separating} if the same property holds for solutions to the corresponding Floer equation. We define this property for higher homotopies and for Floer solutions associated with other moduli problems as well. \end{definition}

%A $\lambda$-separating datum is called \textit{local} if $H|_{\iota(Y)\setminus V_2}$ extends to $\iota(Y)$ as a smooth proper exhaustion function $\tilde{H}$ which is free of critical points outside a compact set. We say that $H$ is locally modeled on $\tilde{H}$.

%\red{We immediately note the importance of this definition even though the role played by $K$ can not be made explicit at this point.}
\begin{lemma}\label{lmSplit}
Let $(H,J)$ be a $\lambda$-separating datum for $K$ with witnesses open sets $V_1$ and $V_2$ as in Definition \ref{dfLambdaSeparation}. Then the $\lambda$-truncated Floer cohomology has a direct sum decomposition
\[
HF^*_\lambda(H)=HF^*_{\lambda,inner}(H)\oplus HF^*_{\lambda,outer}(H),
\]
corresponding to classes represented by periodic orbits contained in $V_1$ and $V_2$ respectively. If $(H^i,J^i)$ are $\lambda$-separating for $i=1,2$, $H^1\leq H^2$ and $(H_s,J_s)$ is a $\lambda$-separating monotone homotopy then the induced continuation map is split. Similar claims hold for the pair of pants product and the $BV$ operator.
\end{lemma}
\begin{proof}
This is an immediate consequence of the definitions.
\end{proof}

Our task now is to produce for any $K$ and $\lambda$ acceleration data for $K$ which are $\lambda$-separating in a a way that is natural with respect to restriction and truncation. The key to this are the robust  $C^0$ estimates we introduce in the next subsection. These state that by considering Floer data which are of sufficiently small Lipschitz constant on a sufficiently large region separating $K$ and $M\setminus\iota(Y)$ we achieve $\lambda$-separation. %The end result in the form we use them is summarized in  subsection \S\ref{SecLipBoundDiff}. The reader who is  willing to take some $C^0$ estimates on faith can skip to \S\ref{SecLipBoundDiff}. %in the the following propositions to produce separating Floer data. This is the key point of the proofs of our locality theorems.  %The proofs of the estimates are postponed to the appendix.

\subsection{Dissipativity and robust $C^0$-estimates}\label{ss-diss-c0}
\subsubsection{Gromov's trick for Floer equation}\label{AppDissRev1}

%\textcolor{red}{we need some more explanations and also references to your thesis but all is clear, we can come back to it later. see my remark below, review it and make it an actual remark.}

A \emph{Riemann surface with $n$ inputs and $1$ output} is a Riemann surface $\Sigma$ obtained from a compact Riemann surface $\overline{\Sigma}$ by forming $n+1$ punctures. We distinguish one puncture as an \emph{output} and refer to the other punctures as \emph{input}. In the the neighborhood of the input $p_i$ we make a choice of \emph{cylindrical coordinates} $\epsilon^-_i:(-\infty,0)\times S^1\to \Sigma$. Near the output $p_0$ we fix cylindrical coordinates $\epsilon^+_0:(0,\infty)\times S^1$. We use the letter $s,t$ to refer to the standard coordinates on $\bR\times S^1$.

Let $\Sigma$ be a Riemann surface with cylindrical ends. A \emph{Floer datum} on $\Sigma$ is a triple $(\alpha,H,J)$ where
\begin{itemize}
\item $H:\Sigma\times M\to\bR$ is a smooth function which is $s$-independent on the ends,
\item $\alpha$ is a closed $1$-form on $\Sigma$ coinciding with $dt$ on the ends, and,
\item $J$ is a  $\Sigma$ dependent $\omega$-compatible almost complex structure on $M$ which is $s$ independent at the ends.
\end{itemize}
The pair $(\alpha,H)$ defines a $1$-form $\fH=H\pi_1^*\alpha$ on $\Sigma\times M$. It can also be considered as a $1$-form on $\Sigma$ with values in functions on $M$. %We will use the letter $\fH$ to denote this $1$-form and write the Floer datum as $(\fH,J)$.
We denote by $X_{\fH}$ the $1$-form  $\alpha\otimes_{C^{\infty}(\Sigma)} X_H$ on $\Sigma$ with values in  Hamiltonian vector fields  on $M$. For the case $\Sigma=\bR\times S^1$, $\alpha=dt$ we write the Floer datum as $(H,J)$. % {\color{red} Do you not want $J$ to be $s$-independent as well?}.

%The data $(\fH:=\alpha\otimes H,J)$ is called a Floer datum \textcolor{red}{(for contractibility it is essential for the Floer data to be pairs $(\alpha, H)$. It is difficult to analyze the fibers of the map $(\alpha, H)\mapsto \alpha\otimes H$ under the monotonicity assumption - assuming you are tensoring over $C^{\infty}(\Sigma)$. This feels very strange as only $\alpha\otimes H$ is used in the equation but I can try to explain how I view it if you want me to. Also note that actually only the $(0,1)$ part of $\alpha\otimes X_H$ is used in the equation.)}

To the Floer datum $(H,\alpha,J)$ we associate an almost complex structure $J_{\fH}$ on $\Sigma\times M$ as follows
\begin{equation}\label{eqGrTrick}
J_{\fH}:=\left(\begin{matrix} j_{\Sigma}(z) & 0  \\ X_{\fH}(z,x)\circ j_\Sigma(z)-J(z,x)\circ X_{\fH}(z,x) & J(x) \end{matrix}\right).
\end{equation}
When $H,J$ are $s$ independent data on the cylinder $\bR\times S^1$ we use the notation $J_H$ for the corresponding almost complex structure.

We say that  $(H,\alpha,J)$ is \emph{monotone} if for each $p\in M$ we have
\begin{equation}\label{eqMonCond}
d\fH|_{\Sigma\times\{p\}}\geq 0.
\end{equation}
In this case the closed form
\begin{equation}\label{eqGrTrick2}
\omega_{\fH}:=\pi_1^*\omega_{\Sigma}+\pi_2^*\omega+ d\fH
\end{equation}
on $\Sigma\times M$ can be shown to be symplectic and  $J_{\fH}$ is compatible with it. Here $\omega_\Sigma$ is any symplectic form compatible with  $j_\Sigma$.  We shall take $\omega_\Sigma$ to coincide with the form $ds\wedge dt$ on the ends.
We denote the induced metric on $\Sigma\times M$ by $g_{J_{\fH}}$. We refer to the metric $g_{J_\fH}$ as the \emph{Gromov metric}. We stress that \emph{a Gromov metric depends on the choice of area form on $\Sigma$}. The projection $\pi_1:\Sigma\times M\to \Sigma$ is $(J_\fH,j_{\Sigma})$-holomorphic. $\pi_1$ is typically far from being a Riemannian submersion. However we have the following lemma.
\begin{lemma}\label{lmGrNonDec}
Under the monotonicity condition \eqref{eqMonCond} the projection $\pi_1$ is  \emph{length non-increasing}.
\end{lemma}
\begin{proof}
It is shown in \cite{groman}[Lemma 5.1] that the 2-form \eqref{eqGrTrick2} is symplectic  and compatible with $J_{\fH}$ for \emph{any} choice of area form $\omega_{\Sigma}$. That is, writing $\beta=\pi_2^*\omega+d\fH$ we have $\beta(v,J_{\fH}v)\geq 0$. It follows that
\[
|v|^2\geq\pi_1^*\omega_{\Sigma}(v,J_{\fH}v)=|(\pi_{1})_*v|^2
\]
where we use that  $\pi_1$ is $(J_\fH,j_{\Sigma})$-holomorphic
\end{proof}

\subsubsection{Gromov's trick and energy}

An observation known as \emph{Gromov's trick} is that $u$ is a solution to Floer's equation
\begin{equation}\label{eqFloer}
(du-X_{J})^{0,1}=0,
\end{equation}
if and only if its graph $\tilde{u}$ satisfies the Cauchy Riemann equation
\begin{equation}
\overline{\partial}_{J_{\fH}}\tilde{u}=0.
\end{equation}
Thus Floer trajectories can be considered as $J_{\fH}$-holomorphic sections of $\Sigma\times M\to\Sigma$. 
%One might naively think this can be fixed by considering some finite area symplectic structure on $\Sigma$. However, the geometry obtained in this way is not bounded rendering the monotonicity techniques irrelevant.

%In particular this allows the use of monotonicity techniques for obtaining $C^0$ estimates. However, this translation is not straightforward. What we get in this way is a $C^0$ estimate on a $J_{\fH}$-holomorphic section given a bound on its topological energy which includes a contribution from the added term $\pi_1^*\omega_{\Sigma}$. The latter is always infinite for the entire curve. One might naively think this can be fixed by introducing some finite area symplectic structure on $\Sigma$. However, the geometry obtained in this way is not bounded rendering the monotonicity techniques irrelevant.   Hence, in this context the monotonicity estimate on the section is only useful in finite area regions inside $\Sigma$. To obtain $C^0$ estimates for the entire curve we shall need to introduce the notion of loopwise dissipativity below.

%In this subsection we liberally reference notions and notations introduced in Appendix \S\ref{AppDissRev} which is a summary of notions and results from \cite{groman}. 
To a Floer solutions $u:\Sigma\to M$ and a subset $S\subset\Sigma$ we can now associate three different non-negative real numbers:
\begin{itemize}
\item The \emph{geometric energy} $E_{geo}(u;S):=\frac12\int_S\|(du-X_{\fH})\|^2$ of $u$.
\item The \emph{topological energy} $E_{top}(u;S):=\int_Su^*\omega+\tilde{u}^*d\fH.$
\item The \emph{symplectic energy} $E(\tilde{u};S):=\int_S\tilde{u}^*\omega_{\fH}$. 
\end{itemize}
We have the relation $E(\tilde{u};S)=E_{top}(u;S)+Area(S)$. For a monotone Floer datum we have, in addition, the relation $E_{geo}(u;S)\leq E_{top}(u;S)$.

%The most important thing to take note of is the use of the Gromov trick for the turning graph $\tilde{u}$ of a Floer solution into a pseudo-holomorphuc curve. The details are 

%Accordingly, there are two related notions of energy of the Floer trajectory on a subset $S$ of the domain $\Sigma$: the energy $E(\tilde{u};S)$ of the pseudo-holomorphic curve  and the geometric energy $E_{geo}(u;S)$ of the Floer solution $u$.

\subsubsection{$C^0$ estimates}
%In the following lemma we refer to the notion of a \emph{monotone Floer datum} $(\alpha,H,J)$ on a Riemann surface $\Sigma$ with cylindrical ends. The precise definition is given in the beginning of Appendix \S\ref{AppDissRev}. We also use the notation $\fH$ for the associated Hamiltonian $1$-form.

Gromov's trick allows transferring \emph{local} results concerning $J_{\fH}$-holomorphic curves to Floer trajectories. This is true in particular for the monotonicity Lemma \ref{lmMonEst++}. However, note that because of the  contribution from the added term $\pi_1^*\omega_{\Sigma}$ the symplectic energy $\int_{\Sigma}\tilde{u}^*\omega$ is always infinite. For this reason, applying Gromov's trick and the monotonicity Lemma for obtaining $C^0$ estimates  is not straightforward. Techniques for doing this were developed by the first named author in \cite{groman}. These techniques are reviewed in Appendix \S\ref{AppDissRev}. 

The remainder of this section is a quantitative version of the results of  \cite[\S5,\S6]{groman}. The end results are summarized in Lemmas \ref{lemmaLipBounEnd} and \ref{lemmaLipBounEnd2}. The proof follows the pattern of the dissipativity estimates of \cite{groman}. These involve  proving a domain-area dependent  estimate on the \emph{diameter} of a Floer solution as in Lemma \ref{lmDiamEstFreeBo} and a domain-area independent estimate on the \emph{distance} between the two boundary components of a Floer cylinder as in Lemma \ref{lmLoopDisLin}. The combination of the two gives the desired $C^0$ estimate.

The setup we will be considering in the entire subsection, as well as throughout the proof of the locality theorem, is the case of a Floer solution passing through a region $U$ so that the Floer datum satisfies some bounds on the region $U$, but is otherwise arbitrary. In order to avoid unnecessary distractions \emph{we shall assume throughout that the underlying almost complex structure is domain independent in the region $U$ and that $H$ is time independent on the ends.} These assumptions are not strictly necessary, but make the presentation clearer. Moreover, note that all our estimates are robust under $C^0$ small perturbations of $J$ and the Hamiltonian vector field $X_H$.

The following Lemma is a special case of the subsequent Lemma \ref{lmDiamEstFreeBo} (the case $\partial^-S=\emptyset$). We prove it separately for expositional reasons.
\begin{lemma}
Let $\Sigma$ be a Riemann surface with cylindrical ends and let $(\alpha,H,J)$ be a monotone Floer datum on $\Sigma$. Let $U\subset M$ be a compact domain whose boundary is partitioned as $\partial U=\partial^+U\coprod\partial^-U$.  Let $S\subset\Sigma$ be a compact connected subset and $u:S\to M$ be a Floer solution for the datum  $(\alpha,H,J)$ such that $u(\partial S)$ is disjoint from the interior of $U$. Suppose that $u(S)$ meets both $\partial^+U$ and $\partial^-U$. Then there are constants $C,D$ depending only on the bounds on $J$ and $\fH$ along $U$ such that
\begin{equation}
E_{top}(u;S)+\Area(S)\geq Cd_{M}(\partial^+U,\partial^-U)-D,
\end{equation}
where $d_M$ is the distance with respect to $g_J$.
\end{lemma}

\begin{remark}
As in Remark \ref{remHeinUsher}, what this Lemma adds to the closely related Usher's lemma \cite{Hein} is that the right hand side goes to infinity with the distance between the boundary components. Lemma \ref{lmDiamEstFreeBo} adds a further relaxation of the requirement that $U$ not meet the boundary of $\partial S$.
\end{remark}

\begin{proof}
Denote by $J_{\fH}$ the almost complex structure produced by the Gromov trick \eqref{eqGrTrick} on $\Sigma\times M$. As discussed above,  the graph $\tilde{u}$ of $u$  is $J_{\fH}$-holomorphic, and, moreover,
we have
\begin{equation}
E(\tilde{u}) =E_{top}(u)+Area(S)%\geq E_{geo}(u)+Area(S).
\end{equation}

Let $g_{\Sigma}$ be a metric on $\Sigma$ as defined in the paragraph right after \eqref{eqGrTrick2}.  Since $U$ is compact, there is a constant $a$ such  that the Riemannian metric $g_J$ is strictly bounded by $a$ on the region $U$.  See Definition \ref{dfStrBound}. %The adjustment for non-strictly bounded is straightforward as noted in Remark \ref{rmNonStrAdj}, but not needed anywhere.
 Let $g_1=g_\Sigma\times g_J$ be the product metric. By Lemma \ref{lmCritABound}, the Gromov metric $g_{J_{\fH}}$ is equivalent on the ball \emph{with respect to $g_1$} of radius $1/a$ around $U$ to the product metric $g_1$ with equivalence constant {$b$} depending on the datum $(\fH,J)$ on $U$. By possibly enlarging $a$ or $b$, we may assume $a=b$. %This produces an estimate
%\begin{equation}
%d_{g_0}(\partial^+K,\partial^- K)\leq a d_{g_{J_\fH}}(\Sigma\times \partial^+ K ,\Sigma\times \partial^- K).
%\end{equation}
In particular, every point $p$ of $\Sigma\times U$ is $a$-bounded with witness supported on the ball with respect to $g_1$ of radius $1/a$ around $p$.

 Picking a path $\gamma$ in $S$ whose image under $u$ connects $\partial^+U$ and $\partial^-U$ we can find $N=\left\lfloor \frac{a}{2}d_{M}(\partial^+U,\partial^- U)\right\rfloor$ disjoint $g_{1}$-balls in $\Sigma\times M$ of radius $1/a$ whose center is in the image of $\tilde{u}$ and which are contained in $S\times U$. The Monotonicity Lemma (Lemma \ref{lmMonEst++}) applied to the Gromov metric now gives
 \begin{equation}
 E_{top}(u;S)+\Area(S)\geq N\frac{c}{a^4} \sim  \frac{c}{2a^3}d_{M}(\partial^+U,\partial^- U).
 \end{equation}
 The constant $D$ in the statement takes care of the case $d_{M}(\partial^+U,\partial^- U)\leq 2a$.
\end{proof}
The Lemma just proven concerned the case where $\partial S$ is mapped by $u$ outside of $U$. The next lemma relaxes this and allows $S$ to have boundary component $\partial^-S$ on which no condition is imposed. In this case the energy $E_{top}(u;S)$ only controls the diameter under $u$ of the smaller set $S'\subset S$ which consists of points whose distance from $\partial^-S$ is bounded away from $0$.
\begin{lemma}\label{lmDiamEstFreeBo}
Let $\Sigma, H,\alpha,J,U$ be as in the previous lemma. Let $S\subset\Sigma$ be a  compact domain whose boundary is partitioned into disjoint closed components  $\partial S=\partial^+S\cup\partial^-S$ snd let $u:S\to M$ be a Floer solution so that $u(\partial^+S)$ does not meet the interior of $U$. Let $S'\subset S$ be a connected subdomain consisting of  points whose distance from $\partial^-S$ is at least $1/a$ where $a>1$ strictly bounds the geometry of $J$ on $U$. Suppose $u(S')$ meets both  $\partial^+U$ and $\partial^-U$.  Then
\begin{equation}
E_{top}(u;S)+\Area(S)\geq Cd_{M}(\partial^+U,\partial^-U)-D
\end{equation}
for  $C,D$ depending only on the bounds on $J$ and $\fH$ along $U$ .
\end{lemma}
 \begin{remark} We suggest the following example for the reader to appreciate this lemma. By the Riemann mapping theorem, there is a bi-holomorphism from the unit disk to an arbitrarily long and thin simply connected domain of say area $1$. The Lemma with $S$ being the unit disk, $\partial^-S=\partial S$ and  $\partial^+S=\emptyset,$ says that under any such a biholomorphism the image of the disc $S'$of radius $0.99$  around the origin  has diameter at most $c_0$, which is a real number independent of the target domain. Indeed  for any $p_1,p_2$ in the image of $S'$ we can take $U$ to be an an annulus with $p_1$ on its inside and $p_2$ on its outside.
 \end{remark}
\begin{proof}[Proof of Lemma \ref{lmDiamEstFreeBo}]
As in the proof of the previous Lemma, without loss of generality the Gromov metric associated with the Floer datum $(\fH,J)$ is $a$-bounded  and $a$-equivalent to the product metric $g_1$ on $\Sigma\times U$. Let $z_-,z_+\in S'$ be a pair of points meeting each boundary respectively. Let $\gamma$ be a path in $S'$ connecting $z_-$ and $z_+$. Let $N$ be a number and $z_i$ for $i=1,\dots N$  be a sequence of points along $\gamma$ with the property that $d_M(u(z_i),u(z_j))\geq2/a$ and $d_M(u(z_i),\partial^{\pm}U)\geq1/a$. We can  take
\begin{equation}
N+1\geq \frac{a}{2} d_M(u(z_-),u(z_+))\geq \frac{a}{2} d_M(\partial^+U,\partial^-U).
\end{equation}

Consider the map $\tilde{u}=id\times u: S\to \Sigma\times M$.  %The product metric and the metric $g_{J_{\fH}}$ are $a$-equivalent for $a$ a constant depending on the restrction of $J,\fH$ to $S\times K$. We therefore have
%\begin{equation}
%\frac1{a}d(u(z),u(z'))<\tilde{d}(\tilde{u}(z),\tilde{u}(z'))<a( d(u(z),u(z'))+d(z,z')).
%\end{equation}
%Without loss of generality $\frac1{a}=\epsilon$.
 Let $S_i\subset S$ be the connected component of $\tilde{u}^{-1}(\tilde{B}_{1/a}(\tilde{u}(z_i))$ containing the point $z_i$.
%\begin{equation}
%S_i=\tilde{u}^{-1}(\tilde{B}_{1/a}(\tilde{u}(z_i)),
%\end{equation}
Here $\tilde{B}_{1/a}$ denotes the ball in $\Sigma\times M$ with respect to the product metric $g_0$.
We claim the interior of $S_i$ does not meet the boundary $\partial S$. Indeed the interior of $S_i$ does not meet $\partial^+S$ since
\begin{equation}
d_{g_1}(\tilde{u}(\partial^+S),\tilde{u}(z_i))\geq d_{M}(u(\partial^+S),u(z_i))\geq d_M(\partial^{\pm} U,u(z_i))\geq1/a.
\end{equation}
Similarly,  the interior of $S_i$   does not meet $\partial^-S$ since $z_i\in S'$, so
\begin{equation}
d_{g_1}(\tilde{u}(\partial^-S),\tilde{u}(z_i))\geq d_{\Sigma}(\partial^-S,z_i)\geq1/a.
\end{equation}

From this it follows that $u(\partial S_i)\cap \tilde{B}_{1/a}(\tilde{u}(z_i))=\emptyset.$
The Monotonicity Lemma (Lemma \ref{lmMonEst++}) thus  guarantees
\begin{equation}
E(\tilde{u};S_i)=Area(S_i)+E(u;S_i)\geq \frac{c}{a^4}.
\end{equation}

%We have
%\begin{equation}
%\tilde{B}_{1/a}(\tilde{u}(z_i))\subset B_a(u(z_i))\times \Sigma.
%\end{equation}
%It follows
%We obviously have that $u(S_i)\subset K$, %On the other hand, we claim $S_i\subset S$. To see this since $z_i\in S'$,% it follows by Lemma \ref{lmGrNonDec} that
%we have
%so $\tilde{u}(S_i)\subset S\times K$. }

On the other hand, if $i\neq j$ then $S_i\cap S_j=\emptyset$. Thus
\begin{align}
E(u;S)+Area(S)&\geq\sum_i Area(S_i)+E(u;S_i)\\
&\geq N\frac{c}{a^4}\notag\\
& \geq \frac{c}{2a^3} d_M(\partial^+U,\partial^-U)-\frac{c}{a^4}.\notag
\end{align}

%This is a somewhat more clever application of monotonicity of Lemma \ref{lmMonEst++}. The argument is as follows.
\end{proof}

For the sake of simplicity,  in the following Lemma we restrict attention  to $H,J$ which are domain independent Floer data on the cylinder. In the statement we refer to the function $\Gamma_{H,J}$ which is introduced in Definiton \ref{dfGamma}.
\begin{lemma}\label{lmLoopDisLin}
For any constant $c>0$ there are constants $\epsilon,R,C>0$ such that the following hold. Let $J$ be a compatible almost complex structure inducing a complete metric. Let $H:M\to\bR$ be a Hamiltonian and let $h_1<h_2\in H( M)$. Suppose on the region $H^{-1}([h_1,h_2])$ we have the geometry of $g_J$ is strictly $c$-bounded (see Definition \ref{dfStrBound}) and  that $H$ has Lipschitz constant $\leq \epsilon$. Then
\begin{equation}
\Gamma_{H,J}(h_1,h_2)\geq C(h_2-h_1-R).
\end{equation}
\end{lemma}
\begin{proof}
We rely on  \cite[Lemma 6.9]{groman} according to which if $\epsilon$ is taken small enough there are constants $R,\delta>0$ so that for any $h\in H(M)$ we have $\Gamma_{H,J}(h,h+R) >\delta$. We now argue as follows. Let $u:[a,b]\times S^1\to M$ be a Floer solution with one end in $\{H<h_1\}$ and the other end in $\{H>h_2\}$. If no such $u$ exists, the claim holds vacuously. We may further assume that $h_2-h_1>R$ or else the statement holds automatically by positivity of geometric energy.  Define functions $f,g:[a,b]\to\bR$ by $f(s)=\min_t H(u(s,t))$ and $g(s):=\max_t H(u(s,t))$. Without loss of generality $g(a)=h_1$ and $f(b)=h_2$ and $g(s)>h_1,f(s)<h_2$ for $s\in (a,b)$. We inductively construct a sequence
\[
s_0=a<s_1<\dots<s_N<b
\]
with the property that $f(s_i)=g(s_{i-1})+R$ and which is a maximal such sequence. Namely, let $s_1$ the smallest $s$ in the interval $[s_0,b]$ for which $f(s)=g(s_0)+R$. Such an $s$ exists since we assume $f(b)-g(a)=h_2-h_1>R$. Inductively, let $s_i$ be the smallest $s$ on the interval $[s_{i-1},b] $ for which $f(s)=g(s_{i-1})+R$.   Take $N=i-1$ for the first $i$ for which no such $s$ exists.

We associate with each $0<i<N$ a subset $S_i\subset [a,b]\times S^1$ as follows. The reader might find it helpful to consult Figure \ref{diss} to follow along. Let
\begin{equation}
\tilde{S}_i=(H\circ u)^{-1}([f(s_i),g(s_i)]).
\end{equation}
let $\tilde{S}_i'$ be the connected component of $\tilde{S}_i$ containing the circle $s=s_i$. Let $a_i=\min\left\{s|(s,t)\in \tilde{S}_i'\right\}$ and $b_i= \max\left\{s|(s,t)\in \tilde{S}_i'\right\}.$ Let
\begin{equation}
S_i^+=\begin{cases} [s_i,s_i+1]\times S^1,&\quad b_i-s_i\geq 2,\\
				\tilde{S}_i'\cap[s_i,b],&\quad b_i-s_i<2.
				\end{cases}
\end{equation}

\begin{figure}
\includegraphics[scale =0.5, trim=1500 250 1500 200,clip]{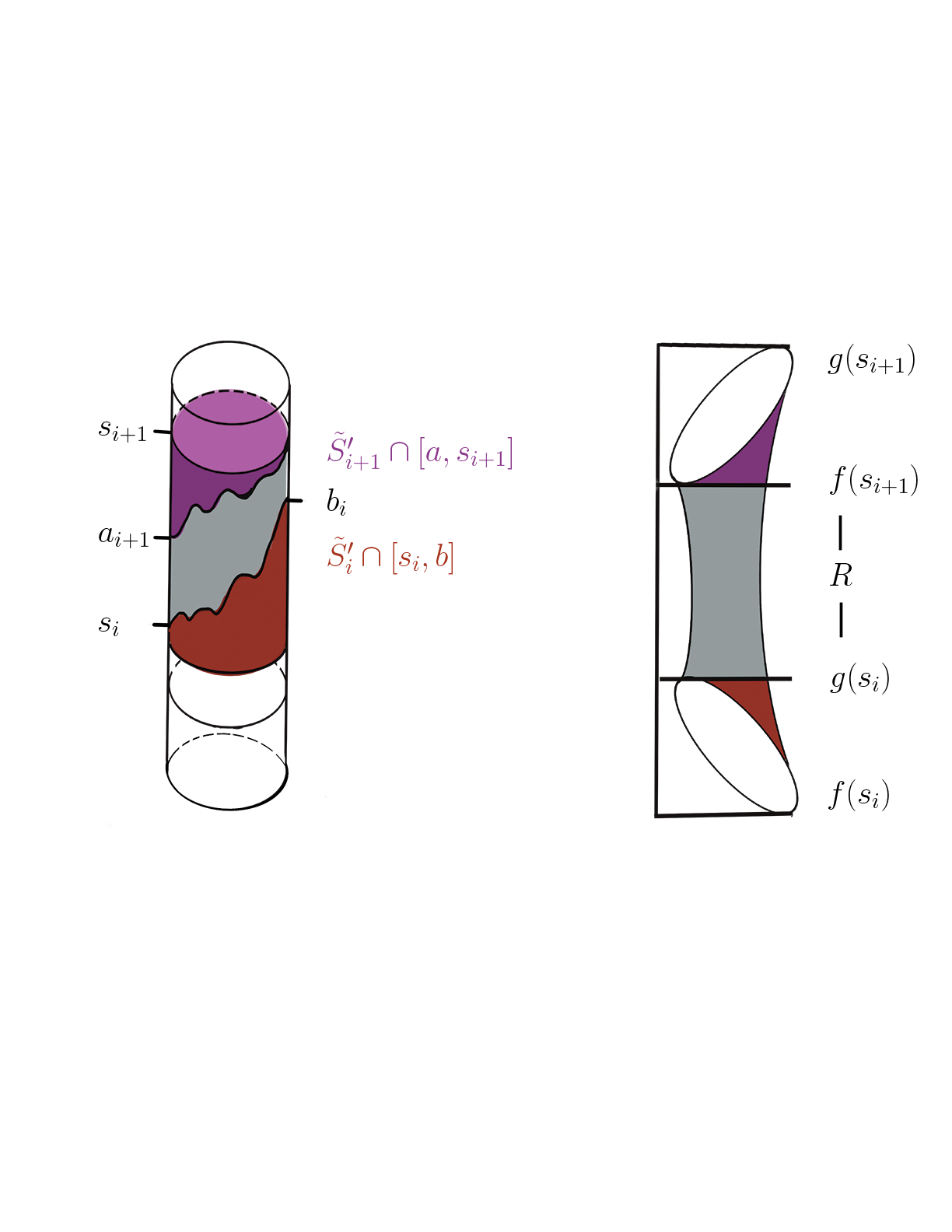}
\caption{}
\label{diss}
\end{figure}
Observe that in any case, we have $S_i^+\subset [a,b]\times S^1$.
%$S_i^+$ be the set $[s_i,s_i+1]\times S^1$ if $b_i-s_i\geq 2$, and $\tilde{S}_i'\cap[s_i,b]$ otherwise.
Similarly, let $S_i^-$ be the set $[s_i-1,s_i]\times S^1$ if $a_i\leq s_i-2$ and $[a,s_i]\cap\tilde{S}'_i$ otherwise. Finally, let $S_i=S_i^+\cup S_i^-$. Observe we have the following
\begin{enumerate}
\item If $|i- j|>1$ then $S_i\cap S_j=\emptyset$. To see this note for any $s\in [s_i,b_i]$ we have $f(s)\leq g(s_i)<f(s_{i+1})$. This implies $s_{i+1}>b_i$. From this it follows that $S_i^+\cap S_{i+1}^+=\emptyset$. Similarly we have $S_{i-1}^-\cap S_{i}^-=\emptyset$. Thus if $S_i\cap S_j\neq 0$ we have either $i=j$ or $|i-j|=1$. 
\item  $E(u;S_i)+Area(S_i)\geq \frac{C}{\epsilon}(g(s_i)-f(s_i))-D$ for $C,D$ as in Lemma \ref{lmDiamEstFreeBo}. Indeed each component of the boundary of $S_i$ is either a distance of $1$ away from the line $s=s_i$ or its image under $u$ does not meet the interior of the the region $K=H^{-1}[f(s_i),g(s_i)]$. Denote by $d$ the distance between the boundary components of $K$. Then by the Lipschitz property of $H$ we have that $g(s_i)-f(s_i)\leq\epsilon d$. The claim thus follows  Lemma \ref{lmDiamEstFreeBo}. For simplicity assume from now on $\epsilon\leq 1$. 
\item We have $Area(S_i)\leq 2$ and therefore, by the first item, $ \sum_iArea(S_i)\leq 4N$. 
\end{enumerate}
Thus summing the inequalities of the second item  and taking $S=\cup_i S_i$ we obtain
\begin{equation}
2E(u;S)\geq\sum_{i=1}^{N}\frac{C}{\epsilon}(g(s_i)-f(s_i))-(4+D)N.
\end{equation}
{The factor of $2$ accounts for the potential of double counting since item 1 allows $S_i\cap S_{i+1}\neq 0$}. 
We have $h_2-h_1\leq \sum_i(g(s_i)-f(s_i))+(N+1)R$. So, writing $\tilde{R}:=\frac{C}{\epsilon}R$ we can convert the last inequality into
\begin{equation}
2E(u)\geq \frac{C}{\epsilon}(h_2-h_1)-(\tilde{R}+D+4)N - \tilde{R}.
\end{equation}
On the other hand by \cite[Lemma 6.9]{groman} we have
\begin{equation}
E(u)\geq N\delta.
\end{equation}
Combining the last two inequalities to eliminate $N$ we obtain
\begin{equation}
E(u)\geq \frac{C\delta}{\epsilon(\tilde{R}+D+4+2\delta)}(h_2-h_1-\tilde{R}).
\end{equation}
{The claim now follows by renaming the constants}.
\end{proof}
\begin{lemma}\label{lemmaLipBounEnd}
Let $J$ be a compatible almost complex structure inducing a complete metric on $M$.  Let $H:M\to\bR$ be a proper Hamiltonian. Let $h_{\pm}\in H(M)$ and let $c$ be a constant  strictly bounding  the geometry of $g_J$ on the region $H^{-1}(h_+,h_-)$. Suppose $H$ is $\epsilon$-Lipschitz where $\epsilon=\epsilon(c)$ is given by Lemma \ref{lmLoopDisLin}. Suppose $u$ is a Floer solution defined along $[a-1,b+1]\times S^1$ and $h_{\pm}$ are the maximum and minimum respectively of $H\circ u$ on $[a,b]\times S^1$. Then
\begin{equation}
E_{top}(u;[a-1,b+1]\times S^1)\geq C(h_+-h_--R)
\end{equation}
for $C=C(c), R=R(c)$.
\end{lemma}
\begin{proof}
Define the functions $f,g$ as in the proof of Lemma \ref{lmLoopDisLin}. If there is an $s\in[a,b]$ such that $g(s)-f(s)>\frac1{3}(h_+-h_-)$  the estimate follows by considering the case $S=[s-1,s+1]\times S^1$ and $\partial^+S=\emptyset$ in Lemma \ref{lmDiamEstFreeBo}. Otherwise, subdivide the interval $[h_-,h_+]$ into equal interval with endpoint labeled $h_0<h_1<h_2<h_3$. Then by assumption there is a pair $s_1,s_2\in[a,b]$ such that $g(s_1)<h_1$ and $f(s_2)>h_2$.  Then
\begin{equation}
E(u;[a,b]\times S^1)>\Gamma_{H,J}(h_1,h_2)
\end{equation}
by definition of $\Gamma_{H,J}$. 
Since  $h_2-h_1= \frac1{3}(h_+-h_-)$ we are done by Lemma \ref{lmLoopDisLin}.
\end{proof}

\begin{lemma}\label{lemmaLipBounEnd2}
Let $\Sigma$ be a Riemann surface with $n$ inputs and $1$ output. Let $\alpha$ be a closed $1$-form on $\Sigma$ equalling $dt$ on the negative ends and $ndt$ on the positive end. Let $H:\Sigma\times M\to\bR$ be smooth such that denoting by $\pi_1:\Sigma\times M\to\Sigma$ the projection we have that $\pi_1\times H$ is proper. Let $\Sigma_0\subset\Sigma$ such that $\Sigma\setminus\Sigma_0$ is a union of cylindrical ends and $H$ is locally domain independent on the complement of $\Sigma_0$. Let $J$ be a compatible almost complex structure inducing a complete metric.  Let $c>0$ be a constant such that on  $H^{-1}(h_1,h_2)$ we have the geometry of $g_J$ is strictly $c$-bounded. Assume that on $\Sigma\setminus\Sigma_0$, $H$ has Lipschitz constant $\leq \frac{\epsilon}{n}$ for $\epsilon=\epsilon(c)$ as in Lemma \ref{lmLoopDisLin}. Let
\[
u:\Sigma\to M
\]
be a map satisfying Floer's equation with datum $(H,\alpha,J)$. Then if $u$ meets both sides of $ M\setminus H^{-1}(h_1,h_2)$ we have
\begin{equation}
E_{geo}(u)+Area(\Sigma_0)+2\geq C(h_2-h_1-R),
\end{equation}
with the constants depending only on $\Sigma_0$ and the bounds on $(\alpha,H,J)$ in the region $H^{-1}(h_1,h_2)$.
\end{lemma}
\begin{proof}
Call a pair of points $p_1,p_2$ \emph{nearby}  if there is a path $\gamma$ connecting $p_1,p_2$ such that $Area(B_1(\gamma))<Area(\Sigma_0)+2$. For appropriately chosen area form on $\Sigma$\footnote{Note the choice of area form affects the constants appearing in the various estimates in this section, via the geometry of the Gromov metric, but plays no role in Floer's equation.}, any two points on $\Sigma_0$ are nearby and any two points in an end of the form $(s,t_0),(s,t_1)$ are nearby. If there are two nearby points such that $u(p_1),u(p_2)\in H^{-1}(h_1,h_2)$ and $H(u(p_1))-H(u(p_2))>\frac1{7}(h_2-h_1)$ we are done by Lemma  \ref{lmDiamEstFreeBo}. Otherwise, {we deduce that the oscillation of $H$  on $\Sigma_0\cap u^{-1}([h_1,h_2])$  is $\leq \frac17(h_2-h_1)$. Since the oscillation of $H$ on $u^{-1}([h_1,h_2])$ is $h_2-h_1$ we deduce there is a subset $[a,b]\times S^1$ in one of the ends so that the oscillation of $H$ on $[a,b]\times S^1\cap u^{-1}([h_1,h_2])$ is at least $\frac{3}{7}(h_2-h_1)$. For each $s\in [a,b]$ the oscillation on $\{s\}\times S^1\cap u^{-1}([h_1,h_2]$ is also $\leq\frac1{7}(h_2-h_1)$. It follows that there is an $I\subset [h_1,h_2]$ of length $\geq\frac17(h_2-h_1)$ and an interval $[a',b']\subset [a,b]$ so that $u( [a',b']\times S^1)$} maps into $H^{-1}((h_1,h_2))$ and so that each boundary is contained in another component of $u^{-1}(I)$. The claim then follows by Lemma \ref{lemmaLipBounEnd}. Either way we are done.

\end{proof}

\subsection{$\lambda$-separation for Hamiltonians with small Lipschitz constant on a region}\label{SecLipBoundDiff}
We now recast the results of the previous section in the form which is most useful for us in the succeeding discussion.

\begin{proposition}[$C^0$ estimate for the Floer differential]\label{prpLipBoundDiff}
Let $J$ be an $\omega$-compatible geometrically bounded almost complex structure. Let $H$ be a smooth proper Hamiltonian, and let $[a,b]$ be an interval so that on the region $H^{-1}([a,b])$ we have that  $H$ is time independent. Let $c$ be a bound on the geometry of $J$ on $H^{-1}([a,b])$ and let $\epsilon$ be a bound on the  the Lipschitz constant of $H$.% considered as a function on $S^1\times H^{-1}([a,b])$.

Then there are constants $\epsilon(c), C(c)$ and $R(c)$  with the following significance. If $\epsilon<\epsilon(c)$, then any solution $u:\bR\times S^1\to M$ to Floer's equation for which
\begin{equation}
\min (H\circ u)<a<b<\max (H\circ u)
\end{equation}
satisfies
\begin{equation}
E(u)>C(c)(b-a-R(c)).
\end{equation}
\end{proposition}
\begin{proof}
This follows immediately from Lemma \ref{lemmaLipBounEnd}.
\end{proof}
\begin{proposition}[$C^0$ estimates for continuation maps with changing constant]\label{prpLipBoundCont}
Let $J, H$ and $[a,b]$ be as in proposition \ref{prpLipBoundDiff}  with some $c>0$ and $\epsilon<\epsilon(c)$. Let $H^s$ be a smooth family of Hamiltonians such that $\partial_sH^s\geq 0$ and such that $\partial_sH^s\equiv 0$ for $|s|\geq 1$. Suppose that on the region $H^{-1}([a,b])$ we have $H^s-H=f(s)$ for some  smooth function $f:\bR\to\bR$. Then any solution $u$ to Floer's equation for which
\begin{equation}
\min (H\circ u)<a<b<\max (H\circ u)
\end{equation}
satisfies
\begin{equation}
E(u)>C(c)(b-a-R(c)).
\end{equation}
\end{proposition}
\begin{proof}
The assumption concerning $H^s-H$ on the region $K=H^{-1}([a,b])$ implies that on $K$ the Gromov metric, which depends only on the derivatives in the direction of $M$, satisfies the same bounds as those of $H$.  The claim now follows  from Lemma \ref{lemmaLipBounEnd2} after taking $\Sigma_0=[-1,1]\times S^1$.
\end{proof}

%The same holds if one considers for an $s$-dependent Hamiltonian $H^s$ provided $H$ has small Lipschitz constant on $H^{-1}([a,b])$ as a function of $\bR\times S^1\times M$ and $(H_{\pm\infty},J)$ is dissipative. In this case, the function $f$ depends in addition on estimate on the size the union of intervals outside of which $\partial_sH=0$. Similarly, given a closed one form $\alpha$ on the pair of pants which coincides with $dt$ on the ends, and an $H:\Sigma\times M\to\bR$  with small Lipschitz constant on $H^{-1}([a,b])$ and so that $H\alpha$ is monotone, then the same holds for solutions to Floer's equation \eqref{eqFlProd}.

%\end{proposition}

%\begin{remark}
%In fact, one can take the function $f(b-a)$ to be linear. However we have no use for this fact.
%\end{remark}

%\begin{remark}
%This  is sufficient for proving the locality isomorphism at the level of $\Lambda$-modules with after making a choice of Floer data for the local model. It also suffices for proving that the locality isomorphism commutes with the $BV$-operator, as long as we choose our local model to be time independent outside of a compact set. To prove independence of the local choice, as well as to prove that the locality isomorphism respects products requires more delicate and involved statements which will be done later in the section.
%\end{remark}
\begin{proposition}[$C^0$ estimates for pairs of pants]\label{prpLipBoundProd}
Let $J, H$ and $[a,b]$ be as in proposition \ref{prpLipBoundDiff}  with some $c>0$ and $\epsilon<\epsilon(c)$.  Denote by $\Sigma$ the pair of pants with two inputs and one output. Let $\alpha$ be a closed $1$-form on $\Sigma$ equaling $dt$ on the inputs and $2dt$ on the output.  Fix once and for all a compact domain $\Sigma_0\subset\Sigma$ whose complement consists of cylindrical ends.  Let $H_{\Sigma}:\Sigma\times M\to\bR$ be a smooth function such that $dH_{\Sigma}\wedge\alpha\geq 0$ and such that $dH_{\Sigma}$ is supported on $\Sigma_0$.  Suppose that on the region $H^{-1}([a,b])$ we have $H_{\Sigma}-H=f_\Sigma$ where $f_\Sigma:\Sigma\to \bR$ is a smooth function. Then any solution $u$ to Floer's equation for the datum $(H_{\Sigma},\alpha,J)$ for which
\begin{equation}
\min (H\circ u)<a<b<\max (H\circ u)
\end{equation}
satisfies
\begin{equation}
E(u)>C(c)(b-a-R(c)).
\end{equation}
\end{proposition}
\begin{proof}
This again follows immediately from Lemma \ref{lemmaLipBounEnd2}.
\end{proof}

\section{Proofs of Theorems \ref{thmLocalityTrunc} and \ref{thmLocalityTruncFunc}}\label{secMainProof}

Let $M$ be a geometrically bounded symplectic manifold and $Y\subset M$ be an open subset. We assume that $Y$ is geometrically of finite type.

\begin{definition}We call an acceleration datum $(H_\tau,J_\tau)$ for $K\subset Y$, considered as a compact subset of $M$, \emph{$\lambda$-separating} if

\begin{enumerate}
\item For each $i\in\mathbb{N}$, $(H_i,J_i)$ is $\lambda$-separating for $K$.
\item For each $i\in\mathbb{N}$, the Floer datum associated to  $(H_\tau, J_\tau)_{\tau\in [i,i+1]}$ is $\lambda$-separating for $K$.
\end{enumerate}
\end{definition}

\begin{definition}
A \emph{Locality datum}  is a triple $(H_{loc}, J_{loc},h_0)$ consisting of a Floer datum $(H_{loc},J_{loc})$ on $Y$ and a real number $h_0$
%\begin{itemize}
	%\item  a compact domain $D\subset Y$,
%	\item a real number $h_0>0$,
%	\item an almost complex structure $J_{loc}$ on $Y$, and
%	\item  a Hamiltonian function $H_{loc}:Y-int(D)\to [h_0,\infty) $.
%\end{itemize}
such  that denoting  $D=H_{loc}^{-1}((-\infty,h_0])$
\begin{itemize}
	\item $J_{loc}$ is  geometrically bounded.%$c$-bounded for some $c>0$. {\color{blue}This means $g_{J_{loc}}$ can be extended to a complete metric on $Y$ which is $c$-bounded on $Y-int(D)$.}
	\item $H_{loc}$ is proper and bounded below.
	\item All $1$-periodic orbits of $H_{loc}$ are contained inside $D$.
	\item The Lipschitz constant of $H_{loc}|_{Y\setminus D}$ is less than $\epsilon(c)$ from Proposition \ref{prpLipBoundDiff} where $c$ is a bound on the geometry of $g_J$ on $Y\setminus D$.
\end{itemize}

In what follows we will generally drop $h_0$ from the notation.
\end{definition}

Note that if $(H_{loc}, J_{loc},h_0)$ is a locality datum so is $(H_{loc}+C, J_{loc},h_0+C)$ for any $C\in\bR$.

\begin{proposition}\label{propLocDatExist}
There exist locality data $(H_{loc}, J_{loc},h_0)$ on $Y$.
\end{proposition}
\begin{proof}
Since $Y$ is geometrically of finite type, there exists a compatible almost complex structure $J$ such that $g_J$ is geometrically bounded and an admissible function $f: Y\to \mathbb{R}$ with no critical points outside of a compact set $K$. Let us choose $f_0\in\mathbb{R}$ so that $f|_K<f_0$ and $f_0$ is a regular value of $f$. % Define $D:=f^{-1}((-\infty, f_0])$.
By Lemma \ref{lemNoNonTOrbits} we can find an $\epsilon_0>0$ such that the Hamiltonian $\epsilon_0 f$ has no non-constant $1$-periodic orbits outside of $K$. Let $\epsilon>0$ be as given by Proposition \ref{prpLipBoundDiff}, let $d$ be the Lipschitz constant of $f$, and define $\epsilon':=min(\epsilon/d,\epsilon_0)$.

Define $J_{loc}=J$, $H_{loc}:=\epsilon' f$ and  $h_0:=\epsilon' f_0$. By construction, all the requirements are satisfied.
\end{proof}

\begin{remark}
The numbered definitions and observations inside the proofs of Theorems \ref{thmLocalityTrunc} and \ref{thmLocalityTruncFunc} will not be used elsewhere in this paper. We put them as such to organize the proof in a better way.
\end{remark}

\begin{proof}[Proof of Theorem \ref{thmLocalityTrunc}]
We identify $Y$ with its image under $\iota$ and use the definitions and notation set in the beginning of the present section. We will follow the plan below:
\begin{enumerate}
\item For fixed locality data $(H_{loc}, J_{loc})$, construct canonical $\Lambda_{\geq 0}$-module maps $$f_{K,\lambda,(H_{loc}, J_{loc})}: SH_{Y,\lambda}^*(K)\to SH_{M,\lambda}^*(K),$$ for all compact subsets $K\subset Y$ and $\lambda\geq 0$.
\item Show that the maps $f_{K,\lambda,(H_{loc}, J_{loc})}$ give a natural transformation of functors. This means that they are compatible with restriction and truncation maps.
\item Show that the maps $f_{K,\lambda,(H_{loc}, J_{loc})}$ are all isomorphisms.
\item Prove that $f_{K,\lambda,(H_{loc}, J_{loc})}$ is independent of locality data $(H_{loc}, J_{loc})$ and hence it can be dropped from notation.
\item Show that the maps $f_{K,\lambda}$ are maps of BV-algebras.
\end{enumerate}

%Recall that we have shown the existence of locality data in Proposition \ref{propLocDatExist}.
%\\

We start with (1) and fix locality data $(H_{loc},J_{loc})$ for $Y$.

\begin{definition}\label{defDeltaComp}
Let $\Delta>0$ be a positive real number. An acceleration datum $(H_\tau ,J_\tau)$ for $K\subset Y$ inside $M$ is called \emph{$\Delta$-compatible with $(H_{loc}, J_{loc})$} if there exists an $h>h_0$ such that  $K$ is contained inside $H_{loc}^{-1}((-\infty, h))$, and for every $\tau\in [1,\infty)$,  on $H_{loc}^{-1}([h,h+\Delta])$: \begin{itemize}
\item $H_\tau$ is time independent and $H_\tau-H_{loc}$ is a constant function.
\item $J_\tau$ is time independent and equal to $J_{loc}$.
\end{itemize}
Let us denote the smallest $h$ with this property by $h(H_\tau,J_\tau)$.
\end{definition}

Note that by construction the Hamiltonian flow of such an $H_\tau$ preserves the region $D_{\Delta,h}:=
H_{loc}^{-1}([h,h+\Delta])
$
as mentioned in this definition, and it has no $1$-periodic orbit inside it. We refer to $D_{\Delta,h}$ as the \emph{$\Delta$-compatibility region.} In what follows we will repeatedly talk about orbits inside and outside relative to such a region. \emph{Inside} refers to the region $H_{loc}^{-1}((-\infty,h))\subset Y$ considered as a subset of $M$ and \emph{outside} refers to the region $M\setminus H_{loc}^{-1}((-\infty,h+\Delta))$.

\begin{obs}\label{propDeltaCompExist}
For any $\Delta>0$ there exists an acceleration datum $(H_\tau,J_\tau)$ for $K\subset Y$ inside $M$ that is $\Delta$-compatible with $(H_{loc}, J_{loc})$. In fact, we can arrange it so that $h(H_\tau,J_\tau)$ is any number larger than $\max\{\max_KH_{loc},h_0\}.$\qed

\end{obs}

 To see this we define $(H_\tau, J_\tau)$ to be equal to $(H_{loc}+\tau,J_{loc})$ on $H_{loc}^{-1}([h,h+\Delta])\subset Y\subset M$. Since $K\subset H_{loc}^{-1}((-\infty, h))\subset Y$ is disjoint from $H_{loc}^{-1}([h,h+\Delta])$ , we are free to define the acceleration data as we please in a precompact neighborhood $U$ of $K$. Moreover, since $H_{loc}^{-1}([h,h+\Delta])\cup U\subset M$ is bounded, that the acceleration data is pre-defined in this region does not pose any difficulty in terms of achieving dissipitavity for the extensions to $M\times S^1$.

By definition of locality data  and Propositions \ref{prpLipBoundDiff} and \ref{prpLipBoundCont} there exist constants $C,R>0$  such that a $\Delta$-compatible acceleration datum is $C(\Delta-R)$-separating. Therefore, for a given $\lambda>0$, we define $\Delta(\lambda):=\frac{\lambda}{C}+R+1>0$. An acceleration datum that is $\Delta(\lambda)$-compatible with $(H_{loc}, J_{loc})$ is $\lambda$-separating.

Let $(H_\tau,J_\tau)$ be an acceleration datum $\Delta(\lambda)$-compatible with $(H_{loc}, J_{loc})$. Then Lemma \ref{lmSplit} produces a splitting
\[
HF^*_\lambda(H_i)=HF^*_{\lambda,inner}(H_i)\oplus HF^*_{\lambda,outer}(H_i),
\]
for every $i\in\mathbb{Z}_{\geq 1}$. %Here we declare the $1$-periodic orbits of $H_i$ that lie inside $H_{loc}^{-1}([h(H_\tau,J_\tau),h(H_\tau,J_\tau)+\Delta])$ the inner orbits and the ones that lie outside $H_{loc}^{-1}([h(H_\tau,J_\tau),h(H_\tau,J_\tau)+\Delta])$ the outer orbits.
These splittings are compatible with continuation maps  in the sense that $HF^*_\lambda(H_i)\to HF^*_\lambda(H_{i+1})$ is the direct sum of the induced maps $$HF^*_{\lambda,inner}(H_i)\to HF^*_{\lambda,inner}(H_{i+1}),$$ and $$HF^*_{\lambda,outer}(H_i)\to HF^*_{\lambda,outer}(H_{i+1}).$$

Let us now define $H_\tau^Y:Y\times S^1\to \mathbb{R}$ and $S^1$-dependent almost complex structure $J_\tau^Y$ on $Y$ for all $\tau\in[1,\infty)$ as follows. On $H_\tau^{-1}((-\infty,h(H_\tau,J_\tau)+\Delta(\lambda)])\times S^1$, we set $H_\tau^Y$ to $H_\tau$ and $J_\tau^Y$ to $J_\tau$. On $H_{loc}^{-1}([h(H_\tau,J_\tau),\infty))\times S^1$, we declare $H_\tau^Y$ to be time independent and equal to $H_{loc}$ up to a constant so that the resulting function is smooth and $J_\tau^Y$ also be to time independent and equal to $J_{loc}$. That this construction makes sense is a direct consequence of Definition \ref{defDeltaComp}.

Clearly, $(H_\tau^Y, J_\tau^Y)$ is an acceleration datum for $K$ inside $Y$. Moreover, using Propositions \ref{prpLipBoundDiff} and \ref{prpLipBoundCont} inside $Y$ and the fact that $H_{loc}$ has no periodic orbits in the region $\{H_{loc}>h\}$ we have commutative diagrams \begin{align*}
\xymatrix{
HF^*_\lambda(H_i^Y)\ar[d]\ar[r]&HF^*_{\lambda,inner}(H_i) \ar[d]\\ HF^*_\lambda(H_{i+1}^Y)\ar[r] &HF^*_{\lambda,inner}(H_{i+1}),}
\end{align*}where the horizontal maps are isomorphisms, left vertical map is the continuation map and the right vertical map is the one that we defined above.

Using the canonical inclusions $HF^*_{\lambda,inner}(H_i)\to HF^*_{\lambda}(H_i)$
and taking direct limits over $i$, we obtain a map $$f_{K,\lambda,(H_{loc},J_{loc})}(H_\tau,J_\tau): SH_{Y,\lambda}^*(K)\to SH_{M,\lambda}^*(K).$$

To be done with part (1) of the plan, we need to show that if $(H_\tau',J_\tau')$ is another acceleration datum $\Delta(\lambda)$-compatible with $(H_{loc}, J_{loc})$ then $$f_{K,\lambda,(H_{loc},J_{loc})}(H_\tau',J_\tau')=f_{K,\lambda,(H_{loc},J_{loc})}(H_\tau,J_\tau)$$.

\begin{obs}
Given acceleration data $(H_\tau,J_\tau)$ and $(H_\tau',J_\tau')$ for $K\subset Y$ inside $M$ that are respectively $\Delta$ and $\Delta'$-compatible with $(H_{loc}, J_{loc}),$ we can find a third acceleration datum $(H_\tau'',J_\tau'')$ that is $\Delta''$-compatible with $(H_{loc}, J_{loc})$ with the property that the interval
\[
[h(H_\tau'',J_\tau''),h(H_\tau'',J_\tau'')+\Delta'']
\]
 contains both intervals
 \[
 [h(H_\tau,J_\tau),h(H_\tau,J_\tau)+\Delta]
 \]
 and
 \[
 [h(H_\tau',J_\tau'),h(H_\tau',J_\tau')+\Delta'],
 \]
 for some $\Delta''>0$. Moreover, we can make sure that $H_i''(t,x)$ is greater than or equal to both $H_i(t,x)$ and $H_i'(t,x)$ for all $i\in \mathbb{N}$ and $(t,x)\in S^1\times M.$\qed

\end{obs}

We choose a third acceleration datum $(H_\tau'',J_\tau'')$ as in this observation (we have $\Delta=\Delta'$). Then we can choose monotone Floer data from $(H_\tau,J_\tau)$ (resp. $(H_\tau',J_\tau')$)  to  $(H_\tau'',J_\tau'')$ which agree up to an $s$-dependent constant with $(H_{loc},J_{loc})$ on the interval $[h(H_\tau,J_\tau),h(H_\tau,J_\tau)+\Delta]$ (resp. $[h(H_\tau',J_\tau'),h(H_\tau',J_\tau')+\Delta']). $ Then Proposition \ref{prpLipBoundCont} and the Hamiltonian Floer theory package imply that $f_{K,\lambda,(H_{loc}, J_{loc})}(H_\tau,J_\tau)$ and $f_{K,\lambda,(H_{loc}, J_{loc})}(H_\tau',J_\tau')$ are both equal to $f_{K,\lambda,(H_{loc}, J_{loc})}(H_\tau'',J_\tau'').$ This finishes the proof of the independence result and hence the item (1) from the master plan.\\

We will refer to the map that we just constructed (i.e. $f_{K,\lambda,(H_{loc},J_{loc})}$) the \emph{$(Y\subset M)$-locality map with data $(H_{loc},J_{loc})$} we will refer to it as just the \emph{locality map} when $M$ is clear from the context. Until we come to part (4) of our plan, we work with fixed $(H_{loc},J_{loc})$ and drop it from notation.

Now we move on to the part (2) of our master plan. We need to show that for every compact pair $K\subset K'$ of subsets of $Y$ and $\lambda'\geq \lambda$, the diagram \begin{align*}
\xymatrix{
SH_{Y,\lambda'}^*(K')\ar[d]\ar[r]&SH_{M,\lambda'}^*(K') \ar[d]\\ SH_{Y,\lambda}^*(K)\ar[r] &SH_{M,\lambda}^*(K),}
\end{align*} commutes. Here the horizontal maps are the locality maps, and the vertical ones are the combination of restriction and truncation maps.

Let's first assume $K=K'$. Then, the statement is trivial because an acceleration datum for $K\subset Y$ that is $\Delta(\lambda')$-compatible with $(H_{loc}, J_{loc})$ is automatically  $\Delta(\lambda)$-compatible with $(H_{loc}, J_{loc})$. Therefore both locality maps in the diagram can be computed using the same acceleration data.

We are now left with dealing with the case $K\subsetneq K'$. Since we have already established compatibility with truncation it suffices to consider the case $\lambda=\lambda'$. We first choose  an  arbitrary acceleration datum $(H_\tau',J_\tau')$ for $K'\subset Y$ that is $\Delta(\lambda)$-compatible with $(H_{loc}, J_{loc})$. Using a slight extension of Observation \ref{propDeltaCompExist}, we can choose an  acceleration datum $(H_\tau,J_\tau)$ for $K\subset Y$ that is $\Delta(\lambda)$-compatible with $(H_{loc}, J_{loc})$ such that \begin{itemize}
\item $h(H_\tau,J_\tau)=h(H_\tau',J_\tau')$
\item $H_i(t,x)\geq H'_i(t,x)$ for every $i\in \mathbb{N}$ and $(t,x)\in S^1\times M.$
\end{itemize} For locality maps constructed with such acceleration data, we get that all Floer trajectories associated with the restriction maps from $K$ to $K'$ of energy $\leq \lambda$ and connecting inner orbits remain inside. This finishes part (2).\\

We move on to part (3). Let $(H_\tau,J_\tau)$ be an acceleration datum for $K\subset M$ that is $\Delta(\lambda)$-compatible with $(H_{loc}, J_{loc})$. Recall that we have a splitting \[
HF^*_\lambda(H_i)=HF^*_{\lambda,inner}(H_i)\oplus HF^*_{\lambda,outer}(H_i),
\] for every $i\in\mathbb{Z}_{\geq 1}$, which is compatible with continuation maps $HF^*_\lambda(H_i)\to HF^*_\lambda(H_{i+1}).$ In particular, we have canonical maps $$HF^*_{\lambda,outer}(H_i)\to HF^*_{\lambda,outer}(H_{i+1}).$$ All we need to show is that the direct limit of the diagram formed by these maps is 0.

Let us choose $(H_\tau,J_\tau)$ with the extra property that for every $i\in \mathbb{N}$, and $x\in M$ such that $x\in H_{loc}^{-1}([h(H_\tau,J_\tau),\infty)) $  and $t\in S^1$, we have  $$H_{i+1}(t,x)\geq H_{i}(t,x)+1.$$This can be easily achieved.

Any Floer solution that contributes to $HF^*_{\lambda,outer}(H_i)\to HF^*_{\lambda,outer}(H_{i+1})$ is contained in the region $M\setminus H_{loc}^{-1}((-\infty, h(H_\tau,J_\tau)))$ because of Proposition \ref{prpLipBoundCont}. Moreover, because of the extra condition we imposed on the acceleration data, the topological energy of each such Floer solution is at least $1$. Now every element in the colimit defining $SH^*_{\lambda,outer}$ comes from an $HF^*_{\lambda,outer}(H_i)$ for some $i$. But each such element has to map to zero inside $HF^*_{\lambda,outer}(H_i+N),$ where $N>\lambda$. This finishes the proof of (3).\\

We come to part (4) and bring back the dependence on $(H_{loc},J_{loc},h_0)$. We need to show that for any $K\subset M,\lambda>0$ (which we fix now), the locality map with data $(H_{loc},J_{loc},h_0)$ is the same as the one with data $(H_{loc}',J_{loc}',h'_0).$ For this we need to introduce a notion of being compatible with both $(H_{loc},J_{loc},h_0)$ and $(H_{loc}',J_{loc}',h'_0)$ simultaneously for an acceleration datum.

\begin{definition}\label{defRoofLocData}
Let $\Delta>0$ be a positive real number. An acceleration datum $(H_\tau ,J_\tau)$ for $K\subset Y$ inside $M$ is called \emph{$\Delta$-compatible with $(H_{loc}, J_{loc},h_0)$ and $(H_{loc}', J_{loc}',h'_0)$} if there are   regions $D_{\Delta,h}:=H_{loc}^{-1}([h,h+\Delta])$ and $D'_{\Delta,h'}:=H_{loc}^{' -1}([h',h'+\Delta])$ such that
\begin{itemize}
\item $h>\max\{h_0,\max H|_K\}$ and $h'>\max\{h'_0,\max H'|_K\}$
\item  $D_{\Delta,h}\cap D'_{\Delta,h'}=\emptyset$   %$h>h_0$ and $h'>h'_0$ such that %$H_{loc}$ is defined on $(H_{loc}')^{-1}(([h',\infty))$ and
 %$$H_{loc}((H_{loc}')^{-1}([h',h'+\Delta]))\subset (h+\Delta,\infty)$$
\item  for every $\tau\in [1,\infty)$ we have that on $D_{\Delta,h}$: \begin{itemize}
\item $H_\tau$ is time independent and $H_\tau-H_{loc}$ is constant function.
\item $J_\tau$ is time independent and equal to $J_{loc}$.
\end{itemize} and on $D'_{\Delta,h'}$: \begin{itemize}
\item $H_\tau$ is time independent and $H_\tau-H_{loc}'$ is a constant function.
\item $J_\tau$ is time independent and equal to $J_{loc}'$.
\end{itemize}
\end{itemize}
\end{definition}

Let $\Delta$ be a real number larger than both $\Delta(\lambda)$ and $\Delta'(\lambda)$. We construct the   locality map with data $(H_{loc},J_{loc})$ using a $\Delta$-compatible acceleration datum $(H_\tau ,J_\tau)$. Since $H_{loc}$ and $H_{loc}'$ are proper, we can find an $h'>h'_0$ such that %$H_{loc}$ is defined on $(H_{loc}')^{-1}(h',\infty))$ and
$$H_{loc}(D'_{\Delta,h'})\subset (h(H_\tau ,J_\tau)+\Delta,\infty).$$ 

\begin{figure}
\includegraphics[width=0.6\textwidth]{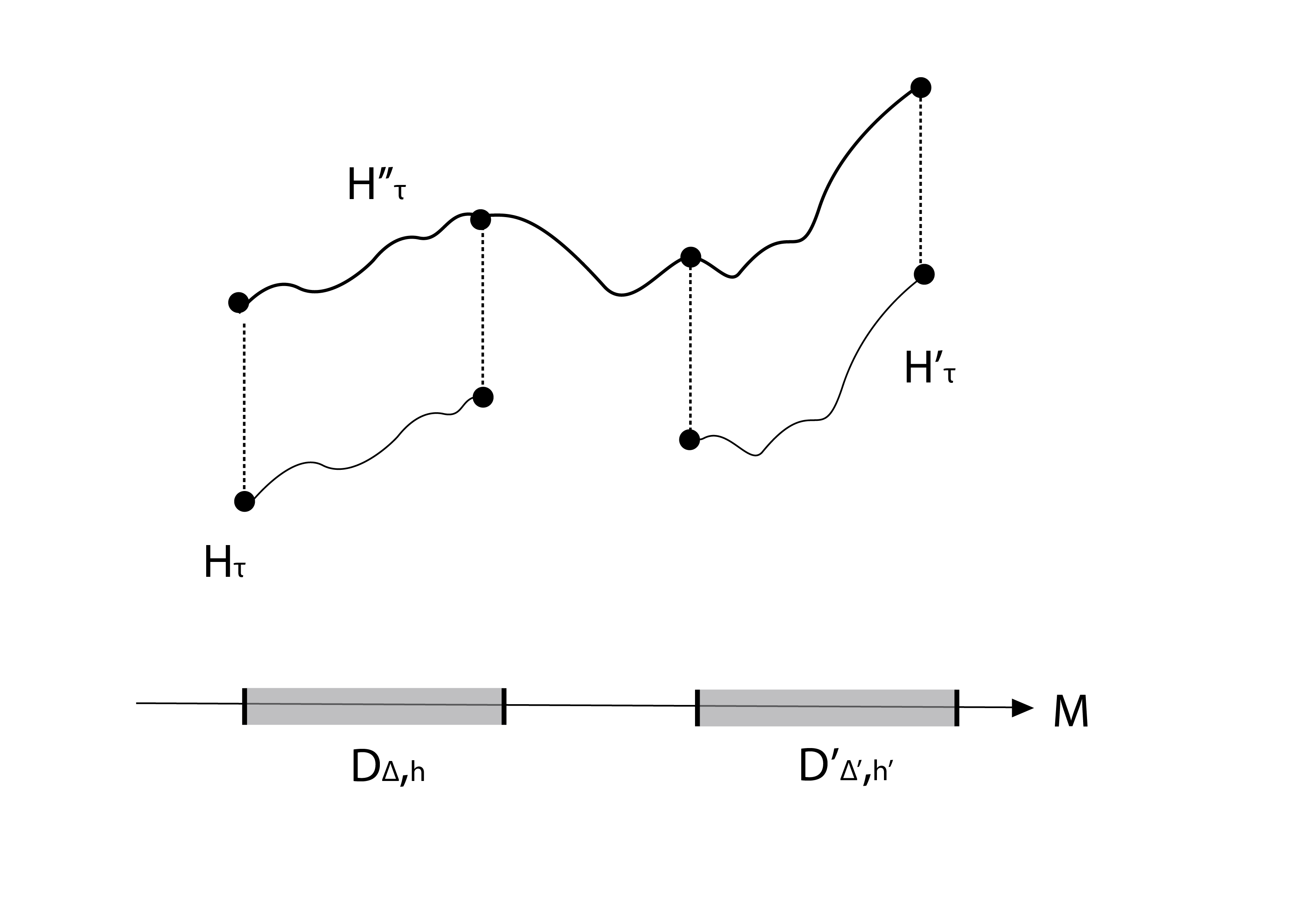}
\caption{A depiction regarding Observation \ref{obs3}}
\label{localitydata}
\end{figure}

Then, we choose acceleration data $(H_\tau' ,J_\tau')$ that is $\Delta$-compatible with $(H_{loc}',J_{loc}')$ satisfying $h(H_\tau' ,J_\tau')=h'$.

\begin{obs}\label{obs3} We can construct an acceleration datum $(H_\tau'' ,J_\tau'')$ that is $\Delta$-compatible with $(H_{loc}, J_{loc})$ and $(H_{loc}', J_{loc}')$ as in Definition \ref{defRoofLocData} with the $h$ in the statement equal to $h(H_\tau ,J_\tau)$ and $h'$ equal to $h(H_\tau' ,J_\tau').$ Moreover, we can make sure that $H_i''(t,x)$ is greater than equal to both $H_i(t,x)$ and $H_i'(t,x)$ for all $i\in \mathbb{N}$ and $(t,x)\in S^1\times M.$ See Figure \ref{localitydata}.
\end{obs}
Notice that we can use $(H_\tau'' ,J_\tau'')$ to construct locality maps in the formalism introduced in part (1) in two different ways. One by the virtue of it being $\Delta$-compatible with $(H_{loc},J_{loc})$ and two by $\Delta$-compatibility with $(H_{loc}',J_{loc}')$. Our goal is to prove that the two locality maps coincide at the level of truncated $SH$. The acceleration data that we construct for $K$ as a compact subset of $Y$,  which we had denoted with a superscript $Y$ above, are not the same in these two different ways. Let us denote these acceleration data by $((H_\tau'')^Y_1,(J_\tau'')^Y_1)$ and $((H_\tau'')^Y_2,(J_\tau'')^Y_2).$%{\color{purple} There needs to be a  diagram here comparing the two locality maps. The text appears to contain the full argument, but a diagram would make it a lot easier to follow.  I think what you need here is a refinement of the triangular diagram with the vertical edge the local continuation map comparing the two local acceleration data and the other two edges are the locality maps. The refinement of the vertical should factor through the  projection which forgets the in between orbits and the other sides should be refined similarly}

We can divide the $1$-periodic orbits of $(H_i'')^Y_2$ also into two groups \begin{itemize}\item the ones that are common with $(H_\tau'')^Y_1,$ called common inner orbits
\item the other ones, called in-between orbits.\end{itemize}

As in the discussion of inner and outer orbits above, we get a splitting \[
HF^*_\lambda((H_i'')^Y_2)=HF^*_{\lambda,common-inner}((H_i'')^Y_2)\oplus HF^*_{\lambda,in-between}((H_i'')^Y_2),
\] similarly compatible with continuation maps. We have $$HF^*_\lambda((H_i'')^Y_1)=HF^*_{\lambda,common-inner}((H_i'')^Y_2)$$ in a way that is compatible with the maps to $HF^*_\lambda(H_i'')$

Hence, we will be done if we can show that the direct limit of $$\ldots\to HF^*_{\lambda,in-between}((H_i'')^Y_2)\to HF^*_{\lambda,in-between}((H_{i+1}'')^Y_2)\to\ldots$$ is zero. This can again be more clearly seen by being slightly more careful in the choice of $(H_\tau'' ,J_\tau'')$ similarly to the step (3) above.

Namely, we choose $(H_\tau'',J_\tau'')$ with the property that for every $i\in \mathbb{N}$, and $x\in M$ such that $H_{loc}(x)\geq h(H_\tau,J_\tau)$ and $H_{loc}'(x)\leq h(H_\tau',J_\tau')$ and $t\in S^1$, we have  $$H_{i+1}''(t,x)\geq H_{i}''(t,x)+1.$$This can be easily achieved and we finish the proof of (4) analogously to step (3).\\

We now get to (5). We will leave compatibility with the BV operator to the reader in its entirety as its proof is another similar use of Proposition \ref{prpLipBoundCont}. We want to prove that $f_{K,\lambda}$ is an algebra homomorphism.

We start by choosing a locality datum $(H_{loc},J_{loc})$ for $Y$ such that $(2H_{loc},J_{loc})$ is also a locality datum. We then choose $(H_\tau,H'_\tau,J_\tau)$ such that\begin{itemize}
\item $(H_\tau,J_\tau)$ is an acceleration datum for $K$ that is $\Delta(\lambda)$-compatible with $(H_{loc},J_{loc})$
\item $(H_\tau' ,J_\tau)$ is an acceleration datum for $K$ that is $\Delta(\lambda)$-compatible with $(2H_{loc},J_{loc}).$
\item $h(H_\tau ,J_\tau)=h(H_\tau' ,J_\tau)$
\item For every $i\in\mathbb{Z}_{>0}$ and $x\in M$, $$\min_tH'_{i}(t,x)> 2\max_t H_{i}(t,x)$$
\end{itemize}

Now consider the pair-of-pants product maps
\begin{equation*}
*:HF^*_{\lambda}(H_i)\otimes HF^*_{\lambda}(H_i)\to HF^*_{\lambda}(H_i')
\end{equation*}
defined as in Section \ref{ssoperations}. Proposition \ref{prpLipBoundProd} proves that the diagram
\begin{align*}
\xymatrix{
HF^*_\lambda(H_i^Y)\otimes HF^*_\lambda(H_i^Y)\ar[d]\ar[r]&HF^*_{\lambda,inner}(H_i)\otimes HF^*_{\lambda,inner}(H_i)\ar[r]\ar[d] &HF^*_{\lambda}(H_i)\otimes HF^*_{\lambda}(H_i)  \ar[d]\\ HF^*_\lambda((H_{i}')^Y)\ar[r] &HF^*_{\lambda,inner}(H_{i}')\ar[r]&HF^*_{\lambda}(H_{i}'),}
\end{align*} is commutative. This finishes the proof using that $f_{K,\lambda}$ is independent of the locality data and $\Delta(\lambda)$-compatible acceleration data that is used to defined it.

\end{proof}
\begin{proof}[Proof of Theorem \ref{thmLocalityTruncFunc}]
Let us identify $Y$ with its image under $\iota_Y$ and $X$ with its image under $\iota_Y\circ\iota_X$ as in the proof of Theorem \ref{thmLocalityTrunc}.

Fix $K\subset X$ compact, $\lambda>0$, and also locality data $\left(H_{loc}^X, J_{loc}^X,h_0^X\right)$ for $X$ and $\left(H_{loc}^Y, J_{loc}^Y,h_0^Y\right)$ for $Y$.
%, neither of whose domain of definition intersect with $K$.
Without loss of generality we assume $h^X_0>\max H_{loc}^X|_K$ and $h^Y_0>\max H_{loc}^Y|_K$.
\begin{definition}
Let $\Delta>0$ be a positive real number. An acceleration datum $(H_\tau ,J_\tau)$ for $K\subset X$ inside $M$ is called \emph{$\Delta$-compatible with both $(H_{loc}^X, J_{loc}^X)$ and $(H_{loc}^Y, J_{loc}^Y)$} if there exist $h^X>h_0^X$ and $h^Y>h^Y_0$ such that %$H_{loc}^X$ is defined on $(H_{loc}')^{-1}([h^Y,h^Y+\Delta])$ and

$$
H_{loc}^Y(p)> h^Y\Rightarrow H_{loc}^X(p) > h^X+\Delta$$
and, moreover, for every $\tau\in [1,\infty)$, on $(H_{loc}^X)^{-1}([h^X,h^X+\Delta])$: \begin{itemize}
\item $H_\tau$ is time independent and $H_\tau-H_{loc}^X$ is a constant real number.
\item $J_\tau$ is time independent and equal to $J_{loc}^X$.
\end{itemize} and on $(H_{loc}^Y)^{-1}([h^Y,h^Y+\Delta])$: \begin{itemize}
\item $H_\tau$ is time independent and $H_\tau-H_{loc}^Y$ is a constant real number.
\item $J_\tau$ is time independent and equal to $J_{loc}^Y$.
\end{itemize}
\end{definition}

An $(H_\tau ,J_\tau)$ that is $\Delta$-compatible with both $(H_{loc}^X, J_{loc}^X)$ and $(H_{loc}^Y, J_{loc}^Y)$ can be constructed using properness similar to Step (4) of the proof of Theorem \ref{thmLocalityTrunc}.
We choose sufficiently large $\Delta>0$ so that such an acceleration datum $(H_\tau ,J_\tau)$ for $K$ inside $M$ works to define both a $X\subset M$-locality map and a $Y\subset M$-locality map for $(K,\lambda)$.

We can also define an acceleration data $(H_\tau^{X\subset Y} ,J_\tau^{X\subset Y})$  for $K$ inside $Y$ by replacing $(H_\tau|_Y ,J_\tau|_Y)$ for every $\tau\geq 1$ with $(H_{loc}^Y, J_{loc}^Y)$ on $(H_{loc}^Y)^{-1}([h^Y,\infty)$. Clearly,  $(H_\tau^{X\subset Y} ,J_\tau^{X\subset Y})$ is $\Delta(\lambda)$-compatible with $(H_{loc}^X, J_{loc}^X).$ Therefore, we can use the data  $(H_\tau^{X\subset Y} ,J_\tau^{X\subset Y})$ to define a $X\subset Y$-locality map for $(K,\lambda).$ The result follows easily.

\end{proof}
\section{Locality for reduced and relative symplectic cohomology}\label{Sec:Torsion}

Let $K$ be a compact subset inside a geometrically bounded symplectic manifold $M$. So far we have discussed results regarding the truncated symplectic cohomologies $SH_{M,\lambda}^*(K)$, where $\lambda\in\mathbb{R}_{\geq 0}$. In this section, we give the definitions of reduced symplectic cohomology $SH_{M,red}^*(K)$ and relative symplectic cohomology $SH_{M}^*(K)$. Then, we discuss locality isomorphisms for these invariants in the same geometric framework as for truncated symplectic cohomology.

We note to the reader that truncated (for a $\lambda\geq 0$), reduced and relative symplectic cohomologies are in general provably different invariants. We believe that they are all useful and we hope the difference in their constructions will be clear to the reader.

\begin{remark}
Let us also remark on why we prefer the adjectives reduced or relative over local. The main reason is that it sounds awkward to define an invariant to be local symplectic cohomology, but then prove that it is local in the sense that we explained in this paper only under some conditions. The name reduced in particular comes from the relationship with taking the reduced cohomology of a chain complex as explained in Section 8.1 of \cite{groman}. See in particular \cite[Theorem 8.4]{groman}. We do not discuss this viewpoint in the present paper.
\end{remark}
\subsection{Reduced symplectic cohomology}

Let us now define the reduced symplectic cohomology (it was called local symplectic cohomology in \cite{groman}, see Section 2.1) of $K\subset M$, $$SH_{M,red}^*(K):=\varprojlim_{\lambda}SH_{M,\lambda}^*(K).$$
Using elementary properties of the inverse limit, we can equip $SH_{M,red}^*(K)$ with a canonical $BV$-algebra structure and deduce the following theorem.

\begin{theorem}\label{thmRedSHDef}
The assignment
\[
\cSH^*_{Y\subset M,red}:(K,\lambda)\mapsto SH^*_{M,red}(K),
\]
which acts on morphisms by the inverse limit of the maps in \eqref{eqrestriction} is a functor
\[
\cSH^*_{Y\subset M,red}: \cK(Y)\to \Lambda_{\geq 0}-BValg.
\]\qed
\end{theorem}

Our locality theorem for truncated symplectic cohomology immediately implies:

\begin{theorem}\label{thmLocalityRed}
\begin{enumerate}
\item Let $Y$ be a symplectic manifold of geometrically finite type and let $M$ be one that is geometrically bounded. Let $\text{dim}(M)=\text{dim}(Y)$ and $\iota:Y\to M$ be a symplectic embedding. Also denote by $\iota$ the induced functor $\cK(Y)\to\cK(M)$. Then there exists a distinguished isomorphism of functors
\begin{equation}
T_{\iota,red}: \cSH^*_{Y\subset Y,red}\simeq \cSH^*_{\iota(Y)\subset M,red}\circ\iota.
\end{equation}
\item Let $X$ and $Y$ be symplectic manifolds of geometrically finite type and let $M$ be one that is geometrically bounded. Let $\text{dim}(M)=\text{dim}(Y)=\text{dim}(X)$ and $\iota_X:X\to Y$, $\iota_Y:Y\to M$ be symplectic embeddings. Then there is an equality of natural transformations $$T_{\iota_Y,red}\circ T_{\iota_X,red}=T_{\iota_Y\circ\iota_X,red}.$$\qed
\end{enumerate}

\end{theorem}

\subsection{Relative symplectic cohomology }

As we have seen in the previous section, the passage from truncated symplectic cohomology to reduced symplectic cohomology is rather straightforward. In our intended applications to mirror symmetry we will be interested in a slightly different invariant called \emph{relative symplectic cohomology}. This invariant was defined in \cite{varolgunes} for $M$ closed and denoted by $SH_M^*(K)$. In the same paper a Mayer-Vietoris property was proved, which is the main reason for the relative symplectic cohomology to be the more relevant version for mirror symmetry.

Let us note from the outset that unlike reduced symplectic cohomology, relative symplectic cohomology (for example when $M$ is closed, where it has already been defined) cannot in general be recovered from the inverse system $SH_{M,\lambda}^*(K).$ Our first aim in this section is to define relative symplectic cohomology $SH_M^*(K)$ for $M$ geometrically bounded.

Let $(H_\tau,J_\tau)$ be an acceleration datum for $K\subset M$. We obtain a $1$-ray of chain complexes over $\Lambda_{\geq 0}$: $$\mathcal{C}(H_\tau):= CF^*(H_1)\to CF^*(H_2)\to\ldots. $$

We define relative symplectic cochain complex by taking the degree-wise completion of the telescope of $\mathcal{C}(H_\tau)$: $$SC_M^*(K,H_\tau):=\widehat{tel}(\mathcal{C}(H_\tau)).$$ Here the telescope is defined as $$tel(\mathcal{C}))=\left(\bigoplus_{i=1}^\infty C_i[1]\oplus C_i\right)$$ with the differential depicted below \begin{align}\label{teles}
\xymatrix{
C_1\ar@{>}@(ul,ur)^{d }  &C_2\ar@{>}@(ul,ur)^{d} &C_3\ar@{>}@(ul,ur)^{d}\\
C_1[1]\ar@{>}@(dl,dr)_{-d} \ar[u]^{\text{id}}\ar[ur]^{f_1} &C_2[1]\ar@{>}@(dl,dr)_{-d} \ar[u]^{\text{id}}\ar[ur]^{f_2}&\ldots\ar[u]^{\text{id}}_{\ldots} }
\end{align}

Completion is a functor $Mod(\Lambda_{\geq 0})\to Mod(\Lambda_{\geq 0})$ defined by \begin{align}
A\mapsto \widehat{A}:\lim_{\xleftarrow[r\geq 0]{}}A\otimes_{\Lambda_{\geq 0}}\Lambda_{\geq 0}/\Lambda_{\geq r}
\end{align} on objects, and by functoriality of inverse limits on the morphisms.

The completion functor automatically extends to a functor $Ch(\Lambda_{\geq 0})\to Ch(\Lambda_{\geq 0})$. Namely, if $(C,d)$ is a chain complex over $\Lambda_{\geq 0}$, then the completion $(\widehat{C},\widehat{d})$ is obtained by applying the completion functor to each graded piece of the underlying graded module, and also to the maps $d_i:C^i\to C^{i+1}$.

\begin{remark}\label{remNotFree}
Note that the underlying $\Lambda_{\geq 0}$- module of $CF^*(H_i)$ is not free if $H_i$ has infinitely many $1$-periodic orbits, but it is degree-wise complete (recall the completed direct sum from Equation \eqref{eq-ham-completed}). Moreover, it is still true that the canonical map $$tel(\mathcal{C}(H_\tau))\otimes_{\Lambda_{\geq 0}}\Lambda_{\geq0}/T^{\lambda}\Lambda_{\geq 0}\to \widehat{tel}(\mathcal{C}(H_\tau))\otimes_{\Lambda_{\geq 0}}\Lambda_{\geq0}/T^{\lambda}\Lambda_{\geq 0}$$ is an isomorphism of chain complexes for every $\lambda\geq 0$.
\end{remark}

Note that we have the following chain of canonical isomorphisms \begin{align*}H^*(\widehat{tel}(\mathcal{C}(H_\tau))\otimes_{\Lambda_{\geq 0}}\Lambda_{\geq0}/T^{\lambda}\Lambda_{\geq 0})&\simeq  H^*(tel(\mathcal{C}(H_\tau))\otimes_{\Lambda_{\geq 0}}\Lambda_{\geq0}/T^{\lambda}\Lambda_{\geq 0})\\&\simeq H^*(tel(\mathcal{C}(H_\tau)\otimes_{\Lambda_{\geq 0}}\Lambda_{\geq0}/T^{\lambda}\Lambda_{\geq 0})) \\&\simeq H^*(\lim_{\rightarrow} (CF(H_i)\otimes_{\Lambda_{\geq 0}}\Lambda_{\geq0}/T^{\lambda}\Lambda_{\geq 0}))\\&\simeq SH^*_{M,\lambda}(K) \end{align*}
In the third equality we used the canonical quasi-isomorphism (see for example Lemma 2.2.1 of \cite{varolgunes})  \begin{align}tel(\mathcal{C})\to\lim_{\rightarrow}(C_i).
\end{align}
We will need the following proposition to prove that the homology of $SC_M^*(K,H_\tau)$ is independent of the acceleration data. This is an extension of the compact case from \cite{varolgunes} to the geometrically bounded case.

\begin{proposition}\label{propRelSHWell}
\begin{itemize}
\item For any two different choices of acceleration data for $K$, $(H_\tau,J_\tau)$ and $(H_\tau',J_\tau'),$
there is a canonical $\Lambda_{\geq 0}$-module isomorphism (comparison map) $$H^*(SC_M^*(K,H_\tau))\simeq H^*(SC_M^*(K,H_\tau'))$$ defined using the Hamiltonian Floer theory package on geometrically bounded manifolds.
 \item The comparison map from an acceleration data to itself is the identity map. Compositions of comparison maps are comparison maps.
 \end{itemize}
\end{proposition}

\begin{proof}
This is an adaptation of the proof of Proposition 3.3.3 (1) of \cite{varolgunes} to open symplectic manifolds. By Lemma 8.13 of \cite{groman}, we can find a third acceleration datum $(H_\tau'',J_\tau'')$ such that $H_\tau,H_\tau'\geq H_\tau''$ for all $\tau$.\footnote{Strictly speaking,  Lemma 8.13 of \cite{groman} is formulated specifically for $H_{\tau},H_{\tau'}$ proper. Here we do not make these properness assumptions. However the proof of  Lemma 8.13 of \cite{groman} adjusts easily.} Using contractibility of dissipative Floer data, we can obtain maps of $1$-rays from $\mathcal{C}(H_\tau'')$ to $\mathcal{C}(H_\tau)$ and $\mathcal{C}(H_\tau')$ (see Equation (3.3.2.1) of \cite{varolgunes}). These induce chain maps $$tel(\mathcal{C}(H_\tau''))\to tel(\mathcal{C}(H_\tau))$$ and $$tel(\mathcal{C}(H_\tau''))\to tel(\mathcal{C}(H_\tau')).$$ By Proposition \ref{propTruncAccDef}  and the chain of isomorphisms above, for any $\lambda >0$, tensoring these maps with $\frac{\Lambda_{\geq 0}}{T^{\lambda}\Lambda_{\geq 0}}$, we obtain quasi-isomorphisms. Hence we obtain an isomorphism as in the statement of the first bullet using Lemma 2.3.5 of \cite{varolgunes}, Remark \ref{remNotFree} and  the fact that  a chain map being a quasi-isomorphism is equivalent to its cone being acyclic.  The fact that this map does not depend on the choices and the second bullet point rely on a further application of the contractibility of dissipative Floer data (see the discussion around Equation (3.3.2.2) of \cite{varolgunes}).
\end{proof}

Therefore, we can make the following definition
\begin{definition}
The \emph{symplectic cohomology of $K$ relative to $M$} is the $\Lambda_{\geq 0}$-module
\begin{equation}
SH^*_M(K):=H^*(SC_M^*(K,H_\tau)).
\end{equation}
When $M$ is clear from the context we refer to $SH^*_M(K)$ as the relative symplectic cohomology of $K$.
\end{definition}

The proof of the following proposition (same as Proposition 3.3.3 from \cite{varolgunes}) is another application of contractibility of dissipative Floer data.

\begin{proposition}
Let $K\subset K'$ be compact subsets of $M$, then we have canonical restriction maps $$SH_M^*(K')\to SH_M^*(K).$$\qed
\end{proposition}

Using the formalism developed in Section 3 of \cite{tonkonog}, we can also equip $SH_M(K)$ with a BV-algebra structure. A full proof of the following statement requires considerable work. Since we do not directly need it in this paper, we omit the proof, which is within the reach of the techniques of \cite{tonkonog}. We will prove this theorem in a more conceptual manner in our future work joint with Mohammed Abouzaid.

\begin{theorem}

The assignment
\[
\cSH^*_{Y\subset M}:K\mapsto SH^*_{M}(K),
\]
which acts on morphisms by the maps in \eqref{eqrestriction} is a functor
\[
\cSH^*_{Y\subset M}: \cK(Y)\to \Lambda_{\geq 0}-BValg.
\]

Moreover, there is a canonical natural transformation of functors $$\cSH^*_{Y\subset M}\to \cSH^*_{Y\subset M,red}.$$\qed
\end{theorem}

We  expect that the locality theorem holds for relative symplectic cohomology in the same way it does for reduced symplectic cohomology.  To prove this we would need work at the chain level much longer than we do in this paper which we avoid in order not to make the paper more technical. On the other hand the following is immediate.

\begin{corollary}\label{corLocRelSH}
Let $Y$ be a symplectic manifold of geometrically finite type and let $M$ be one that is geometrically bounded. Let $\text{dim}(M)=\text{dim}(Y)$ and $\iota:Y\to M$ be a symplectic embedding.

We define $\cK(Y)_{i-red}$ to be the category of compact subsets $K$ of $Y$ for which the canonical maps $$SH_M^i(\iota(K))\to SH_{M,red}^i(\iota(K))$$ and $$SH_Y^i(K)\to SH_{Y,red}^i(K)$$ are both isomorphisms.

Also denote by $\iota_{i-red}$ the induced functor $\cK(Y)_{i-red}\to\cK(M)$. Then there exists a distinguished isomorphism of functors
\begin{equation}
T_{\iota}: \cSH^i_{Y\subset Y}\mid_{\cK(Y)_{i-red}}\simeq \cSH^i_{\iota(Y)\subset M}\circ\iota_{i-red}.
\end{equation}
\end{corollary}

We have omitted the straightforward functoriality statement that is the corollary of Theorem \ref{thmLocalityRed} (2) for brevity.

%\begin{remark} Using the Mayer-Vietoris property for relative symplectic cohomology one can extend the compact subsets of $Y$ for which the locality isomorphisms hold to ones that are covered by objects of $\cK(Y)_{i-red}$. Note that this requires the use of the five-lemma and hence to get locality in degree $i$, one needs to make assumptions on degrees $i-1,i$ and $i+1$. Since we have no use for this in the current paper we omit a careful discussion.\end{remark}

Corollary \ref{corLocRelSH} is clearly not that useful. One of its shortcomings is that it requires knowledge of the symplectic cohomology of $K$ relative to $M$. In the next section, we will develop sufficient criteria for the reduced to relative comparison maps to be isomorphisms and prove a more useful version which only requires knowledge of symplectic cohomology of $K$ relative to $Y$.

\subsection{Homologically finite torsion chain complexes}\label{ss-tor-fin}

We start by recalling a simple version of the Mittag-Leffler condition from the theory of inverse limits. Consider an inverse system of abelian groups indexed by non-negative real numbers $$M_r\to M_s,\text{ for every }r\geq s.$$
\begin{definition}
Such a system satisfies the \emph{Mittag-Leffler condition} if there exists an $R\geq 0$ such that $M_r\to M_s$ is surjective for every $ r\geq s >R.$
\end{definition}
\begin{remark} This is a strong form of the Mittag-Leffler condition that is enough for our purposes. We refer the reader to \cite{weibel} for a discussion of the Mittag-Leffler condition (Definition 3.5.6) and in general for a discussion of inverse limits (Section 3.5). In this reference only inverse systems that are indexed by non-negative integers are considered. That our inverse system is indexed by real numbers is only a cosmetic difference as integers are final inside the real numbers.\end{remark}

As the chief difference between reduced and relative symplectic cohomologies is the order in which we apply homology and completion functors, and completion involves an inverse limit in its construction, it is not a surprise that the Mittag-Leffler condition makes an important appearance.

\begin{proposition}\label{propMitt} If the inverse system $SH^{i-1}_{M,\lambda}(K)$ satisfies the Mittag-Leffler condition, the canonical map $$SH_M^i(K)\to SH_{M,red}^i(K)$$ is an isomorphism.\end{proposition}

\begin{proof} Note that the Mittag-Leffler condition implies that $R^1\lim_{\leftarrow}(SH^{i-1}_{M,\lambda}(K))=0$.
The result immediately follows from the Milnor exact sequence (the cohomological variant mentioned after Theorem 3.5.8 of \cite{weibel})  $$0\to R^1\lim_{\leftarrow}(SH^{i-1}_{M,\lambda}(K))\to H^i(SC_M^*(K,H_\tau))\to\lim_{\leftarrow}(SH^i_{M,\lambda}(K)))\to 0.$$
\end{proof}

This result is still not that useful yet because it is not clear how one would check that $SH^{i-1}_{M,\lambda}(K)$ satisfies the Mittag-Leffler condition. It turns out that there is a more checkable sufficient condition on $SH^{i}_{M}(K)$ for this to happen.\\

Note that completion involves not only an inverse limit but also tensor products with the torsion (and hence non-flat) modules $\frac{\Lambda_{\geq 0}}{T^{\lambda}\Lambda_{\geq 0}}$. Such tensor products also do not commute with taking homology and this is measured by the universal coefficient formula, which we now recall.
 \begin{lemma}\label{lemUCT}For any degree-wise torsion free cochain complex $C$ over $\Lambda_{\geq 0}$, we have the following short exact sequence for every $i\in\mathbb{Z}$ and $\lambda\geq 0$:
 $$0\to H^i(C)\otimes_{\Lambda_{\geq 0}}\frac{\Lambda_{\geq 0}}{T^{\lambda}\Lambda_{\geq 0}}\to H^i\left(C\otimes_{\Lambda_{\geq 0}}\frac{\Lambda_{\geq 0}}{T^{\lambda}\Lambda_{\geq 0}}\right)\to Tor^1\left(H^{i+1}(C),\frac{\Lambda_{\geq 0}}{T^{\lambda}\Lambda_{\geq 0}}\right)\to 0.$$
 \end{lemma}
\begin{proof}
Consider the long exact sequence of the following SES of cochain complexes: $$0\to C\to C\to C\otimes_{\Lambda_{\geq 0}}\frac{\Lambda_{\geq 0}}{T^{\lambda}\Lambda_{\geq 0}}\to 0,$$where $C\to C$ is multiplication by $T^\lambda$. Splitting the  LES into short exact sequences at the terms  $H^i\left(C\otimes_{\Lambda_{\geq 0}}\frac{\Lambda_{\geq 0}}{T^{\lambda}\Lambda_{\geq 0}}\right)$, we prove the result.
\end{proof}

We will need the following immediate corollary.

\begin{corollary}\label{corUCT} Let $\lambda'>\lambda\geq 0$. If the canonical map $$Tor^1\left(H^{i+1}(C),\frac{\Lambda_{\geq 0}}{T^{\lambda'}\Lambda_{\geq 0}}\right)\to Tor^1\left(H^{i+1}(C),\frac{\Lambda_{\geq 0}}{T^{\lambda}\Lambda_{\geq 0}}\right)$$ is surjective, then so is $$H^i\left(C\otimes_{\Lambda_{\geq 0}}\frac{\Lambda_{\geq 0}}{T^{\lambda'}\Lambda_{\geq 0}}\right)\to H^i\left(C\otimes_{\Lambda_{\geq 0}}\frac{\Lambda_{\geq 0}}{T^{\lambda}\Lambda_{\geq 0}}\right).$$
\end{corollary}
\begin{proof}
This follows from the naturality of the exact sequence of Lemma \ref{lemUCT} under truncations and the snake lemma.
\end{proof}

\begin{definition}\label{def-tor} Let $V$ be a module  over the Novikov ring $\Lambda_{\geq 0}$. For any element $v\in V$ we define the torsion $\tau(v)$ to be the infimum over $\lambda$ so that $T^{\lambda}v=0$. We take $\tau(v)=\infty$ if $v$ is non-torsion. We define the maximal torsion
\begin{equation}
\tau(V):=\sup_{v:\tau(v)<\infty}\tau(v).
\end{equation}
If there are no torsion elements we take $\tau(V)=-\infty$.
\end{definition}

\begin{remark}
It would be more appropriate to call $\tau(v)$ the \emph{$T$-torsion} as it is possible to have integral torsion in $\Lambda_{\geq 0}$-modules. We hope this will not cause confusion.
\end{remark}

The following simple looking definition is crucial for our purposes.

\begin{definition}For any $C$ graded cochain complex over $\Lambda_{\geq 0}$, we define the \emph{$i$- homological torsion of $C$} as the maximal torsion of $H^i(C)$. If the $i$-homological torsion of $C$ is less than $\infty$, we say that it has \emph{homologically finite torsion at degree $i$}.\end{definition}

\begin{lemma}\label{lemTorCons}
Let $V$ be a module over $\Lambda_{\geq 0}$ and assume that its maximal torsion is equal to the real number $\lambda_0$. Let $\lambda'\geq\lambda\geq \lambda_0$, then the canonical map $$Tor^1\left(V,\Lambda_{\geq 0}/T^{\lambda'}\Lambda_{\geq 0}\right)\to Tor^1\left(V,\Lambda_{\geq 0}/T^{\lambda}\Lambda_{\geq 0}\right)$$ is an isomorphism.
\end{lemma}

\begin{proof}
By definition $$Tor^1\left(V,\Lambda_{\geq 0}/T^{\lambda}\Lambda_{\geq 0}\right)=ker\left(V\xrightarrow[]{T^{\lambda}\cdot} V\right).$$ But, by the assumption: $$ker\left(V\xrightarrow[]{T^{\lambda}\cdot} V\right)=ker\left(V\xrightarrow[]{T^{\lambda'}\cdot} V\right)=ker(V\xrightarrow[]{T^{\lambda_0}\cdot} V).$$
\end{proof}

\begin{lemma}\label{lemTorsMitt}
Assume that $C$ has homologically finite torsion at degree $i$, then the inverse system $$H^{i-1}\left(C\otimes_{\Lambda_{\geq 0}}\frac{\Lambda_{\geq 0}}{T^{\lambda'}\Lambda_{\geq 0}}\right)\to H^{i-1}\left(C\otimes_{\Lambda_{\geq 0}}\frac{\Lambda_{\geq 0}}{T^{\lambda}\Lambda_{\geq 0}}\right)$$ is Mittag-Leffler.
\end{lemma}

\begin{proof}
By Lemma \ref{lemTorCons} applied to $V=H^i(C)$ and Corollary \ref{corUCT}, we get that $$H^{i-1}\left(C\otimes_{\Lambda_{\geq 0}}\frac{\Lambda_{\geq 0}}{T^{\lambda'}\Lambda_{\geq 0}}\right)\to H^{i-1}\left(C\otimes_{\Lambda_{\geq 0}}\frac{\Lambda_{\geq 0}}{T^{\lambda}\Lambda_{\geq 0}}\right)$$ is surjective as long as $\lambda'>\lambda$ is sufficiently large. We get the result.
\end{proof}

Let us finally give the proof of Proposition \ref{prop-intro-rel-red} from the introduction.

\begin{proof}[Proof of Proposition \ref{prop-intro-rel-red}]
The statement follows immediately by combining Lemma \ref{lemTorsMitt} and Proposition \ref{propMitt}. 
\end{proof}

We finally come to the upshot of this section. First, we make a definition.

\begin{definition}
A compact subset $K$ of a geometrically bounded symplectic manifold $M$ has \emph{homologically finite torsion in degree $i\in\mathbb{Z}$} if $SH^i_M(K)$ has finite torsion.
\end{definition}

The following is our main theorem.

\begin{theorem}\label{thmLocalityRelFinal}
Let $X$ and $Y$ be symplectic manifolds of geometrically finite type and let $M$ be one that is geometrically bounded. Let $\text{dim}(M)=\text{dim}(Y)=\text{dim}(X)$ and $\iota_X:X\to Y$, $\iota_Y:Y\to M$ be symplectic embeddings.

\begin{enumerate}
\item Let $K\subset Y$ be a compact subset with homologically finite torsion in degree $i$. Then there is a distinguished isomorphism of modules $$SH_Y^i(K)\to SH_M^i(\iota_Y(K))$$ called the locality isomorphism.

\item Let $K_1\subset K_2\subset Y$ be compact subsets with homologically finite torsion in degree $i$. Then the diagram
\begin{align*}
\xymatrix{
SH^i_Y(K_2)\ar[d]\ar[r]&SH^i_M(\iota_Y(K_2))\ar[d]\\ SH^i_Y(K_1)\ar[r] &SH^i_M(\iota_Y(K_1)),}
\end{align*} is commutative, where the horizontal arrows are locality isomorphisms and vertical arrows are restriction maps.

\item Let $K\subset X$ be a compact subset with homologically finite torsion in degree $i$. Then the composition of the locality isomorphisms $$SH_X^i(K)\to SH_Y^i(\iota_X(K))\to SH_M^i(\iota_Y(\iota_X(K)))$$ is the locality isomorphism $$SH_X^i(K)\to SH_M^i(\iota_Y\circ\iota_X(K))$$ for $\iota_Y\circ\iota_X:X\to M.$
\end{enumerate}
\end{theorem}

\begin{proof} We have already done the heavy work in proving Theorem \ref{thmLocalityTrunc} and Theorem \ref{thmLocalityTruncFunc}. This version almost immediately follows  by Corollary \ref{corLocRelSH}.

If $K$ has homologically finite torsion in degree $i$, then we get that the inverse system $SH^{i-1}_{Y,\lambda}(K)$ satisfies the Mittag-Leffler condition by Lemma \ref{lemTorsMitt}. Using Theorem \ref{thmLocalityTrunc}, we deduce that the same holds for $SH^{i-1}_{M,\lambda}(\iota(K))$. Finally, using Proposition \ref{propMitt}, we get that the canonical maps $$SH_M^i(\iota(K))\to SH_{M,red}^i(\iota(K))$$ and $$SH_Y^i(K)\to SH_{Y,red}^i(K)$$ are both isomorphisms. Theorem \ref{thmLocalityRed} finishes the proof.

\end{proof}

\subsection{An example of a non-homologically finite torsion compact subset}\label{ss-non-fin-tor}
We give an example of a geometrically bounded symplectic manifold $M$ and a compact set $K$ which does not have homologically finite torsion at degree $1.$ For more details (particularly for the index computations) see Section 5 of \cite{GroVar}.

Consider the symplectic manifold  $M=\mathbb{R}^2\times T^2$ with the symplectic form $\omega=dp_1dq_1+dp_2dq_2$, where $p_1,p_2$ are coordinates on $\mathbb{R}^2$ and $q_1,q_2$ are angular coordinates on $T^2$. Denote by $$\pi: \mathbb{R}^2\times T^2\to \mathbb{R}^2$$ the standard Lagrangian fibration forgetting the angular part. We use the standard $T^2$ translation invariant trivialization of the canonical bundle.

\begin{figure}
\includegraphics[width=\textwidth]{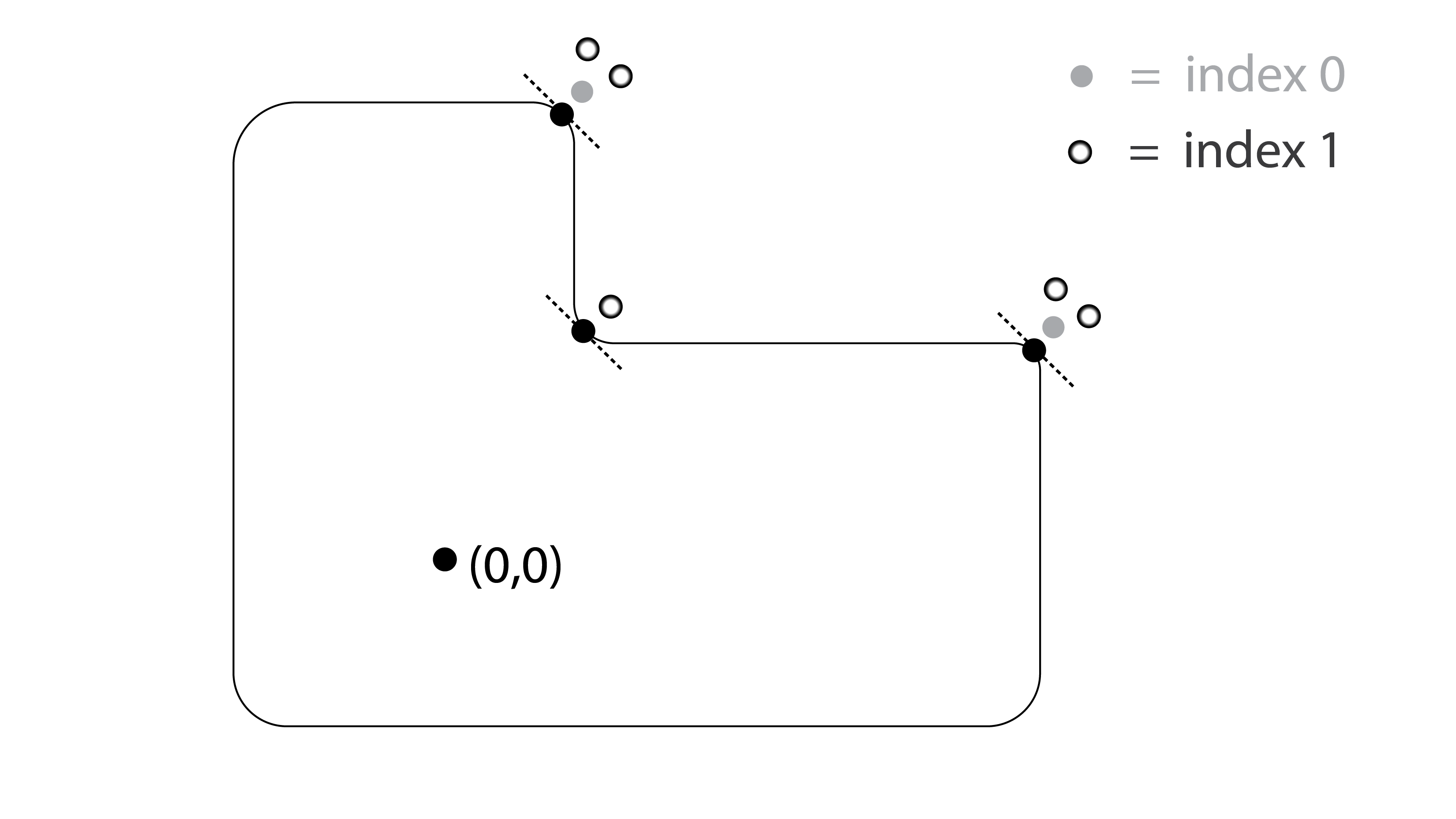}
\caption{The depiction of the set $P$ and approximate locations of the $1$-periodic orbits of index $0$ and $1$ in homology class $(a,a)$.}
\label{fig-non-convex}
\end{figure}

Consider the set $P\subset \mathbb{R}^2$ that is the union of two rectangles $$[-1,1]\times [-1,A]\cup [-1,B]\times [-1,1]$$ with its corners rounded (see Figure \ref{fig-non-convex}) with $A,B>1$. Let $K:=\pi^{-1}(P)$. We will show that $K$ does not have finite torsion at degree 1.

Note that $p_1\partial_{p_1}+p_2\partial_{p_2}$ is a Liouville vector field  that is positively transverse to $\partial K$. Let us denote by $$e^{\rho}: (\mathbb{R}^2-\{0\})\times T^2\to \mathbb{R}$$ the resulting exponentiated Liouville coordinate with $e^{\rho}(\partial K)=1$.

Consider an acceleration datum consisting of functions $H_{\tau}$ that are slight time dependent perturbations of functions that are $C^2$ small inside the region $\{e^{\rho}=0.9\}$, that that are of the form $H_{\tau}=h_{\tau}(e^{\rho})$ outside the region $\{e^{\rho}=0.9\}$ for $h_{\tau}$ a convex function which is linear at infinity.

The Floer continuation maps preserve the homology class of an orbit. So there is an extra grading on $\lim_{\rightarrow} CF(H_i)$ by $H_1(T^2,\mathbb{Z})$. We will prove that there exists $c>0$ such that for all $a\in\mathbb{N},$  $$\varnothing\neq d\left(\lim_{\rightarrow} (CF^0(H_i))_{(a,a)}\right)\subset T^{ca} \left(\lim_{\rightarrow} (CF^1(H_i))\right),$$ where the subscript $(a,a)$ means that we are considering the homogeneous piece of degree $(a,a)$. This implies the  non-torsion finiteness that we claim.

For every large enough integer subscript Hamiltonian ($H_i$) in the acceleration data, there are two index $0$ orbits $\alpha_1,\alpha_2$ in homology class $(a,a)$ with $a>0$ and five index $1$ orbits, $\alpha_1^{\pm},\alpha_2^{\pm}$ and $\beta$.  The action of an orbit $\gamma$ representing the class $(a,a)$ and occurring in the fiber over some $(x_1,x_2)\in\bR^2$ 
is (approximately due to perturbations and the constant term) given by
\begin{equation} 
\cA_{H_i}(\gamma)\sim -ax_1-ax_2.
\end{equation}
%equal to minus the dot product of $(a,a)$ with their position vectors in $\mathbb{R}^2$.
In particular, the action difference between $\alpha_i$ and $\gamma$ goes do infinity as $a\to\infty$ for $\gamma\in\{\alpha_j^{\pm},\beta\}$ and $j\neq i$.  

The orbits $\alpha_i^{\pm}$ and $\alpha_i$ for $i=1,2$ are obtained from a Morse-Bott torus orbit of $H_i$ by an arbitrarily small perturbation. One can show similarly to \cite{CFHW96} that the matrix entry of the differential between $\alpha_i$ and $\alpha_i^{\pm}$ zero if the perturbations are small enough. Thus if we show that for all natural $a$ we have 
\begin{equation}\label{eqCFnot0}
 d(\lim_{\rightarrow} (CF^0(H_i))_{(a,a)})\neq\varnothing
 \end{equation}
 it will follow that the torsion is not homologically finite.

To this end note that while the relative symplectic cohomology $SH_M(K)$ over the Novikov ring depends on $K$, when $K$ is a Liouville domain and $M$ its completion,  the $\bZ$-module obtained by substituting  $T=1$ is independent of the choice of domain $K$ which is used to express $M$ as a completion. Indeed in this case  by \cite{groman} $SH_M(K)$ is the homology of a complex computed by Hamiltonians that are linear at infinity with respect to the Liouville coordinate determined by $K$. It is proven in \cite{seidelbiased} (along with the Erratum \cite{seidelerratum} published in author's website) that the resulting homology is independent of the choice of $K$. The results of \cite{groman} give another way of proving this without relying on maximum principles. Now, over $\bZ$, it is well established that $SH^0(\mathbb{R}^2\times T^2,dp_1dq_1+dp_2dq_2))$ (a la Viterbo \cite{Vit}) in grade $(a,a)$ is $1$-dimensional.  This implies \eqref{eqCFnot0}.

\section{Complete embeddings of symplectic cluster manifolds}\label{sec:EignenRayCompleteEmb}
\subsection{Nodal integral affine manifolds}\label{ss-nodal-int}

For $k\in\mathbb{Z}$, consider the linear map $A_k:\mathbb{R}^2\to\mathbb{R}^2$ defined by \begin{equation} v\mapsto v-k\cdot det(e_1,v)\cdot e_1.\end{equation}Here $e_1$ is the first standard basis vector. We define the integral affine manifold $B_k^{reg}$ by gluing  $$U_1:= \mathbb{R}^2\setminus \{(x,y)\mid y=0, x\geq 0\}$$ and $U_2:= (0,\infty)\times (-\epsilon,\epsilon)\subset \mathbb{R}^2$ for an arbitrary $\epsilon>0$ along $U_{12}:= U_1\cap U_2\subset U_1$ and $U_{21}:= U_1\cap U_2\subset U_2$ with the transition map $\phi:U_{21}\to U_{12}$ defined by $$\phi(x,y)=\begin{cases}A_k(x,y),&  y>0 \\ (x,y),& y<0.\end{cases}$$

\begin{remark}\label{rem-PL} The integral affine structure on $B_k^{reg}$ is in particular a PL structure.  Note that with this PL-structure the continuous map $\mathbb{R}^2\setminus \{0\}\to B_k^{reg}$ extending the identity map $U_1\to U_1$ is a PL-homeomorphism.\end{remark}

\begin{definition}\label{def-nodal-int-aff}
Let $B$ be a two dimensional topological manifold with a finite number of special points $N\subset B$ where $B^{reg}:=B-N$ is equipped with an integral affine structure. If each point $n$ in $N$ admits a neighborhood $U$ such that $U\setminus\{n\}$ is integral affine isomorphic to a punctured neighborhood of the origin in $B_k^{reg}$ with $k\geq 1$ (see Remark \ref{rem-no-neg}), we call $B$ a \emph{nodal integral affine manifold}. The elements of $N$ are called \emph{nodes}, and the ones of $B^{reg}$  regular points. We call the positive integer $k$ above the \emph{multiplicity} of $n$. 
\end{definition}

We define $B_k$ as the nodal integral affine manifold obtained by adding the origin back to $B_k^{reg}$ using Remark \ref{rem-PL}.  An \emph{embedding of a nodal integral affine manifold into another} is defined to be a topological embedding which sends nodes to nodes and is integral affine on the regular locus.

We will sometimes use the phrase \emph{a node with multiplicity $0$} to mean a regular point.

\begin{remark}A nodal integral affine manifold $B$ can be equipped with the structure of a PL-manifold. We start with the induced PL structure on $B^{reg}$. Then, we use Remark \ref{rem-PL} and the standard PL structure on $\mathbb{R}^2$ to extend it to $B$. The result does not depend on any choices and from now on we consider $B$ with this PL-structure.\end{remark}

\begin{definition}\label{def-eigenray} Let $B$ be a nodal integral affine manifold and $n\in N$ a node. We call the oriented image of a continuous map $\phi: [0,E)\to B$ for some $E\in (0,\infty]$ a \emph{local eigenray of $n$} if \begin{itemize}
\item $\phi(0)=n$,
\item $\phi$ restricted to $\phi^{-1}(B^{reg})$ is an integral affine immersion,
\item For any $a\in [0,E)$ such that $\phi(a)\in N$, there exists an open neighborhood $U$ of $a$ such that for $t\in U\setminus \{a\}$ the tangent vector $\phi'(t)$ is fixed under the linear holonomy $T_{\phi(t)}B\to T_{\phi(t)}B$ around a loop which bounds a disk $D$ with $D\cap N=\phi(a)$.
\end{itemize}

A local eigenray that is not contained in any other one is called an eigenray. An eigenray for which $E=\infty$ and $\phi$ is in addition injective is called a \emph{proper eigenray}.\end{definition}

From the definitions and a quick analysis of the nodal integral affine manifolds $B_k$, it is clear that for each node there are two, not necessarily proper, eigenrays. \\

Let $B$ be a nodal integral affine manifold. There are two induced lattices: $\Lambda$ on $TB^{reg}$ and $\Lambda^*$ on $T^*B^{reg}$. We can canonically construct a symplectic manifold by $X(B^{reg}):=T^*B^{reg}/\Lambda^*$ with a Lagrangian torus fibration $X(B^{reg})\to B^{reg}$. This fibration has a preferred Lagrangian section that we call the \emph{zero section}.

One can extend the smooth structure on $B^{reg}$ to a smooth structure on $B$, the symplectic manifold $X(B^{reg})$ to a symplectic manifold $X(B),$ and  the Lagrangian torus fibration $X(B^{reg})\to B^{reg}$ to a nodal Lagrangian torus fibration%\footnote{this means an almost toric fibration with only focus-focus type singularities, see Definition \ref{def-nodal-lag}.} 
$\pi:X(B)\to B$ as defined in the introduction. This is done by gluing in a \emph{local model}. That is, a nodal Lagrangian torus fibration over some open disk $\mathbb{D}_\epsilon\subset\mathbb{C}$ with a single critical value at the origin to $X(B^{reg})\to B^{reg}$ for each node of $B$.  We now give some details on this construction which we call \emph{filling in the nodal fibers.}

In order to fill in the nodal fibers, we need to make two choices at each node of $B$: what we are gluing (the local model) and how we are gluing. We note the non-trivial fact that all local models have Lagrangian sections that do not intersect the focus-focus singularities. By Proposition 4.14 of \cite{symington}, the integral affine structure on the punctured disk induced by any local model is isomorphic to a punctured neighborhood of the origin in $B_k^{reg}$, where $k$ is the number of singularities in the local model. The choice of such a Lagrangian section of the local model (which will be matched with the zero section) along with a nodal integral affine embedding of the punctured disk into $B$ (origin mapping to the node in question) determines how the local model is glued in. We refer to the extra data that we need to fill in the nodal fibers as the \emph{fine data}.

To sum up, a nodal integral affine manifold $B$ along with a choice of fine data gives rise to a Lagrangian nodal fibration $\pi: M\to B$, which induces the given nodal integral affine structure on $B$. We refer  to such $\pi$ as a  nodal Lagrangian torus fibration \emph{compatible} with $B$.

\begin{remark}
Using the general result that two homeomorphic surfaces are diffeomorphic, it follows that any two choices of smooth structures on $B$ as above give rise to diffeomorphic smooth manifolds. On the other hand, we do not know whether the identity map is necessarily a diffeomorphism.
\end{remark}

\begin{remark}
Germs of local models for a node with multiplicity one were classified by Ngoc \cite{san} up to fiber preserving symplectomorphisms. For higher multiplicity nodes, this is a more recent result \cite{alvaro}.
\end{remark}

We note the following result that will be useful.

\begin{theorem}\label{thm-local}
For some $\epsilon>0$ let $\pi: X\to \mathbb{D}_\epsilon$ and $\pi': X'\to \mathbb{D}_{\epsilon}$ be two local models with the same number of singularities. Assume that the induced integral affine structures on $\mathbb{D}_\epsilon\setminus\{0\}$ are the same. Fix Lagrangian sections  $s: \mathbb{D}_\epsilon\to X$ and $s': \mathbb{D}_{\epsilon}\to X'$ which do not pass through any singular point. Then, for every $0<\delta<\epsilon$ we can find a symplectomorphism $$\Phi: X\to X',$$ which is fiber preserving outside of $\overline{\mathbb{D}_{\delta}}$, induces the identity map on $\mathbb{D}_\epsilon\setminus \overline{\mathbb{D}_\delta}$ and (in the same region) sends $s$ to $s'$.

\end{theorem}
\begin{proof}
Using the technique introduced in Section \ref{ss-3d}, we can reduce to the case where the local models have only one singularity. In that case, the proof combines the main result (Theorem 2.1) of \cite{san}, Symington's Theorem (Theorem 4.6.1) from Evans' notes \cite{evans} and Lemma 3.2 of \cite{solomon}. We omit further details.
\end{proof}

A \emph{nodal slide} is an operation which inputs a nodal integral affine manifold $B$ and outputs another one. We take a node $n$ with multiplicity $k>0$ and a local eigenray $l$ with a defining map $\phi: [0,E)\to B$ for some $E\in (0,\infty)$ which we assume extends continuously to the injective map $\phi: [0,E]\to B$. Assume that $\phi(E)$ is a node with multiplicity $a\geq 0$. We can then modify the nodal integral affine structure on $B$ near $l$ so that now we have a node with multiplicity $k-1$ at $n$ and a node with multiplicity $a+1$ at $\phi(E)$. This operation was introduced in \cite{koso} with the name ''moving the worms". We omit further details on the construction but note that the operation allows a node to slide over (or even into) other nodes as long as it approaches and leaves them in monodromy invariant directions.

\begin{proposition}\label{prop-nodal-slide} Let nodal integral affine manifolds $B$ and $B'$ be related by nodal slides and consider compatible nodal Lagrangian fibrations $\pi : M\to B$ and  $\pi': M'\to B'$ with Lagrangian sections $s$ and $s'$ that do not pass through the critical points. 

If  $B$ and $B'$  are identical in the complement of a disjoint union $A\subset B$ and $A'\subset B'$ of compact straight line segments, then there is a symplectomorphism $\Phi: M\to M'$ so that there exist arbitrarily small neighborhoods $U$ and $U'$ of $A$ and $A'$ with the properties \begin{itemize}
\item $\Phi$ restricts to a fiber preserving symplectomorphism $M-\pi^{-1}(U)\to M'-(\pi')^{-1}(U').$
\item The  induced map $B-U\to B'-U'$ is the identity map.
\item $M-\pi^{-1}(U)\to M'-(\pi')^{-1}(U')$ sends $s$ to $s'$.
\end{itemize}
\end{proposition}

This was proved by \cite{symington} using a Moser argument. Below, at the end of Section \ref{ss-3d}, we sketch another proof using our three dimensional viewpoint on semi-toric fibrations.

{In} case $B=B'$ and $A=A'=\varnothing$, that is when there is actually no nodal slide, the statement is already non-trivial and is a consequence of Duistermaat's results \cite{duistermaat} along with Theorem \ref{thm-local}.  The result in this case explains the dependence of a compatible nodal Lagrangian torus fibration on the choice of fine data. 

\begin{remark}\label{rem-no-neg}
One might wonder why we assumed $k\geq 1$ in Definition \ref{def-nodal-int-aff}. Of course, $k=0$ corresponds to regular points. The reason we did not allow $k\leq -1$ is because for such $k$, one cannot fill in the fiber above the origin in the Lagrangian fibration $X(B_k^{reg})\to B_k^{reg}$ so that the result is a nodal Lagrangian torus fibration. This follows from Proposition 4.14 of \cite{symington} along with the fact that no neighborhood of the origin in $B_k$ can be embedded into a $B_l$ with $l>0$, where the origin maps to the origin and the embedding is integral affine on $B_k^{reg}$. The latter can be shown using that there is no convex polygon with two sides that contains the origin in its interior inside $B_k$ (while there is for all $B_l$ with $l>0$).
\end{remark}

\subsection{Eigenray diagrams, symplectic cluster manifolds, and  geometric boundedness}\label{SecEigenrayDiagrams}

%A nodal integral affine manifold $B$ is said to be \emph{complete} if for any integral affine chart $U\subset B$ any sequence of points in $U$ with bounded integral affine coordinates has a subsequence which converges in $B$. 
%
%We now consider the construction of a particular class of complete integral affine manifolds. 

A ray in $\mathbb{R}^2$ is the image of a map of the form $[0,\infty)\to \mathbb{R}^2$, $t\mapsto x+vt$, where $x,v\in \mathbb{R}^2$. Let us define an \textit{eigenray diagram} $\mathcal{R}$ to consist of the following data:

\begin{enumerate}
    \item A finite set of pairwise disjoint rays $l_i$, $i=1,\ldots, k,$   in $\bR^2$ with rational slopes.
        \item A finite set of points on each ray including the starting point. Let us call the set of all of them $N_{\mathcal{R}}$.
    \item A map $m_{\mathcal{R}}: N_{\mathcal{R}}\to \mathbb{Z}_{\geq 1}$.
\end{enumerate}

For each ray $l$ in $\cR$, we call the sum $$\sum_{n\in l\cap N_\cR}m_\cR(n)$$ the total multiplicity of $l$. If $n\in N_\cR$ is contained in the ray $l$, we define $l^n\subset l$ to be the subray starting at $n$.

\begin{remark}
Eigenray diagrams are special cases of Symington's base diagrams.
\end{remark}

\begin{remark}\label{rem-eigenray-def}
An eigenray diagram can be equivalently described by the finite multiset $\{\underbrace{l^n,\ldots,l^n}_{m_R(n)}\}_{n\in N_\cR}$, which is a finite multiset of rays with rational slopes any two of which are either disjoint or so that one is contained in the other. In fact, it is more appropriate to consider a somewhat weaker condition where we allow rays that are non-disjoint to intersect as long as they do so non-transversely. We refer to such a structure as a \emph{weak eigenray diagram}. {A cluster presentation as discussed in the introduction produces a multiset $\{\pi(L_p)\}_{p\in Crit\pi}$ which in general constitutes only a weak eigenray diagram.} A straightforward generalization of\emph{branch moves} can be used to turn a weak diagram to an actual eigenray diagram. Moreover, for cluster presentations this can be achieved by replacing some tails in the cluster presentation with their opposites. See the proof of Proposition \ref{prpClustEquiv}. While this process involves the choice of which tails to replace by their opposites, we ignore this slight ambiguity and abusively refer to \emph{the} eigenray diagram determined by a cluster presentation. We allow this abuse since the ambiguity has no effect on the integral affine structure.
\end{remark}

Let us define three operations on eigenray diagrams that we will need later:

\begin{itemize}
\item A \emph{nodal slide} of an eigenray diagram $\cR$ is an eigenray diagram $\cR'$ with the same number of rays  which are all either contained in or contain the rays of $\cR$. This induces a correspondence between the rays of $\cR$ and $\cR'$. We further require that the total multiplicities of the corresponding rays are the same.
\item A \emph{node removal} from an eigenray diagram $\cR$ means choosing an arbitrary element of $n\in N_\cR$, which is contained in a ray $l$, and \begin{enumerate}
\item if $m_\cR(n)=1$ and the total multiplicity of the ray that contains $n$ is $1$, we remove $l$ altogether,
\item if $m_\cR(n)=1$, the total multiplicity of the ray that contains $n$ is more than $1$ and $n$ is the boundary of $l$, we remove $n$ and modify $l$ to start at the next element of $N_\cR$,
\item if $m_\cR(n)=1$, the total multiplicity of the ray that contains $n$ is more than $1$ and $n$ is not the boundary of $l$, we remove $n$ from $N_\cR$,
\item if $m_\cR(n)>1$, then change $m_\cR(n)$ to $m_\cR(n)-1$.
\end{enumerate}After a node removal we end up with a new eigenray diagram. In the multiset description of Remark \ref{rem-eigenray-def}, a node removal at $n\in N_\cR$ means removing an element of the form $l^n$ from the multiset.
\item Let $\cR$ be an eigenray diagram. We call an element $n\in N_\cR$ on ray $l$ \emph{mutable} if $l\cap N_\cR=\{n\}$ and the straight line $w$ that contains $l$ does not intersect any ray of $\cR$ other than $l$. Applying a \emph{branch move} at a mutable $n\in N_\cR$ means to define a new eigenray diagram as follows\begin{enumerate}
\item Delete the ray $l$ and replace it with $\overline{w\setminus l}$
\item Let $e\in\mathbb{Z}^2$ be a primitive vector in the direction of $l$. Consider the PL map $\Psi:\bR^2\to \bR^2$ defined by $$\Psi(v)=\begin{cases}
v+m_\cR(n) det(e,v)e,& det(e,v)\geq 0\\
v,&\text{otherwise}
\end{cases}$$
\item Replace each ray $l'$ other than $l$ in $\mathcal{R}$ with $\Psi(l')$.
\item Replace $N_\cR$ with $\Psi(N_\cR)$ and the function $m_\cR$ with $m_\cR\circ \Psi^{-1}|_{\Psi{-1}(N_\cR)}$.
\end{enumerate}
\end{itemize}

Each eigenray diagram $\cR$ gives rise to a nodal integral affine manifold $B_\cR$ with a $PL$ homeomorphism $\psi_\cR: \mathbb{R}^2\to B_{\mathcal{R}} $ such that
\begin{itemize}
\item $\psi_\cR(N_{\mathcal{R}})$ is the set of nodes of $B_{\mathcal{R}}.$
  \item $\psi_\cR$ restricted to the complement of the rays in $\cR$ is an integral affine isomorphism onto its image.
    \item The multiplicity of a node $\psi_\cR(n)$ is $m_{\mathcal{R}}(n)$.
    \item For any $n\in N_\cR$, $\psi_\cR(l^n)$ is an eigenray of $\psi_\cR(n)$.
\end{itemize}  The construction involves  starting with the standard integral affine $\mathbb{R}^2$ and doing a modification similar to the one in the definition of $B_k$ for each element $n$ of $N_\cR$ along $l_n$. One can refer to this operation as {a} \emph{nodal integral affine surgery}. The order in which these surgeries are made does not change the resulting nodal integral affine manifold.

\begin{remark} A given nodal integral affine manifold may be isomorphic to $B_{\mathcal{R}}$ for multiple eigenray diagrams $\cR$. The main source of examples for this phenomenon are the branch moves. \end{remark}

Note that the restriction $\bR^2\setminus N_\cR\to B_\cR\setminus \psi_\cR(N_\cR)$ of the map $\psi_\cR$ is differentiable only in the complement of the rays of $\cR$. The restriction $\bR^2\setminus \bigcup l_j\to B_\cR\setminus \psi_\cR(\bigcup l_j)$ is a diffeomorphism, where $l_j$ are the rays of $\cR$.

We can equip $B_{\mathcal{R}}$ with a smooth structure as explained in the previous section as part of the filling in the nodal fibers process.  Let us call these \emph{compatible smooth structures}. Note that $B_\cR$ with a compatible smooth structure is abstractly diffeomorphic to $\bR^2.$

%We omit the proof of the following Lemma.
%\begin{lemma}\label{lem-sm-away-from-other-rays}
%Equip $B_{\mathcal{R}}$ with a compatible smooth structure. Let $l_i$, $i=1,\ldots, k$ be the rays of $\cR$. Consider an  affine function $\mathbb{R}^2\to \mathbb{R}$ that is constant along $l_i$ for some $i$. Then the function $h\circ\psi_\cR$ is a smooth function on $B_\cR-\bigcup_{i\neq j}l_j$.
%\end{lemma}

Any eigenray diagram $\mathcal{R}$ gives rise to a symplectic manifold equipped with a nodal Lagrangian torus fibration after making a choice of fine data. By abuse of notation we will denote each one of these by  $\pi_{\mathcal{R}}: M_{\mathcal{R}}\to B_{\mathcal{R}}$ and refer to them as compatible almost toric fibrations. This is an abuse of notation as different choices are known to not give rise to fiberwise symplectomorphic Lagrangian fibrations. On the other hand, it is well-known that different choices give rise to symplectomorphic $M_{\mathcal{R}}$ (Corollary 5.4 of \cite{symington}). A much more detailed statement  is the special case of Proposition \ref{prop-nodal-slide} that was mentioned right after that statement.\\

Let us now connect all this with the discussion of symplectic cluster manifolds from Section \ref{sss-intro-cluster}. We refer the reader to that section for definitions. Let us first explicitly prove the following elementary result that we stated in the introduction.

\begin{lemma}\label{lem-tail-eigenray}
Let $X^4$ be a symplectic manifold and $\pi: X\to B$ be a nodal Lagrangian torus fibration. Let $L$ be a Lagrangian tail emanating from a focus-focus point $p$ which surjects under $\pi$ onto the smooth ray $l$ emanating from $m=\pi(p)$. Then, $l$ is an eigenray of $p$ in the sense of Definition \ref{def-eigenray}. 
\end{lemma}

\begin{proof}

{Let $\phi:[0,\infty)\to B$ be a smooth parametrization of $l$. By definition $\phi$ is an embedding. Fix an arbitrary orientation of the vector bundle $TB|_l\to l$. For each $b\in l,$ the tangent space of $l$ is contained in the kernel of some covector $\alpha_b\in T^*_bB$. We can choose $\alpha_b$'s such that $\alpha_b(v)>0$ if the ordered pair $\phi'(\phi^{-1}(b)),v$ is a positive basis of $T_bB$. For $b\neq m$, $L$ being Lagrangian implies that the symplectic dual tangent vector to $d\pi^*_x\alpha_b\in T_x^*X$ for all $x\in L\cap \pi^{-1}(b)$ is tangent to the circle $x\in L\cap \pi^{-1}(b)$\footnote{Here one needs to analyze the cases $T_x\pi^{-1}(b)$ and $T_xL$ being equal or not separately but they both lead to this conclusion.}. We can replace each $\alpha_b$ with a unique positive scalar multiple of it so that the flow that these vectors define on $L\cap \pi^{-1}(b)$ is $1$-periodic. This proves that the connected components of $l\setminus \text{critv.}(\pi)$ are all integral affine submanifolds. }

{It is a basic result of Zung from \cite{Zung} that for a smooth disk $D\subset B$ containing only one node, the monodromy invariant class in $H_1(\pi^{-1}(b),\mathbb{Z})$ for $b\in \partial D$ is the vanishing class. This finishes the proof using that the homology class of a loop in a nodal fiber is trivial if it is constant or if it never intersects the focus-focus singularities.}

\end{proof}

Our goal is to prove the following.

\begin{proposition}\label{prpClustEquiv}
Any symplectic manifold of the form $M_{\cR}$ is a 4-dimensional symplectic cluster manifold. Conversely, any 4-dimensional symplectic cluster manifold is symplectomorphic to $M_{\mathcal{R}}$ for some eigenray diagram $\mathcal{R}$.
\end{proposition}

The proof of the {converse} direction will rely on the following proposition concerning nodal integral affine manifolds. 
\begin{proposition}\label{propEignerayCharac}
Let $B$ be a {simply connected nodal integral affine manifold. Suppose there is a choice of proper eigenrays $l_p$ for each node $p\in N_B$ such that \begin{itemize}
\item For each $p$, one (and hence all) of the defining maps $\phi:[0,\infty)\to l_p$,  is a proper map. 
\item For any pair $p_1\neq p_2\in N_B$ we have either $l_{p_1}\cap l_{p_2}=\emptyset$ or one is included in the other. \item $B$ is weakly geodesically complete, i.e. an affine geodesic starting at any point $b\in B^{reg}$ in any direction can be extended indefinitely  unless it hits $\bigcup_{p\in N_B}l_p$.\end{itemize} Then there is an eigenray diagram $\cR$  and a PL-homeomorphism  $B\to B_{\cR}$ which is an integral affine isomorphism outside of the nodes.} 
\end{proposition}
\begin{proof}
We prove this by induction on the number $m$ of points in $N_B$. First consider the case $m=0$ where $B$ is a simply connected geodesically complete integral affine surface. Then {a well known result by Auslander-Markus (the first corollary on page 145 of \cite{Auslander}) says that $B$ is integral affine isomorphic to $\bR^2$.}

We now proceed with the induction. Let $p\in N_B$ so that $l_p$ does not contain any of the eigenrays $l_{q}$ for $q\neq p$. Let $k$ be the multiplicity at $p$. {Consider the model $B_k$ and an eigenray $l_0$ in $B_k$. By definition, there is an embedding of an open neighborhood of the node $0$ in $B_k$ into $B$ as a nodal integral affine manifold\footnote{Recall that this was defined after Definition \ref{def-nodal-int-aff}}, which sends $0$ to $p$. Since a local eigenray needs to be sent to a local eigenray, we can assume that the points of $l_0$ are sent to $l_p$. We assume that the portion of $l_0$ that lies in this neighborhood is connected and moreover that the neighborhood can be shrunk to satisfy certain properties without mention. Let us call this embedding the local embedding. We claim that in fact there exists an open neighborhood $V$ of $l_0$ in $B_k$ and an embedding of nodal integral affine manifolds $\psi: V\to B$ mapping $l_0$ to $l_p$ which extends the local embedding. By shooting affine geodesics from $l_p$ in a transverse direction appropriately, one can extend the local embedding to a continuous map $\tilde{\psi}: \tilde{V}\to B$ such that $\tilde{V}^{reg}$ maps into $B^{reg}$ as an integral affine immersion,  where $\tilde{V}$ is an open neighborhood of $l_0$.  Using that $l_p$ is properly embedded as in the first bullet point of the statement, we can construct a neighborhood $V\subset \tilde{V}$ in which the restriction of $\tilde{\psi}$ is an injection and gives the desired embedding $\psi$ (cf. Lemma 7.2 of \cite{Kosinski}). We can also easily arrange $V$ to be the interior of a smooth submanifold with boundary. 

We denote the image of $\phi$ by $U\subset B$. $U$ has smooth boundary, and we can find a neighborhood of $\partial U$ which is integral affine isomorphic to a neighborhood of the boundary of some neighborhood $S$ of the non-negative part of the $x$ axis in $\bR^2$ in such a way that the isomorphism extends to a PL homeomorphism $S\to U.$} 
We define a new integral affine manifold $\tilde{B}=(B-U)\cup S$.\footnote{{In the parlance of Section \ref{ss-surgery}, this could be called an   \emph{integral affine anti-surgery}.}}

We now prove that $\tilde{B}$ is weakly geodesically complete. Note there is a PL homeomorphism $\iota:B\to\tilde{B}$ which is a nodal integral affine isomorphism on the complement of $l_p$.  Let $x\in\tilde{B}$ and let $\gamma$ be an affine ray in the complement of the eigenrays $\iota(l_q)$, $q\neq p\in N_B$ of $\tilde{B}$ emanating from $x$ and maximally extended in the positive direction. If $\gamma$ never passes through $\iota(l_p)$ then $\gamma$ is the image under $\iota$ of an affine ray in $B$ and thus extends at least until it hits an eigenray which is not $l_p$ by assumption. It remains to deal with the case that $\gamma$ intersects $\iota(l_p)$. We divide into two cases. If $l_p$ is contained in another eigenray $l_q$, then the complement of the eigenrays of $\tilde{B}$ is integral affine isomorphic to the complement of the eigenrays of $B$ and the weak geodesic completeness is automatically preserved. 

It remains to consider the case where $l_p$ is not properly contained in any other eigenray. Let $s_0$ be such that $\gamma(s_0)\in \iota(l_p)$ and consider the affine geodesic $\gamma_{s_0}:s\mapsto\gamma(s+s_0)$. We show that $\gamma_{s_0}$ does not intersect $\iota(l_p)$ again, which would finish the proof. Assume otherwise, $\gamma_{s_0}$ is contained in $B':=\tilde{B}\setminus \bigcup_{q\neq p\in N_B}\iota(l_q)$ which is a simply connected integral affine manifold and satisfies for some $s_1>  s_0$ that $\gamma(s_1)\in l_p$. By simple connectedness there is an affine immersion  of $B'$ into $\bR^2$ via a developing map which allows us to pull back a flat Riemannian metric on $B'$ so that the affine geodesics are geodesics of the metric.

 The loop formed by concatenating $\gamma([s_0,s_1])$ to the segment $\delta\subset \iota(l_p)$ connecting $\gamma(s_0)$ with $\gamma(s_1)$ encloses a simply connected region. By flatness of the metric and the vanishing of the geodesic curvature of $\gamma$ and $\delta$ the Gauss-Bonnet formula reads $2\pi=\alpha_0+\alpha_1,$ where $\alpha_0$ and $\alpha_1$ are the turning angles of the tangent vector to $\gamma * \beta$ at $\gamma(s_0)$ and $\gamma(s_1)$ respectively \footnote{See \cite[\S 4-5]{dCCS}.}. Relying again on flatness of the connection and geodesicity  we find that the tangent vector to $\gamma$ at $s_0$ is parallel to the tangent vector to $\gamma$ at $s_1$ with an analogous statement holding for $\delta$. Thus the turning angle from $\delta$ to $\gamma$ at one vertex is cancelled by the turning angle from $\gamma$ to $\delta$ at the other vertex. That is, $\alpha_0+\alpha_1=0$. This produces the desired contradiction.

% If $\gamma$ passes through $l_p$ a finite number of times, one easily deduces the same. It remains to deal with the seemingly implausible possibility that $\gamma$ meets $l_p$ an infinite number of times. Since $\gamma$ doesn't hit any eigenray of $\tilde{B}$ it is contained in $\tilde{B}\setminus \bigcup_{q\neq p\in N}l_q$ which is simply connected and  contains the strip $V$. Thus if $x_1,x_2\in l_p$ are two successive points hit by $\gamma$ at times $t_1,t_2$ respectively, then $\gamma|_{[t_1,t_2]}$ together with the affine segment in $l_p$ connecting $x_1$ with $x_2$ enclose a simply connected region.  By flatness it follows that $\gamma'(t_1)$ is parallel to $\gamma'(t_2)$ in $V$. Moreover, the same equality holds for the entire connected components of $\gamma^{-1}(V)$ containing $t_1$ and $t_2$ respectively. As a cosequence, there is a $\delta>0$ so that for any $t$ such that $\gamma(t)\in l_p$, $\gamma$ extends to $\gamma(t+\delta)$. Since $\gamma$ hits $l_p$ an infinite number of times, and $\gamma$ cannot hit $l_p$ more than once without first leaving $V$,  it  follows that $\gamma$ is defined for all positive time. }

Thus by the inductive hypothesis $\tilde{B}=B_{\tilde{\cR}}$ for some eigenray diagram $\tilde{\cR}$. $B$ is obtained from $\tilde{B}$ by a nodal integral affine surgery replacing $S$ with $V$. This is the same as introducing an additional ray $l$ to $\tilde{\cR}$ to form a new eigenray diagram $\cR$. 
\end{proof}

Armed with Proposition \ref{propEignerayCharac}, we can now proceed to prove Proposition \ref{prpClustEquiv}. 
\begin{proof}[Proof of Proposition \ref{prpClustEquiv}]
Given an eigenray diagram we can construct a cluster presentation for $M_{\cR}$ by lifting an eigenray of multiplicity $m$ to an $m$-tuple of pairwise disjoint Lagrangian tails. The set of all these tails can be taken to be pairwise disjoint. Indeed if a pair of Lagrangian tails lies over disjoint eigenrays they are automatically disjoint. If their eigenrays overlap, then the are both a family of loops representing the same homology class in the torus fibers, and thus can be made disjoint. {Eigenray diagrams satisfy the weak completeness condition since on the complement of the rays the intergral affine structure is inherited from $\bR^2$ which is geodesically complete.}

Conversely,  let $X$ be symplectic 4-manifold equipped with a cluster presentation $P=\{\pi,\{L_p\}_{p\in Crit\pi}\}$. Let $B$ be the base of $\pi$ with its induced nodal integral affine structure. {Note that it suffices to prove that $B$ is nodal integral affine isomorphic to $B_\cR$ for some eigenray diagram $\cR$ by Proposition \ref{prop-nodal-slide} (in the case where there is no nodal slide, $B=B'$, $A=A'=\varnothing$)}.
 
{Consider the eigenrays $l_p:=\pi(L_p)$ (recall Lemma \ref{lem-tail-eigenray}), for $p\in Crit\pi$. If $l_p$ and $l_q$ intersect, they must do so non-transversely, or else the corresponding Lagrangian tails would intersect. Let us call a pair $l_p,l_q$ such that $l_p\cap l_q$ is non-empty and neither of them is contained in the other a bad pair. If there were no bad pairs among $l_p:=\pi(L_p)$, $p\in Crit\pi$, we could associate exactly one eigenray to each node of $B$ and the conditions of Proposition \ref{propEignerayCharac} would be satisfied. }

{To deal with bad pairs, we note that if  $l_p,l_q$ is a bad pair then the other eigenray $l^{-}_p$ emanating from $p$ is fully contained in $l_q$. Moreover, if we assume that $l_q$ is a proper eigenray and it satisfies the first bullet point of Proposition \ref{propEignerayCharac}, then the same properties are automatically inherited by $l^{-}_p$. It is clear that we can do this replacement finitely many times to obtain a set of eigenrays $l_p'$, $p \in Critv.(\pi)$ which contain no bad pairs.}

{We claim that $B$ and the eigenrays $l_p'$ satisfy the conditions of Proposition \ref{propEignerayCharac}.  The properness as in the first bullet point follows from the fact that the Lagrangian tails we were given are properly embedded. The second bullet point is satisfied after the modifications as in the previous paragraph. For the third bullet point it remains to show that these modifications preserve weak geodesic completeness. The proof of this is exactly the same as the proof of weak completeness in the inductive step of Proposition \ref{propEignerayCharac}. We omit further details.}

\end{proof}

Using nodal slides and branch moves one immediately sees that a symplectic cluster manifold can be symplectomorphic to $M_\mathcal{R}$ for many different eigenray diagrams.

Consider the following important basic cases.

\begin{itemize}

    \item If we choose no rays in our eigenray diagram, we end up with the Lagrangian fibration $T^*T^2\to \mathbb{R}^2$. Let us denote this case by $\pi_0:M_0\to B_0$.
    \item If we choose one ray and only one node with multiplicity one, we end up with a symplectic manifold diffeomorphic to $\mathbb{C}^2-\{xy=1\}$. Let us denote this case by $\pi_1: M_1\to B_1$. Note that all choices of rays result in symplectomorphic manifolds. This is a completion of the Auroux fibration on $\mathbb{C}^2-\{xy=1\}$ with an explicit Kahler structure. See Section \ref{ss-surgery} for another concrete model.
\end{itemize}
\begin{proposition}\label{prpClstGeoBd}
Symplectic cluster manifolds are geometrically of finite type.\end{proposition}

\begin{proof}
We know that every symplectic cluster manifold is symplectomorphic to $M_{\cR}$ for some eigenray diagram $\mathcal{R}$. Let $\pi_\cR:M_\cR\to B_\cR$ be a compatible nodal Lagrangian torus fibration. To construct a geometrically bounded almost complex structure on $M_{\cR}$ we will first construct a Riemannian metric on $B^{reg}_{\cR}$.   A metric $g_B$ on $B^{reg}_{\cR}$ gives rise to a compatible almost complex structure $J_0$ on $M^{reg}:=\pi_{\cR}^{-1}(B^{reg}_{\cR})$ by the following procedure.  Consider the defining identification of $M^{reg}\subset M_{\cR}$ with $X(B^{reg})=T^*B^{reg}_{\cR}/\Lambda^*$. The sublattice $\Lambda^*\subset  T^*B^{reg}_{\cR}$ induces a flat connection $\nabla$ on $T^*B$ whose local flat sections are the real linear combinations of local sections of $\Lambda^*$. The connection $\nabla$ descends to a flat Ehresmann connection on $\pi_{\cR}$. This induces a splitting $TM^{reg}=H\oplus V$ into horizontal and vertical bundles.  Then for any $p\in M^{reg}$ we have canonical isomorphisms $V_p=T^*_{\pi(p)}B^{reg}$ and $H_p=T_{\pi(p)}B^{reg}$. The metric $g _B$ induces an isomorphism $T_{\pi(p)}B^{reg}=T^*_{\pi(p)}B^{reg}$. Denote by $j_p:H_p\to V_p$ the map induced by the latter isomorphism. Using $j$ we define an almost complex structure $J_0$ by $J_0(h+v)=jh-j^{-1}v$ for $h\in H,v\in V$.  $J_0$ is clearly a compatible almost complex structure on $M^{reg}$. To get an almost complex structure $J$ on $M_\cR$ we pick a compact set $K$ containing all the nodal fibers and modify $J_0$ arbitrarily inside $K$ so that it extends to an almost complex structure on $M_\cR$.

We now show that $g_B$ can be chosen in such a way that $g_J$ for $J$ as above is geometrically bounded. Let $B'_{\cR}\subset B_{\cR}$ be the complement of the eigenrays and let $g'_B$ be the standard flat metric on $B'_{\cR}$ considered as a subset of $\bR^2$.  %and let $M'_{\cR}:=\pi_{\cR}^{-1}(B'_{\cR})$. Then $M'_{\cR}$ is symplectomorphic to $T^*B'_{\cR}/\bZ^2=B'_{\cR}\times\bT^2$. Let $J'$ be the almost complex structure  $M'_{\cR}$ obtained from the standard almost complex structure on $T^*B'_{\cR}\subset T^*\bR^2=\bC^2$. To obtain an almost complex structure on $M_\cR$
We need to modify $g'_B$ near the pre-image of each eigenray so that it extends across.
We are only interested in the behavior near infinity, so we shall specify this for a subray $l'_i$ of each eigenray $l_i$ whose closure does not contain any point of $N_{\cR}$. Fix a neighborhood $V_i$ of $l'_i$ that is integral affine isomorphic to the strip $l'_i\times (-2,2)$ so that $l'_i\subset V_i$ maps to $l'_i\times\{0\}$. %Let $U_i=\pi_{\cR}^{-1}(V_i)$.
Fix a non-zero vector $v$ tangent to $l_i$ and   pointing in the direction of $l_i$. Then there there is an action of the additive monoid $\bR_+$ on $V_i$ defined by
\begin{equation}\label{EqtransAction}
t\cdot (x_1,x_2)\mapsto (x_1+tv,x_2).
\end{equation}
The restriction of $g'_B$ to $V_i\setminus l'_i$ is invariant under this action. Indeed, {the action preserves $V_i\setminus l'_i$ and is an isometry of the standard metric.} Let $p$ be the boundary point of $l'_i$ and let $g''_B$ be a Riemannian metric on $V_i$ which is invariant under the $\bR_+$ action. Let $g_B$ be obtained by gluing together the metrics $g'_B,g''_B$ in a way that is compatible with the $\bR^+$ action. Doing this for all eigenrays produces a Riemannian metric $g_B$ on the complement of a  sufficiently large compact set $K\subset B_{\cR}$. We modify  $g_B$ near the boundary of $K$ and extend smoothly to obtain a Riemannian metric on all of $B_{\cR}$.

We show that the resulting $J$ is geometrically bounded. Any $x\in M_{\cR}$ which is outside of a large enough compact set satisfies either that on the ball $B_1(x)$ the metric is the standard flat metric or that $B_1(x)\subset U_i:=\pi_{\cR}^{-1}(V_i)$ for some $i$. In the first case we have the bounds on sectional curvature and injectivity radius for the flat metric. In the second case, there is an induced action of $\bR_+$ on $U_i$ by isometries {covering the action of \eqref{EqtransAction}. Moreover, the orbit of any strip of the form $\pi^{-1}_{\cR}((0,\epsilon)\times (-2,2))$ covers $\pi^{-1}_{\cR}(V_i)$.}  So $B_1(x)$ is isometric to $B_1(p)$ for $p$ in some a priori compact set producing a priori bounds on the geometry. It remains to establish completeness. The tail of any Cauchy sequence $x_i$ is contained in a ball of radius $1$ around a point and the closure in $M_{\cR}$ of any such ball is Cauchy complete by the same argument.

It remains to construct an admissible  function $f:M_{\cR}\to\bR_+$ all of whose critical points are contained in a compact set. Outside of a compact set the projection $\pi_{\cR}$ is a Riemannian submersion with respect to the metrics $g_J$ on $M_{\cR}$ and $g_B$ on $B_{\cR}$. Moreover, given a function $h:B_{\cR}\to \bR$, we have the relation
\begin{equation}
X_{h\circ\pi}=J\nabla h
\end{equation}
where $\nabla h$ is the gradient with respect to $g_B$ and we identify $T_{\pi(p)}B$ with $H^*_p$.  We have that $\nabla J$ is bounded by the same argument as for the sectional curvature. Thus to construct an  $f$ it suffices to construct a function $h:B_{\cR}\to \bR$ all of whose critical points are  contained in a compact set and so that $h$ has gradient and Hessian bounded from above.
%This metric is not the metric $g_B$ we started with, but away from the neighborhoods $V_i$ it is the Euclidean metric whereas on the neighborhoods $V_i$ the metric is preserved under the $\bR_+$ action.

To construct $h$ let $\tilde{h}:\bR^2\to\bR_+$ be a proper piecewise linear function which for some compact $K\subset \bR^2$ is smooth on $B'_{\cR}\setminus K$. Moreover, assume the level sets of $h$ meet each of the $l_i$ transversely with gradient pointing in the direction of increasing $l_i$. Fix a smooth function $\rho:(-2,2)\to[0,1]$ which is identically $1$ on $(-2,-1)\cup(1,2)$ and identically $0$ on a neighbourhood of $0$. Let $h$ be a function equalling $\tilde{h}$ away from the $V_i$ and some fixed compact set $K$ and  equalling
\begin{equation}
h=\rho(x_2)\tilde{h}+(1-\rho(x_2))x_1
\end{equation}
on $V_i$. $h$ is further extended to $K$ so as to be smooth and proper. It is clear that $h$ has its gradient and Hessian bounded from above. It remains to show the critical points are contained in a compact set. For this it suffices to estimate $\|\nabla h\|$ from below. The gradient is a non-zero  locally constant vector on the complement of $\cup V_i\cup K$. So it remains to estimate the gradient on $V_i$. We have
\begin{equation}
\nabla h=\rho\nabla\tilde{h}+(1-\rho)\nabla x_1+\rho'(\tilde{h}-x_1)\nabla x_2.
\end{equation}
The  assumption about the gradient of $\tilde{h}$ near $l_i$ guarantees that writing this as a linear combination of  $\nabla x_1$ and $\nabla x_2$, the coefficient of $\nabla x_1$ is a non-zero constant.
\end{proof}

%Finally we formulate the main proposition of this section. The statement will be sharpened somewhat and proved later.
%
%\begin{proposition}\label{prop-node-remove-pre}
%Let $\mathcal{R}$ be an eigenray diagram, and $\mathcal{R}'$ be a node removal of $\mathcal{R}$. Then, there is a symplectic embedding $$\iota_{\mathcal{R}',\mathcal{R}}:M_{\mathcal{R}'}\to M_\mathcal{R}$$ for every choice of critical point over the modified node.
%\end{proposition}
%
%As a corollary, note that we obtain an embedding of $T^*T^2$ into a symplectic cluster manifold for each of its eigenray diagram representations. If these eigenray diagram representations differ by branch moves (as opposed to nodal slides) then the resulting embeddings are provably non-Hamiltonian isotopic.

\subsection{A three dimensional perspective on generalized semi-toric manifolds}\label{ss-3d}

This section is a preparatory one for the next two sections. It starts with a discussion of nodal Lagrangian fibrations on four dimensional symplectic manifolds equipped with a Hamiltonian circle action, which only has fixed points of weight $(1,-1)$ (see the discussion near Lemma \ref{lem-S-inv-J} for the definition) and no finite stabilizers. The key point here is a correspondence between circle action invariant Lagrangian foliations and certain smooth foliations on the three dimensional quotient space, see Proposition \ref{prop-3d-smooth-st} for the precise statement. The basic idea is not new, see Theorem 1.2 (3) of \cite{Gross00}, but it appears that the specific analysis of focus-focus singularities (as in Lemma \ref{lem-Eliasson}) might be. In the next section, where we will define the surgery/anti-surgery operation that was mentioned in Section \ref{sss-idea}, we will take advantage of the fact that Lagrangian foliations on two dimensional symplectic manifolds are simply codimension 1 foliations and hence are objects belonging to smooth topology.

The second point of this section is an algorithmic analysis of the nodal integral affine structure induced on the base of a nodal Lagrangian torus fibration that is invariant under the fixed circle action from the previous paragraph. Such a fibration corresponds to a foliation by circles on the three dimensional quotient space. Moreover, the Hamiltonian that generates the circle action descends to give a submersion of the quotient space over the real line. The level sets of this submersion are equipped with canonical symplectic structures by the symplectic quotient construction. The leaves of the foliation by circles are contained in these level sets. The upshot is that by considering the symplectic areas that lie in between leaves, we compute the desired nodal integral affine structure, more specifically we construct an eigenray diagram representation for it. The results are summarized in Propositions \ref{prop-eig-proc} and \ref{prop-lag-to-eig}.

We end the section with an indication of how the nodal slide symplectomorphisms are constructed using the three dimensional picture. Instead of a general proof, which would require setting up notation that will not be used later, we give an example that we believe explains the idea.\\

 Let us recall the definition of a focus-focus singularity as stated in Definition 1.1 of \cite{san2}\footnote{where there is also a rigorous proof of Eliasson's normal form theorem for such singularities} for the convenience of the reader.
\begin{definition}\label{def-definition1.1}
For a smooth map $F = (f_1, f_2):M\to \mathbb{R}^2$ on a symplectic 4-
manifold $M$ with $\{f_1,f_2\}=0$ and $df_1$ and $df_2$ linearly independent almost everywhere, $m$ is a \emph{critical point of focus-focus type} if \begin{itemize}
\item $dF(m) = 0$;
\item the Hessians $H_m(f_1)$ and $H_m(f_2)$ are linearly independent;
\item there exist a symplectic basis $e_1,f_1,e_2,f_2$ on $T_mM$ such that these Hessians are linear combinations of the focus-focus quadratic forms $e^1\wedge f^1+e^2\wedge f^2$ and $e^1\wedge f^2-e^1\wedge f^2$.\end{itemize}
\end{definition}

 In what follows it will be more convenient to talk about submersions over a given smooth base instead of foliations, which is what we did above in the introduction to this section. 

\begin{definition}\label{def-nodal-lag}
Let $X^4$ be a symplectic manifold and $B^2$ a smooth manifold. A smooth map $p: X\to B$ is called a nodal Lagrangian submersion if  \begin{itemize}
                                                                                                                          \item At every regular point $x\in X$, $ker(dp_x)\subset T_xX$ is a Lagrangian subspace.
                                                                                                                          \item Every critical point of $p$ is of focus-focus type.
                                                                                                                        \end{itemize}
If $p$ is also proper with connected fibers, we call it a nodal Lagrangian torus fibration.
\end{definition}

Let us start with a local discussion. Consider $\mathbb{C}^2$ with its standard Kahler structure and denote the complex coordinates by $z_1$ and $z_2$. The Hamiltonian function  $\mu(z_1,z_2):=\pi(|z_1|^2-|z_2|^2)$ generates a $S:=\mathbb{R}/\mathbb{Z}$ action on $\mathbb{C}^2$ by $\theta\cdot (z_1,z_2)=(e^{2\pi i\theta}z_1, e^{-2\pi i\theta}z_2)$. We have a smooth map $Hopf: \mathbb{C}^2\to \mathbb{C} \times\mathbb{R}$ defined by $$(z_1,z_2)\mapsto \left(2\pi z_1z_2, \mu(z_1,z_2)\right).$$ The fibers of $Hopf$ are precisely the orbits of the $S$-action.

\begin{lemma}\label{lem-Eliasson}
Consider a smooth map $f: \mathbb{C} \times\mathbb{R}\to \mathbb{R}^2$ of the form $(g,pr_{\mathbb{R}})$. The map $f\circ Hopf: \mathbb{C}^2\to \mathbb{R}^2$ is an $S$-invariant nodal Lagrangian fibration with a single focus-focus singularity at the origin if and only if $f$ is a submersion.
\end{lemma}
\begin{proof}
Note that $$\mu=pr_{\mathbb{R}}\circ f \circ Hopf.$$ It is easy to see that the only critical point of $Hopf$ is the origin. Hence, if $f$ is a submersion, then $f\circ Hopf$ restricted to $\mathbb{C}^2\setminus \{0\}$ is a submersion. Moreover, it automatically satisfies the Lagrangian condition as well, because $ker(d(f\circ Hopf)_x)$ for $x\neq 0$ is spanned by the Hamiltonian vector $X_\mu(x)$ along with a vector tangent to the level set of $\mu$ passing through that point.

The only non-trivial point to check is that $(0,0)$ is a regular point of $f$ if and only if  $f\circ Hopf$ has a focus-focus type singularity at the origin, which is its only singularity. The focus-focus type condition on this singularity can be expressed concretely as a condition on the Hessians of the components of $f\circ Hopf$ with respect to the standard real coordinates $Re(z_1),Im(z_1),Re(z_2),Im(z_2)$ on $\mathbb{C}^2$ as in the last two bullet points of Definition \ref{def-definition1.1}. Since the components of $Hopf$ are homogeneous quadratic functions these Hessians only depend on the differential of $f$ at $(0,0).$ Therefore, without loss of generality, we can assume that $f$, or equivalently $g$, is a linear map: $$g(x,y,z)=ax+by+cz,$$ where $x+iy$ is the complex coordinate on $\mathbb{C}$, $z$ is the coordinate on $\mathbb{R}$ and $a,b,c$ are real numbers.

Of course $(0,0)$ is a regular point of $f$ if and only if $a$ or $b$ is non-zero. Moreover, the second bullet point in Definition \ref{def-definition1.1} is satisfied also if and only if   $a$ or $b$ is non-zero.  Hence, we need to show that as long as  $a$ or $b$ is non-zero the third bullet point in Definition \ref{def-definition1.1} is also satisfied. We can assume without loss of generality that $b=0$ because the Hamiltonian action on $\mathbb{C}^2$ by $\theta\cdot (z_1,z_2)=(e^{2\pi i\theta}z_1, z_2)$ preserves the function $\mu$ and rotates function $2\pi z_1z_2.$ In the case $a\neq 0$ and $b=0$, one can construct the desired symplectic basis by hand. 
\end{proof}

Before we introduce the general setup of this section let us recall some basic facts about symplectic representations of the circle group $S$. Assume that $(V,\Omega)$ is a finite dimensional symplectic vector space and that we have a representation $S\to Sp(V)$. 

\begin{lemma}\label{lem-S-inv-J}
There exists an $S$-invariant $\Omega$-compatible complex vector space structure $J:V\to V$.
\end{lemma}
\begin{proof} Let us choose a compatible complex structure $J$ on $V$, which gives rise to the maximal compact subgroup $U(V,J)\subset Sp(V).$ Using basic Lie theory we can find an element $A\in Sp(V)$ such the image of the representation lies in the conjugate subgroup $A^{-1}U(V,J)A.$ Conjugating $J$ with $A$, we find the desired complex structure.

A more constructive argument is to choose an arbitrary $S$-invariant inner product on $V$ by averaging and then use the polar decomposition argument to construct a compatible complex structure which is automatically $S$-invariant.\end{proof}

Fixing such a $J$, we can find a splitting of the unitary representation $V$ into one dimensional unitary subrepresentations. The latter are classified by integers. The multi-set of these integers is canonically defined and referred to as the weights of our representation.\\

Let $W^4$ be a symplectic manifold with a Hamiltonian $S$ action generated by $J:W\to \mathbb{R}$.  Assume that the action has finitely many fixed points and is free away from those. Denote by $P$ the set of fixed points. We make the crucial assumption:

\begin{itemize}
\item \emph{The action of $S$ on $T_pW\simeq\bC^2$ has weights $1$ and $-1$ for all $p\in P$.}
\end{itemize}

Consider the orbit space $Z:= W/S$ and let  $Z^\#:= Z\setminus q(P)$, where $q:W\to Z$ is the projection. Then $Z$ has a canonical topology, $Z^\#$ has a canonical smooth structure, and the  continuous map $\tilde{J}: Z\to \bR$ induced from $J$ by the universal property of the quotient is a smooth submersion when restricted to $Z^\#$.

\begin{proposition}\label{prop-3d-smooth-st}
$Z$ can be equipped with a smooth structure compatible with its topology and the smooth structure on $Z^\#$ such that \begin{enumerate}
\item $q: W\to Z$ is a smooth map.
\item $\tilde{J}$ is a smooth submersion.
\item Let $B$ be an arbitrary smooth two dimensional manifold. There is a one to one correspondence between the set of submersions $f: Z\to B$ which factorize\footnote{If a map $A\to C$ can be written as a composition $A\xrightarrow[]{g} B\to C$, we say that $g$ factorizes it.} $\tilde{J}$ and the set of $S$-invariant nodal Lagrangian submersions $\pi:W\to B$ with focus-focus singularities at $P$. The correspondence sends $f$ to $\pi=f\circ q.$
\end{enumerate}
\end{proposition}
\begin{proof}
Let $p$ be a fixed point of the $S$-action. By the equivariant Darboux theorem it follows that $p$ has an $S$ invariant open neighborhood $V$ such that there is a commutative diagram of maps

\begin{align*}
\xymatrix{
V\ar[d]_q\ar[r]^\Phi&\mathbb{C}^2 \ar[d]^{Hopf} \\ q(V)\ar[r]^\phi&\mathbb{C}\times\mathbb{R}}
\end{align*} with the following properties:

\begin{itemize}
\item $\Phi(p)=0$.
\item $\Phi$ is a smooth embedding.
\item $\phi$ is s topological embedding
\item $\phi|_{q(V)\setminus q(p)}$ is a smooth embedding
\end{itemize}

Using a $\phi$ as above at each point of $q(P)\subset Z$, we extend the smooth structure on $Z^{\#}$ to $Z$. The choice of $\Phi$ is not unique and we simply make an arbitrary choice at every critical point. With this smooth structure the conditions in the statement are satisfied by the local results discussed above.
\end{proof}

Let us call a smooth structure of $Z$ as in Proposition \ref{prop-3d-smooth-st} a \emph{compatible  smooth structure} and a submersion $f: Z\to B$  as in (3) \emph{an admissible submersion}.

We fix a compatible smooth structure and an admissible submersion $f: Z\to B$ that is an $S^1$ bundle. Then the map  $\pi=f\circ q: W\to B$ is in fact a nodal Lagrangian torus fibration. Now we will describe a procedure to compute the nodal integral affine manifold structure on $B$ induced by $\pi$ .

 By definition, $J$ descends to a smooth function $j:B\to \mathbb{R}$ such that $J=j\circ \pi.$ Let us make some further assumptions tailored to the application that we have in mind that will simplify the discussion:

\begin{itemize}
\item The image and level sets of $j$ are connected.
\item There exists an open subset  $B_{out}\subset B$ disjoint from $\pi(P)$ such that for every $c\in\mathbb{R}$ with $j^{-1}(c)$ non-empty, $j^{-1}(c)\cap B_{out}$ is non-empty and connected, and $j^{-1}(c)\setminus B_{out}$ is connected. See Figure \ref{bout}.
\item We are given integral affine coordinates $g: B_{out}\to \mathbb{R}^2$ such that $x_2\circ g=j$ and $x_1$ increases as we go from $j^{-1}(c)\cap B_{out}$ to $j^{-1}(c)\setminus B_{out}$ whenever these are both non-empty for some $c\in\mathbb{R}$.
\end{itemize}

\begin{figure}
\includegraphics[width=0.6\textwidth]{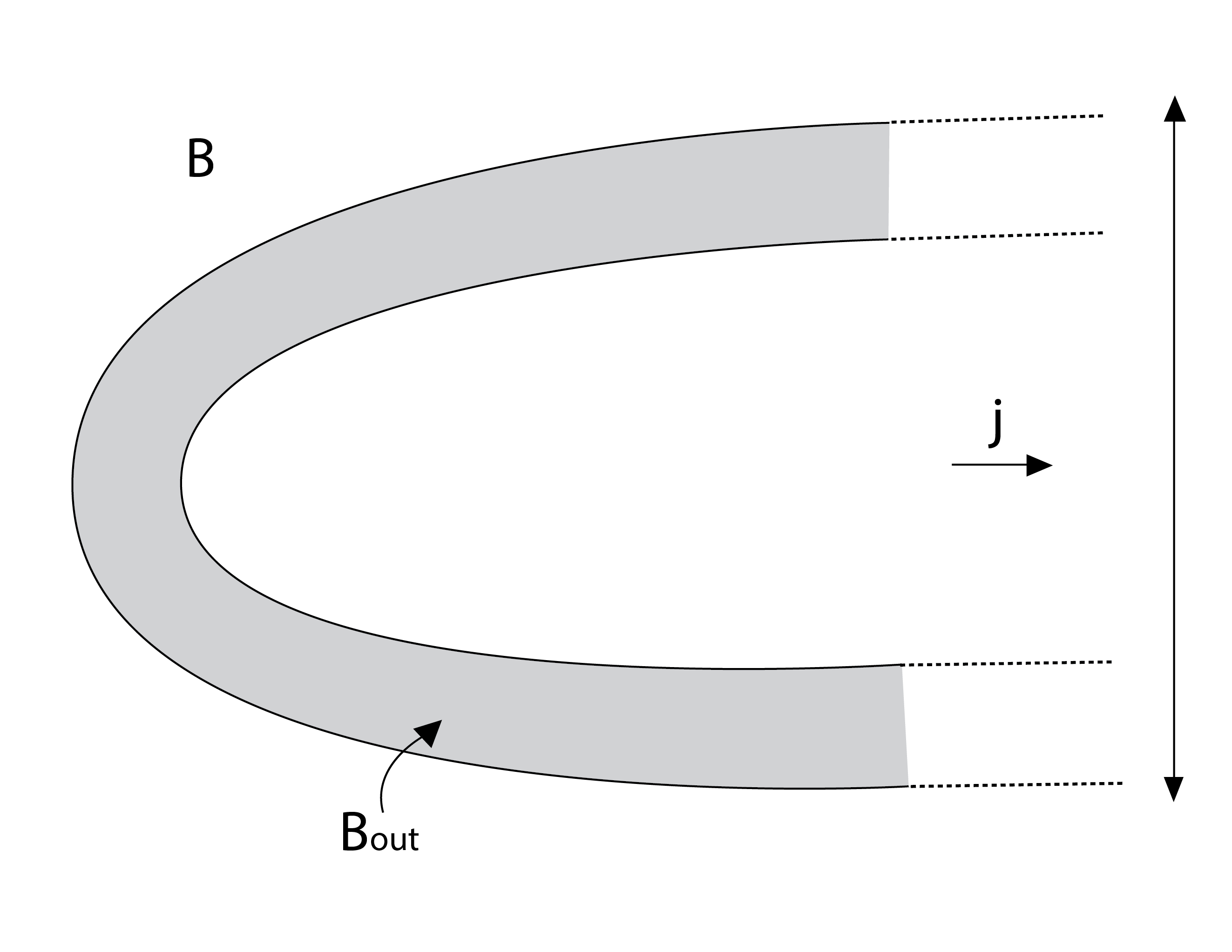}
\caption{}
\label{bout}
\end{figure}

\begin{proposition}\label{prop-eig-proc}

The following procedure extends $g$ to an injective continuous map $g^e: B\to \mathbb{R}^2$.

Let $x\in B$ with $j(x)=c$. Choose $y\in j^{-1}(c)\cap B_{out} $ and consider the domain $C\subset \tilde{J}^{-1}(c)$ whose boundary is the disjoint union of $f^{-1}(x)$ and $f^{-1}(y)$. We define $$g^e(x) :=g(y)+\left(\left|\int_C\omega_c\right|,0\right).$$
Here $\omega_c$ is the symplectic structure on $\tilde{J}^{-1}(c)$ as a symplectic reduction of the Hamiltonian action of $S$ .

\end{proposition}

\begin{proof}
The reduced symplectic structures on $\tilde{J}^{-1}(c)\setminus q(P)$ do not extend to $\tilde{J}^{-1}(c)\cap q(P)$ smoothly (if the latter is non-empty), but it follows from the equivariant Darboux theorem and the formula given in Equation (23) of \cite{BM09} (the case $n=1$ is relevant here) that the integral in the statement makes sense even if $C$ contains $q(P)$. The same formula implies the desired continuity as well. Injectivity follows from the simplifying  assumptions along with the fact that $\omega_c$ are area forms.

To see that $g^e$ extends $g$, we reinterpret the integrals that appear in statement as flux integrals on $W$. Consider a $C$ as in the statement and assume that it is disjoint from $q(P)$. Choose a smooth lift $C'\subset W$ of $C$ along $q$. It follows from the definition of symplectic reduction that \begin{equation}\label{eq-flux-red} \left|\int_C\omega_c\right|=\left|\int_{C'} \omega\right|,\end{equation}where the integral on the right is of course a flux integral by definition. This finishes the proof.

\end{proof}

Inside $g^e(B)\subset \mathbb{R}^2$ we have the set $g^e(\pi(P))$ that is disjoint from $g^e(B_{out})$. For each $n\in g^e(\pi(P))$, we consider the horizontal ray that starts at $n$ and does not intersect $g(\mathcal{L}_{out})$ (i.e. goes in the positive direction). Denote the union of all these rays by $Cut$. It  follows from Equation \eqref{eq-flux-red} that the restriction $g^e|_{B\setminus (g^e)^{-1}(Cut)}$ is an integral affine map. Moreover, the integral affine structure in the punctured neighborhood of a point in $\pi(p)$, for any $p\in P$, is determined by $k=\#\mathcal{L}_{\pi(p)}\cap q(P)$ using Proposition 4.14 of \cite{symington}.

\begin{remark}
Note that the map $g^e$ is a special case of the cartographic maps defined by \cite{Ratiu}.
\end{remark}

The upshot is that the map $g^e$ along with the data of the number of points belonging to $q(P)$ in each fiber of $\pi$ determines the nodal integral affine structure of $B$. Thinking of the connected components of $Cut$ as rays, this is of course nothing but the data of an eigenray diagram, along with the open set $g^e(\mathcal{L})$ which contains all the rays. Hence, we can apply the constructions of the previous section - but replacing $\mathbb{R}^2$ with $g^e(B)$. By construction the nodal integral affine manifold obtained from this eigenray diagram on $g^e(B)$ reproduces $B$ with its existing nodal integral affine manifold structure. The following summarizes the discussion.

\begin{proposition}\label{prop-lag-to-eig}
The nodal Lagrangian torus fibration $\pi: W\to B$ is compatible with the  eigenray diagram on $g^e(B)$ that we just described.\end{proposition}

Let us end this section with the promised sketch of the remaining argument in the proof of Proposition \ref{prop-nodal-slide} from Section \ref{ss-nodal-int}. We will show the argument in an example, which we believe contains the essential idea. Consider the eigenray diagrams $\cR$ and $\cR'$ both with rays contained in the non-negative real axis and where \begin{itemize}
\item $\cR$ has a single node with multiplicity $2$ at the origin.
\item $\cR'$ has a node at the origin and another one at $(1,0)$ both with multiplicity $1$.
\end{itemize}

Consider a nodal Lagrangian torus fibration $\pi: M_\cR\to B_\cR$ compatible with $\cR$. Notice that the function $$J:=pr_2\circ \psi^{-1}_\cR\circ \pi$$ generates an $S$ action on $M_\cR$ that satisfies the conditions of Proposition \ref{prop-3d-smooth-st}.

We will prove that there exists another nodal Lagrangian fibration $\pi': M_\cR\to B_\cR,$ where we are thinking of $B_\cR$ as a smooth manifold, such that \begin{itemize} \item For an arbitrarily small neighborhood $U$ of $\psi_\cR([0,1]\times \{0\})\subset B_\cR$, on $M_\cR\setminus \pi^{-1}(U)$ we have $\pi=\pi'$.  \item For a set $V$ of the form $\psi_\cR([-a,\infty)\times [-b,b])\subset B_\cR$ with $a,b>0$ which contains $U$, and with $g: B_\cR\setminus V\to \mathbb{R}^2$ defined as restriction of $\psi_\cR^{-1}$, the procedure explained after Proposition \ref{prop-eig-proc} outputs precisely the eigenray diagram $\cR'$.
\end{itemize}

 Let $Z:=M_\cR/S$ and equip it with a compatible smooth structure. By part (3) of Proposition \ref{prop-3d-smooth-st}, $\pi$ corresponds to an admissible submersion $f: Z\to B_\cR$. We construct the desired $\pi'$ by modifying $f$ inside $U$ to another admissible submersion $f'$. Let us consider the fibers of $f$ in $J^{-1}(0)$ with its symplectic structure $\omega_0$. We see both points of $q(P)$ inside the same fiber $\alpha$ of $f$. We choose $f'$ so that one is contained in fiber $\beta$ and the other in $\gamma$ with the following properties: \begin{itemize}\item $\beta$ is Hamiltonian isotopic to $\alpha$ and it coincides with it near the point of $q(P)$ that they share\item the area from $\beta$ to   $\gamma$  is equal to $1$.\end{itemize} We can arrange $f'$ so that it agrees with $f$ outside of the preimage of $U$. Figure \ref{slide} explains what we just said. Defining $\pi'$ to be the Lagrangian fibration corresponding to $f'$ as in  Proposition \ref{prop-3d-smooth-st} (3), we get the desired result.

\begin{figure}
\includegraphics[width=0.6\textwidth]{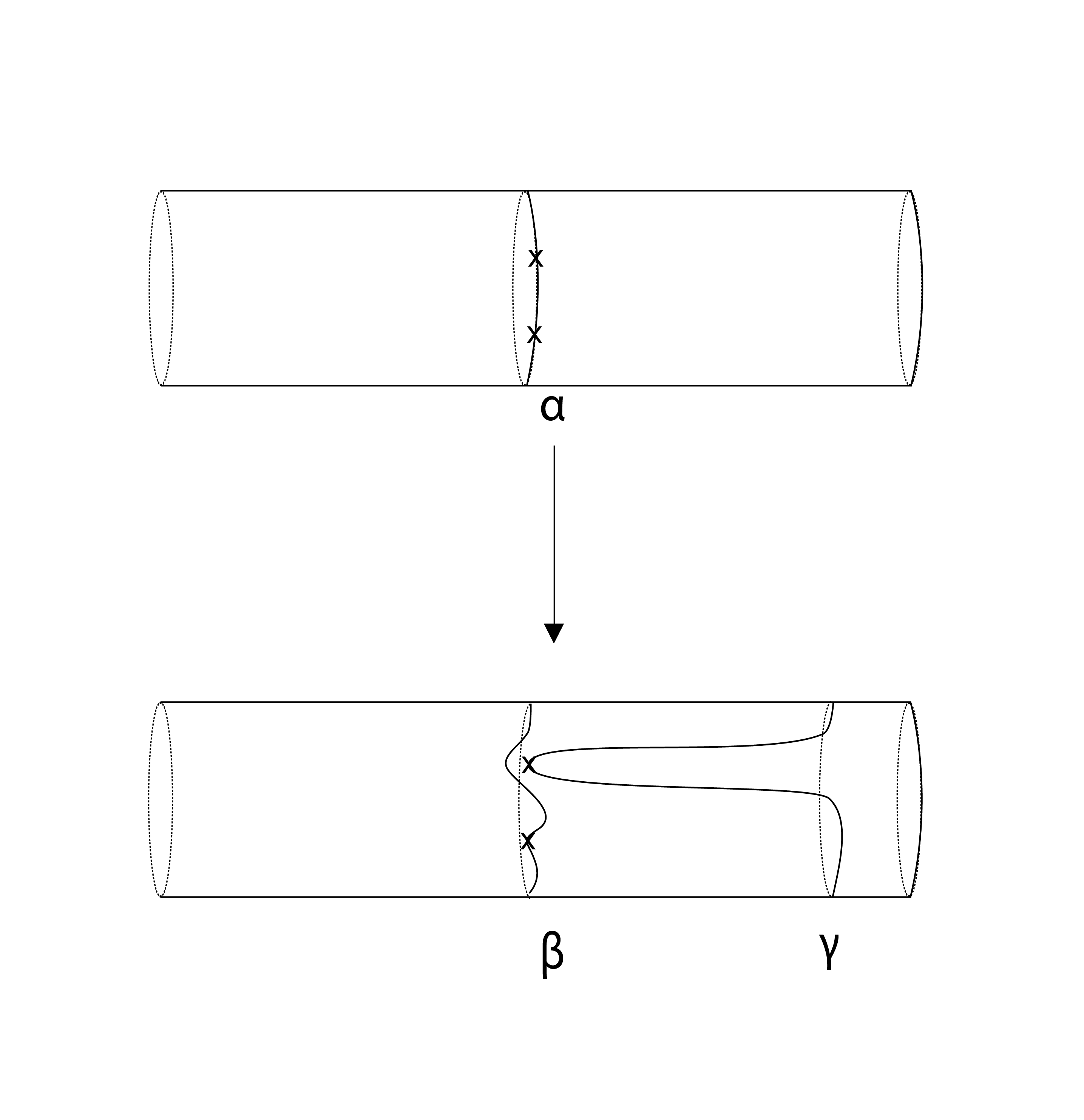}
\caption{A depiction explaining the change of foliation in the reduced space after a nodal slide.}
\label{slide}
\end{figure}

\subsection{Integrable surgery/anti-surgery in four dimensional Lagrangian submersions}\label{ss-surgery}
In this section, we formalize the notions of integrable surgery and anti-surgery. Integrable surgery is an explicit operation that takes in \begin{itemize} \item a four dimensional symplectic manifold $M$ \item a properly embedded Lagrangian plane $L$ \item a nodal Lagrangian submersion $\pi: M\setminus L\to B$ \item a Hamiltonian circle action near $L$ that preserves the submersion and rotates $L$,\end{itemize} and outputs a nodal Lagrangian submersion $\pi^{new}: M\to B$. Anti-surgery on the other hand, starts with the same data except that instead of a nodal Lagrangian submersion on $M\setminus L$, we have one on $M$ and we produce a nodal Lagrangian submersion on $M\setminus L$. As their name suggests these operations are inverse to each other in a suitable sense.

\begin{remark}
The key idea in the construction of integrable surgery/anti-surgery is the one formalized in Proposition \ref{prop-3d-smooth-st} especially part (3). It would be possible to simplify the discussion if we only wanted to discuss anti-surgery, but we preferred to give a unified statement and refer to Proposition \ref{prop-3d-smooth-st} even when a weaker statement would suffice. In particular, for anti-surgery the fact that if we start with a nodal Lagrangian submersion then we end up with a Lagrangian submersion with only focus-focus type singularities is trivial as it is directly inherited from the original submersion. For surgery on the other hand, we add new singularities to the submersion and the fact that they are of focus-focus type is non-trivial. This is one of the main points of Proposition \ref{prop-3d-smooth-st} with the relevant analysis done in Lemma \ref{lem-Eliasson}.
\end{remark}

In both operations if we start with a nodal Lagrangian \emph{torus fibration}, we end up with a nodal Lagrangian torus fibration. This is the case in which integrable surgery/anti-surgery is most useful. A crucial point is that even though the fibrations before and after the surgery/anti-surgery are over the same smooth manifold $B$, the two nodal integral structures on $B$ are different. In the case most relevant to us where the original fibrations are complete and we choose $L$ to be a Lagrangian tail, the effect of anti-surgery is to remove an eigenray from the corresponding eigenray diagram - we are using the terminology of Remark \ref{rem-eigenray-def} in this sentence. The analysis that goes into the proof of this statement is contained in Propositions \ref{prop-eig-proc} and \ref{prop-lag-to-eig}. The point that is specific to this case is that removal of a one dimensional submanifold from a symplectic surface does not change its symplectic area (see Section \ref{sss-idea}.)

We now introduce the local model. Consider the $S$-action on $\mathbb{C}^2$ as in the previous section and the $S$-invariant Lagrangian plane $L\subset \mathbb{C}^2$ given by $$L:=\{Re(z_1z_2)\geq 0\}\cap \{Im(z_1z_2)= 0\}\cap \{|z_1|^2-|z_2|^2=0\}=\{(z,\bar{z})\mid z\in\mathbb{C}\}.$$ Let $l:=Hopf(L)=\mathbb{R}_{\geq 0}\times \{0\}\subset \mathbb{C}\times \mathbb{R}$. It is elementary to show that $l$ has an open neighborhood $U$ such that there exists a diffeomorphism $$
\Phi: U\to U\setminus l$$ which preserves the projection to $\mathbb{R}$ and is the identity near the boundary of $U$. By the latter we mean that there is a closed subset $U^{in}\subset U$ containing $l$ such that $\Phi$ is the identity map on $U\setminus U^{in}$. We give an explicit example in Lemma \ref{lem-ex-slide}. Let us call such a neighborhood $Hopf^{-1}(U)$ a \emph{sliding neighborhood} of $L$ and $\Phi$ a witness. 

\begin{figure}
\includegraphics[width=\textwidth]{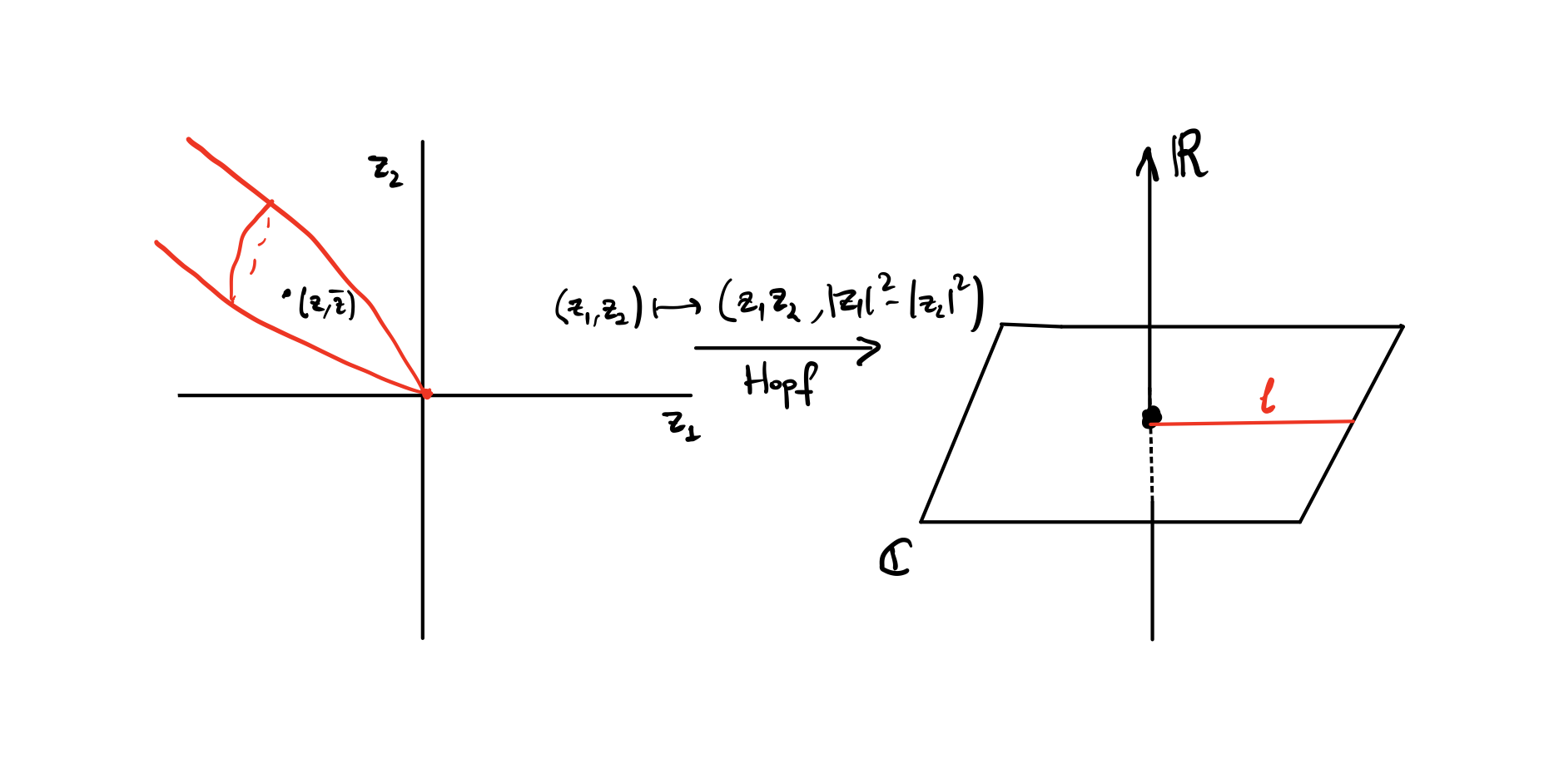}
\caption{A schematic depiction of the Hopf map. In the domain we show a cartoon image of the Lagrangian plane $L$ and in the target is a real picture of its image under the Hopf map.}
\label{fighopf}
\end{figure}

\begin{remark} We actually do not need $\mathbb{C}^2\setminus Hopf^{-1}(U)$ in what follows. We just used the Hopf fibration picture to give a convenient model for the equivariant Weinstein neighborhood of a Lagrangian plane with the circle action that rotates the plane.\end{remark}

\begin{lemma}\label{lem-ex-slide} Choose arbitrary positive real numbers $a,b,c$ such that $a,b>c$. Identify $\bC$ with $\bR^2$ in the standard way and let $$U= (-b,\infty)\times (-a,a)\times (-a,a)\subset \mathbb{C}\times \mathbb{R}.$$
Then, there exists a diffeomorphism $G: U\setminus l\to U$ with the following properties
\begin{enumerate}
	\item $G$ preserves $pr_\mathbb{R}$
	\item $G$ is the identity outside of  $V:=(-c,\infty)\times (-c,c)\times (-c,c)$.		
	\end{enumerate}
\end{lemma}
\begin{proof}
We choose $G$ of the form $$(x,y,z)\mapsto (x+\psi(x,y,z),y,z),$$where $\psi: U\setminus l\to \mathbb{R}$ is a non-negative function equal to $0$ on $U\setminus V$, converges to $+\infty$ on $l$ and satisfies $\partial_x\psi\geq 0$. It is easy to construct such a function.

It follows that $G$ is injective and surjective. Moreover, the Jacobian of $G$ is clearly invertible at every point, which finishes the proof.
\end{proof}

Let us also note the following useful lemma.

\begin{lemma}\label{lem-slide-ab}
Any $S$-invariant neighborhood of $L$ contains a sliding neighborhood.
\end{lemma}
\begin{proof}
Any $S$-invariant neighborhood of $L$ is of the form $Hopf^{-1}(W)$, where $W$ is a neighborhood of $l$ in $\bC\times \bR$. It will be convenient to identify $\bC\times \bR$ with $\bR^2\times \bR=\mathbb{R}^3$ and denote the coordinates by $x,y,z$. In this identification $pr_\bR$ is the $z$-coordinate and $l$ is the non-negative part of the $x$-axis.

We will construct a smooth embedding of $(-b,\infty)\times (-a,a)\times (-a,a)$ for some $a,b>0$ into $W$, which is the identity map when restricted to $l$ and preserves both $y$ and $z$ components. This will finish the proof by the previous lemma.

Choose a $b>0$ such that $[-b,\infty)\times\{0\}\times\{0\}$ is contained inside $W$ and an $a>0$ such that $\{-b\}\times  (-a,a)\times (-a,a)$ is also contained inside $W$. Choose a smooth function $h: W\to [0,1]$ with the properties \begin{itemize}
\item $h$ is equal to $1$ on $[-b,\infty)\times\{0\}\times\{0\}$.
\item The support of $h$ is contained in $W$.
\end{itemize}
Now consider the vector field $h\partial_x$ on $W$. By construction, this vector field is complete and let us denote its flow by $\phi^t$. We define the desired embedding $(-b,\infty)\times (-a,a)\times (-a,a)\to W$ by $$(x,y,z)\mapsto \phi^{x+b}(-b,y,z).$$ 

\end{proof}

The following is a consequence of the (proof of) equivariant tubular neighborhood theorem, existence of an invariant complementary Lagrangian subbundle at a Lagrangian submanifold setwise preserved under a symplectic compact Lie group action\footnote{This can be proved by first choosing one without the invariance property, then identifying all possible choices fiberwise with symmetric bilinear forms on the dual to the tangent space and using the averaging trick.}  and the equivariant Weinstein-Darboux neighborhood theorem (Theorem 3.2 of \cite{eqwein}).

% {\color{purple} A little more is required. To apply the theorem you cited, you need to have the equivariant tubular neighborhood embedding  intertwine the the symplectic form on the restriction of the tangent bundle to $L$ and $L'$ respectively. The Weinstein neighbourhood theorem (for Lagrangians, in the non-equivariant case)) does this by picking  an almost complex structure and using the associated exponential map. In the equivariant case, I guess, you just need to choose the almost complex structure to be invariant. I'm expanding on this because I havn't found a reference for the equivariant Lagrangian Weinstein neighborhood theorem (hopefully, there is a one). By the way, the latter theorem has the implication the  equivariant symplectomorphism type of the action in a neighborhood of an invariant Lagrangian is determined by the equivariant diffeomorphism type on the Lagrangian, which is what you are implicitly relying on when specifying the action only on $L'$. 
%%}
\begin{proposition}\label{prop-eq-wein}
Let $W$ be a four dimensional symplectic manifold and let the Lagrangian $L'\subset W$ be the image of a proper embedding $\phi: \bR^2\to W$. Assume that there is a symplectic $S$ action on $W$ that preserves $L'$ setwise and the induced action on $\bR^2$ rotates the plane around the origin. Then there is an $S$ invariant neighborhood $V$ of $L$ inside $\mathbb{C}^2$ and an $S$-equivariant symplectic embedding $V\to W$ sending $L$ to $L'$.
\end{proposition}

By Proposition \ref{prop-3d-smooth-st}, for any smooth manifold $B^2$ there is a one to one correspondance between $S$ invariant nodal Lagrangian submersions $Hopf^{-1}(U)\to B$  with a single  focus-focus singularity at $0$ and submersions $U\to B$ which factorize $pr_\mathbb{R}$. Similarly, there is a correspondence between $S$ invariant Lagrangian submersions $Hopf^{-1}(U)\setminus L\to B$ and submersions $U-l\to B$ which factorize $pr_\mathbb{R}$.  On the other hand using the diffeomorphism $\Phi$, we obtain a one-to-one correspondence between submersions $U\to B$ and submersions $U-l\to B$ which respects the factorization of $pr_\mathbb{R}$ property. Putting these three correspondences together, we obtain the local version of our operation. Using the fact the $\Phi$ is the identity near the boundary of $U$, we obtain our integrable anti-surgery by implanting this operation into symplectic manifolds equipped with nodal Lagrangian submersions. Let us now spell this out formally.

\begin{definition}
Let $X$ be a four dimensional symplectic manifold and let the Lagrangian $L'\subset X$ be the image of a proper embedding $\phi: \bR^2\to X$. Assume that there is a symplectic $S$ action in a neighboorhood $W$ of $L'$ that preserves $L'$ setwise and the induced action on $\bR^2$ rotates the plane around the origin up to conjugation by a diffeomorphism $\bR^2\to \bR^2$. Combining Propositions \ref{lem-slide-ab} and \ref{prop-eq-wein}, we obtain an $S$-equivariant symplectic embedding $$f: V\to W\subset X$$ for some sliding neighborhood $V=Hopf^{-1}(U)$ of $L$ with witness $\Phi: U\to U\setminus l$ such that $f(L)=L'$. Let us also fix a closed subset $U^{in}\subset U$ such that $l$ is contained in $U^{in}$ and $\Phi$ restricted to $U\setminus U^{in}$ is the identity map. Set $V^{in}:=Hopf^{-1}(U^{in})$.
\begin{itemize}
\item Let $\pi: X\to B$ be a nodal Lagrangian submersion such that $\pi$ restricted to $W$ is $S$ invariant and has only one critical point at the fixed point of the $S$-action on $L'$.  Then, we can produce a nodal Lagrangian submersion $$\pi^{new}: X\setminus L'\to B$$ as follows. On $X-f(V^{in})$, we set $\pi^{new}=\pi$; and on $f(V-L)$, we define $$\pi^{new}= \pi\circ f\circ \Phi^{-1}\circ Hopf\circ f^{-1}.$$ We say that  $(X\setminus L', \pi^{new})$ is obtained from $(X,\pi)$ using \emph{integrable anti-surgery} along $L'$ supported at $W$. 
\item Let $\pi: X\setminus L'\to B$ be a nodal Lagrangian submersion such that $\pi$ restricted to $W\setminus L'$ is $S$ invariant and has no critical points. Then, we can produce a nodal Lagrangian submersion  $$\pi^{new}: X\to B$$ as follows. On $X-f(V^{in})$, we set $\pi^{new}=\pi$; and on $f(V)$, we define $$\pi^{new}= \pi\circ f\circ \Phi\circ Hopf\circ f^{-1}.$$ This is indeed a nodal Lagrangian submersion because of Proposition \ref{prop-3d-smooth-st}. We say that $(X,\pi)$  is obtained from $(X\setminus L', \pi^{new})$ using \emph{integrable surgery} along $L'$ supported at $W$. 
\end{itemize} 

\end{definition}

\begin{lemma}\label{lmSurgPresPrp}
If $\pi$ is proper, so is $\pi^{new}$.
\end{lemma}
\begin{proof}
$\pi^{new}$ can alternatively be written as follows. Let $Z$ be the orbit space of $U:=\pi^{-1}(\pi(f(V)))$. Let $q: U\to Z$ denote the quotient map. Let $\chi:Z\setminus q(L')\to Z$ be the map which is identity outside of $f(V)/S$ and is intertwined with $\Phi$ under the map induced by $f$.  Let  $\tilde{\pi}:Z\to \pi(U)$ be defined by $\pi=\tilde{\pi}\circ q$. Then each of the maps $\tilde{\pi},q,\chi$ is proper. But $\pi^{new}|_U=\tilde{\pi}\circ\chi\circ q$. This proves the claim.
\end{proof}

%\begin{remark}\label{rem-eq-wein}
%It is not difficult to find such embeddings $f: V\to X$ as in the statement. Assume that we are given a nodal Lagrangian submersion $\pi: X\setminus L'\to B$, where  $L'\subset X$ is a properly embedded Lagrangian plane and $\pi$ has no critical points in a neighborhood of $L'$. All we need is a Hamiltonian circle action near $L'$ that fixes $L'$ setwise and rotates it around an origin (up to relabeling by a diffeomorphism),  and moreover preserves the Lagrangian fibers. One can then use the equivariant Weinstein tubular neighborhood theorem together with Lemma \ref{lem-slide-ab} to construct an embedding $f:V\to X$ along which one can do integrable surgery and obtain a nodal Lagrangian submersion $\pi^{new}:X\to B$.
%\end{remark}

\begin{remark}\label{rem-Loo} This is a long remark about an application of integrable surgery, which is not used elsewhere in the paper. Consider the symplectic manifold $X=\mathbb{C}^2\setminus \{xy=1\}$ with the form $$Im\left(\frac{dx\wedge dy}{xy-1}\right).$$ We will construct a nodal Lagrangian torus fibration $X\to B$ with a single singularity and such that $B$ with its induced nodal integral affine structure is isomorphic to $B_1$. $X$ is therefore a symplectic cluster manifold and is symplectomorphic to what is denoted by $M_1$ in this paper.

Let us consider the embedded Lagrangian plane $L':=\{x=0\}$ and the embedding $$f_x: (\mathbb{C}^*)^2\to X, (s,t)\mapsto \left(s, \frac{t+1}{s}\right),$$ whose image is precisely $X-L'$. The symplectic form on $X-L'$ is easily computed to be $Im\left(\frac{ds}{s}\wedge \frac{dt}{t}\right)$ and on $X-L'$ we can consider the Lagrangian torus fibration $$(s,t)\mapsto (\log |s|,\log |t|).$$ This is of course nothing but the standard Lagrangian torus fibration on $T^*T^2$ denoted by $\pi_0: T^*T^2\to B_0=\mathbb{R}^2.$

Consider the action $\theta\cdot (x,y)=(e^{2\pi i\theta}x, e^{-2\pi i\theta}y)$ near $L'$ (the action is defined on all of $X$ but we only need it near $L'$, which is an important point for further applications). Then one can do integrable surgery to the Lagrangian torus fibration on $X\setminus L'$ obtained by pushing-forward $\pi_0$ along $L'$ and define the result of our operation to be the desired $X\to B$. The completeness of this Lagrangian nodal fibration follows immediately from Propositions \ref{prop-eig-proc} and \ref{prop-lag-to-eig} and the fact that removing a one dimensional subset does not change the area of a symplectic surface.

This method immediately generalizes to give a proof of Theorem \ref{thm-Loo}.
\end{remark}

\subsection{Proof of Theorem \ref{tmCompEmbLagTails}}\label{SecProofOfThm2}
 
Recall that in the introduction before the statement of Theorem \ref{tmCompEmbLagTails} we had mentioned that we would in fact prove a more refined version, where we also discuss how the complete embeddings interact with the nodal Lagrangian fibrations. Below we state and prove that statement.

\begin{theorem}\label{prop-node-removal}
Let $\mathcal{R}$ be an eigenray diagram, and $\mathcal{R}'$ be obtained from it by a node removal at $n\in N_\cR$. Consider compatible nodal Lagrangian torus fibrations $\pi_\cR: M_\cR\to B_\cR$ and $\pi_{\cR'}: M_{\cR'}\to B_{\cR'}$.  Choose a Lagrangian tail $L_n$ lying above $\psi_\cR(l_n).$

Then, there is a symplectic embedding $$\iota_{\mathcal{R}',\mathcal{R}}:M_{\mathcal{R}'}\to M_\mathcal{R}$$ whose image is $M_\cR\setminus L_n$. Moreover, we can choose $\iota_{\mathcal{R}',\mathcal{R}}$ so that for an arbitrarily small\footnote{for simplicity let us mean by this that for any fixed compact set in the complement one can choose a neighborhood disjoint from it for which the result is true, but one could prove it for an actually arbitrarily small neighborhood} neighborhood $U$ of $\psi_\cR(l_n\cup N_\cR)$ and $U':=\psi_{\cR'}\circ (\psi_{\cR})^{-1}(U) $ we have \begin{itemize}
\item $\iota_{\mathcal{R}',\mathcal{R}}$ restricts to a fiber preserving symplectomorphism $M_{\cR'}-\pi_{\cR'}^{-1}(U')\to M_\cR-\pi_\cR^{-1}(U)$.
\item The  induced map $B_{\cR'}-U'\to B_\cR-U$ is equal to $\psi_{\cR}\circ (\psi_{\cR'})^{-1}$.
\item $M_{\cR'}-\pi_{\cR'}^{-1}(U')\to M_\cR-\pi_\cR^{-1}(U)$ sends the zero section to the zero section.
\end{itemize}
\end{theorem}

\begin{proof} Without loss of generality, assume that $l_n$ is the non-negative part of the real axis. Let us choose $U$ so that the nodes of $B_\cR$ that are contained in $U$ all lie inside $\psi_\cR(l_n)$ and that $U$ is simply connected. Then, the integral affine function $x_2\circ \psi_\cR^{-1}$ on $U\cap B_\cR^{reg}$ extends to a smooth function $j: U\to \mathbb{R}$. We now impose further restrictions on $U$. Namely, there exists an open subset $U_{out}\subset U$  so that $\psi_\cR(l_n)$ is contained in $U\setminus\overline{U_{out}}$; and $U$, $U_{out}\subset U$ and $\psi_\cR^{-1}: U_{out}\to \mathbb{R}^2$ satisfies the three bullet points listed before Proposition \ref{prop-eig-proc}.

The function $$J:= \pi_\cR\circ j: \pi_\cR^{-1}(U)\to \mathbb{R}$$ generates a Hamiltonian circle ($S$) action whose fixed points are the critical points of $\pi_\cR$ with values inside $U$. $S$ action preserves $L_n$ setwise and rotates it around $n$. Of course, inside $\pi_\cR^{-1}(U)$ where it is defined, $S$ also preserves $\pi_\cR$. Hence, we can do integrable anti-surgery to $(M_\cR,\pi_\cR)$ along $L_n$ supported at an $S$ invariant neighborhood not containing any other critical points of $\pi_\cR$ and contained inside $\pi_\cR^{-1}(U\setminus\overline{U_{out}})$. By Lemma \ref{lmSurgPresPrp} we obtain a nodal Lagrangian torus fibration $\pi^{new}: M_\cR\setminus L_n\to B_\cR$. 

Now we consider the continuous extension $U\to \mathbb{R}^2$ that is produced by Proposition \ref{prop-eig-proc} for $\pi^{new}$ restricted to $\pi_\cR^{-1}(U)\setminus L_n$. We claim that its image is the entire  $\pi_\cR^{-1}(U)$. The reason is that even after we remove $L_n/S$, each reduced space inside the quotient  $\pi_\cR^{-1}(U)/S$ has infinite area. We can also analyze the eigenray diagram data on $\pi_\cR^{-1}(U)$ that appears in Proposition \ref{prop-lag-to-eig}. By construction all the nodes in this eigenray diagram will be contained in the intersection of the $x$ axis with $\pi_\cR^{-1}(U\setminus\overline{U_{out}})$ and the rays will go along the $x$-axis in the positive direction. The total multiplicity of these nodes will be one less than the total multiplicity of the nodes originally contained in $l_n$.

Using that $\pi^{new}=\pi_\cR$ outside of $\pi_\cR^{-1}(U\setminus\overline{U_{out}})$, we therefore deduce that the nodal integral affine structure on $ B_\cR$ induced by $\pi^{new}$ is given by an eigenray diagram obtained from $\cR$ by the removal of $l_n$ and possibly some nodal slides. Using nodal slide symplectomorphisms of Proposition \ref{prop-nodal-slide}, we finish the proof.
\end{proof}

\begin{proof}[Proof of Theorem \ref{tmCompEmbLagTails}]
This immediately follows from Theorem \ref{prop-node-removal} and Proposition \ref{prpClustEquiv}.
\end{proof}

\subsection{Relative symplectic cohomology for symplectic cluster manifolds}\label{ss-rel-koso}
Let $\mathcal{R}$ be an eigenray diagram and $\pi_{\mathcal{R}}: M_{\mathcal{R}}\to B_{\mathcal{R}}$ be a compatible nodal Lagrangian torus fibration. We claim that $M_{\cR}$ admits a unique (up to homotopy) trivialization of $\Lambda^{top}_\mathbb{C}TM_\cR$ such that the Lagrangian torus fibers are Maslov zero.

Let us justify existence first. Let $M_\cR^{reg}:=\pi_\cR^{-1}(B_\cR^{reg})$. Let us denote by $V$ the Lagrangian distribution on $M_\cR^{reg}$ obtained as the vertical subbundle of $\pi_\cR$. Choosing an orientation of $B_\cR^{reg}$ we can define the element $$\frac{\partial}{\partial \theta_1}\wedge \frac{\partial}{\partial \theta_2}\in \Gamma(M_\cR^{reg}, \Lambda^{top}_\mathbb{R}V) ,$$ where $\theta_1,\theta_2$ are arbitrary oriented angle coordinates on fiber tori. 

We have a canonical isomorphism of complex vector bundles over $M_\cR^{reg}$: $$T_JM_\cR^{reg}\simeq V^\mathbb{C}:=V\otimes_\mathbb{R}\mathbb{C}$$ such that $V$ is sent identically to $V\otimes_\mathbb{R}\mathbb{R}\subset V\otimes_\mathbb{R}\mathbb{C}$, where $J$ is an arbitrary compatible almost complex structure. For two such $J_1$ and $J_2$, an elementary argument shows that the induced isomorphism between $T_{J_1}M_\cR^{reg}$ and $T_{J_2}M_\cR^{reg}$ belongs to the homotopy class of isomorphisms obtained using the contractibility of compatible almost complex structures.

Now $\frac{\partial}{\partial \theta_1}\wedge \frac{\partial}{\partial \theta_2}\otimes 1$ gives a trivialization $\Omega^{reg}$ of $$\Lambda^{top}_\mathbb{C}T_JM_\cR^{reg}\simeq (\Lambda^{top}_\mathbb{R}TV)^\mathbb{C}$$ such that the natural pairing $<d\theta_1\wedge d\theta_2, Re(\Omega^{reg})>$ is everywhere positive. It is easy to see that the Lagrangian torus fibers are Maslov zero with respect to this trivialization.

%For any integral affine manifold $B$, the bundle $\Lambda^{top}_\mathbb{C}TX_B$ is isomorphic to the complexification of $\pi^*\Lambda^{top}_\mathbb{R}TB,$ where $\pi$ is the Lagrangian torus fibration $\pi:X_B\to B$. Choosing a smooth volume form on $B_\cR^{reg}$, we obtain a trivialization of $\Lambda^{top}_\mathbb{C}TM_\cR$ in the complement of the singular fibers. 

Near the singular fibers we pick trivializations using holomorphic volume forms on local models, see \cite[Section 5]{Auroux} for the once pinched singular fiber and use fiberwise abelian covers in the multiply covered case. The Lagrangian fibers in these models are special Lagrangians, which we assume have phase $0$.  Therefore for any node $b\in B_\cR$, there is a neighborhood open disk $D_b\subset B_\cR$, a complex almost complex structure $J_b$ on $X_b:=\pi^{-1}_\cR(D_b)$ and a trivialization $\Omega_b$ of $\Lambda^{top}_\mathbb{C}T_{J_b}X_b$ which also makes $<d\theta_1\wedge d\theta_2, Re(\Omega_b)>$ everywhere positive. 

Now choosing a partitions of unity subordinate to $$B_\cR=B^{reg}_\cR\cup \bigcup D_b,$$ we can patch together $\Omega^{reg}$ and $\Omega_b$'s to obtain the desired trivialization.

For uniqueness, note that by dividing two trivializations we find a map $M_\cR\to \mathbb{C}^*$. The trivializations are homotopic if the corresponding class $a$ in $H^1(M_\cR;\mathbb{Z})$ is zero. Let $F$ be a Lagrangian torus fiber. We know that $$H^1(M_\cR;\mathbb{Z})\to H^1(F;\mathbb{Z})$$ is an injective map. The fact that $F$ is Maslov $0$ for trivializations show that $a$ maps to $0$ and hence it is zero to begin with as desired.

%Note that the canonical bundles are defined using different compatible $J$'s (call them $J_1$ and $J_2$). After choosing an appropriate smoothly varying orientation of each smooth fiber of $\pi_\cR$, on every vertical tangent space $V_p$ over $B_\cR^{reg}$ we obtain a canonical generator of $det_\mathbb{R}(V^\vee)$ as $d\theta_1\wedge d\theta_2$ with $\theta_1,\theta_2$ oriented angle coordinates. The complex vector space $T_pM$ with respect to any compatible $J$ is canonically isomorphic to $V_p\otimes \mathbb{C}.$ Using this we obtain an isomorphism between the canonical bundles with respect to $J_1$ and $J_2$, which is homotopic to the isomorphisms we obtain using the contractibility of almost complex structures. We can complexify $d\theta_1\wedge d\theta_2$ and pair it with the trivializations for $J_1$ and $J_2$ that we had constructed. A quick computation shows that we obtain $1$ in both cases and hence we can patch up the two trivializations using a partitions of unity.. 

Using this grading datum, for $K\subset B_{\mathcal{R}}$ a compact subset, we define $$\mathcal{F}_{\mathcal{R}}(K):=SH^0_M(\pi^{-1}(K))\otimes_{\Lambda_{\geq 0}}\Lambda.$$ Restriction maps turn $\mathcal{F}$ into a presheaf. In fact, much more non-trivially, using Theorem 1.3.4 of \cite{varolgunes} and the fact\footnote{see Section 5 of \cite{GroVar} for details on index computations} that $$SH^*_M(\pi^{-1}(K))=0, \text{ for }*<0,$$ we obtain the following.

\begin{theorem}
$\mathcal{F}_{\mathcal{R}}$ is a sheaf with respect to the Grothendieck topology of compact subsets and finite covers.
\end{theorem}

Using our locality theorem along with local computations we can fully compute this sheaf restricted to polygonal domains (which are allowed to contain nodes in their interior). This will appear in our forthcoming paper. Here we will be content with explaining the role played by the locality theorem that we proved in the present paper.

Let us note a sample local result.

\begin{proposition}
Let $P\subset B_0=\mathbb{R}^2$ be a compact convex rational polygon. Define the Kontsevich-Soibelman algebra of functions $KS(P)$ as the completion of $\Lambda[(\mathbb{Z}^n)^{\vee}]$ with respect to the valuation $$val(\sum a_{\alpha}z^{n_{\alpha}})=\min_{\alpha, p\in P}(val(a_{\alpha})+n_{\alpha}(p)).$$

Then $SH_M^0(\pi_0^{-1}(P),\Lambda)$ is canonically isomorphic to $KS(P)$.
\end{proposition}

As an immediate corollary we obtain the following.

\begin{theorem}
Let $\cR$ be an eigenray diagram. Let us take a compact convex polygon $P$ in $\mathbb{R}^2$ which is disjoint from all the rays of $\mathcal{R}$. Then there is an induced  isomorphism $$\mathcal{F}_{\mathcal{R}}(\psi_\cR(P))\to KS(P).$$
\end{theorem}

We stress that the isomorphism in the statement is determined by an eigenray diagram representation. Representing symplectic cluster manifolds by different eigenray diagrams, we obtain distinct locality isomorphisms of the same form. The simplest example of this phenomenon was explained in the introduction Section \ref{ss-WC}.

\subsection{Mirror symmetry}\label{ss-ms}

In this section we will write down some concrete conjectures regarding mirror symmetry for cluster symplectic manifolds. Their proof is the subject of work currently in progress.

Let us start with a simplifying assumption. Let us call an eigenray diagram \textbf{simple} if

\begin{itemize}
\item There are no parallel rays.
\item The only nodes are the starting points of the rays.
\item All multiplicities are $1$.
\end{itemize}

\begin{remark}
Our techniques extend to non-simple eigenray diagrams as well but we want to keep things simple here.
\end{remark}

Let us call an exact eigenray diagram (recall from Remark \ref{remExacEigenray}) \emph{strongly exact} if there exists a domain in $B_{\mathcal{R}}$ which contains all the nodes and which is convex with respect to the integral affine structure. If a point in $\bR^2$ is contained in all the lines containing the rays of $\mathcal{R}$, let us call it a valid basepoint. Given  a strongly exact simple eigenray diagram $\mathcal{R}$ and a choice of valid basepoint, we can consider the Viterbo symplectic cohomology of the Liouville manifold $M_{\mathcal{R}}$ (again recall from Remark \ref{remExacEigenray})) as defined in \cite{seidelbiased}, in particular the commutative $k$-algebra $SH^0_{Vit}(M_{\mathcal{R}})$.

Assume (on top of strong exactness) that $SH^0_{Vit}(M_{\mathcal{R}})$ is smooth. In this case, we expect that the Stein analytic space $Spec^{an}(SH^0_{Vit}(M_{\mathcal{R}})\otimes\Lambda)$ (see Section 5.1 of \cite{temkin}) is a bona fide non-archimedean mirror to $M_\cR$. This is strongly supported by recent results of Pomerleano \cite{pomerleano} (e.g. Corollary 1.4). To be precise, by being mirror we mean that the Conjectures \ref{conj-Stein-cont}-\ref{con-BV} below hold.

As examples let $Y_0=Spec(\Lambda[\xi^{\pm},\eta^{\pm}]) $ and $Y_1=Spec(\Lambda[x,y,u^{\pm}]/(u(xy-1)-1).$ We define the canonical inclusion $\iota: Y_0\to Y_1$ by $$x\mapsto \xi, y\mapsto\frac{1}{\xi}\left(1+\frac{1}{\eta}\right)\text{ and } u\mapsto \eta.$$ We can analytify these: $\mathcal{Y}_0, \mathcal{Y}_1$ and $\mathcal{Y}_0\to \mathcal{Y}_1$. Our claim is that these are mirror to $M_0$, $M_1$ and $M_0\to M_1$. To construct the candidate mirror $\mathcal{Y}_\cR$ for a general simple eigenray diagram $\mathcal{R}$ we need to do gluing.

Let us call a pair $(v,f)$, where $v\in \mathbb{Z}^2$ is primitive (in particular non-zero) and $f\in \mathbb{R},$ a primitive integral vector with flux. We can associate to each simple eigenray diagram with $k$ rays a $k$-tuple $\{(v_i,f_i)\}_{i=1}^k$ of primitive integral vectors with fluxes. The primitive vectors are the primitive directions of the rays of $\mathcal{R}$ and fluxes are given by $det(r_i,v_i)$ where $r_i$ is the position (with respect to the origin which is an arbitrary choice) of the node $i$. Let us call this the seed data of the eigenray diagram.

Let $\{(v_i,f_i)\}_{i=1}^k$ be the seed data of $\mathcal{R}$. For each $i$, we choose an integral affine transformation of $\mathbb{R}^2$ that sends $(v_i,f_i)$ to $((1,0),0)$ and consider the corresponding automorphism $T_i$ of  $Y_0$. Then, we take $k$ copies of $Y_1$ and glue them to each other along the images of $\iota\circ T_i: Y_0\to Y_1$. We obtain a scheme $Y_{\mathcal{R}}$.

\begin{remark}
This scheme also admits a description as a Kaliman modification of a complete toric surface over $\Lambda$, but one has to blow-up along points with non-zero valuations in the toric divisor.
\end{remark}

Let $\mathcal{Y}_{\mathcal{R}}$ be the analytification of $Y_{\mathcal{R}}.$

\begin{remark}
As special cases of Theorem \ref{tmCompEmbLagTails} above we obtain complete embeddings $$\iota_i: M_{1}\to M_{\mathcal{R}},$$ for every $i=1,\ldots ,k$. Let us call the images of these embeddings $V_i$. We also obtain an embedding $$\iota: M_0\to M_{\mathcal{R}},$$ whose image we call $W$. These satisfy the following properties \begin{enumerate}
\item $\bigcup_{i\in I} V_i= M_{\mathcal{R}}$.
\item $V_i\cap V_j =W$ for any $i\neq j$.
\end{enumerate}
The defining decomposition of $\mathcal{Y}_{\mathcal{R}}$ should be thought of as a mirror to this decomposition.
\end{remark}

\begin{conjecture}\label{conj-strongly-exact-mirror}
When $\cR$ is strongly exact and $SH^0_{Vit}(M_{\mathcal{R}})$ is smooth, then $\mathcal{Y}_{\mathcal{R}}$ and $Spec^{an}(SH^0_{Vit}(M_{\mathcal{R}})\otimes\Lambda)$ are canonically isomorphic.
\end{conjecture}

In order to explain the conjectures that spell out what we expect from mirror symmetry, we need to introduce some terminology. Let us define the following Grothendieck topology on the nodal integral affine manifold $B_{\mathcal{R}}$ for a simple eigenray diagram $\mathcal{R}$. An admissible polygon is roughly any compact domain with piecewise linear boundary where each of the boundary edges have rational slope and the boundary does not contain any node. An admissible covering is any finite covering of these admissible polygons by admissible polygons. From now on we consider $\mathcal{F}_{\mathcal{R}}$ as a sheaf over this Grothendieck topology, and continue to denote it by $\mathcal{F}_{\mathcal{R}}$

Let $\mathcal{Y}$ be a $\Lambda$-analytic space in the sense of Berkovich. We call a map $\mathcal{Y}\to B_{\mathcal{R}}$ continuous if  the preimage of each admissible polygon is a compact analytic domain of $\mathcal{Y}$ and the preimages of admissible coverings are admissible coverings in $\mathcal{Y}$. We call such a continuous map Stein if the preimage of a sufficiently small convex admissible polygon is an affinoid domain.

\begin{conjecture}\label{conj-Stein-cont}
There is a Stein continuous map $\mathcal{Y}_{\mathcal{R}}\to B_{\mathcal{R}}$ such that the push-forward of the structure sheaf of $\mathcal{Y}_\cR$ is isomorphic to $\mathcal{F}_\cR.$
\end{conjecture}

In our future work, we will give another construction of a non-archimedean mirror for which this conjecture is much more obvious. This is a more local construction that directly glues affinoid domains together. As of now we do not know whether these two non-archimedean mirrors are the same, but we expect them to be. To spell this out as a precise conjecture would require us to introduce too much new notation so we omit it.

We expect that the next conjecture follows from the spectral sequence in \cite{varolgunes}.% but we do not give details here. There are still some technical details to be worked out.

\begin{conjecture}\label{conj-poly-vec}
There exists a ring isomorphism $$SH_{M_{\mathcal{R}}}^k(M_{\mathcal{R}},\Lambda)\to\bigoplus_{p+q=k} H^p(\mathcal{Y}_{\mathcal{R}},\Lambda^qT\mathcal{Y}_{\mathcal{R}}),$$ where the right hand is the Cech cohomology with respect to an affinoid or Stein cover.
\end{conjecture}

To make a final speculation note that we can equip $Y_\mathcal{R}$ with an algebraic volume form, which gives rise to an analytic volume form on $\mathcal{Y}_{\mathcal{R}}$.

\begin{conjecture}
The Stein continuous map from Conjecture \ref{conj-Stein-cont} is a non-archimedean SYZ fibration.
\end{conjecture}

We equip the polyvector fields in Conjecture \ref{conj-poly-vec} with a differential by transporting the deRham differential.
\begin{conjecture}\label{con-BV}This differential is compatible with the BV operator on the A-side.\end{conjecture}

\section{$3$-dimensional examples: Semitoric fibrations}\label{Sec:GrossFibrations}

In \cite[Theorem 2.4]{Gross00}  Gross introduced a class of examples of special Lagrangian torus fibrations on complex toric Calabi-Yau's after removing a certain complex codimension one smooth subvariety. We will restrict ourselves to the three complex dimensional case. The important features for us are that these are singular Lagrangian torus fibrations $\cL: M\to B$ which admit a Lagrangian section, which are invariant under a global Hamiltonian $2$-torus action, and whose set of singular values is a 1-dimensional graph which will allow integrable anti-surgery. More precisely,

\begin{definition}\label{def-3d-semitoric}
A \emph{semitoric 3-fibration} consists of the data $(M,H,B,f_{\mu},\cL)$ where
\begin{itemize}
\item  $M$ is a symplectic manifold of dimension $6$,
\item  $H$ is a two dimensional torus with dual Lie algebra $\fh^*$,
\item  $B$ is $3$-manifold equipped with a smooth submersion $f_{\mu}:B\to\fh^*$ making it into an $\bR$-bundle,
\item $\cL:M\to B$ is a Lagrangian torus fibration with singularities. That is, $\cL$ is smooth proper surjection so that the non-singular fibers are Lagrangian.   
\end{itemize}
We require the data to satisfy the following conditions.
\begin{itemize}
\item The map $\mu:=f_{\mu}\circ\cL$ generates an $H\simeq\bT^2$ action with no finite stabilizers.
\item The set of singular values of $\cL$ is a $1$-dimensional trivalent graph $\Delta\subset B$ which is contained in a smooth section of $f_{\mu}$. We refer to the locus of these singular values as the \emph{discriminant locus $\Delta\subset B$}.
\item The singular points of $\cL$ are precisely those contained in lower dimensional orbits of $H$. Moreover, each singular fiber contains exactly one connected lower dimensional orbit of $H$. 
\item $\cL$ admits a Lagrangian section $\sigma:B\to M$ which does not pass through any singular point.
%\item {\color{blue}Each point $p\in M$ has an $H$-invariant neighborhood $V$, an $H$-invariant complex structure $J$ on $V$, and a non-vanishing holomorphic volume form $\Omega$ which is preserved by $H$.}
\end{itemize}
%together with a projection $f_{\mu.  A \emph{ semi-toric Lagrangian fibration of $M$ over $B$}  is a singular Lagrangian torus fibration $\cL:M\to B$ such that $\tilde{\cL}=(\tilde{H},\tilde{\mu})$ where $\tilde{\mu}$ generates an $H\simeq\bT^{n-1}$ action whose lower dimensional orbits lie over a codimension $2$ stratified set in $\fh^*$ and such that the critical points of $\tilde{\cL}$ coincide with the lower dimensional orbits. For such a fibration we define $I^{\pm}_0$ by analogy to the above.
%We say that $\tilde{\cL}$ is \emph{complete} if for each $c\in\fh^*$ we have that $I_0^{\pm}({\mu}^{-1}(c))=\bR$.
\end{definition}
Instead of recalling Gross' construction in full generality here, we construct these Lagrangian fibrations in two examples. 

\begin{example}\label{ex-gross1}
Let $X=\bC^3$ and let $\pi:\bC^3\to\bC$ be defined by $(z_1,z_2,z_3)\mapsto z_1z_2z_3$. Let $M=X\setminus\pi^{-1}(1)$.  We initially consider $M$ as a symplectic manifold with a symplectic structure $\omega_0$ obtained by restricting the standard one from $\bC^3$. This symplectic structure is not a priori geometrically bounded. However, we shall see below there is a canonical way to inflate $\omega_0$ near the divisor so that it becomes geometrically bounded.

Let $G=\bR^3/\bZ^3$ with its diagonal action of $\bC^3$ and let $H\subset G$ be the $2$-torus which preserves the function $\pi$. Concretely, $H$ is the sub-torus defined by the equation $x_1+x_2+x_3=0$ and $\fh^*$ is the quotient $\left(\bR^3\right)^*/\text{span}\{e_1^\vee+e_2^\vee+e_3^\vee\}$. The open subset $M\subset X$ is $H$-invariant, so it carries an $H$ action. The action of $H$ is Hamiltonian (as is that of $G$). 

Let $B=\bR^3$ and let $\cL:M\to B$ be given by 
$$
(z_1,z_2,z_3)\mapsto \left(\ln |1-z_1z_2z_3|, \frac1{2}(|z_1|^2-|z_3|^2),\frac1{2}(|z_2|^2-|z_3|^2)\right).
$$
Then $\cL$ is a Lagrangian torus fibration on $M$. Indeed, the map $\cL$ is easily seen to be proper. To see that the components commute with each other note the last two components of $\cL$ generate the action of $H$. Finally, the function $\pi$ and therefore the Hamiltonian $H_0(x)=\ln |1-\pi(x)|$ for $x\in M$ is constant along orbits of the torus $H$ and so it commutes with the last two components.

Identifying $\fh^*$ with $\bR^2$ via the basis which is dual to the basis $\{e_1-e_2,e_1-e_3\}$ of $\fh$ we let $f_{\mu}$ be the projection $\bR^3\to\bR^2\simeq\fh^*$ to the last two components. 

We verify that $(M,H,B,f_{\mu},\cL)$ satisfies all the conditions. 
\begin{itemize}
\item By construction $f_{\mu}\circ\cL$ generates the $H$ action which has no finite stabilizers.
\item The set of singular points of $\cL$ is the set of critical points of $\pi=z_1z_2z_3$. That is, the set of points where two coordinates vanish. These are also precisely the points with non-trivial stabilizers for the action of $H$. The orbits of these points form circles, except the origin where it is just a point.

The singular values are thus contained in the plane $\{x_3=0\subset B=\bR^3\}$. The subset of this where two given coordinates among $z_1,z_2,z_3$ vanish projects under  $\cL$ to a properly embedded trivalent graph $\Delta$ with one vertex. For each point $x=(x_1,x_2,0)\in\Delta$ the set of critical points of $\cL$ in the preimage $\cL^{-1}(x)$ is contained inside an intersection of two coordinate hyperplanes and is an orbit with moment value $(x_1,x_2,0)$. All such orbits are lower dimensional and connected. 

\item Finally, a Lagrangian section $\sigma$ may be constructed as follows. All the critical points occur in the fiber $\pi^{-1}(0)$. Pick any properly embedded real line $\gamma$ in $\bC\setminus\{1\}$ with $\lim_{t\to\infty}\gamma(t)=1$, $\lim_{t\to-\infty}\gamma(t)=\infty$ and $\gamma(t)\neq 0$ for all $t$. For any $t\in \bR$ the fiber $\pi^{-1}(t)$ is symplectomorphic to $\bR^2\times\bT^2$. So we define $\sigma$ by taking the zero section in one such fiber and using symplectic parallel transport along $\gamma$.  To see that the parallel transport does not go off to infinity note that symplectic parallel transport preserves the moment map $\mu:=f_{\mu}\circ\cL$. Indeed $\mu$ generates an action which is tangent to the fibers of $\pi$. Now use the fact that  $\mu$ is proper along the fibers of $\pi$.
\end{itemize}
\end{example}

We consider now one example where $M$ is not exact. It is not hard to see that $M$ is in fact not even exact in the complement of any compact set. 

\begin{example}\label{ex-gross2}
Consider the quasi-projective variety $X=\cO(-1)\oplus\cO(-1)\to \bC P^1$. As a complex algebraic variety it can be obtained as the projective GIT quotient of $\bC^4$ with the trivial line bundle by the $\bC^*$ action
\begin{equation}\label{eqGitQuot}
t\cdot(z_1,z_2,z_3,z_4,\lambda)=(tz_1,tz_2,t^{-1}z_3,t^{-1}z_4, t^{-1}\lambda).
\end{equation}
Concretely, the complex points of $X$ can be identified with the quotient of  $(\bC^2\setminus\{(0,0)\})\times \bC^2$ by the same action (restricted to $\lambda=0$). 
On the other hand, all semi-stable orbits are stable in this case and the GIT quotient \eqref{eqGitQuot} corresponds by the Kempf-Ness theorem to a symplectic reduction of $\bC^4$ by the Hamiltonian $S^1$ action  generated by the moment map $\mu(z_1,z_2,z_3,z_4)=|z_1|^2+|z_s|^2-|z_3|^2-|z_4|^2$. The latter action is obtained from \eqref{eqGitQuot} ($\lambda=0$) by restricting to $t\in S^1\subset\bC^*$. The symplectic quotient involves the choice of a real number $c\in\bR_{>0}$ so that $X=\mu^{-1}(c)/S^1$. Considering $X$ as a symplectic quotient for $c=1$, it inherits a Hamiltonian $$G:=\bT^4/\{(t,t,t^{-1},t^{-1})\mid t\in S^1\}$$ action with moment map $\mu^{G}:X\to \{x_1+x_2-x_3-x_4=0\}\subset (\mathfrak{t}^4)^*$. This map is induced from the restriction (domain and target) of the standard moment map $\mathbb{C}^4\to (\mathfrak{t}^4)^*$ (with image $\bR_{\geq 0}^4$) to $\mu^{-1}(1)\to \{x_1+x_2-x_3-x_4=1\}.$ In particular the moment map image is obtained by slicing $\bR_{\geq 0}^4$ with the hyperplane $\{x_1+x_2-x_3-x_4=1\}$. Note that the Lie algebra $\mathfrak{t}^4$ has an implicit basis $e_1,e_2,e_3,e_4$ and the linear coordinates $x_1,x_2,x_3,x_4$ on $(\mathfrak{t}^4)^*$ are defined with respect to this basis. On $(\mathfrak{t}^4)^*$ we also get the dual basis $e_1^\vee,e_2^\vee,e_3^\vee,e_4^\vee.$

Denote by $H\hookrightarrow G$ the sub-torus whose action preserves the product $z_1z_2z_3z_4$. $H$ acts on $M$ in a Hamiltonian way. We have that $\fh^*$ is naturally isomorphic to $$\{x_1+x_2-x_3-x_4=0\}/\text{span}(e_1^\vee+e_2^\vee+e_3^\vee+e_4^\vee).$$
The moment map $X\to \fh^*$ is the composition of $\mu^{G}$ with the canonical quotient map.

The global function $\pi:\bC^4\to\bC, (z_1,z_2,z_3,z_4)\mapsto z_1z_2z_3z_4$ is constant on the orbits of the $\bC^*$ action  \eqref{eqGitQuot}. So, it descends to a global function on $X$, which we still denote by $\pi$.  As before, consider the variety $M=X\setminus \pi^{-1}(1)$. %We initially consider $M$ as a symplectic manifold with a symplectic structure $\omega_0$ obtained by restriction. This symplectic structure is not a priori geometrically bounded as it is obtained by removing a divisor. However, we shall see below there is a canonical way to inflate $\omega_0$ near the divisor so that it becomes geometrically bounded.

%It is defined by the equations $\sum_{i=1}^4x_i=0$ and $x_1+x_2-x_3-x_4=0$. We fix the basis $\{x_1-x_3,x_2-x_4\}$ for $\fh^*$. With this choice $H$ is identified with $\bT^2$ and the moment map $\mu^H$ of the $H$ action  is given by $\mu=(H_{F_1},H_{F_2})$ where $H_{F_1}=|z_1|^2-|z_3|^2$ and $H_{F_2}=|z_2|^2-|z_4|^2$.

$M$ carries the function $H_0(x)=\ln |1-\pi(x)|$ for $x\in M$. The function $\pi$ and therefore the Hamiltonian $H_0$ is constant along orbits of the torus $H$. It follows that combining $H_0$ with the moment map $X\to \fh^*$ restricted to $M$ we obtain a Lagrangian torus fibration $\mathcal{L}:M\to \bR\times \fh^*$ with singularities. We define $f_{\mu}$ to be the projection to $\fh^*$.

Verification that  this data conforms with the definition of a semi-toric 3-fibration is the same as in the previous example and is omitted.

\end{example}

%\begin{remark}
%The discussion above can be generalized to any toric Calabi-Yau variety obtained as a quotient of $\bC^{n+k}$ by a $k$-subtorus of $SU(n+k)$ with its standard action. %Moreover, one can construct such a toric Calabi-Yau variety for any smooth tropical curve in $\bR^2$ thought of as the dual Lie algebra of the $2$-torus. For details see Groman19.
%\begin{figure}
%\includegraphics[trim={3cm 20cm 2cm 3cm}, scale=0.8]{ToricDiagrams.pdf}
%\caption{some toric diagrams}
%\label{FigToric}
%\end{figure}
%\end{remark}

\begin{figure}
\includegraphics[width=0.9\textwidth]{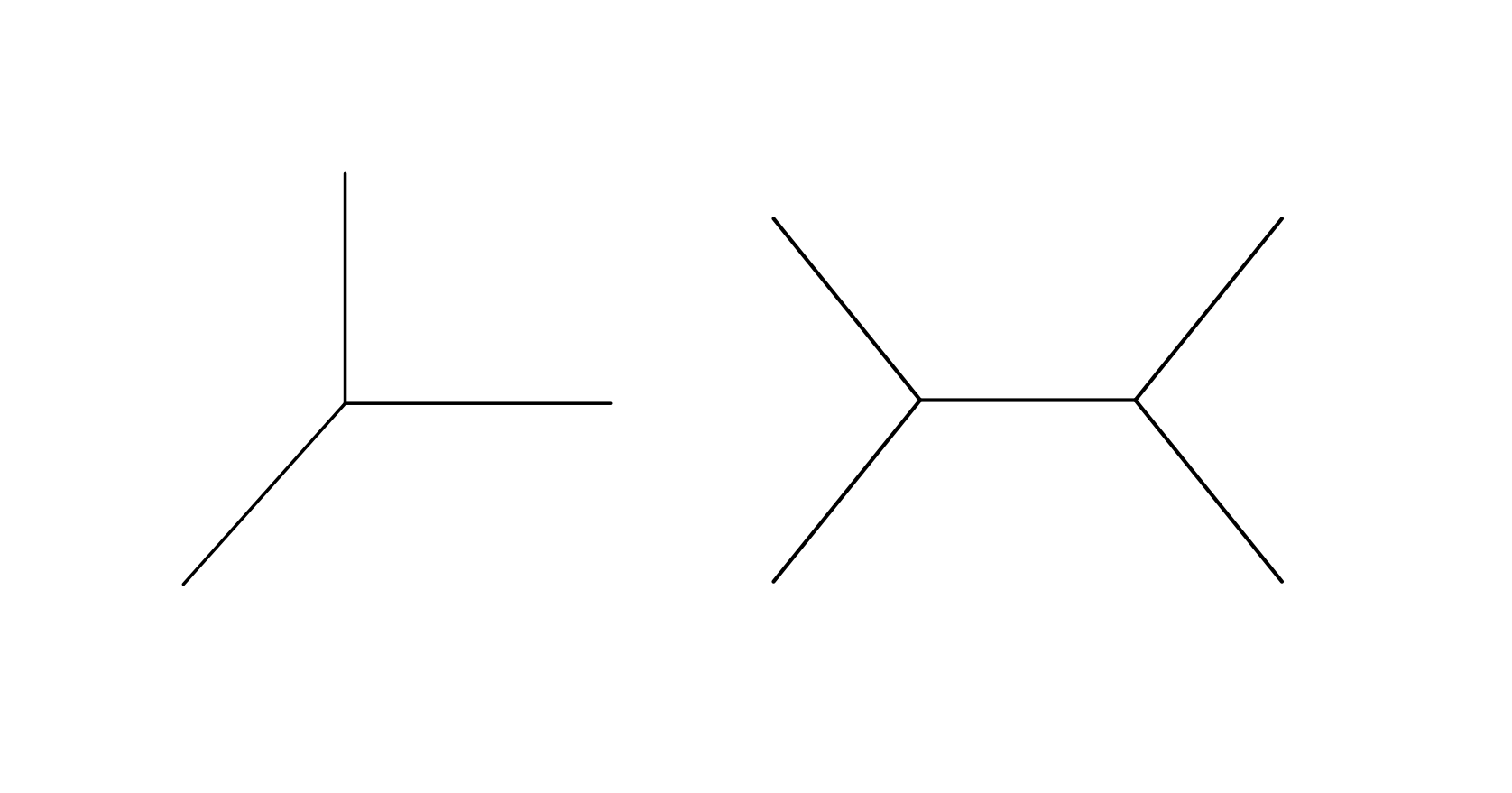}
\caption{The diagram $\Delta$ in the case of Example \ref{ex-gross1} (left) and Example \ref{ex-gross2} (right)}
\label{figDiagarm}
\end{figure}

We proceed with the general discussion. Given a semi-toric 3-fibration we pick a basis for $\fh^*$, inducing a pair of functions $\eta_1,\eta_2: B\to\bR^2$, and a transverse coordinate $\eta_0:B\to\bR$ such that $\Delta\subset \eta^{-1}(0)$. Such a coordinate exists by the assumption that $f_{\mu}$ is an $\bR$-bundle with a given section whose image contains $\Delta$\footnote{The image of the section divides each fiber into two connected components. Because the base is contractible we can make a choice of a smoothly varying positive/negative components. Choosing a Riemannian metric on $B$, we can define $\eta(x)$ for $x$ belonging to $\bR$ fiber $F$ as $\pm$ the length of the part of $F$ between the section and $x$.}.  This produces  a diffeomorphism $\eta=(\eta_0,\eta_1,\eta_2):B\to\bR^3$. In this notation, we  write
\begin{equation}
(H_0,H_{F_1},H_{F_2})=\eta\circ{\cL}
\end{equation}

The set of regular values in the base $B_{reg}:=B\setminus \Delta$ carries an integral affine structure induced by the fibration $\cL$. This structure is defined by taking the set of integral affine functions on a simply connected $V\subset B_{reg}$ to consist of functions $f: U\to\bR$ such that the Hamiltonian flow of $f\circ \cL$ is $1$-periodic. From this description it is clear that the functions $\eta_1,\eta_2$ are integral affine coordinates.

The function $\eta_0$, on the other hand, is not integral affine.  Moreover, there is a possibility that the lines $f_{\mu}=Const$ in $B$ map under integral affine coordinates to intervals which are not the entire real line. This motivates the following definition. 

\begin{definition}\label{dfHorizontalCompleteness}
We say that a semitoric SYZ fibration is \emph{horizontally complete} if for any $c\in\fh^*\setminus f_{\mu}(\Delta)$ we have that the line $f_{\mu}^{-1}(c)$ is integral affine isomorphic to $\bR$ while for $c\in f_{\mu}(\Delta)$ we have that
$f_{\mu}^{-1}(c)\setminus\Delta$ is integral affine isomorphic to $\bR\setminus\{0\}$. 
\end{definition}
%{\color{gray}
%This means the integral affine structure induced by $\cL$ is \emph{incomplete}.  

%To make this more precise we introduce the following definition.}

%{\color{gray}\begin{definition}
%Let $B$ be a topological manifold, let $B_0\subset B$ be an open and dense subset carrying an integral affine structure. We refer to $B$ with this data as an \emph{integral affine manifold with singularities}. We say that $B$ is \emph{complete} if the following property holds: \emph{If $U$ is an integral affine chart and $\{x_i\in U\}$ is a sequence of points with bounded integral affine coordinates, then there is a subsequence $x_{i_j}$ which converges in $B$}.
%\end{definition}}

\begin{lemma}\label{lmHorizontalCompletion}
Given a semitoric 3-fibration $(M,H,B,f_{\mu},\cL)$ there exist a  symplectic manifold $\overline{M}$, a smooth $3$-manifold  $\overline{B}\simeq\bR^3$ and a {horizontally complete} semitoric 3-fibration $\overline{\cL}:\overline{M}\to \overline{B}$ fitting into a commutative diagram
\begin{equation}
\xymatrix{M\ar[d]^{\cL}\ar[r]&\overline{M}\ar[d]^{\overline{\cL}}\\B\ar[rd]\ar[r]&\overline{B}\ar[d]\\&\fh^*},
\end{equation}
where the upper horizontal arrow is a symplectic embedding and the lower horizontal arrow is a smooth embedding.
\end{lemma}

\begin{remark}
It can be shown that  $(\overline{M},\overline{\cL})$ is {geodesically complete in the sense that all rays which don't hit $\Delta$ extend indefinitely. Moreover, it is} unique up to fiberwise symplectomorphism {with this completeness property}. We don't use this so we omit a detailed discussion, but  we refer to $(\overline{M},\overline{\cL})$ as \emph{the completion of $(M,\cL)$}. 
\end{remark}
\begin{proof}

Before diving into the details of the proof we explain the main idea. We have that the components $\eta_1,\eta_2$ of $\eta$ are integral affine coordinates on $B$ while the coordinate $\eta_0$ is not. We cover the complement of $\Delta\subset B$ by a pair of subsets $V^{\pm}$ as follows. Let $P^{\pm}=\eta^{-1}(f_{\mu}(\Delta)\times\bR_{\pm})\subset B$ and let $V^{\pm}=B\setminus P^{\mp}$.  
\begin{figure}
\includegraphics[width=\textwidth]{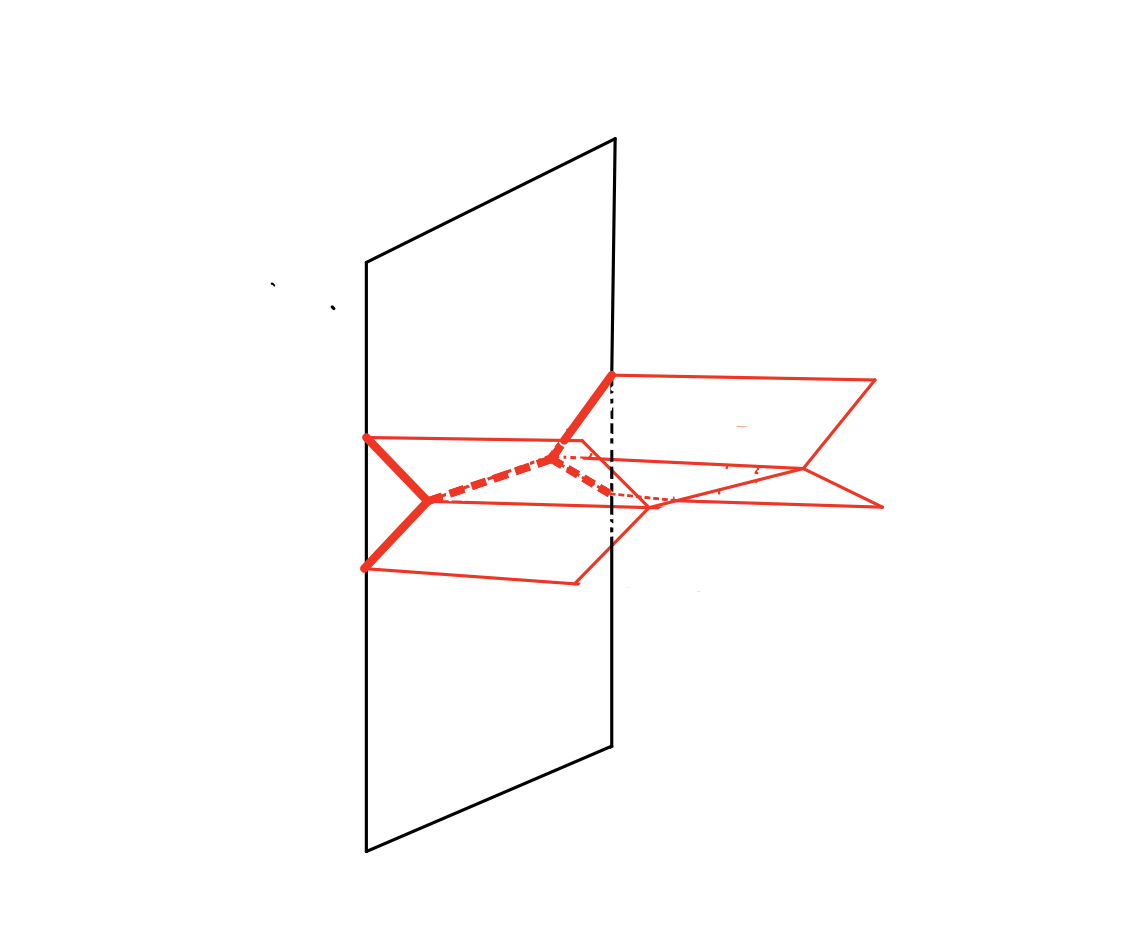}
\caption{The set $P^+$ in the case of Example \ref{ex-gross2}.}
\label{figP}
\end{figure}
The subsets $V^{\pm}$ are simply connected and so we replace the coordinate $\eta_0$ by an integral affine coordinate $i^{\pm}$ so that $i^{\pm},\eta_1,\eta_2$ form an integral affine coordinate system on $V^{\pm}$. The lack of horizontal completeness is symptomized by the possibility that along lines defined by $(\eta_1,\eta_2)=const$ the functions $i^+$ and $i^-$ are bounded above or below respectively. 
\begin{figure}
\includegraphics[width=\textwidth]{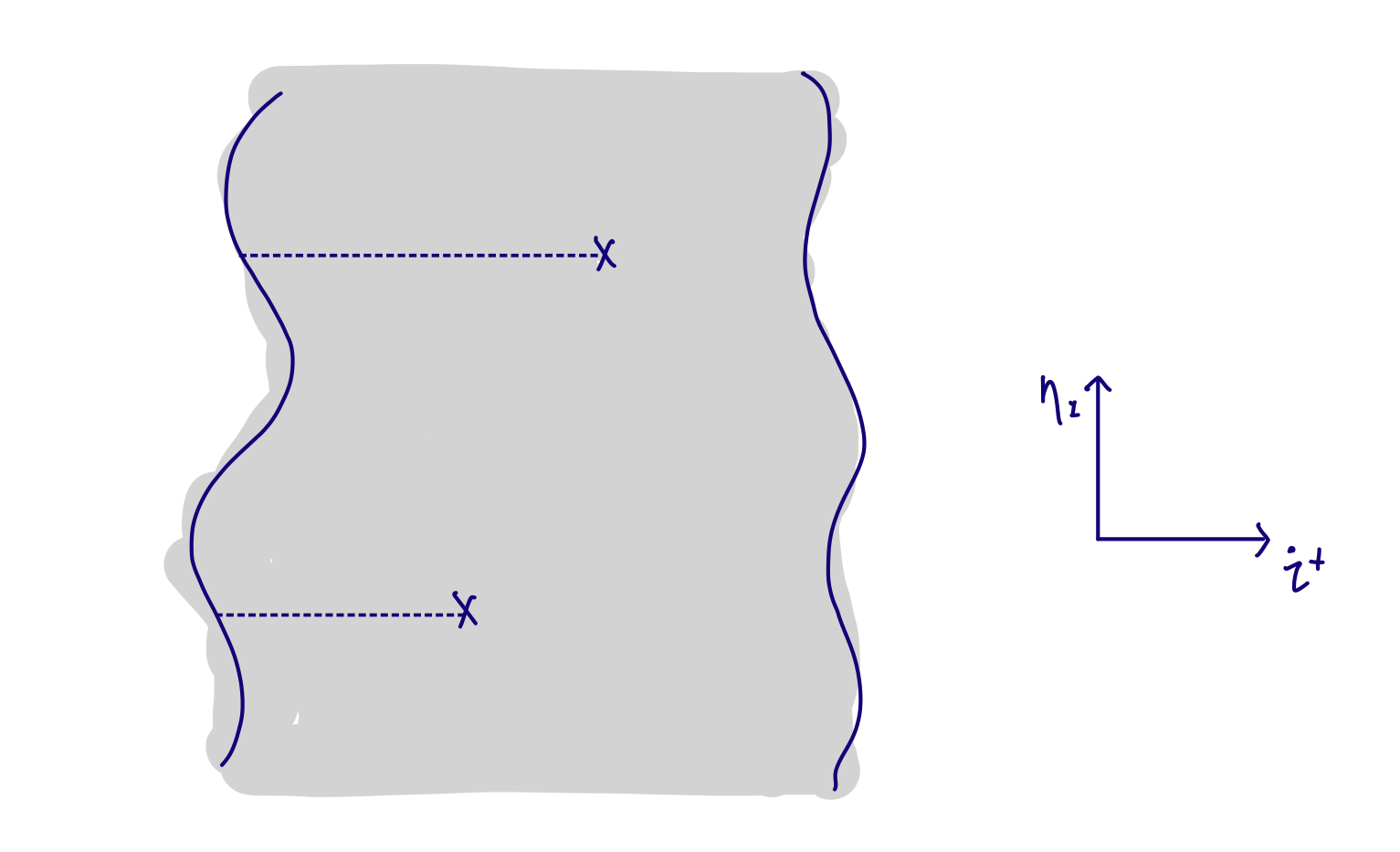}
\caption{Proof of Lemma \ref{lmHorizontalCompletion}. A slice $\eta_2=Const$ of the image of {$V^+$} in the case of example \ref{ex-gross2} under the affine coordinates $(i^+,\eta_1,\eta_2).$ Due to the incompleteness, the image has boundary given by the thick vertical lines.}
\label{figEmbedding}
\end{figure}

To deal with this, write $\phi^{\pm}=(i^{\pm},\eta_1,\eta_2)$. Then we show that the maps $\phi^+,\phi^-$ define open embeddings of $B$ into $\bR^3$.  For any $\epsilon>0$, the hypersurface $\eta_0^{-1}(\pm\epsilon)$ maps under $\phi^{\pm}$ to some hypersurface $S^{\pm}$ of $\bR^3$ which disconnects $\bR^3$ into two components. The horizontal completion $\overline{B}$ of $B$ is achieved by attaching the points of $\bR^3$ to the right of the image $S^+$ and to the left of  $S^-$, where the gluing is carried out via the identifications given by $\phi^{\pm}$ respectively. The horizontal completion $\overline{M}$ of $M$ is then achieved in the obvious way since near each of the hypersurfaces $S^{\pm}$ the fibrations is a product fibration. 

Now to the details. Let  $U^{\pm}:=\cL^{-1}(V^{\pm})\subset M$. Pick $p\in V^+\cap V^-$ and let $L_p:=\cL^{-1}(p)$. Note that $L_p$ is a smooth fiber.  Pick a primitive generator $\gamma$ of $H_1(L_p;\bZ)$ transverse to the torus orbits of $(H_{F_1},H_{F_2})$. Let $I_0^{\pm}:U^{\pm}\to \bR$ be the function giving the symplectic flux of the cylinder traced by transporting $\gamma$ from  the point $p$ along any path contained in $V^{\pm}$.

We can factor  $I_0^{\pm} =i^{\pm}\circ \cL$ for some unique function $i^{\pm}:V^{\pm}\to\bR$. We show now that the triple of functions $(i^{\pm},\eta_1,\eta_2):V^{\pm}\to\bR^3$ is an open embedding. For this let $f^{\pm}=i^{\pm}\circ\eta^{-1}$. We have that $\partial_{x_0}f^{\pm}(x_0,x_1,x_2)\neq 0$ for all $x\in \eta^{-1}(V^{\pm})$.  Indeed, otherwise the functions $(I^{\pm}_0,H_{F_1},H_{F_2})$ would be dependent functions in contradiction to the Arnold-Liouville theorem. From this it follows that the map $\eta^{-1}(V^{\pm})\to\bR^3$ defined by
\begin{equation}
(x_0,x_1,x_2)\mapsto (f^{\pm}(x_0,x_1,x_2), x_1,x_2)
\end{equation}
is an injective immersion and thus, being of full dimension, an  open embedding. This is equivalent to $(i^{\pm},\eta_1,\eta_2)$ being an open embedding.

Under the above embedding, for any $\epsilon>0$ the hypersurfaces $\eta_0=\pm\epsilon$ map to hypersurfaces $S^{\pm}\subset\bR^3$ which are each the graph of a function $x_0=g^{\pm}(x_1,x_2)$. By picking a Lagrangian section of $\cL$ in the region $\eta_0\geq\epsilon$ (resp. $\eta_0\leq-\epsilon$) we obtain a symplectic fiberwise embedding of $\eta\circ\cL^{-1}(\{x_0<\epsilon\})$ (resp. $\eta\circ\cL^{-1}(\{x_0>-\epsilon\})$) into the open submanifold of $\bR^3\times\bT^3=T^*\bT^3$ given by $\{x_0<g^+(x_1,x_2)\}$ (resp. $\{x^0>g^-(x_1,x_2)\}$).   Use this to define the completions $\overline{U}^{+}\to\overline{V}^{+}$ (resp. $\overline{U}^-\to\overline{V}^-$) by attaching the regions to the right of $S^+$ (resp. to the left of  $S^-$) times $\bT^3$. Finally, define
\begin{equation}
\overline{M}:=(\eta\circ \cL)^{-1}([-\epsilon,\epsilon])\cup \overline{U}^{+}\cup \overline{U}^{-}
\end{equation}
and,

\begin{equation}
\overline{B}:=\overline{V}^{+}\cup B\cup\overline{V}^{-}.
\end{equation}
$\overline{M}$ is equipped with a Lagrangian torus fibration over $\overline{B }$ obtained by extending $\cL$ in the obvious way, which fits into a commutative diagram as in the statement of the Lemma. %{\color{red}The horizontal completeness is immediate from the construction.} %It remains to verify that $\overline{B}$ considered as an integral affine manifold with singularities, with the integral affine structure induced from $\overline{\cL}$ is {\color{red}geodesically} complete. {\color{red} Consider the obvious extension of $f_{\mu}$ to $\overline{B}$  which we continue to denote by $f_{\mu}$.   Let $\overline{P}:=f_{\mu}^{-1}(\Delta)$ and let $\overline{P}^{\pm}$ be the components of $\overline{P}\setminus\Delta$. Observe that $\overline{B}^{\pm}=B\setminus \overline{P}^{\mp}$, and therefore that $\overline{B}_{reg}=\overline{B}^+\cup\overline{B}^-$. Observe also  that $\overline{B}$ has the property that all lines of the form $f_{\mu}=c$ for $c\not\in f_{\mu}(\Delta)$ are integral affine isomorphic to $\bR$. For $c\in f_{\mu}(\Delta)$ each component of the complement of $\Delta$ is integral affine isomorphic to $(0,\infty)$. It follows that $\overline{B}^{+}$ is integral affine isomorphic to an open and dense subset of $\bR\times \fh^*$ obtained by removing the graph of some smooth function $g^+$ defined on $f_{\mu}(\Delta)$ together with the points above it. The same description applies to $\overline{B}^{-}$ with some smooth function $g^{-}$. It can be shown that $g^+=g^-$, but we won't use this fact, so we don't prove it.  

\end{proof}

%Consider projection $\bR^3\to\bR^2$ forgetting the third coordinate. We denote by $f_{\mu}(\Delta)$ the image of $\Delta$ under this map. The discriminant locus is given in integral affine coordinates   by a pair of functions $g^{\pm}:f_{\mu}(\Delta):\to\bR$. These functions are completely determined once the basepoint $p$ and the class $\gamma$ are fixed and are thus in invariant of $\cL$.
\begin{example}
Another way to achieve the completion in the case of the complement of an anti-canonical divisor in a toric Calabi-Yau 3-fold is to fix $M,\cL$ and modify the symplectic form as follows. Let $\omega_1$ be the symplectic form on $\bC\setminus\{1\}\simeq \bC^*$ defined by pull-back of the standard form $ds\wedge dt$ on the cylinder $\bR\times S^1$. Then consider the symplectic form $\omega=\omega_0+\pi^*\omega_1$. Then $\omega$ is $H$ invariant and the system $\cL$ is still involutive. Moreover, it is easy to see that $\cL$ is now horizontally complete.  %Henceforth we shall continue discussing $M,\cL$, but we shall be referring to the modified symplectic form with respect to which $\cL$ is complete. As before, $H_0$ is not an affine coordinate, and $\Delta\subset H^{-1}(0)$. However, the torus fibration is now complete.
\end{example}

Proceeding with the general discussion  we assume from now on that $\cL$ is horizontally complete and show that $(M,\omega)$ admits a pair of complete embeddings of $T^*\bT^3$ covering the complement of the critical points of $\cL$. More precisely, consider as before the subsets $P^{\pm}=\eta^{-1}(\eta(\Delta)\times\bR_{\pm})\subset B$. Then we prove the following proposition
\begin{proposition}\label{proComlpeteT3embedding}
There exist co-isotropic stratified 4-dimensional subspaces $C^{\pm}\subset\cL^{-1}(P^{\pm})$ such that denoting $M^{\pm}:=M\setminus C^{\mp}$ there is a symplectomorphism
\begin{equation}
\iota^{\pm}:M^{\pm}\to T^*\bT^3.
\end{equation}
Moreover, for any convex open neighborhood $\cN(P^{\mp})\subset \bR^3$ of the stratified space $P^{\mp}$ we can choose $\iota^{\pm}$   to agree  on  $M\setminus \cL^{-1}(\cN(P^{\mp}))$ with the embedding into $T^*T^3$ defined by the action angle coordinates.
\end{proposition}

While Proposition \eqref{proComlpeteT3embedding} is somewhat heavy in notation, both the statement and the proof  are in complete analogy with the discussion of integrable anti-surgery in dimension $2$ in the case of a single node. The reader may find it helpful to understand the analogy before diving into the proof. In the case of the focus focus singularity, we have a pair of affine rays $l^{\pm}$ emanating from the nodal value. These are characterized by the constancy of the moment map of the $S^1$ action.  We then construct a pair of Lagrangian tails $L^{\pm}$ by smoothly choosing a lift of each point of $l^{\pm}$ to an orbit of the $S^1$ action. An essential step in the construction is presenting the orbit space of the complement to a tail as a union of symplectic cylinders of infinite area in both directions. 

In the present case, instead of a singular value we have a $1$-dimensional set $\Delta$ of singular values, and instead of a global $S^1$ action we have a global $\bT^2$ action. For each singular value $x\in\Delta$ we have a pair $l^\pm$ of affine rays emanating from $x$ on which the moment map is constant. The sets $P^{\pm}$ are each the union of such rays going off in the $\pm$ direction. The sets $C^{\pm}$ are a smooth lift of each point of $P^{\pm}$ to an orbit of the $2$-torus action. The generalization of the notion of a Lagrangian tail in this dimension thus turns out to be a $4$-dimensional coisotropic space.

As in the case of nodal singularities, the key to the construction of the complete embeddings is the structure of the $\bT^2$ orbit space $Y=M/H$. {In this case we will again present the complement $(M\setminus C^{\pm})/H$ as a family of symplectic cylinders of infinite area in both directions.} Denote by $Y^{\#}\subset Y$ the regular orbits. From now on we identify $H$ with $\bT^2$ using the basis  that we chose for $\fh^*$.  Consider the commutative diagram
\begin{equation}
\xymatrix{M\ar[rd]^{\cL}\ar[r]^q& Y\ar[d]^{\tilde{\cL}}\\& B}
\end{equation}
Then we have
\begin{lemma}
There is a fiberwise homeomorphism $\psi:B\times S^1\to Y$ fitting into a commutative diagram
\begin{equation}
\xymatrix{B\times S^1\ar[rd]^{p_1}\ar[r]^{\psi}& Y\ar[d]^{\tilde{\cL}}\\& B}.
\end{equation}
Moreover, $\psi$ restricts to a diffeomorphism
\begin{equation}
B\times S^1\setminus (\Delta\times \{1\})\to Y^{\#}.
\end{equation}
\end{lemma}
%\begin{lemma}
%There exists a homeomorphism $\psi:Y\to B\times S^1$. Moreover, $\psi$ restricts to a diffeomorphism
%\begin{equation}
%Y^{\#}\to  B\times S^1\setminus \Delta\times\{1\}.
%\end{equation}
%Denote by $\phi$ the angle coordinate on $S^1$. Then there is some positive function $g:B\to \bR_+$ such that for each $c\in \fg^*$ the symplectic form $\tilde{\omega}_c$ on $\psi^{-1}(f_{\mu}^{-1}(c)\times S^1)$ induced by symplectic reduction is of the form
%\begin{equation}
%\tilde{\omega}_c=g(c,\eta_0)d\eta_0\wedge d\phi
%\end{equation}
%away from fibers containing a critical point .
%\end{lemma}
\begin{proof}
%Let $\sigma$ be a global section of $\cL$ which does not pass through any critical point.
The Hamiltonian $H_0$ commutes with the $\bT^2$ action and so induces a smooth flow $t\mapsto\phi_t$ on $Y^{\#}$. For a point $p\in M$ denote by $H\cdot p\in Y$ the $\bT^2$-orbit through $p$. Fix a section $\sigma$ of $\cL$ which does not pass through any singular point.  Consider the map
\begin{equation}
B\times\bR\to Y^{\#},\quad (z,t)\mapsto \phi_t(H\cdot\sigma(z)).
\end{equation}
Let $\tau:B\to\bR\cup\{\infty\}$ be the function assigning to $z\in B$ the infimal number so that $\phi_\tau (H\cdot \sigma(z))=H\cdot\sigma(z)$. Note $\tau(z)=\infty$ exactly for $z\in\Delta$. By the Implicit Function Theorem $\tau$ is smooth on $B\setminus\Delta$ since the flow $\phi_t$ has no fixed points. Define an equivalence relation $\sim$ on $B\times\bR$ by $(z,t)\sim (z,t+\tau(z))$ whenever $\tau(z)<\infty$. Let $Y^o=B\times\bR/\sim$. Note that each fiber of $Y$ over $\Delta$ has exactly one  singular orbit. So, we diffeomorphically identify $Y^o$ fiberwise with $B\times S^1\setminus\Delta\times\{1\}$. We then have a diffeomorphism
\begin{equation}
\psi: Y^o\to Y^{\#}.
\end{equation}

Extend $\psi$ to a map $B\times S^1\to Y$ by mapping each point of the form $(z,1)$ for $z\in \Delta$ to the unique lower dimensional orbit in the fiber over $z$. Then this  extended map is a  continuous bijection by construction. Moreover it is proper. Indeed, the pre-image of a compact set projects to a bounded subset of $B$ and so is closed and bounded.  $Y$ is Hausdorff since it is the quotient of $M$ by a compact Lie group. By general nonsense, a continuous proper bijection to a Hausdorff space is closed.  It follows that the extended map $\psi$ is a homeomorphism.

%The claim about the symplectic form is as in the nodal case except that in this case we cannot refer to San for finiteness of the area in the neighborhood of a puncture. Instead we refer to Castano-Bernhard.

\end{proof}

Henceforth we consider $Y$ as a smooth manifold with the smooth structure induced by $\psi$. Let $\tilde{C}^{\pm}:=\psi(P^{\pm}\times\{1\})\subset Y$. Denote by $P^{\pm}_{\epsilon}$ the $\epsilon$ neighborhood of $P^{\pm}$ in $B$.  Let $\tilde{C}^{\pm}_{\epsilon}:=\psi(P^{\pm}_{\epsilon}\times[e^{-i\epsilon},e^{i\epsilon}])$.

\begin{lemma}
There exists a diffeomorphism $\chi^{\pm}:Y\setminus \tilde{C}^{\pm}\to Y$ preserving the fibers of $f_{\mu}\circ\psi$ and which is identity away from $\tilde{C}^{\pm}_{\epsilon}$. %and such that the image of the intersection of  $Y\setminus \tilde{C^{\pm}}$ with each such fiber is symplectomorphic to the complete cylinder $\bR\times S^1$.
\end{lemma}
\begin{proof}
Working with the global identification of $X$ with $B\times S^1$, this claim is a straightforward exercise in differential topology. We nevertheless spell out the details. Let $h:\fh^*\times S^1\setminus\Delta\times\{1\}\to\bR_+$ be  a function which is supported in an $\epsilon$ neighborhood of $f_{\mu}(\Delta)\times\{1\}$ and converges to infinity on $\Delta\times\{1\}$. Let $f=f_{\mu}^{-1}\circ h$. Let $g^+:\bR\to [0,1]$ be a smooth monotone function which is supported on $[-\epsilon,\infty)$ and is identically $1$ on $[0,\infty)$. Let $g^-:\bR\to [0,1]$ be defined by $g^{-1}(t)=g^+(-t)$.

Let
\begin{equation}
\tilde{\eta}^{\pm}_0:=\eta_0\pm g^{\pm}f
\end{equation}

Define
\begin{equation}
\chi^{\pm}:(\eta_0,\eta_1,\eta_2,\phi)\mapsto (\tilde{\eta}^{\pm}_0,\eta_1,\eta_2,\phi).
\end{equation}

We verify that this is indeed a diffeomorphism as desired. First note the Jacobian is triangular with functions that are pointwise $\geq 1$ on the diagonal. It thus suffices to verify that the fiber $(h_1,h_2,\phi)=Const$ in $Y\setminus\tilde{C}$ is mapped bijectively to the fiber $(\eta_1,\eta_2,\phi)=Const$ in $Y$.  Injectivity is clear since $\frac{\partial \tilde{\eta}^{\pm}_0}{\partial \eta_0}\geq 1$. As for surjectivity, the image $\chi((\eta_1,\eta_2,\phi)=Const)$ is an interval by continuity, so it suffices to verify that it is unbounded in both directions. This is clear from the definition of $\tilde{\eta}^{\pm}_0$. Preservation of the fibers and the support property are both immediate by construction.

%Let $V^{\pm}\subset Y$ be given as the pre-image under $\psi$ of the subset $P^{\pm}_{\epsilon}\times [-\epsilon,\epsilon]_\phi$.  Continue as in the nodal case.
\end{proof}

The fibers of $f_{\mu}\circ\tilde{\cL}:Y\setminus \tilde{C}^{\pm}\to\fh^*$
each carry a sympelctic structure by symplectic reduction.
\begin{lemma}
The fibers are each symplectomorphic to the complete cylinder $\bR\times S^1$.
\end{lemma}
\begin{proof}
Abbreviate $F:=f_{\mu}\circ\tilde{\cL}$. First we consider  fibers $F_x$ over points $x\not\in f_{\mu}(\Delta)$. Let $\gamma$ be a generator of $H_1(F_x;\bZ)$ and  let $\tilde{\gamma}$ be a lift of $\gamma$ to $H_1(q^{-1}(F_x);\bZ)$ . It suffices to show that the flux of $\gamma$ along fibers of the projection to {$\fh^*$} is unbounded both above and below. This is readily seen to be equivalent to the flux of $\tilde{\gamma}$ along fibers of {$f_{\mu}$} being unbounded both above and below. {This follows by horizontal completeness}.%{\color{gray}The flux of $\tilde{\gamma}$ is an integral affine coordinate while $q^{-1}(F_x)$ is non-compact and contained in a single integral affine chart. So the claim follows by completeness.}

We now consider the slightly more involved case where $x\in f_{\mu}(\Delta)$. For definiteness consider the case of $Y\setminus C^-$. The fiber $F_x$ is then the complement of a half infinite ray in a cylinder, so half the fibers of the projection to $B$ are contractible and we cannot talk about the flux of $\gamma$ along the fibers in the negative direction. However it suffices to show that each component of the complement of a non-contractible fiber of the projection to $B$ has infinite area. This clear for the positive component $F$ arguing by flux and {horizontal} completeness as before. For the negative component $N$ we argue as follows. Let $N'$ be obtained from $N$ be filling the ray back in. Since the ray has measure $0$, we have $area(N')=area(N)$. Let $N''\subset N'$ be an $S^1$ invariant subset which does not contain the singular point. Then the $area(N')\geq area(N'')$. Now arguing again by flux and {horizontal} completeness we have that $area(N)\geq area(N'')=\infty$.
\end{proof}

\begin{proof}[Proof of Proposition \ref{proComlpeteT3embedding}]
We take $C^{\pm}:=q^{-1}(\tilde{C}^{\pm})$. %Denote by $\tilde{\cL}:X\to B$ the unique map defined by $\cL=\tilde{\cL}\circ q$.
Let $M^{\pm}:=M\setminus C^{\mp}$. We define a map $\cL^{\pm}:M^{\pm}\to B$ by
\begin{equation}
\cL^{\pm}:=\tilde{\cL}\circ \chi^{\pm}\circ q.
\end{equation}
Then
\begin{itemize}
\item $\cL^{\pm}$ is proper, since it's a composition of proper maps.
\item $\cL^{\pm}$ is a smooth submersion, since it's a composition of smooth submersions.
\item $\cL^{\pm}$ has Lagrangian fibers. For this note that the components of $\cL^{\pm}$ are $\tilde{H}_0,H_{F_1},H_{F_2}$ where $\tilde{H}_0=\eta_0\circ\chi\circ q$. Now since $\tilde{H}_0$ is factors through the quotient map $q:M\to X$ it is constant on the flowlines of the  $H_{F_i}$ and thus Poisson commutes with each of  them.
\end{itemize}

It follows that $\cL^{\pm}$ is smooth Lagrangian torus fibration over $B$ with no singular fibers. Since $B$ is diffeomorphic to $\bR^3$ it follows that there exists a global Lagrangian section. Picking a base point $q\in B$ and a basis $\gamma_0,\gamma_1,\gamma_2$ for $H_1(L;\bZ)$, the Arnold Liouville Theorem  produces a commutative diagram
\begin{equation}
\xymatrix{M^{\pm}\ar[d]^{\cL^{\pm}}\ar[r]&\bR^3\times\bT^3\ar[d]\\
		B\ar[r]^{i^{\pm}}&\bR^3}
\end{equation}
where the upper horizontal arrow is a symplectic immersion which is a diffeomorphism in each fiber and the lower one is a smooth immersion. To conclude the claim it suffices to prove that the map $i^{\pm}$ is a bijection. For this note we may take the last two components of $i^{\pm}$ to be given by the functions $\eta_1,\eta_2$. This corresponds to taking $\gamma_1,\gamma_2$ to span the homology of a $\bT^2$ orbit.  Thus to conclude the claim we need to establish for each $c\in\fh^*$  the line $f_{\mu}^{-1}(c)$ is mapped under the first component $i^{\pm}_0$ bijectively to  real line. First note that the line $f_{\mu}^{-1}(c)$ is non-empty. Indeed, $f_{\mu}:B\to\fh^*$ is a surjection. %By continuity the line $f_{\mu}^{-1}(c)$  gets mapped to an interval. To conclude we need to show that the interval is unbounded in both directions. For this write $I^{\pm,\tilde{\cL}}_0=i_0\circ q$.
Without loss of generality, the  base point $q$ is in $B\setminus P^{\pm}$ and lies on $f_{\mu}^{-1}(c)$. Observe now that the symplectic area of a cylinder traced out by $\gamma$ over a line $f_{\mu}=c$ in any way is the same as the area of its image under $q$ with respect to the reduced symplectic form. Now, we have shown that  reduction $(f_{\mu}\circ\tilde{\cL})^{-1}(c)$ is symplectomorphic to a complete cylinder. Thus for any submersion to $\bR$ the image of the associated integral affine coordinate is $\bR$.

\end{proof}

\appendix
\section{Floer theory on open manifolds}\label{AppDissRev}
In this appendix review the notion of \emph{dissipative Floer data} introduced in \cite{groman}. The properties that we need are summarized in Propositions \ref{tmSummary00} and \ref{tmSummary}. These are applied in \S\ref{FloerPackage} to produce
 the \emph{Hamiltonian Floer theory package} going into the construction of symplectic cohomology type invariants of compact subsets  on an open manifold.  The results presented in the present appendix are mostly contained  in \cite{groman}. %The sole innovation here is the use of bounded dissipative Floer data where \cite{groman} uses exclusively proper Floer data. We indicate the necessary adjustments.

\subsection{Intermittent boundedness}
For a Riemannian metric $g$ on a manifold $M$ and a point $p\in M$ we denote by $\inj_g(p)$ the radius of injectivity and by $\Sec_g(p)$ the maximal sectional curvature at $p$. We drop $g$ from the notation when it is clear from the context.
\begin{definition}\label{dfStrBound}
Let $(M,g)$ be a Riemannian manifold. For $a\geq 1$, the metric $g$ is said to be \emph{strictly $a$-bounded} at a point $p\in M$ if the closed ball $B_{1/a}(p)$ is compact,  for all $x\in B_{1/a}(p)$ we have $|\Sec(x)|\leq a^2,$ and for all $x\in B_{\frac1{2a}}(p)$ we have $\inj(x)\geq\frac1{2a}$.  For a  subset $K\subset M$ we say that the geometry of $g$ is strictly $a$-bounded on $K$ if it is strictly $a$-bounded at every point of $K$. We say that the geometry of $g$ is \emph{$a$-bounded}  on $K$  if there exists another Riemannian metric $g'$ on $M$ which is strictly $a$-bounded at for all $p\in K$, and for all $x\in B_{1/a}(K;g')$, the ball of radius $1/a$ around $K$ with respect to the metric $g'$, we have that
\begin{equation}
\frac1{a}|v|_{g'}<|v|_g<a|v|_{g'}.
\end{equation}
In this situation we say that \emph{$g$ is $a$-bounded on $K$ with witness supported on $V=B_{1/a}(K;g')$}.
\end{definition}
%Note that if $g$ is a bounded with witness $g'$ as in the definition, then the ball $B_1(p;g)$ contains the ball $B_{1/a}(p,g')$.

\begin{remark}
In all the examples we are aware of,  there exists $J$ so that the metric $g_J$ on the underlying manifold $M$ is strictly bounded. The main reason for introducing the more cumbersome but flexible notion of non-strict boundedness is it is easier to establish it for the Gromov metric associated with a Floer datum. This is the content of the next lemma.
\end{remark}

\begin{lemma}[A criterion for $a$-boundedness of the Gromov metric.]\label{lmCritABound}
Let $J$ be an $\omega$-compatible almost complex structure on $M$. Suppose $J$ is strictly $a$-bounded on some compact $K\subset M$. Let $U\subset\Sigma$ be a compact set on which the Riemannina metric determined by  $\omega_{\Sigma}$ and $j_{\Sigma}$ is strictly $a$-bounded.  Let $\fH$ be a Hamiltonian function valued $1$-form for which $|\nabla \fH|<a$ on the ball $B$ of radius $1/a$ around $U\times K$ with respect to the product metric. Here we consider $\fH$  as a $1$-form on $\Sigma\times M.$ Then $g_{J_\fH}$ is $a$-bounded on $U\times K$ with witness supported on $B$.

\end{lemma}

\begin{proof}
 Examining equations \eqref{eqGrTrick} and \eqref{eqGrTrick2} we see that the metric $g_{J_{\fH}}$ is $a$ equivalent to the product metric $g_{j_{\Sigma}}\times g_J$ on $\Sigma\times M$ .
\end{proof}

\begin{remark}\label{rmNonStrAdj}
The hypothesis of strict boundedness in Lemma \ref{lmCritABound} can be weakened to mere boundedness at the price of replacing the constant $a$ bounding the Gromov metric by the constant $a^2$ and replacing $B$ by an appropriate open neighborhood $B'$ of $U\times K$.
\end{remark}
\begin{remark}
The proof of Lemma \ref{lmCritABound} is sufficiently robust to allow $C^0$-small domain dependent perturbations of $J$ while maintaining $a$-boundedness.
\end{remark}
The key to obtaining $C^0$ estimates is the is the following monotonicity lemma.  For a $J$-holomorphic map $u:\Sigma\to M$ and for a measurable subset $U\subset \Sigma$ write
\begin{equation}
E(u;U):=\int_Uu^*\omega.
\end{equation}
\begin{lemma}\label{lmMonEst++}[\textbf{Monotonicity} \cite{Sikorav94} ]
Let $g_J$ be strictly $a$-bounded at $p\in M$. Let $\Sigma$ be a compact Riemann surface with boundary and let $u:\Sigma\to M$ be $J$-holomorphic such that $p$ is in the image of $u$ and such that
\[
u(\partial \Sigma)\cap B_{1/a}(p)=\emptyset.
\]
Then there is a universal constant $c$ such that
\begin{equation}
E\left(u;u^{-1}(B_{1/a}(p))\right)\geq\frac c{a^2}.
\end{equation}
If instead, we only require that $g_J$ is $a$-bounded at $p$ with witness supported on some neighborhood $V$ of $p$, we get the inequality
\begin{equation}
E\left(u;u^{-1}(V)\right)\geq\frac c{a^4}.
\end{equation}
%If $g_J$ is quasi-isometric to an $a$-bounded metric with quasi-isometry constant $A$, the same inequality holds but with $c$ replaced by $\frac{c}{A^2}$.
\end{lemma}

From this it is deduced in \cite{Sikorav94} that if $J$ is geometrically bounded, then there is an a priori estimate on the diameter of a $J$-holomorphic curve in terms of its energy. However, the topology of the space of geometrically bounded almost complex structures is unknown and is not likely to be contractible. Instead,  \cite{groman}  utilizes the local nature of the monotonicity inequality to replace geometric boundedness by the weaker condition of intermittent boundedness. This still produces a priori estimates but is also a contractible condition in an appropriate sense .

\begin{definition}[Intermittently bounded Riemannian metric]\label{dfIntBounded}
We say that a complete Riemannian metric $g$ is \emph{intermittently bounded}, abbreviated \emph{i-bounded}, if there is a constant $a\geq 1$, an exhaustion $K_1\subset K_2\subset \dots$ of $M$ by compact sets, and open neighborhoods $V_i$ of $\partial K_i$,  such that the following holds.
\begin{enumerate}
    \item $V_i\cap V_{i+1}=\emptyset$.
    \item $g$ is $a$-bounded on $\partial K_i$ with witness supported on $V_i$. % {\color{red} If this were to be strict $a$-boundedness, I would be completely fine with the definition. With the equivalence part the definition becomes ambigous. Do you mean that this statement is true for the restriction of $g$ to $\partial K_i$ as a Riemannian metric or do you mean that there is a Riemannian metric on $M$ which is strictly $a$-bounded on $\partial K_i$ and the equivalence inequalities for this metric and the original one are satisfied at the points of  $\partial K_i$? I am guessing the latter. Once what equivalent means is fully spelled out above, it should be easier to say this.}
\end{enumerate}
The data $(a,\{K_i,V_i\}_{i\geq 1})$ is called taming data for $(M,g)$. We say the \emph{taming data is supported on an open  $V\subset M$} if $V_i\subset V$ for all $i$.

Given a length non-increasing map $\pi:M\to N$ of complete Riemannian manifolds, we say that $g$ is \emph{intermittently bounded relative to $\pi$} if there is a finite open cover $\cC$ of $N$ such that  for each $U\in\cC$ there is an exhaustion $K_1\subset K_2\subset \dots$ of $\pi^{-1}(U)$ and open neighborhoods $V_i$ of $\partial K_i$
\begin{enumerate}
	\item $\pi|_{K_i}$ is proper
   	 \item $V_i\cap V_{i+1}=\emptyset$
    	\item $g$ is $a$-bounded on $\partial K_i$ with witness supported on $V_i$.
\end{enumerate}

%For a symplectic manifold $(M,\omega)$, an $\omega$-compatible almost complex structure $J$ is called \textbf{i-bounded} if $g_J$ is i-bounded. The symplectic form $\omega$ is said to be \textbf{i-bounded} if it admits an i-bounded almost complex structure. For an i-bounded $(M,\omega)$ denote by $\mathcal{J}_{i.b.}(M,\omega)$ the space of i-bounded almost complex structures.
\end{definition}
%{\color{blue}
%\begin{rem}
%We stress again that the definition of $a$-boundedness on a set $K$ concerns the behavior of $g$ on a neighborhood of $K$. %For example, in the definition of intermittent boundedness relative to $\pi$, the assumption involves  the metric on the region $B_partial K_i;g)1(\$ which is contained in $B_1(U;g)$, not merely in $U$ itself.
%\end{rem}}
\begin{example}
If $g$ is geometrically bounded, meaning that $g$ is equivalent to an $a$-bounded metric $g'$ on $M$, it is i-bounded. In this case, we can take the taming data to be $\{a, K_i=B_{3i a}(p;g'),B_a(\partial K_i;g')\}$ for some arbitrary point $p\in M$.
\end{example}
\begin{remark}
In \cite{groman}, instead of a fixed constant $a$, there are constants $a_i$ such that the geometry is $a_i$ bounded on $\partial K_i$. These are required to satisfy $\sum \frac1{a_i^4}=\infty$. We have no use for this weakening in the present setting.
\end{remark}

\begin{comment}
\begin{definition}
Let $\Sigma$ be a Riemann surface with boundary and $(M,\omega)$ a symplectic manifold. Let $\tilde{\omega}$ be a symplectic form on $\Sigma\times M$. A $\tilde{\omega}$-compatible almost complex structure $J$ is said to be \emph{vertically $i$-bounded} if $g_J$ is  and there is a
constant $a$ and an exhaustion $K_1\subset K_2\subset \dots$ of $M$ by precompact sets  that the following holds.
\begin{enumerate}
    \item $d( \Sigma\times K_i,\Sigma\times \partial K_{i+1})> \frac2{a}.$
    \item $J$ is $a$-bounded for every point on $\Sigma^o\times \partial  K_i$.
\end{enumerate}
The data $(a,\{K_i\}_{i\geq 1})$ is called taming data for $(\Sigma\times M,J)$.

\end{definition}
\end{comment}
\begin{definition}[Intermittently bounded Floer data]
\begin{itemize}
\item We say that an almost complex structure is intermittently bounded if $g_J$ is an intermittently  bounded Riemannian metric.
\item  An $s$-independent Floer datum $(H,J)$ is said to be intermittently bounded if the almost complex structure   $J_H$ on $\bR\times S^1\times M$ is intermittently bounded relative to the projection $\pi_1:\Sigma\times M\to \Sigma$.
\item More generally, given a Riemann surface $\Sigma$ with cylindrical ends we say that a monotone Floer datum $(\alpha,H,J)$ on $\Sigma$ is intermittently bounded if $g_{J_{\fH}}$ is intermittently bounded relative to the projection $\pi_1:\Sigma\times M\to \Sigma$. %there is a finite cover $C$ of $\Sigma$  so that
%\begin{itemize}
	%\item The induced cover of $\Sigma\times M$ has non-zero Lebesgue number with respect to the metric $g_{J_{\fH}}$.
	%\item  The restriction of $J_{\fH}$ to $\overline{U}\times M$ is  intermittently bounded for every element $U\in C$. By this we mean that $J_{\fH}|_{\overline{U}\times M}$ has an extension to $\Sigma\times M$ which is intermittently bounded.
	%\end{itemize}
\item A $k$-parameter family $(\alpha,H,J)_{t\in[0,1]^k}$ of i-bounded Floer data on a fixed Riemann surface  is said to be \emph{uniformly i-bounded} if there is an $\epsilon>0$ such that for each $t_0\in [0,1]^k$ there is a cover $C$ of $\Sigma$ as in the previous item so that for each element $U$ of the cover , the taming data $(a,\{K_i,V_i\})$ for  $J_{\fH_t}|_{\overline{U}\times M}$  can be chosen fixed on the $\epsilon$ neighborhood of $t_0$.
\end{itemize}
\end{definition}
\begin{remark}
Note that the condition of intermittent boundedness of a monotone Floer datum $(\alpha,H,J)$ is weaker than the requirement of intermittent boundedness of the metric $g_{J_{\fH}}$. The latter condition would produce $C^0$ estimates for \emph{arbitrary} $J_{\fH}$-holomorphic curves whereas the given condition produces $C^0$ estimates only for $J_{\fH}$-holomorphic \emph{sections} of the projection $\Sigma\times M\to \Sigma$. We need the weaker condition for the existence of monotone continuation maps between arbitrary ordered pairs of Floer data.% {\color{red} I cannot understand the last sentence. What is arbitrary ordered data?} .
\end{remark}
%\begin{example}\label{exIboundTam}
%Suppose now $f:M\to\bR$ is the distance from some point $p\in M$ and that at each point $x\in M$ the metric $g_J$ is $f(p)$-bounded, then $g_J$ is still $i$-bounded. For this case consider the sequence of real numbers $b_i$ obtained from the set $\cup_{n=1}^{\infty}\{n+k/n|0\leq k<n\}\subset\bR$ with its standard order.  Then the sequence $(K_i=f^{-1}(0,b_{3i}), a_i=\lceil b_{3i}\rceil)$ constitutes taming data for $g_J$. Indeed, by assumption, the metric is $a_i$ bounded on $K_i$ and the series $\sum\frac1{a_i^2}$ is readily seen to diverge.

%\end{example}

Intermittently bounded Floer data satisfy the following two properties.
\begin{proposition}\label{prpMonHo}
Let $(J_1,H_1)$ and $(J_2,H_2)$ be intermittently bounded and suppose $H_1\leq H_2$. Then there exists a monotone intermittently bounded Floer datum $(\fH,J_z)$ on the cylinder $\bR\times S^1$ which equals $(H_1,J_1)$ for $s\ll 0$ and $(H_2,J_2)$ for $s\gg 0$. Moreover, any such pair of intermittently bounded monotone homotopies is connected  by a uniformly intermittently bounded family of monotone homotopies. %This holds also if one or both of the  homotopies is broken. %{\color{red} why just one?}
\end{proposition}
\begin{proof}
Fix two disjoint open sets $V_1,V_2\subset M$ such that there is taming data for $J_{H_i}$ supported in $\bR\times S^1\times V_i$ for $i=0,1$. For the existence of such disjoint sets see the proof of  Theorem 4.7 in \cite{groman}. We spell details for constructing a \emph{monotone} homotopy. We may assume that each of the $V_i$ is a disjoint union of pre-compact sets. Let $\chi: M\to[0,1]$ be a function which equals $0$ on $V_0$ and $1$ on $V_1$. Let $f:\bR\to\bR$ be a monotone function which is identically $0$ for $t\leq 0$ and identically $1$ for $t\geq 1/3$. Let $g:M\times \bR\to[0,1]$ be defined by
\[
g(x,s)= f(1-s)f(s)\chi(x)+1-f(1-s).
\]
Then $g$ is monotone increasing in $s$, identically $0$ for all $x$ when $s\leq0$ and identically $1$ for all $x$ when $s\geq 1$. Take $H_s=g(x,s)H_1+(1-g(x,s))H_0$. Then $H_s$ is also monotone increasing in $s$. Moreover, $H$ is fixed and equal to $H_0$ on $(-\infty,2/3]\times S^1 \times V_0$ and to $H_1$ on $[1/3,\infty)\times S^1\times V_1$. Let $J_s$ be any homotopy from $J_0$ to $J_1$ which is fixed and equal to $J_0$ on $(-\infty,2/3]\times S^1 \times V_0$   and to $J_1$ on $[1/3,\infty)\times S^1\times V_1$. Then $J_{H_s}$ is i-bounded as a monotone Floer datum since it coincides with $J_{H_0}$ on $[1/3,\infty)\times S^1\times V_1$ and with $J_{H_1}$ on $[1/3,\infty)\times S^1\times V_1$. Contractibility of the set of all such homotopies is similar.

\end{proof}

For the second property we first recall the definition of \emph{geometric energy}. For a solution $u$ to Floer's equation \eqref{eqFloer} we define the geometric energy
\begin{equation}
E_{geo}(u):=\int_{\Sigma}\left\|du-X_{J}\right\|^2.
\end{equation}
\begin{proposition}\label{prpC0IB}
Let $\Sigma$ be a connected Riemann surface with cylindrical ends and let $(H,\alpha,J)$ be an intermittently bounded monotone Floer datum on $\Sigma$. Then there is a metric $\mu:(\Sigma\times M)^2\to\bR_{\geq 0}$ and a constant $C$ such that the following holds
\begin{itemize}
\item A subset of $\Sigma\times M$ is bounded with respect to $\mu$ if and only if it is bounded with respect to $g_{J_{\fH}}$.
\item
Let $u:\Sigma\to M$ be a solution to Floer's equation and let  $p_1,p_2\in\Sigma.$ % let $\gamma$ be a path  in $\Sigma$ connecting $p_1$ and $p_2.$
Then
\begin{equation}
\mu(\tilde{u}(p_1)),(\tilde{u}(p_2)) \leq C\left(E_{geo}(u)+d(p_1,p_2)+1\right)%\int_{B_{1}(\gamma)}\omega_{\Sigma}\right).
\end{equation}
Here the distance $d(p_1,p_2)$ is with respect to the metric $g_{j_{\Sigma}}$.
%Here $B_{1}(\gamma)$ is the set of points in $\Sigma$ whose distance from $\gamma$ with respect to the metric $g_{j_{\Sigma}}$ is at most $1$.
\end{itemize}
Given a uniformly i.b. family of Floer data, the metric $\mu$ and the constant $C$ can be chosen to be fixed for the entire family.
\end{proposition}

\begin{proof}
By definition $g_{J_{\fH}}$ is intermittently bounded relative to $\pi_1$. Let $\cC$ be an open cover and for $U\in\cC$ let $K_1\subset K_2\dots\subset\Sigma\times M$ be an exhaustion as in the  definition of intermittently bounded relative to $\pi_1$. Let $a=\max_{U\in\cC} \{a_U\}$.

Given a path $\alpha$ in $\Sigma\times M$ let $N_a(\alpha)$, for any $a\geq 1,$ be  the maximal integer  so that there are $N$ distinct points $\{x_i\}$ in the image of $\alpha$ satisfying
\begin{itemize}
\item $x_i$ is $a$-bounded with witness supported in an open neighborhood $V_i$.
\item $i\neq j\Rightarrow V_i\cap V_j=\emptyset$.
\end{itemize}
For $x\neq y\in\Sigma\times M$  define
\begin{equation}
\mu_a(x,y):=\min_{\alpha}N_a(\alpha)+1+\lceil d(\pi_1(x),\pi_1(y))\rceil,
\end{equation}
where the minimum is over all paths in $\Sigma\times M$ \emph{which are lifts with respect to $\pi_1$ of a path in $\Sigma$}.  Call such paths \emph{admissible}. Then $\mu_a$ is readily seen to be a metric.  It is also clear that for $x\in\Sigma\times M$ and any real number $R$, the $\mu_a$-ball of radius $R$ is bounded with respect to the standard metric. Indeed, a bound on the distance $\mu_a$ implies for each admissible path $\gamma$ and each $U\in\cC$ a bound on the number of levels $\partial K_i$ transversed by the part the part of $\gamma$ in $\pi^{-1}(U)$. The claim follows by finiteness of the cover $\cC$. The converse is obvious.

Finally,  by monotonicity, for any $i$ we have
\begin{equation}
E(\tilde{u};\tilde{u}^{-1}(V_i))\geq\frac{c}{a^4}.
\end{equation}
But
\begin{equation}
E(\tilde{u};\tilde{u}^{-1}(V_i))=E_{geo}(u;\tilde{u}^{-1}(V_i))+\int_{\tilde{u}^{-1}(V_i)}\omega_{\Sigma}.
\end{equation}

Now $V_i$ can be taken to be the ball of radius $2/a$ with respect to some metric $g'$ which is $a$ equivalent to $g$. So, $V_i\subset B_1(x_i;g)$.  Since $\pi_1$ does not increase lengths, we have
\begin{equation}
\tilde{u}^{-1}(V_i)\subset \tilde{u}^{-1}( B_1(x_i;g))\subset B_{1}(\pi_1(x_i))
\end{equation}
where on the right we consider the geodesic ball with respect to the metric on $\Sigma$.  It follows that for a minimizing geodesic $\gamma$ in $\Sigma$ going from $p_1$ to $p_2$ we have
\begin{equation}
\mu(\tilde{u}(p_1),(\tilde{u}(p_2)) \leq  \frac{a^4}{c}\left(E_{geo}(u;B_{1}(\gamma))+\int_{B_{1}(\gamma)}\omega_{\Sigma}\right)+d(p_1,p_2)+1.
\end{equation}
Clearly,
\begin{equation}
\int_{B_{1}(\gamma)}\omega_{\Sigma}\leq C'( d(p_1,p_2)+1),
\end{equation}
for an appropriate constant $C'$ depending on the metric on $\Sigma$. The proposition follows by rearranging the last two inequalities.

%For $x\in\pi^{-1}(\overline{U})$  let $i_U(x)$ be the the maximal $i$ for which $x$ is not in the interior of $K_i$. Let $i(x)=\max_{U\in\cC,x\in\overline{U}}i_U(x)$.  Finally, let $f(x):=(\pi_1(x),i(x))$. Then $i$ is proper.

%Let $f_U:\Sigma\times M$ be defined by
%\begin{equation}
%f_U(x)=i(x)+\frac{d(x,K_i)}{d (K_i, \partial K_{i+1})},
%\end{equation}
%where the distance is measured with respect to $g^U_{J_{\fH}}$. Let $a_U$ be the constant such that $J^U_{\fH}$ is $a_U$ bounded near $K_i$ for each $i$. Take  $f(x)=\min_{U\in \cC}f_U(x)$. Then $f$ is clearly proper.
%To define the constant $C$, , let $b$ be the Lebesgue number of $\cC$ and take
%\begin{equation}
%C:=\min\left\{  \frac{a^2}{c},\frac{\pi}{b^2}\right\},
%\end{equation}
%where $c$ is the constant of Theorem \ref{lmMonEst++}. Call a point $x\in\Sigma\times M$ \emph{a}-bounded if the metric $g_{J_{\fH}}$ is $a$-bounded at $x$. Let $N$ be the maximal integer  so that there are $N$ distinct points $\{x_i\}$ in the image of $\gamma$ under $\tilde{u}$ satisfying
%\begin{itemize}
%\item $x_i$ is $a$-bounded.
%\item $i\neq j\Rightarrow d(x_i,x_j)\geq\frac2a$.
%\end{itemize}
%Then $N\geq \left |i(\tilde{u}(p_1))-i(\tilde{u}(p_2))\right|$.

\end{proof}
\subsection{Loopwise dissipativity}
There are two ways in which compactness may fail for the set of solutions to Floer's equation on an open manifold. One is the development of a singularity at some $z$ in the domain. This problem is taken care of by Proposition \ref{prpC0IB} provided we consider only Floer data for which the Gromov metric is intermittently bounded. The other way in which compactness may fail is by breaking at infinity. To take care of this we introduce the following notion.

We refer to a Floer solution on a domain of the form  $[a,b]\times \bR/\bZ$, with $a<b\in\bR$ as a \emph{partial Floer trajectory}. We do not take $a,b$ fixed in the following.

\begin{definition}\label{dfGamma}
 Let $F:M\to\bR$ be a function and $r_1<r_2$ be real numbers. For a smooth $H:\bR/\bZ\times M\to\bR$  and an $\omega$-compatible almost complex structure $J$  define $\Gamma^F_{H,J}(r_1,r_2)$ as the infimum over all $E$ for which there is a partial Floer trajectory   $u$ of geometric energy $E$  with one end of $u$ contained in $F^{-1}((-\infty,r_1))$ and the other end in $F^{-1}((r_2,\infty))$. Note that $\Gamma^F_{H,J}(r_1,r_2)$ may take the value of infinity. Let us also define $\Gamma_{H,J}(r_1,r_2):=\Gamma^H_{H,J}(r_1,r_2).$
\end{definition}
We say that $(H,J)$ is \emph{loopwise dissipative (LD)} if for some (hence any)  proper bounded below function $F$, we have $\Gamma^F_{H,J}(r_1,r)\to\infty$ as $r\to\infty$ for all $r_1$. We say that $(H,J)$ is \emph{robustly loopwise dissipative (RLD)} if  in the uniform $C^1\times C^0$ topology determined by $g_J$ there is an open neighborhood of the datum $(H,J)$ such that all elements are loopwise dissipative. 

% If this holds for some function $\Gamma:\bR_+\times\bR_+\to\bR_+\cup \{\infty\}$ satisfying $\Gamma^F_{H,J}\geq \Gamma$ we say that $(H,J)$ is $\Gamma$-LD. We say that $(H,J)$ is \emph{robustly loopwise dissipative (RLD)} if there is a function $\Gamma:\bR\times\bR\to\bR$ and an open neighborhood of $(H,J)$ in $C^1\times C^0$ all elements of which are $\Gamma$-LD. {\color{red} when are these two notions after just LD used in this paper?}

Finally, we can introduce the main definition of this appendix.

\begin{definition}\label{dfDissipative}
A Floer datum $(H,J)$ is \emph{dissipative} if it is satisfies
\begin{enumerate}
\item The Gromov metric $g_{J_H}$ is intermittently bounded.
\item $(H,J)$ is robustly loopwise dissipative.
\end{enumerate}
A monotone Floer datum $(\fH,J_z)$ on a Riemann surface $\Sigma$ is called \emph{dissipative} if the Gromov metric $g_{J_\fH}$ is intermittently bounded and for all ends $i$ we have that $(H_i,J_i)$ is dissipative.
A family of dissipative  Floer data  on a fixed Riemann surface and with fixed dissipative data on the ends is said to be \emph{ dissipative} if it is uniformly intermittently bounded.
\end{definition}

Before proceeding further we comment on the definition of dissipative data.
\begin{itemize}
\item  The set of dissipative Floer data is is closed under the action of symplectomorphisms. That is, if $(H,J)$ is dissipative, so is $(H\circ\psi,\psi^*J)$ for any symplectomorphism $\psi$.
\item It follows from the definition that if $(H,J)$ is dissipative, so is any datum which coincides with $(H,J)$ outside of a compact set up to the addition of a constant. Also, a sufficiently small perturbation of a dissipative datum in the uniform $C^1\times C^0$ topology determined by $g_J$  is dissipative.
%\item {\color{red} remove?}
%In \cite{groman} we restricted the definition of dissipativity to proper functions. However in the context of local $SH$ this is rather inconvenient as it is much preferable to work with bounded Hamiltonians which are dissipative. The existence of such data is non-trivial, but, as we shall show, true. In theorem  \ref{tmSummary} we summarize the main properties of dissipative Floer data with references for the proof in \cite{groman}. We indicate the main ideas of the proofs and adjustments that are necessary for  bounded data.
\item For the full Hamiltonian Floer theory package it is also necessary to introduce a more general notion of a dissipative family of Floer data which allows broken data. For a definition see \cite{groman}[Definition 5.6].%This involves considering family of Floer data with thick thin decompositions so as to allow broken Floer data. {\color{red} this needs to be expanded. maybe a reference could help? it is not possible to understand what it says.}

\end{itemize}

It is easy to produce intermittently bounded Floer data on a geometrically bounded manifold or to verify intermittent boundedness of a given Floer datum. The condition of loopwise dissipativity on the other hand refers to the space of solutions to a non-linear partial differential equation. Thus it is not a priori clear from the definition that dissipative data exist on a general geometrically bounded symplectic manifold.  In this regard we state the following general proposition that is proven in \cite{groman}.
\begin{theorem}[Existence of dissipative data] \label{tmSummary0}
If $J$ is geometrically bounded and $H$ is proper and has sufficiently small Lipschitz constant outside of a compact set then $(H,J)$ is dissipative.
\end{theorem}
\begin{proof}
This is Theorem 6.6 in \cite{groman}.
\end{proof}
\begin{itemize}
\item The theorem guarantees there is an abundance of dissipative Floer data on any geometrically bounded manifold. Indeed on such a manifold one can take the distance function to a point and perturb it slightly to obtain a proper smooth Lipschitz function, then scale it so the Lipschitz constant becomes small.

\item At first sight  Theorem \ref{tmSummary0} appears to be of limited value since it imposes severe restrictions on the growth of $H$ at infinity. However, as shown in \cite{groman}, by the continuity property of Floer cohomology we can use Hamiltonians of small Lipschitz constant to construct Floer cohomologies of all proper Hamiltonians.

\item The set of dissipative Floer data is  not limited to those guaranteed to by the first part of the theorem. In particular, the linear at infinity Hamiltonians  which are typically considered in the literature together with contact type almost complex structures are dissipative if the slope is not in the period spectrum. This is shown in \cite{groman}.
\end{itemize}

We conclude with the two properties of dissipative data that we will be using.

\begin{proposition}[A priori estimate]\label{tmSummary00}
A dissipative Floer datum  determines for every positive real number $E$ and any compact set $K\subset M$ a constant $R=R(E,K)$ such that the following holds. If $u$ is any solution to Floer's equation corresponding to the given datum, and $u$ has energy $\leq E$ and intersects $K$, then the image of $u$ is contained in the ball %{\color{red} The use of ``ball" around a compact set is non-standard. If you want to use it we should explain it above.}
$B_R(K)$ of radius $R$ around $K$. Moreover, the constant $R(E,K)$ does not change if the Floer datum is altered outside of the ball $B_R(K)$. Given a  dissipative family of Floer data, the constant $R(E,K)$ can be taken uniform in the family. %\end{enumerate}
\end{proposition}
\begin{proof}
For $H$ proper this is Theorem 6.3 in \cite{groman}. The proof does not really change when $H$ is not assumed proper. We briefly spell this out for the case of a single Floer datum. For such a datum $\Sigma$ there is a connected compact Riemann surface $\Sigma_0\subset\Sigma$ with boundary, such that $\Sigma\setminus\Sigma_0$ consists of ends. First observe there is an a priori compact set $K'$ determined by $K$ and $E$ such that $u(\Sigma_0)$ meets $K'$. Indeed, either $\Sigma_0$  meets $K$ in which case there is nothing to prove, or one of the ends meets $K$ and the claim follows by loopwise dissipativity since each end shares a boundary with $\Sigma_0$. By Proposition \ref{prpC0IB} we have an a priori estimate on the $\mu$-diameter of $\tilde{u}(\Sigma_0)$ as well as on the $\mu$-diameter of $\tilde{u}(s,\cdot)$ where $s$ is in any of the ends. It follows by loopwise dissipativity again that there is a compact set $K''$ determined a priori by $K,E$ such that for all $s$,   $\tilde{u}(s,\cdot)$ intersects $K''$. Applying Proposition \ref{prpC0IB} again we obtain an estimate as desired.

%an a priori compact set $ Combined with loop-wise dissipativity of the ends we can find a priori compact set $K'$ such that $\tilde{u}(\Sigma_0)$ as well as each of the ends meets $K'$. Applying loopwsie dissipativity again

%For the sake of self containment we spell out the details. By assumption there is a proper function $F$ such that
%\begin{equation}
%\lim_{h_2\to\infty}\Gamma^{F}_{H,J}(h_1,h_2)=\infty
%\end{equation}
%for each fixed value of $h_1$. Let $s_0\in\bR$ such that for some $t_0\in S^1$ we have $u(s_0,t_0)\in K$. Let $f=\max_tF\circ u(s_0,\cdot)$. Then it follows from Proposition \ref{prpC0IB} that there is an a priori constant $h_1(E,K)$ (depending on the function $F$) such that $f\leq h_1$. %For a proof of this see Lemma \ref{lmDiamEstFreeBo} below taking $S=[s_0-1,s_0+1]\times S^1$.
%Let $h_2=h_2(h_1)$ such that $\Gamma^{F}_{H,J}(h_1,h_2)>E$. Then, by definition of the function $\Gamma$, for any $s\in\bR$ we have that $u(s,\cdot)$ meets the compact subset $F^{-1}(-\infty, h_2)$. Reasoning as before, there an a priori constant $h_3=h_3(h_2,E)$ such that for each $s\in\bR$ we have $\max_tF\circ u(s,\cdot)<h_3.$
\end{proof}

%\subsection{dissipative Floer data and $C^0$ estimates}
%In the following theorem we summarize the results from \cite{groman} that are relevant for general definitions of local Floer homology. Later we discuss some more refined statements which are required for the locality isomorphism. Let $\Sigma$ be punctured Riemann surface and let $(H,J)$ be a be Floer datum on $\Sigma$.
%The following Theorem summarizes the results from \cite{groman}  that are necessary for the basic definitions of Floer cohomology on an open manifold.

%From the definition of dissipative Floer data, it is not all clear that such data exist on a general geometrically bounded symplectic manifold. While the condition of intermittent boundedness is easily satisfied, the condition of loopwise dissipativity refers to the space of solutions to a non-linear partial differential equation. In this regard we state the following general theorem that is proven in \cite{groman}.

%We can now state the following theorem summarizing the results from \cite{groman} that are relevant for general definitions of local symplectic cohomology.

%{\color{red} I think this statement does not belong here. It should be moved to the previous section.}
\begin{proposition}\label{tmSummary}

%\item \label{tmSummaryIt2}
If $H_1\leq H_2$ and $(H_i,J_i)$ are dissipative, then there exists a monotone dissipative homotopy from $(H_1,J_1)\to (H_2,J_2)$. Moreover, any such pair of dissipative homotopies are connected by a dissipative family of monotone homotopies. This holds also if one or both of the homotopies is broken.
%\item\label{tmSummaryIt3}
\end{proposition}
\begin{proof}
This follows from  Proposition \ref{prpMonHo} by Definition \ref{dfDissipative}.
\end{proof}

\bibliographystyle{plain}
\bibliography{../locality} 
\end{document}